\documentclass[11pt]{amsart}

\usepackage[colorlinks=true,allcolors=blue!70!black,bookmarksopen=true]{hyperref}

\usepackage{url}
\usepackage{amsmath}
\usepackage{amssymb}
\usepackage{amscd}
\usepackage{manfnt}
\usepackage{enumitem} % MB. Changed from enumerate, which sucks
\usepackage[all]{xy}
\usepackage{graphicx}
\usepackage{MnSymbol}
\usepackage{ulem} % for strikethrough
\usepackage[paper=a4paper, text={155mm,218mm},centering]{geometry}
\usepackage{tikz}
\usetikzlibrary{decorations.markings}
\usetikzlibrary{decorations.pathreplacing}
\usetikzlibrary{intersections,patterns}
\usepackage{pgfplots}
\usepackage{standalone}

\usepackage{comment}
\usepackage[colorinlistoftodos]{todonotes}

\numberwithin{section}{part}
\newtheorem{theorem}{Theorem}[section]
\newtheorem*{theorem*}{Theorem}
\newtheorem{proposition}[theorem]{Proposition}
\newtheorem{lemma}[theorem]{Lemma}
\newtheorem{corollary}[theorem]{Corollary}

\newtheorem{addendum}[theorem]{Addendum}

\theoremstyle{definition}
\newtheorem{definition}[theorem]{Definition}
\newtheorem{example}[theorem]{Example}

\newtheorem{convention}[theorem]{Convention}
\newtheorem{condition}[theorem]{Condition}
\newtheorem{cond}[theorem]{Condition}

\theoremstyle{remark}
\newtheorem{remark}[theorem]{Remark}
\newtheorem*{remark*}{Remark}
\numberwithin{equation}{section}

\newcommand{\ol}[1]{\overline{#1}}

\newcommand{\dx}[1]{\frac{\partial\phantom{x}}{\partial x_{#1}}}
\newcommand{\dy}[1]{\frac{\partial\phantom{y}}{\partial y_{#1}}}

\newcommand{\eg}{\varepsilon_\gamma}
\newcommand{\toiso}{\xrightarrow{\simeq}}

\newcommand{\Diff}{\operatorname{Diff}}
\newcommand{\Emb}{\operatorname{Emb}}
\newcommand{\Image}{\operatorname{Image}}

\newcommand{\x}{\pmb{x}}
\newcommand{\y}{\pmb{y}}
\newcommand{\hra}{\hookrightarrow}

\newcommand{\imra}{\looparrowright}
\newcommand{\ra}{\longrightarrow}

\newcommand{\smfrac}[2]{\mbox{\footnotesize$\displaystyle\frac{#1}{#2}$}} % small medium frac
\newcommand{\tmfrac}[2]{\mbox{\large$\frac{#1}{#2}$}} % tiny medium frac

\def\Str{\mathcal{S}}
\def\Z{\mathbb Z}
\def\R{\mathbb R}
\def\Q{\mathbb Q}
\def\N{\mathbb N}
\def\O{\Omega}
\def\lbirth{\mathfrak{r}}
\def\birth{\mathfrak{u}}
\def\ftheta{\mathfrak{s}}
\def\fa{\mathfrak{a}}
\def\fb{\mathfrak{b}}
\def\wt#1{\widetilde{#1}}
\def\wh#1{\widehat{#1}}

\def\H{\mathbb{M}}
\def\Ha{\H_{\mathrm{a}}}
\def\Hd{\H_{\mathrm{d}}}

\def\cF{\mathcal{F}}

\def\cA{\mathcal{A}}
\def\cR{\mathcal{R}}
\def\cD{\mathcal{D}}
\def\cI{\mathcal{I}}
\def\cK{\mathcal{K}}
\def\cL{\mathcal{L}}
\def\cP{\mathcal{P}}
\def\cN{\mathcal{N}}

\def\mj{\mathfrak{j}}
\def\xmo{\xi_{2:n}}
\def\xtan{\xi_{\ttan}}
\def\wxtan{\wt{\xi}_{\ttan}}

\def\Ttan{E_{\ttan}}
\def\UUU{\Upsilon}

\def\sm{\setminus}

\def\Mtwo{M_{p_+}} % the part that is

\def\Mone{M_{\textrm{\upshape orig}}} % the part that is replaced
\def\Mfng{M_{\textrm{\upshape fng}}} % the new part
\def\Mnew{M_{\textrm{\upshape new}}} % everything
\def\sh{t_{\textrm{\upshape hit}}}
\def\efng{\varepsilon_{\textrm{\upshape fng}}}

\def\esafe{\varepsilon_{\textrm{\upshape safe}}}
\def\evs{\varepsilon_{\textrm{\upshape vs}}}
\def\feta{\phi_{\lambda}}
\def\Csafe{C_{\textrm{\upshape safe}}}
\def\Cvs{C_{\textrm{\upshape vs}}}
\def\rhofng{\rho_{\textrm{\upshape fng}}}
\def\eeta{\varepsilon_R}
\def\erho{\varepsilon_\Sigma}
\def\fng{\textrm{\upshape fng}}
\def\oout{\textrm{\upshape out}}
\def\iinn{\textrm{\upshape inn}}
\def\iin{\textrm{\upshape in}}
\def\mmid{\textrm{\upshape mid}}
\def\bbot{\textrm{\upshape bot}}
\def\ttop{\textrm{\upshape top}}
\def\ttan{\textrm{\upshape tan}}
\def\bbad{\textrm{\upshape bad}}
\def\eend{\textrm{\upshape end}}
\def\ccoor{\textrm{\upshape coor}}
\def\ttang{\textrm{\upshape tang}}
\def\piin{\partial_{\iin}}
\def\pout{\partial_{\oout}}
\def\ptan{\partial_{\tan}}
\def\xit{\overline{\xi}}
\def\xih{\widehat{\xi}}
\def\xitt#1{\wt{\xi}_{#1}}
\def\etat{\overline{\eta}}

\def\guiding{\varpi}
\def\auxfield{\upsilon}

\DeclareMathOperator\codim{codim}

\DeclareMathOperator\ind{ind}

\DeclareMathOperator\sspan{span}
\DeclareMathOperator\Int{Int}
\DeclareMathOperator\Id{Id}
\DeclareMathOperator\Iim{Im}

\DeclareMathOperator\bd{bd}
\DeclareMathOperator\pr{pr}

\definecolor{linkcolor0}{rgb}{0.45, 0.15, 0.15}
\definecolor{linkcolor1}{rgb}{0.15, 0.15, 0.45}
\definecolor{c2626d8}{RGB}{38,38,116}
\definecolor{c00ffff}{RGB}{0,255,255}
\definecolor{c0000db}{RGB}{0,0,219}
\definecolor{ce7a0ff}{RGB}{231,160,255}
\definecolor{c650000}{RGB}{101,0,0}
\definecolor{c009100}{RGB}{0,145,0}
\definecolor{cd82626}{RGB}{116,38,38}
\definecolor{ce6f815}{RGB}{30,248,110}

\newcounter{nparcount}
\def\npar#1#2{\stepcounter{nparcount}%
  \colorbox{green!10}{$?^{\thenparcount}$}\marginpar{\colorbox{green!10}{\begin{minipage}{2cm}\fontsize{5}{6}\selectfont\color{green!20!black}${}^{\arabic{nparcount}}$#1\end{minipage}}\\ \colorbox{blue!10}{\begin{minipage}{2cm}\fontsize{5}{6}\selectfont\color{green!20!black}
${}^{\arabic{nparcount}}$#2\end{minipage}}}}
\title{Link concordance implies link homotopy}

\author{Maciej Borodzik}
\address{Institute of Mathematics, University of Warsaw, ul. Banacha 2,
02-097 Warsaw, Poland}
\email{mcboro@mimuw.edu.pl}

\author{Mark Powell}
\address{School of Mathematics and Statistics, University of Glasgow, University Place, Glasgow, G12 8QQ, United Kingdom}
\email{mark.powell@glasgow.ac.uk}

\author{Peter Teichner}
\address{Max Planck Institut f\"ur Mathematik, Vivatsgasse 7, 53111 Bonn, Germany}
\email{teichner@mpim-bonn.mpg.de}

\makeatletter
\@namedef{subjclassname@2010}{\textup{2010} Mathematics Subject Classification}
\makeatother
 %, secondary: }

\begin{document}

\begin{abstract}
We show that link concordance implies link homotopy for immersions of codimension at least two. As a consequence, we prove that every link $\sqcup^r S^n \hra S^{n+2}$ is link homotopically trivial for $n\geq 2$. This means that there is a {\em link} homotopy that (for each time parameter) maps distinct component $n$-spheres disjointly into $S^{n+2}$. In other words, beyond the classical dimension (of embedded circles in $S^3$) there are no `linking modulo knotting' phenomena in codimension two. To date, this was only known for $n=2$.

In our proofs we follow, expand, and complete unpublished notes of the third author developing stratified Morse theory for generic immersions, where the $d$-th stratum is  given by points that have $d$ preimages under the generic immersion. We  include a discussion of gradient like vector fields, their strata preserving flows, and Cerf theory. This generalizes the case of embeddings, having only two strata, as studied by Perron, Sharpe, and Rourke,  and expanded by the first two authors.

Two vital operations in Cerf theory are the rearrangement and cancellation of critical points. In the setting of stratified theory, there are additional rearrangement and cancellation obstructions arising from intersections of ascending and descending membranes for critical points of the Morse function restricted to various strata.

 We show that both additional obstructions vanish in codimension at least three, implying a smooth proof of Hudson's result that embedded concordance implies isotopy.
In codimension at least two, we show that only the rearrangement obstruction vanishes and we introduce finger moves that eliminate the cancellation obstruction. This is done carefully and only at the expense of introducing new self-intersection points into the components of the immersion. Therefore, our moves keep distinct components disjoint and hence preserve the link homotopy class.
\end{abstract}

\maketitle

\newpage

\setcounter{tocdepth}{1}
\tableofcontents

\part{Main results and outlines of proofs.}\label{part:intro}
%\chapter{Introduction}

\section{Link maps, link homotopy, and link concordance} \label{sec:link maps}

A continuous map $c \colon X \to Y$ is called a {\em link map} if it keeps distinct connected components of $X$ disjoint in $Y$. In other words, if and only if the induced map $\pi_0(X) \to \pi_0(\Image(c))$ is a bijection. A {\em link homotopy} is a homotopy through link maps.
A {\em link concordance} is a link map $C \colon X \times [0,1] \to Y \times [0,1]$ such that $C^{-1}(Y \times  \{i\}) = X \times \{i\}$ for $i=0,1$.

This notion of link concordance is an analogue for link maps of the usual notion of embedded link concordance of embedded links. However note that when restricted to embedded links, our notion of link concordance does not give the notion of embedded link concordance one finds in the literature, since $C$ need not be an embedding, even if $C|_{X \times \{0,1\}}$ is an embedding.
%\ypar{New paragraph alerting the reader to the non-standard usage of `link concordance'.}

We focus initially on the case that $Y$ is a sphere $S^N$ and $X$ is a disjoint union of spheres $S^{n_i}$ with $n_i<N$. 

Milnor~\cite{Milnor:1954-1,Milnor:1957-1} introduced link homotopies for classical links $L \colon \sqcup^r S^1 \hra S^3$, to distinguish the phenomenon of `knotting' from that of `linking'. The Hopf link (formed by two unknots) is not link homotopic to an unlink, as detected by its linking number.  Milnor's invariant $\ol{\mu}_{123}$ shows that the Borromean rings are also link homotopically essential. 

The fact that Milnor's invariants are also concordance invariants led to the question of whether link concordance implies link homotopy.
This was proven in 1979 in the classical dimension by Giffen~\cite{Giffen-1979-1} and Goldsmith~\cite{Goldsmith-1979-1} (there is also a later account by Habegger~\cite{Habegger-1992-1}).  We extend this result to all dimensions, provided the codimension is at least two.

\begin{theorem}\label{thm:link concordance}
For any dimensions $n_i \leq N-2$, two link maps $L, L' \colon \sqcup_{i=1}^r S^{n_i} \to S^N$ are link homotopic if and only if they are  link concordant.
\end{theorem}

In his PhD thesis \cite{Bartels_thesis}, Bartels showed that in codimension two and for $n \geq 2$, every embedded link $L \colon \sqcup^r S^{n} \hra S^{n+2}$ is link concordant to the unlink. Prior to this, Massey and Rolfsen~\cite{Massey-Rolfsen} had introduced some high dimensional link homotopy invariants and asked whether link concordance implies link homotopy, and also whether every embedding $S^2\sqcup S^2 \hra S^4$ is link homotopically trivial. Our paper shows that both questions have an affirmative answer. Indeed the latter holds much more generally: by combining Theorem~\ref{thm:link concordance} with~\cite{Bartels_thesis} we obtain the following.

\begin{corollary}\label{cor:embedded links}
Every smooth embedding $\sqcup^r S^{n} \hra S^{n+2}, n\geq 2,$ is link homotopically trivial.
\end{corollary}

%This follows from .
%, where he showed that in codimension two, every embedded link $L \colon \sqcup^r S^{n} \hra S^{n+2}, n\geq 2,$ is link concordant to the unlink.
%repetition so commented out.

Corollary~\ref{cor:embedded links} was proven in \cite{Bartels-Teichner-1999-1} for $n=2$, where some of the techniques of this paper were introduced. Later \cite{ST-group-2spheres} showed by different methods that a link map $S^2\sqcup S^2 \to S^4$ with one (topologically) embedded component is link homotopically trivial.  Fenn and Rolfsen constructed the first nontrivial link map, exhibited in \cite[Figure~4]{Fenn-Rolfsen}. It has one self-intersection of each component.
The main result of \cite{ST-group-2spheres} showed that the group (under connected sum) of link maps $S^2\sqcup S^2 \to S^4$
is a free module over the commutative ring $\Z[x_1,x_2]/(x_1\cdot x_2)$, freely generated by the Fenn-Rolfsen link map.

We end our link homotopy discussion for embedded spherical links by looking at the codimension one case.
%However for the case of spheres in Euclidean space a stronger result holds, as we discuss next.
The Schoenflies theorem holds for $n\neq 3$; see~\cite[Proposition~D,~p.~112]{Mil65} for $n \geq 4$, implying that a smooth codimension one submanifold $K \subseteq \R^{n+1}$ that is diffeomorphic to $S^n$ is ambiently isotopic to the round $n$-sphere $S^n\subseteq \R^{n+1}$.  We shall refer to $K$ as a (unparametrized) {\em knot} and define a codimension one link $L\subseteq\R^{n+1}$ to be an {\em ordered} sequence $(K_1, \dots, K_r)$ of pairwise disjoint codimension one knots.
From the Schoenflies theorem it follows that such links are classified up to isotopy by their dual tree.

\begin{definition}\label{def:dual}
The {\em dual tree} $t(L)$ of a codimension one link $(K_1, \dots, K_r)=L\subseteq\R^{n+1}$ has $\pi_0(\R^{n+1}\smallsetminus L)$ as set of vertices, with a specified {\em root vertex} for the unbounded component.  For each component $K_i$ of $L$ there is a unique edge $e_i$ in $t(L)$ whose boundary consists of the two connected components of $\R^{n+1}\smallsetminus L$ whose closures meet $K_i$.
 \end{definition}

\begin{proposition}\label{prop:cod1}
For $n\neq 3$, two smooth codimension one links $L, L' \subseteq\R^{n+1}$ are ambiently isotopic if and only if they are link concordant.
Moreover, a generically immersed link concordance induces an isomorphism on dual trees, and if there is an isomorphism $t(L)\cong t(L')$ of rooted, edge-ordered trees, then it is unique and is induced by an ambient isotopy from $L$ to $L'$.
%\ypar{Changed the statement here as per email.}
%any isomorphism $t(L)\cong t(L')$ of rooted, edge-ordered trees is induced by an ambient isotopy from $L$ to $L'$.
\end{proposition}

The notion of a link concordance between unparametrized links is discussed in Definition~\ref{def:unparametrized}. Remark~\ref{rem:oriented edges} contains our result on parametrised links modulo link concordance, classified by $t(L)$ with oriented internal edges. We could not find results on link concordance in codimension one in the literature, so we prove them in Section~\ref{sec:cod1}, where we also discuss {\em parametrised} knots $S^n\hra \R^{n+1}$ and links, and their relation to exotic smooth structures on~$S^{n+1}$.

The classification of {\em unordered} codimension one links  follows from Proposition~\ref{prop:cod1}, by allowing non-order preserving isomorphisms of rooted tree.  Proposition~\ref{prop:cod1} is stronger, namely it implies that there is an ambient self-isotopy of $L$ that permutes its components if and only if the rooted tree $t(L)$ admits a self-isomorphisms with the given permutation of edges. To see this, apply Proposition~\ref{prop:cod1} with $L'$ the same link as $L$ but with a different ordering, so that the self-isomorphism of the rooted tree becomes an edge-order preserving isomorphism $t(L)\cong t(L')$.

\begin{figure}[ht]
  \begin{tikzpicture}
    \draw(0,0) node [yscale=-1] {\includegraphics[width=9cm]{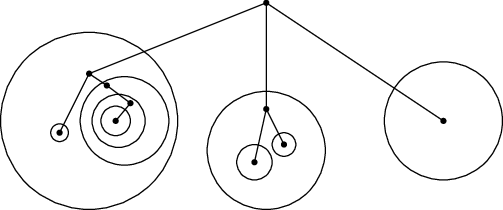}};
    \draw(0.2,-2.1) node  {root};
\end{tikzpicture}
\caption{Any rooted tree arises as dual tree in dimensions $n\geq 1$.}
\label{fig:dual tree}
\end{figure}
%\npar{Could we draw the root at the bottom? Like in any honest tree...}{That was meant to be an Australian tree, that grows upside-down, as we all know.}

Note the subtle distinction in Proposition~\ref{prop:cod1} between the first statement `link concordance implies isotopy', the most surprising part, and the second statement ``a {\em generically immersed} link concordance {\em induces} an isomorphism''. Figure~\ref{fig:need immersion} shows an example where a (non-generically immersed) link concordance does not induce a tree isomorphism (even though the trees are isomorphic).

It follows that approximating immersions are a very useful tool in the study of link homotopy, even in codimension one. Their existence is guaranteed by the following well-known result, which holds in all codimensions at least one, and that we shall apply in the rest of the paper in codimension at least two.

\begin{proposition} \label{prop:immersed}
Assuming $n_i \leq N-1$ for all $i$, a link map $L \colon \sqcup_{i=1}^r S^{n_i} \to S^N$ is link homotopic to an immersion. Similarly, a link concordance $(\sqcup_{i=1}^r S^{n_i}) \times [0,1] \to S^N \times [0,1]$ that restricts to an immersion on the boundary is link homotopic $($rel.\ boundary$)$ to an immersion.
\end{proposition}

\begin{proof}
Note that for compact domains, link maps are open in the space of all continuous maps. Hence a smooth approximation (in the compact-open topology) will stay a link map (and not map onto $S^N$, i.e.\ lie in the parallelisable manifold $\R^N$). Then one shows that $L$ is covered by a monomorphism of tangent bundles, and so by the $h$-principle is thus homotopic to an immersion.
To see that $L$ is covered by a bundle map, fix an inclusion map $\iota_i \colon \R^{n_i+1} \to \R^N$, and let $L_i \colon S^{n_i} \to \R^N$ be the $i$th component of $L$. Then the composition
\[TS^{n_i} \to TS^{n_i} \oplus \varepsilon \cong S^{n_i} \times \R^{n_i +1} \xrightarrow{(L_i, \iota_i)} \R^N \times \R^N \cong  T\R^N\]
covers $L_i$ and is fibrewise injective, as desired. %\ypar{MP: added a proof here.}
By \cite[Thm.5.10]{Hirsch-immersions}, this codimension $\geq 1$ immersion can be chosen arbitrarily close to $L$ and is thus link homotopic to $L$. All  arguments have relative versions.
\end{proof}

%The Schoenflies Theorem holds in the topological category for all $n$ and the same discussion applies to show that link homotopy implies topological isotopy for (locally flat) topological embeddings $L \colon  \sqcup^r S^{n} \hra S^{n+1}$, up to some reflections.

\section{Link concordance implies link homotopy for immersions}

The proof of Theorem~\ref{thm:link concordance} consists of a number of steps. In the first step, we use Proposition~\ref{prop:immersed} to improve our link concordance to an immersion.
In the second step, we apply the following main result of this paper that holds for all immersions. It was announced in \cite{Pete_hab} where the strategy of our current proof was outlined but not completed.

\begin{theorem}\label{thm:concordance immersion}
Fix a closed $(n-1)$-manifold $A$ and a compact $(n-1+k)$-manifold $Y$, both smooth.
 Consider a link concordance $G\colon A\times[0,1]\to Y\times[0,1]$ that is also an immersion, such that $G|_{A\times\{0\}}=g_0\times\{0\}$, $G|_{A\times\{1\}}=g_1\times\{1\}$ with $($link$)$ immersions $g_i \colon A\to Y$.

If the codimension $k \geq 2$, then there is a regular link homotopy $G_\tau$ $($rel.\ boundary$)$ with $G_0=G$ and such that $G_1$ is a level-preserving link immersion.   In particular $g_0$ and $g_1$ are regularly link homotopic.

Moreover, if the codimension $k \geq 3$ and $G$ is an embedding, then $G_\tau$ may be chosen to be an embedding for all $\tau \in [0,1]$. In particular $g_0$ and $g_1$ are ambiently isotopic.
\end{theorem}

\begin{remark}\label{rem:h-principle} Let $V_n(n+k)$ be the Stiefel manifold of orthonormal $n$-frames in $\R^{n+k}$. It follows from the $h$-principle and the fibre bundle (with projection given by the last vector)
\[
V_{n-1}(n-1+k) \ra V_n(n+k) \ra S^{n-1+k}
\]
that for $k\geq 2$, a concordance immersion is regularly homotopic (rel.\ boundary) to a level preserving immersion, i.e.\ to the track of a regular homotopy. However, we see no argument along these lines that would imply that the outcome is a link homotopy, i.e.\ that disjoint components would be kept disjoint during the regular homotopy that comes from the $h$-principle. To obtain the essential output in Theorem~\ref{thm:concordance immersion}, namely a link homotopy, we have to add more geometric techniques.
\end{remark}

\begin{figure}[h]
  \begin{tikzpicture}
    \draw(-6,0) node {\includegraphics[width=2.5cm]{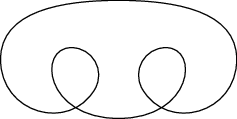}};
    \draw(-2.5,0) node {\includegraphics[width=2.5cm]{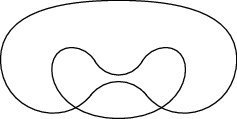}};
    \draw(1,0) node {\includegraphics[width=2.5cm]{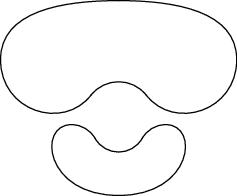}};
    \draw(4.5,0) ellipse (1.3cm and 0.7cm);
   \end{tikzpicture}
\caption{Immersed concordance does not imply regular homotopy in codimension one. Going from left to right: the singular link $S^1 \looparrowright \R^2$ to the left is transformed
by a saddle move to obtain a singular $2$-component link. Then, a Whitney move transforms the singular link to a $2$-component unlink. Finally,
a minimum cancels one of the components. The trace of these moves provides a singular concordance between the singular link on the left and
the unknot on the right.  The winding numbers of the immersions are different so they are not regularly homotopic. }
\label{fig:cod1}
\end{figure}

Figures~\ref{fig:cod1} and~\ref{fig:trefoil} imply that our result  is the best possible in various ways.
The example in Figure~\ref{fig:cod1} shows that Theorem~\ref{thm:concordance immersion} is not true in codimension one, even with only one component. We use that the winding number in $\Z$ is a regular homotopy invariant of immersions $S^n \imra \R^{n+1}$, which for $n=1$ is the Whitney-Graustein Theorem~\cite{Whitney-closed-curves}. The example shows that the winding number can change along an immersed concordance (even though it is invariant modulo~2), and hence immmersed concordance does not imply regular homotopy in general.
We do not know whether link concordance implies (non-regular) link homotopy for codimension one (non-spherical) link maps.
%\ypar{Added this sentence here.}
%\ypar{Added ref to Whitney's paper.}

Figure~\ref{fig:trefoil} shows that in codimension two there exist concordances that are not isotopic to the trace of any isotopy. Even though the minimum and maximum of the trefoil knot cancel in the domain 1-manifold, they cannot be cancelled ambiently, otherwise the trefoil knot would be trivial (as obstructed for example by the Alexander polynomial or Arf invariant).
In this case the points at the top and bottom are still isotopic, via a different, simpler isotopy.
However one dimension higher there are codimension two examples showing that concordance does not imply isotopy.
In Figure~\ref{fig:stevedore} below we note that the stevedore knot is concordant to the unknot, but is not isotopic to it. Again, the last claim can be seen using the Alexander polynomial.
%\ypar{MP: I rewrote this because the point is still isotopic to the point, even if the long trefoil is a non-standard concordance. }

\begin{figure}[h]
 \includegraphics[height=4cm]{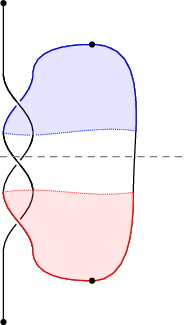}
 \caption{A nontrivial cancellation obstruction in codimension two. The intersection of membranes with the planar level presented as a horizontal dashed line is shown in Figure~\ref{fig:trefoil in the plane}.}
\label{fig:trefoil}
\end{figure}
%\mpar{I would remove the black semicircle on the right and straighten the figure so it's a trefoil arc, from bottom to top (aka the easiest nontrivial concordance).}

Theorem~\ref{thm:concordance immersion} nevertheless has the following consequence in dimension three.
\begin{proposition}\label{prop:braids}
Every classical link $\sqcup^r S^1\hra S^3$ is link
homotopic to the closure of a braid with $r$ strands.
\end{proposition}

This result was first proven in \cite{Habegger-Lin:1990-1} by an inductive procedure
using Milnor's classification of almost trivial links in terms of a certain
quotient of the fundamental group \cite{Milnor:1954-1}. Our proof is an alternative direct
geometric argument.
We prove Proposition~\ref{prop:braids} in Section~\ref{sec:braid}.
It may help the reader of the proof of Theorem~\ref{thm:concordance} to first understand the low dimensional case, which can be more easily visualised.

For the proof of Theorem~\ref{thm:concordance immersion} we develop  Morse theory on stratifications that come from generic immersions, including gradient-like vector fields, their flows and membranes. This  technical core of our paper is described in the next section.

\section{Morse immersions -- an outline of the proof of Theorem~\ref{thm:concordance immersion}}\label{sub:immersed_morse_intro}

Let $N$ and $\O$ be smooth manifolds, $F \colon \O \to \R$ a Morse function  and $G\colon N\imra\O$ a generic immersion.
A \emph{Morse immersion} is a pair $(F,G)$ that satisfies suitable compatibility conditions between $F$ and $G$, given in detail in Definition~\ref{def:immersed_morse_function}. Roughly speaking, we require that the restriction of $F$ to each $G$-multiple point stratum
\[
\O[d] := \{x \in \O \, \colon \,  |G^{-1}(x)| = d\}, \quad d\geq 0
 \]
is again a Morse function. In case the generic immersion $G$ is fixed throughout some argument, we also refer to $F$ as an {\em immersed Morse function} (for $G$).

\begin{figure}[h]
 \includegraphics[width=5cm]{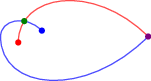}
\caption{Intersecting membranes in a planar level for the minimum and maximum in Figure~\ref{fig:trefoil}.}
\label{fig:trefoil in the plane}
\end{figure}

As with classical Morse theory, the key is to understand the behaviour near critical points, how critical points confer topological information, and how Morse immersions can be manipulated in order to cancel or rearrange critical points. This sometimes means changing $F$ alone, sometimes $G$ alone and in some cases, both maps have to be changed.

The main tool for studying Morse immersions is a suitable generalisation of a gradient-like vector field, a so-called gradient-like immersed vector field, in short
a \emph{grim} vector field. Roughly, a vector field $\xi$ on $\O$ is {\em grim} for a  Morse immersion $(F,G)$ if $\partial_\xi F\ge 0$, with equality
only at critical points of a restriction $F|_{\O[d]}$, and $\xi$ is tangent to all the strata $\Omega[d]$. A more precise definition involves local coordinates for the pair $(F,\xi)$
near critical points of the restrictions $F|_{\O[d]}$; see Definition~\ref{def:grim}. This characterisation was introduced in \cite{Pete_hab} and is a direct generalisation of
the definition of an embedded gradient like vector field due to Sharpe \cite{Sha}; see also \cite{BP}.

Grim vector fields are not linear near critical points of $F|_{\O[d]}$ for $d>0$. Therefore the notions of stable (descending)
and unstable (ascending) manifolds of a critical point require revision. This leads to the construction of descending and ascending \emph{membranes} in Section~\ref{sec:hypermembranes}, which are manifolds with corners, described via a product formalism in \cite{Pete_hab} but not here.
These membranes were used in \cite{Bartels-Teichner-1999-1} if $\dim(Q)=5$ and generalise those studied by Rourke \cite{Ro}, Perron \cite{Pe}, Sharpe \cite{Sha}, and the first two authors \cite{BP} in the context of embedded manifolds.

One of the key properties of  membranes is as follows. Suppose $p\in \O[d]$ is a critical point of $F|_{\O[d]}$. The stable and unstable manifolds of $\xi$ for $p$ can be defined on $\O[d]$.
The ascending and descending membranes have strictly higher dimension and correspond to trajectories limiting to $p$ as $t\to \mp \infty$. The membranes associated with a point on $\O[d]$ do not lie entirely within $\O[d]$ unless $d=0$. The intersections of membranes of different critical points away from $\O[d]$ give rise to an obstruction
to rearrangement or cancellation that is not seen in classical Morse theory (see \cite{BP} for a detailed discussion of the embedded case).
From this it follows that it might not be possible to rearrange (respectively cancel) two critical points $p,q$ of $F|_{\O[d]}$, that could be rearranged (respectively cancelled) in the classical setting. This obstruction gives a Morse--theoretical explanation of the existence of non-ambiently isotopic embeddings and immersions. For example, Figure~\ref{fig:trefoil in the plane} explains from this point of view why the trefoil knot is nontrivial.

To be more precise, a careful dimension counting argument combined with a generalised Morse--Smale condition shows that these extra obstructions for rearrangements of critical points are void if the codimension is at least two. In codimension at least three, there are also no obstructions for cancelling a pair of immersed critical points that would be cancellable in the \ non-immersed case. In the present paper, our primary focus is on codimension two, where rearrangement is possible, but cancellation
can be obstructed. %This contains link theory of embeddings $\sqcup S^n\hra S^{n+2}$.

The main result of this paper implements the idea that these cancellation obstructions can be removed at the expense of adding self-intersections into the components of the domain of $G \colon N\imra \O$.
Suppose the codimension is $2$ and $p,q$ are critical points of $F|_{\O[1]}$. Assume they can be cancelled in this stratum, i.e.\ there is precisely one trajectory on $\O[1]$
of a Morse--Smale gradient-like vector field $\xi$ that connects $p$ and $q$. As mentioned above, there might be other trajectories of $\xi$  in $\O$ connecting $p$ and $q$. A dimension counting argument and compactness shows that there can be only finitely many such trajectories.
Our main technical tool is the introduction of (a sequence of) \emph{finger moves} that decrease the number of trajectories from $p$ to $q$ that stay outside $\O[1]$. The
price to pay is the introduction of an $(n-2)$--dimensional sphere of self-intersections in a component of $N$ for each finger move. The Morse function $F$ will have some critical
points on this sphere, so we introduced  new critical points on $\O[2]$.  After a finite number of such moves, the original critical points $p$ and $q$ can be cancelled.  An example of this is illustrated in Figure~\ref{fig:stevedore}.

\begin{figure}[h]
  \input{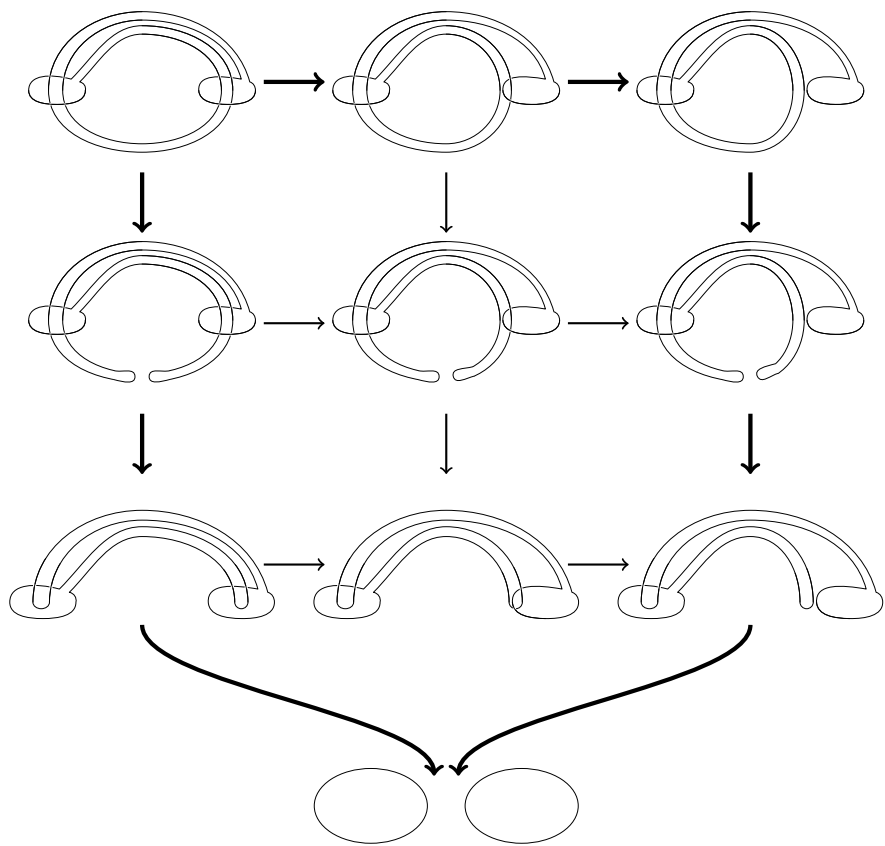}
\caption{A an embedded concordance and a link homotopy (rel.\ boundary). The arrows from the top-left corner to the bottom represent an embedded concordance
from the stevedore knot to the unknot (the obvious deletion of a circle is not presented). The arrows from top-left to the top-right, and then down, represent an immersed concordance with one minimum and one saddle point.
The middle route downwards represents the midpoint of a regular homotopy between these two concordances.
Going only along the top row from the left to the right gives an immersed concordance from the stevedore knot to the unknot (the knot in the top-right corner is already an unknot).
In the  right-then-down route, the saddle point and the minimum can be cancelled.
A detailed explanation is given in Section~\ref{sec:61}.}\label{fig:stevedore}
\end{figure}

%To explain the relationship with existing results in Section~\ref{subsection:existing-results}.
%Recall that a generic immersion is an immersion in general position; every immersion $N \looparrowright \O$ can be perturbed to a generic immersion, and generic immersions are \emph{stable}, which means~\cite{GG} that for a generic immersion $f$ there is an open set $U \ni f$ in the space of smooth maps $C^{\infty}(N,\Omega)$ such that for every $f' \in U$ one can find diffeomorphisms $g \colon N \to N$ and $h \colon \Omega \to \Omega$ with $f'\circ g = h \circ f$.
%We return to a detailed study of paths
%of immersions in Subsection~\ref{sub:paths_of_immersions}.

%\section{Outline of the proof of Theorem~\ref{thm:concordance immersion}}

Now let us describe an outline of the entire proof.
After a perturbing $G \colon A \times [0,1] \imra Y \times [0,1] $ we may assume that the composition
\[
f_0 \colon A \times [0,1] \xrightarrow{G} Y \times [0,1] \xrightarrow{\pr_2} [0,1]
\]
is a Morse function.  It is connected to the projection $f_1 := \pr_2 \colon A \times [0,1] \to [0,1]$ by a path $f_\tau \colon A \times [0,1] \to [0,1]$ of \emph{generalised Morse functions}, namely functions that are either Morse or have isolated birth or death points.  Moreover, this path
can be chosen to be independent of $\tau$ in a neighbourhood of $A\times\{0,1\}$.

The question is whether we can {\em lift} such a path $f_\tau$ to a regular homotopy $G_\tau$ of $G$.
More precisely, we seek a regular homotopy $G_\tau$ and a homotopy
$\pr_2 \circ G_\tau \simeq f_\tau$ through generic paths of generalised Morse functions that agree with $f_0$ and $f_1$ at $\tau=0,1$ respectively.  Here generic means that there are only finitely many births and deaths, and that they occur at different $\tau$ values.   If such a regular homotopy $G_\tau$ exists, then at the end of the regular homotopy there are no critical points, and so $G_1$ is level preserving, and is therefore the trace of a regular homotopy.

More generally, suppose $(F_0, G_0)$ is a Morse immersion with $G_0 \colon N\looparrowright\O$ a generic immersion and $F_0\colon\O\to \R$ a Morse function. Consider $f_0:=F_0 \circ G_0$ and assume that
 $f_\tau$ is a smooth path of functions on~$N$.
 Assume that $f_\tau$ is Morse for all but finitely many
parameters of $\tau$, where a single birth or death occurs.  We aim to {\em lift} the path of functions $f_\tau$ to a path $(F_\tau, G_\tau)$ of Morse immersions in the sense that $F_\tau \circ G_\tau$  and $f_\tau$  are homotopic (relative their endpoints) through generic paths of generalised Morse functions on~$N$.
The obstructions to lifting only arise if:
\begin{itemize}
  \item there is a $\tau$ with a death in $f_\tau$, that is a pair of critical points $p,q$ on $N$ is cancelled, but there is an obstruction (from intersecting membranes) for cancellation of the pair $p,q$ in $\O$, or
    \item there are two critical points $p,q$ on $N$ that get rearranged, see Figure~\ref{fig:Hopf}. In general, such rearrangement need not be possible in $\O$ but we show that in codimension at least two it is possible, provided the critical point of higher index of $f_\tau$ goes below the critical point of lower index.
  \end{itemize}

We solve the first problem in codimension at least two by performing a regular homotopy that introduces self-intersections in the components of $N$, but that removes the cancellation obstruction.
To see that a regular homotopy is necessary in general to obtain a level-preserving map $G_1$, and that this cannot be done via embeddings, recall that there are many pairs of knots $S^n \subseteq S^{n+2}$ that are concordant but not isotopic.

As an example for the reordering obstruction, let $G \colon N=S^1 \sqcup S^1 \hra \O =\R^3$ be the Hopf link.
Let $f_\tau$ be a path of functions with $f_0 = \pr_3 \circ g$ that pulls the circles apart as in Figure~\ref{fig:Hopf}, i.e.\  there is a $c \in \R$ such that $f_1$ maps one circle above $c$ and the other circle below $c$.  Note that $f_\tau$ moves a critical point of index $1$ below a critical point of index $0$ and that this can be realized by a regular homotopy $G_\tau$ in the sense that $f_\tau = \pr_3\circ G_\tau$.

However, the obvious $G_\tau$ is not a link homotopy because for some $\tau$, the image $G_\tau(N)$ is connected. In fact, there is no way we can lift the path~$f_\tau$ by a link homotopy $G_\tau$ because the Hopf link has nontrivial linking number.

\begin{figure}[h]\label{fig:Hopf}
  \input{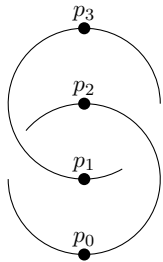}
\caption{A nontrivial reordering obstruction. The critical point $p_2$ can be abstractly moved below $p_1$, but it cannot be moved in the embedded case, without creating self-intersections or adding extra critical points.}
\end{figure}

One of the main technical results in this paper is the following Path Lifting Theorem, which gives a criterion for such lifts of $f_\tau$ to exist.
We use the notation $\bd N$ for the boundary of a manifold $N$.
We fix compact manifolds $N$ and $\O$ of dimensions $n$ and $n+k$ respectively.
\begin{theorem}[Path Lifting]\label{thm:path_lifting-intro}
 Let $(F_0\colon\O\to\R, G_0\colon N\imra\O)$ be a Morse immersion such that $G_0^{-1} \bd(\O) = \bd N$ and $F_0$ is constant on connected components of $\bd \O$.
 Starting with $f_0:=F_0\circ G_0$, let $f_\tau \colon N\to \R$ be a generic path of functions that contains no rearrangements where  a critical point of higher index goes below a critical point of a lower index.
%  Suppose $f_\tau$ does not depend on $\tau$ in a neighbourhood of $\bd N$, and that $f_\tau$ restricted to this neighbourhood has no critical points.
%This can be arranged, so we don't need it as assumption. Brevity is a virtue here.

If the codimension $k\ge 2$ then $f_\tau$ lifts, i.e.\ there exist:
  \begin{itemize}
%    \item a path of function $\wt{f}_\tau\colon N\to\R$ such that $\wt{f}_0=f_0$ and $\wt{f}_1=f_1$ and the paths $\wt{f}_\tau$ and $f_\tau$
%      are homotopic in the space of generic paths of functions.
    \item a regular homotopy $G_\tau\colon N\to\O$ $($rel.\ boundary$)$ such that $G_\tau$ is a generic immersion for all but finitely many $\tau$.
    Moreover, if $G_0$ is a link map then so is $G_\tau$ for all $\tau$.
    \item a generic path of Morse functions $F_\tau\colon\O\to\R$ with $(F_1,G_1)$ a Morse immersion and such that $F_\tau\circ G_\tau$ and $f_\tau$ are homotopic in the space of generic paths of functions.
      \end{itemize}

  Furthermore, if $G_0$ is an embedding, then $G_\tau$ can be chosen to be an ambient isotopy if $k\ge 3$, or $k=2$ and the path $f_\tau$ has no deaths.
\end{theorem}
\begin{remark*}
  In the statement of the Path Lifting Theorem~\ref{thm:path_lifting}, we will make  the assumptions on \emph{genericity} of the path $f_\tau$ precise.
\end{remark*}

From the Path Lifting Theorem, there is a direct way to deduce the Link Concordance Implies Link Homotopy Theorem~\ref{thm:link concordance}.
Apply the Path Lifting Theorem to $N= A \times [0,1], \Omega = Y \times [0,1]$ and let $f_\tau$ be a generic path connecting the given $\pr_2 \circ G_0$ to $\pr_2$. Then we get a path of Morse functions $F_\tau$ with $F_0=\pr_2$ and hence no $F_\tau$ has a critical point on the top dimensional stratum $\O[0]$.

We construct the regular homotopy by identifying the level sets of the composition $F_1\circ G_1$. Each such level set is shown to be diffeomorphic to $A$; the restriction of $G_1$ to these level sets yields a regular homotopy. We refer to Proposition~\ref{prop:crossingforhomotopy}
for more details. %\ypar{MB. Added a new variant.}
\begin{comment}
For each $\tau$, the flow of $\nabla F_\tau$ (with respect to some fixed choice of Riemannian metric on $Y \times [0,1]$) gives rise to a diffeomorphism $\Phi_\tau \colon Y \times [0,1] \to Y \times [0,1]$ with $\Phi_0 = \Id$.  \npar{Phi-tau is NOT an ambient isotopy since it's not level preserving}{MB. Formally it's not. But we can always choose a family of gradient-like vector fields for $F_\tau$, that are, say, length $1$. Then $\Phi_\tau$ is level-preserving. Is that right? PT: I don't think that's right because pseudo-isotopies are exactly the kind of thing you get - and those are not always isotopic to isotopies. But we don't need it anyway.}
%Then $\Phi_\tau$, for $\tau \in [0,1]$, is an ambient isotopy.
The composition $H_\tau := \Phi_\tau \circ G_\tau$
 is again a regular homotopy. Moreover, $\pr_2 \circ \Phi_1 = F_1$ and so
\[
\pr_2 \circ H_1 = \pr_2 \circ \Phi_1 \circ G_1 = F_1 \circ G_1 = f_1=\pr_2.
\]
Therefore, $H_1$ is level preserving as desired. Moreover, if $G_0$ is a link map then so are $G_1$ and $H_1$.
\end{comment}

\section{A guide through the paper}

The paper is split into six parts. % Gallia omnis divisa'st in partes tres...
After the introduction (Part~\ref{part:intro}), in Part~\ref{part:immersed_morse_thy} we develop Immersed Morse Theory. We give the definition
of an immersed Morse function and a gradient-like immersed vector field (called a \emph{grim vector field} for short). We prove many technical results about the existence
of grim vector fields. If $G\colon N\looparrowright\O$ is a generic immersion, we show how to pass from a grim vector field on $\O$ to a gradient-like vector field on $N$ and vice versa. These technical constructions are important while proving results on path lifting.
A key term from Part~\ref{part:immersed_morse_thy} is that of \emph{membranes} (Section~\ref{sec:hypermembranes}), which are analogues of stable and unstable manifolds from classical Morse theory.

Part~\ref{part:just_paths} deals with families of Morse functions. Theorems on rearrangement and cancellation of critical points are proved for
immersed manifolds and interpreted as creating a one parameter family
of functions such that rearrangement, respectively cancellation, occurs along the family.  We also recall the rudiments of classical Cerf theory. Our focus is on one-parameter phenomena, that is: birth, rearrangement, and death. We recast Cerf's results on uniqueness (up to homotopy) of paths of rearrangements in the language of gradient vector fields.

In Part~\ref{part:pathlifting}, we introduce the concept of \emph{path lifting}. This means that we are given an immersion $G\colon N\looparrowright\O$ and a Morse function $F_0\colon\O\to\R$. We suppose that there is a path of functions $f_\tau\colon N\to\R$, $\tau\in[0,1]$, such that $F_0\circ G=f_0$. We want to find a family of functions $F_\tau\colon\O\to\R$ and a regular homotopy $G_\tau \colon N \to \O$ such that $f_\tau=F_\tau\circ G_\tau$ and such that $F_\tau$ has no births away from $N$. The Path Lifting Theorem~\ref{thm:path_lifting} gives precise criteria under which such a lift can be found.

The proof of the Path Lifting Theorem in codimension two involves the finger move, which is stated at the end of Part~\ref{part:pathlifting} but proven in Part~\ref{part:finger}.
The finger move is one of the central notions of the paper. It allows us to trade intersections of membranes for extra double points of $G(N)$.  The whole of Part~\ref{part:finger} comprises the construction together with the proof that it has the desired effect.

Part~\ref{part:examples} contains examples and applications. The first example is the proof of Proposition~\ref{prop:braids}. The second example gives the details of the example in Figure~\ref{fig:stevedore}. It describes a finger move on a concordance in $S^3 \times [0,1]$ bounded by the stevedore knot $6_1$ and the unknot, converting that cylinder $S^1 \times [0,1] \subseteq S^3 \times [0,1]$, via a regular homotopy, to the trace of a regular homotopy between $6_1$ and the unknot.

Both examples serve as an introduction to the finger move in low dimensional situations where it can be more easily visualised. The reader who wishes to skip the technical construction of the finger move can read this part first. The prerequisite knowledge is the notion of a membrane and the statement of the Cancellation Theorem~\ref{thm:grimcanc}.

We prove Proposition~\ref{prop:cod1} on codimension one links in Section~\ref{sec:cod1}. Section~\ref{section:regular-homotopies-surfaces}
shows applications of our results to regular homotopies of surfaces.

\begin{remark*}
  We take the opportunity to point out that the proof of Vector Field Integration Lemma in \cite{BP} contains a mistake. The function constructed in \cite[Lemma 3.10]{BP}
is smooth on trajectories, but it need not be continuous in general: it can have jumps when a trajectory hits a neighbourhood of a critical point, while a nearby trajectory
does not hit this neighbourhood.
We give a rigorous proof of that result in this paper in a more general setting, that is for immersed manifolds.
\end{remark*}

\subsection{Assumptions and conventions}
Throughout, all manifolds will be compact and smooth, and all immersions and embeddings will be smooth.  Manifolds can be orientable or nonorientable, and need not be connected.

The notion of a `boundary' can refer to the boundary of a manifold, or to the point-set boundary of a subset $U \subseteq \O$.  In the latter case we write $\partial U$ for  $\ol{U} \sm \mathring{U}$, the closure minus the interior.  To avoid potential confusion %, for example in the case that $U$ is a codimension zero submanifold of $\O$,
we shall use  $\bd M$ when we mean the manifold boundary of a manifold $M$.  The notation $\bd M$ is somewhat unusual, but we prefer to avoid ambiguity.

Whenever a manifold has nonempty boundary, we will assume that the Morse function $F$ is such that for each connected component $\bd_i \O$ of the boundary $\bd \O$, there is an $a_i \in \R$ such that $F(\bd_i \O) = \{a_i\}$, i.e.\ $F$ is constant on each connected component of the boundary of $\O$.  For product manifolds $\O=Y \times I$, we will usually assume that $Y$ has no boundary, $F(Y \times \{a\}) = a$ for $a=0,1$ and $F(Y\times[0,1])=[0,1]$.

\subsubsection*{Acknowledgements}

The first and second authors thank the Max Planck Institute for Mathematics in Bonn, the University of Regensburg, and the Banach Center in Warsaw, for hospitality during the time that part of this paper was written.
MB was supported by OPUS 2019/B/35/ST1/01120 grant.
MP was partially supported by EPSRC New Investigator grant EP/T028335/2 and EPSRC New Horizons grant EP/V04821X/2.

\part{Morse immersions}\label{part:immersed_morse_thy}

% \section{Overview of Part~\ref{part:immersed_morse_thy}}

We are now going to develop  {\em Morse immersions}, i.e.\ stratified Morse theory on the multiple-point sets of a generic immersion.
The key notion is Definition~\ref{def:immersed_morse_function}, a specialisation of the more general notion of a stratified Morse function; see Definition~\ref{def:Morse_function}, and compare~\cite{GM}.

The second definition, and actually the most important one, is that of a gradient-like immersed Morse vector field, or \emph{grim vector field} for short.
We study the ascending and descending submanifolds of critical points, which following Perron~\cite{Pe} and Sharpe~\cite{Sha} we call \emph{membranes}, as in the embedded case \cite{BP}.

The results on grim vector fields that we prove are mostly technical, but they are needed in the following parts of the paper. First, we
prove existence results, which, given the specific local form of a grim vector field, are not completely trivial. A related result
extends a suitable vector field from a submanifold to a grim vector field.

Grim vector fields are used to study the change of level sets of an immersed Morse function. In particular, we show that if a Morse function
has no critical points on the zeroth and first stratum, then it induces regular homotopy between level sets. This observation is used in the proof that immersed link concordance implies regular link homotopy.

\section{Stratified manifolds}\label{sec:stratified_manifolds}

The reader acquainted with work of Goresky and McPherson might ask about the relation of our immersed Morse theory with the stratified
Morse theory developed in \cite{GM}. The aim of this section is to review their approach and show how it relates to ours. In short, Goresky--McPherson develop Morse theory in much broader context, but they do not spend much time studying trajectories of vector fields on general stratified spaces. Their motivations were somewhat different to ours.

Recall that a \emph{stratification} $\Str$ of a topological space $\O$ is a set partition of $\O$ into smooth manifolds $\{\O[d]\}_{d \in \mathcal{I}}$, called the {\em strata}.
In the present paper the indexing set $\mathcal{I}$ will always be a subset of the nonnegative integers $\mathbb{N}_0$. We will assume the following two conditions.
\begin{enumerate}[label=(S-\arabic*)]
  \item \emph{Local finiteness:} every point in $\O$ has a neighbourhood that meets only finitely many strata.\label{item:S1}
	\item \emph{Local triviality:} for every $p \in \O[d]$, there is an open set $U \ni p$ in $\O$ with a stratum preserving diffeomorphism
	  $\varphi \colon U \toiso S \times T$ where $S$ is an open set in $\R^k$ for some $k$, trivially stratified, and $T$ is a stratified manifold with a point stratum $\{q\} \subseteq T$, and $U \cap \O[d]$ is identified with $S \times \{q\}$ under $\varphi$.\label{item:S2}
	\end{enumerate}

The definition of a stratification might be too general for various results to hold. Often one imposes regularity conditions on the stratification, which is easiest to do when $\O$ is a smooth manifold and the strata $\O[d]$ are smooth submanifolds. We will assume the following regularity conditions.

\begin{definition}\label{def:Whitney}
A stratification of $\O$ satisfies \emph{Whitney's conditions} if for every pair of strata, represented locally by disjoint submanifolds $X,Y \subseteq \O$, the following hold.

\begin{enumerate}[label=(W-\arabic*)]
  \item  Whenever a sequence of points $x_1, x_2,\dots\in X$ converges in $\O$ to a point $y\in Y$, and the sequence of tangent planes $T_{x_m}X$ converges to a plane $T \subseteq T_y\O$, then~$T$ contains the tangent plane to $Y$ at $y$, i.e.\ $T \supseteq T_y Y$. \label{item:W1}
  \item For each sequence $x_1, x_2,\dots$ of points in $X$ and each sequence $y_1, y_2,\dots$ of points in~$Y$, both converging to the same point~$y$ in $Y$, such that the sequence of secant lines $L_m$ between $x_m$ and $y_m$ converges to a line~$L$, and the sequence of tangent planes $T_{x_m}X$ converges to a plane $T$, we have that $L \subseteq T$.\label{item:W2}
\end{enumerate}
The limits of secant lines and of tangent planes are taken in the topology on the Grassmannian bundle of $T\O$ associated with $1$-dimensional (respectively $(\dim X)$-dimensional) subspaces. The Grassmannian space at each fibre inherits a topology as a quotient of the space of $\dim \O \times 1$ (respectively $\dim \O \times \dim S'$) matrices with rank 1 (respectively $\dim X$).
\end{definition}

Given a generic immersion $G\colon N\looparrowright \O$, it defines a stratification of $\O$ by declaring
\[\O[d] := \{w \in \O \, \colon \, |G^{-1}(w)| = d\}, \quad d\geq 0.\]

Another important example of a stratified manifold is given by manifolds-with-corners. Recall that $N$ is an $n$-dimensional manifold-with-corners if~$N$
is a second countable Hausdorff topological space locally modelled on $\R^{n-d}\times\R_{\ge 0}^d$, for some $d \geq 0$ that depends on the point in~$N$.
We define a stratification of $N$ by declaring that $w\in N[d]$ if a neighbourhood of $w$ in $N[d]$ is diffeomorphic
to an open subset of $\R^{n-d}\times\R_{\ge 0}^d$, via a diffeomorphism sending~$w$ to~$0$.
%but no open subset of $N$ containing $w$ is diffeomorphic to an open subset of $\R^{n-d+1}\times\R_{\ge 0}^{d-1}$.
This type of stratification occurs in this article when discussing ascending and descending membranes of vector fields.

The following result will be used in Section~\ref{sec:morse-smale-condn}  to guarantee the analogue of the Morse-Smale condition in immersed Morse theory. Its proof can be found in \cite[Section I.1.3]{GM}.

\begin{lemma}[Stratified General Position]\label{disjoint}
  Let $\Str$ be a locally trivial stratification  of a closed manifold $\O$ satisfying Whitney's conditions, i.e.\  \ref{item:S1},\ref{item:S2},\ref{item:W1},\ref{item:W2} all hold for $\Str$.  Let $Y_1,Y_2 \subseteq \O$. Assume that for every stratum $S\in\Str$ the intersections $Y_i\cap S$ are compact submanifolds of $S$.

Then there exists a stratum preserving diffeomorphism $h$ of $\O$, stratum preserving isotopic to the identity, such that $h(Y_1\cap S)$ intersects $Y_2\cap S$ transversely for all strata $S\in\Str$.
In particular, if
\[
 \dim (Y_1\cap S) + \dim (Y_2\cap S) < \dim S
\]
then $h(Y_1\cap S)$ is disjoint from $Y_2\cap S$.
\end{lemma}

Having defined stratifications we can now  introduce stratified Morse functions.
The definition makes sense for any stratification, and so we give it in this generality, however studying Morse functions on stratifications violating Whitney's conditions is much more involved, and we will not do so.

As usual, given a smooth function $f \colon L \to \R$, where $L$ is a manifold of dimension $\ell$, a point $p \in L$ is a \emph{critical point} if there is a chart $\varphi \colon U \to L$ with $0 \in U \subseteq \R^{\ell}$ and $\varphi(0) = p$, such that $\left(\tmfrac{\partial(f \circ \varphi)}{\partial x_i}(0)\right)_{i=1}^{\ell}$ vanishes.  The critical point $p$ is \emph{nondegenerate} if the Hessian matrix of second partial derivatives $\left(\tmfrac{\partial^2(f \circ \varphi)}{\partial x_i\partial x_j}(0)\right)_{i,j=1}^{\ell}$ is nondegenerate.

\begin{definition}[Stratified Morse function]\label{def:Morse_function}
	A \emph{Morse function} on a stratification $\Str$ of $\O$ is a smooth function $F\colon \O\to \R$ such that:
	\begin{enumerate}[label=(M-\arabic*)]
	  \item\label{item:M1} for every stratum $S$, the restriction $F|_S$ has no degenerate critical points;
	  \item\label{item:M2} for every stratum $S$ and for every $p\in S$ and every limit plane
		  $Q=\lim_{p_i \to p} T_{p_i}S'$ we have that $dF_p(Q)\neq 0$, where the $p_i\in S'\neq S$ converge to $p$.
		  	\end{enumerate}
As above the limit of tangent planes is taken in the topology on the Grassmannian bundle of $T\O$ associated with $(\dim S'$-dimensional) subspaces.

	A \emph{critical point} of $F$ is by definition an ordinary critical point of one of the restrictions $F|_S$ and the \emph{index} of such a critical point is the ordinary index of the Morse function $F|_S$.
	
	If, in addition to \ref{item:M1} and \ref{item:M2}, all critical values are distinct, then $F$ is called an \emph{excellent Morse function} on $\Str$.
\end{definition}
Condition~\ref{item:M2} might sound technical, but its intuitive meaning is that a critical point of the restriction $F|_S$ does not lie on deeper strata. See \cite[Section I.1.4]{GM} for a more detailed discussion. We return to~\ref{item:M2} in \ref{item:IM2} of Definition~\ref{def:immersed_morse_function}, which gives a further specialisation of Definition~\ref{def:Morse_function} to immersed Morse functions.

%\npar{(W2) and (M2) are quite hard (for me) to understand. Is there a simplest example, perhaps with a picture that can explain what they say, or what violation looks like and why it is bad? }{ Added a paragraph.}

\section{Generic immersions}\label{sec:transverse}

%Generic immersions provide an important instance of stratified manifolds.
%To set up the notation, we recall the definition of a transverse intersection of multiple submanifolds, and then that of a generically immersed manifold. A useful resource is Ekholm~\cite[Section 3]{Ekholm}.

Linear subspaces $V_1,\dots, V_\ell$ of a vector space $V$ `are in general position', `meet generically', `intersect transversely', or `are transverse' if the diagonal map
\[
  V \ra \bigoplus_{i=1}^\ell (V/V_i)
\]
is an epimorphism. If $V$ is equipped with an inner product then this condition translates into the orthogonal complements of the $V_i$ being linearly independent in $V$.
If $M_1,\dots,M_\ell$ are smooth submanifolds of a smooth manifold $\O$, one says that they `meet, or intersect, transversely' at a point $x\in\O$ if the tangent spaces $T_xM_1,\dots, T_xM_\ell$ are transverse in $T_x\O$.

%Choose normal bundles $\nu_1,\dots,\nu_\ell$ to $M_1,\dots,M_\ell$
%in $\O$. Suppose $x\in M_1\cap \dots \cap M_\ell$. We say that the manifolds $M_1,\dots,M_\ell$ intersect \emph{transversely at $x$},
%if the normal bundles at $x$ to $M_1,\dots,M_\ell$ are linearly independent, that is,
%\begin{equation}\label{eq:linearly_independent}
%\dim(\nu_1(x)+\dots+\nu_\ell(x))=\dim\nu_1(x)+\dots+\dim\nu_\ell(x).
%\end{equation}
%An equivalent definition is inductive, compare e.g. \cite{Ekholm}. We say that $M_1,M_2$ intersect transversely at $x$, if $T_xM_1\oplus T_xM_2=T_x\O$. In that case, $M_1\cap M_2$ is a submanifold. We then say that $M_1,\dots,M_\ell$ intersect transversely, if $M_1,\dots,M_{\ell-1}$ intersect transversely, and $(M_1\cap\dots\cap M_{\ell-1}),M_\ell$ intersect transversely. An elementary argument in linear algebra involving
%induction over the number of branches shows that the two definitions of transversality are equivalent.

A smooth map $G \colon N \to \O$ is an \emph{immersion} if $dG \colon T_pN \to T_{G(p)}\O$ is injective for every $p \in N$.  Our immersions will be assumed \emph{neat}, which means that $G^{-1} (\bd \O) = \bd N$ and $G$ is transverse to $\bd \O$.

%there exist collar neighbourhoods of $\bd N$ and $\bd \O$ such that in these coordinates, for $x \in \bd N$ and $t \in [0,1]$ we have that $G(x,t) = (G(x,0),t)$.  %Following Ekholm~\cite{Ekholm}, we define generic immersions.

%For $d\in\N_0$ define $N[d] \subseteq N$ to be the set of points in $x \in N$ such that $G^{-1}(G(x))$ has cardinality $d$, i.e.\ where $G$ is $d$ to one.

For a map $G \colon N\to \O$ and $d\in\N_0$, recall that  $\O[d] \subseteq \O$ is the set of points $x \in\O$ such that $G^{-1}(x)$ has cardinality $d$, i.e.\ there exist exactly $d$ points $u_1,\dots, u_d\in N$ such that $G(u_i)=x$.

\begin{definition}\label{def:generic-immersion}
For closed $N$, an immersion $G \colon N \to \O$  is \emph{generic}, denoted $N \looparrowright \O$, if for all $x \in \O[d]$, the linear subspaces $dG(T_{u_1}), \dots , dG(T_{u_d})$, for $G(u_i)=x$, meet generically in $T_x\O$.
If $\bd N$ is nonempty, an immersion $G$ is \emph{generic} if it is neat and the two restrictions $G|_{\bd N} \colon \bd N \to \bd \O$ and $G|_{N \sm \bd N} \colon N \sm \bd N \to \O \sm \bd \O$ are both generic immersions.
%Equivalently, we could require that the double $G \cup G \colon N \cup_{\bd N} N \looparrowright \O \cup_{\bd \O} \O$ is generic.
%\begin{enumerate}
%  \item   \item suppose $dk \leq n+k$ and $u_1,\dots,u_d \in N[d]$ are distinct points with $z:= G(p_1) = \cdots = G(p_d) \in \O$. Take $U_1,\dots,U_d\subseteq N$ to be the open subsets such that $u_i\in U_i$ and $G|_{U_i}$ is an embedding. Then $G(U_1),\dots,G(U_d)$ intersect transversely at $z$,
%    in the sense that the normal bundles are linearly independent.
%\end{enumerate}

Note that the genericity condition implies that $\O[d]$ is empty whenever $dk > n+k$.
   \end{definition}

A subspace of $\O$ will be called a \emph{generically immersed manifold} if it is the image of a manifold $N$ under a generic immersion $G \colon N\imra\O$.
From now on $N$ and $\O$ will be assumed to be compact and immersions will be usually assumed generic.
When we study paths of immersions in later parts, we will have to consider the minimal generic failure to a generic immersion at isolated points.

\begin{proposition}
For a generic immersion $G \colon N \looparrowright \O$ the following hold:
  \begin{itemize}
\item $\O[d]$ is a submanifold of codimension $(d-1)k$ in $\O$ for all $d\in\N_0$,	
\item The closure of $\O[d]$ is the union $\bigcup_{d'\ge d} O[d']$,
\item $\{\O[d], d\in \N_0\}$ forms a stratification of $\O$ satisfying conditions  \ref{item:S1}, \ref{item:S2}, \ref{item:W1}, and~\ref{item:W2} from Section~\ref{sec:stratified_manifolds}.
	\end{itemize}
\end{proposition}

\begin{proof}[Sketch of proof]
The stratification has only finitely many nonempty strata and hence it is also locally finite, i.e.~\ref{item:S1} holds.  For local triviality, we let $x \in \O[d]$. Then we take $S$ to be an open neighbourhood of $x$ in $\O[d]$.  The final transversality condition of Definition \ref{def:generic-immersion} implies that we have a local description of the normal bundle to $\O[d]$ at $x$ as a stratified manifold $T \cong \R^{(d-1)k}$  with the origin as a stratum. Thus \ref{item:S2} holds.

In order for a sequence of points to occur as in \ref{item:W1} and \ref{item:W2}, there must be an open set $U \cong \R^n$ in $N$ containing the limit point $y$ and the tail of the sequence $(x_m)$, such that $G|_U$ is a smooth embedding and $y$ lies in a deeper stratum than the $x_m$.  In this case the Whitney conditions \ref{item:W1} and \ref{item:W2} are automatic.
%  The result follows from implicit function theorem.
\end{proof}

For example, the fact that $\O[d]$ has codimension $(d-1)k$ has the following consequences  in low dimensions. A surface $N$ generically immersed in a 4-manifold $\O$ cannot have triple points, because in that case $d=3$ and $k=2$, so triple points would have codimension $4$ in $N$. For generic immersions of 3-manifolds in 5-manifolds, the double points are arcs and circles, and the triple points are again empty.

The next result is a standard consequence of the multijet transversality theorem; see for instance
\cite[Section III.3]{GG}. We will prove some generalisations of this result while studying paths of immersions in Subsection~\ref{sub:paths_of_immersions}.

\begin{lemma}\label{lem:generic-immersion}
Generic immersions form an open set and dense subspace of the space of all immersion under the Whitney $C^\infty$-topology. %Every immersion can be perturbed to a generic immersion by an arbitrarily small perturbation.
Moreover, generic immersions are {\em stable} in the sense that every $G \colon N\imra \O$ has an open neighbourhood in $C^\infty(N,\O)$ in which each map differs from $G$ only by ambient isotopies of $N$ and $\Omega$.
\end{lemma}

%\begin{remark*}
%  We will use the conventions from above Definition~\ref{def:generic-immersion} throughout:
%$\O$ is dimension $n+k$;
% $N$ is dimension $n$;
%  $G \colon N \looparrowright \O$ is a generic immersion; and $M:= G(N)$. So the codimension is denoted by~$k$.
%\end{remark*}

%\begin{definition}[Stratification induced by a generic immersion]
%  Given a generically immersed manifold $M \subseteq \O$, represented as $G(N)$ for some manifold $N$, define the $d$-th \emph{strata} of $M$ as $\O[d]=G(N[d])$. We will also write $\O[0] := \Omega\setminus M$.
%  We call $d$ the \emph{depth} of the stratum $\O[d]$.
%\end{definition}

%The following result is well-known to  experts. We include a quick sketch for the reader's convenience.
%
%\begin{lemma}
%  The induced stratification satisfies \ref{item:S1}, \ref{item:S2}, \ref{item:W1}, and~\ref{item:W2} from Section~\ref{sec:stratified_manifolds}.
%\end{lemma}

\begin{definition}
	Let $G \colon N\imra \O$ be a generic immersion with $x\in \O[d]$. A \emph{branch} of $G$ through $x$
	is a neighbourhood $V \subseteq N$ of some point $v\in G^{-1}(x)$ such that the restriction $G|_{V}$ is one-to-one.
\end{definition}

Note that $G$ and $v\in G^{-1}(x)$ determine the germ of a branch $V$ and hence there are exactly $d$ distinct germs of branches through $w\in \O[d]$; see Figure~\ref{fig:8branches}.
\begin{figure}
  \includegraphics[width=3cm,angle=90]{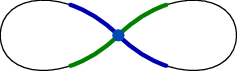}
  \caption{Two local branches passing through a double point of a figure 8 in the plane.}\label{fig:8branches}
\end{figure}

\section{Jet extensions and the Thom-Boardman stratification}\label{sec:strat_image}

\subsection{Multijet transversality}\label{sub:multijet}
For two smooth manifolds $X$ and $Y$, and an integer $r$, the \emph{jet space} $J^r(X,Y)$ is the space of Taylor expansions
up to order $r$ of maps from $X$ to $Y$ (a formal definition can be found in \cite[Section II.2]{GG}). For a $C^r$-smooth map
$G\colon X\to Y$,
we consider its \emph{jet extension} $\mj^rG\colon X\to J^r(X,Y)$.
The \emph{source map},
assigning to a Taylor expansion of a map at a point $x\in X$, the point $x$ itself,
is denoted $\alpha\colon J^r(X,Y)\to X$.

We now recall the construction of multijet spaces, referring to \cite[Section II.4]{GG} for more details.
Choose an integer $s\ge 1$, where the case $s=1$ recovers standard jet spaces. The space $X^{(s)}\subseteq X^{\times s}=X\times\dots\times X$ is the space of distinct $s$-tuples of points in $X$, i.e.\ the ordered configuration space of $s$ points in $X$.
We define $J^r_s(X,Y)$ to be the preimage of $X^{(s)}$ under the $s$-fold product of the source map $\alpha^{\times s}\colon J^r(X,Y)^{\times s}\to X^{\times s}$. The space $J^r_s(X,Y)$ is called the \emph{multijet space}.
For a smooth map $G\colon X\to Y$, there is the notion of a \emph{multijet extension} $\mj^r_sG\colon X^{(s)}\to J^r_s(X,Y)$, assigning to an $s$-tuple $x_1,\dots,x_s$ of pairwise distinct points in $X$, the Taylor expansions up to order $r$ at these points.
The multijet extension has the property that $\alpha^{\times s} \circ \mj^r_sG = \Id|_{X^{(s)}}$.

Before we state the Multijet Transversality Theorem we briefly introduce  residual subsets and their basic properties.
%\ypar{The space $C^\infty(X,Y)$ is not locally compact.}

\begin{definition}[Residual subset, compare \expandafter{\cite[Definition III.3.2]{GG}}]\label{def:residual}
  A subset $A \subseteq T$ is \emph{residual} in $T$ if it is a countable intersection of subsets of $X$, each of whose interiors is dense in $X$. A space $T$ is called a \emph{Baire space} if any residual subset of $T$ is dense.
\end{definition}
Any Banach space is complete, hence a Baire space. In particular, if $X,Y$ are smooth manifolds, then the space of $r$-times differentiable maps
$C^r(X,Y)$ is a Baire space. The next result has a more complex proof, because $C^\infty$ is only a Fr\'echet space.
\begin{proposition}[see \expandafter{\cite[Proposition III.3.3]{GG}}]\label{prop:baire}
  Suppose $X,Y$ are two smooth manifolds. Then  $C^\infty(X,Y)$ is a Baire space.
\end{proposition}

Knowing that a subset is residual is often more useful than knowing it is dense. One reason is that countable intersections of dense subsets are not necessarily dense (let $(q_i)_{i=1}^\infty$ be an enumeration of the elements of $\Q$ and consider $A_i = \Q \sm \{q_i\}$), but the analogous property is true for residual subsets.
%\npar{I thought it would be helpful to explain why residual is useful in these lemmas and the discussion between them. }{Good point!}
We record this as a lemma, whose proof is immediate from the definitions. %\ypar{Removed the proof as requested, and added this sentence.}

\begin{lemma}\label{lem:countable-int-residual-is-residual}
  Suppose that $A_i \subseteq X$ is residual for $i=1,2,\dots$. Then $\bigcap_{i=1}^\infty A_i$ is residual.
  %\npar{It is nice to have, but maybe it is too obvious to be written.}{I think lemma is okay, but we shouldn't be adding a proof.}
\end{lemma}

%\begin{proof}
%  For each $i$, $A_i = \bigcap_{j=1}^\infty A_{ij}$, for subsets $A_{ij} \subseteq A_i \subseteq X$ such that $\mathring{A}_{ij}$ is dense. Then $\bigcap_{i=1}^{\infty} A_i = \bigcap_{i,j=1}^{\infty} A_{ij}$, with each  $\mathring{A}_{ij}$ dense, and so indeed $\bigcap_{i=1}^{\infty} A_i$ is residual.
%\end{proof}

The following result~\cite[Theorem II.4.13]{GG} will be extensively used in the future.

\begin{theorem}[Multijet Transversality Theorem]\label{thm:multijet}
  For any submanifold $S\subseteq J^r_s(X,Y)$ the set of maps $G\in C^\infty(X,Y)$ whose $s$-fold multijet extension $\mj^r_sG$ is
  transverse to $S$ is residual in $C^\infty(X,Y)$.
\end{theorem}

%Since  $C^\infty(X,Y)$ is locally compact and Hausdorff, and since the set of maps whose multijet extension is transverse to $S$ is residual, it follows from Lemma~\ref{lem:residual-implies-dense} that this set is dense in $C^\infty(X,Y)$.
A residual set need not be open, in general.
The following corollary translates Theorem~\ref{thm:multijet} into a ready-to-use criterion. To prove it, use Theorem~\ref{thm:multijet}
in the case $S$ is not a manifold, but a stratified space. This does not pose problems: one simply inducts on the strata (note that we
strive to prove that the set of transverse maps is residual, not open, which requires extra conditions). We refer to the discussion
in \cite[Section 29]{Arn}.
%\npar{We need to say how we deduce it, since $S$ is not a submanifold. Apply Theorem~\ref{thm:multijet} to the strata inductively?}{Yes. But I'd like not to enter into details.}
%\npar{Surely there is some middle ground? `To deduce this result we apply Theorem~\ref{thm:multijet} to the strata inductively.' For example. }{Added a discussion and a reference. \cite{Arn} is notoriously hard to give a pinpoint citation, I tried my best. It is page 232 in my edition.}

\begin{corollary}\label{cor:multijet}
  Suppose $S$ is a Whitney stratified subset of $J^r_s(X,Y)$.
  \begin{itemize}
    \item[(i)]\label{item-cor-multijet-i} If $\codim S>s\dim X$, the space of maps from $H\colon X\to Y$,
  whose $s$-fold $r$-th multijet extension misses $S$, is residual in $C^\infty(X,Y)$.
    \item[(ii)]\label{item-cor-multijet-ii} If $\codim S>s\dim X+1$, then the space of paths $H\colon X \times [0,1]\to Y$ such that for all $\tau\in[0,1]$, the $s$-fold $r$-th jet extension
  of $H_\tau$ misses $S$, is residual in $C^\infty(X \times [0,1],Y)$.
    \item[(iii)]\label{item-cor-multijet-iii} More generally, if $\codim S > s\dim X+k$, then the space $k$ parameter families $H \colon X \times D^k \to Y$ such that for all $\tau \in D^k$, the
$s$-fold $r$-th jet extension of $H_\tau$ misses $S$, is residual in $C^\infty(X \times D^k,Y)$.
  \end{itemize}
\end{corollary}

The third item implies both the first and second, but since the first and second are the main statements we will need, we state these explicitly too.

Theorem~\ref{thm:multijet} generalises the standard Thom transversality theorem (case $s=1$). For $s=1$,
the space of maps missing a subset $S$ satisfying the hypothesis of Corollary~\ref{cor:multijet} is not only residual, but open-dense. Openness can be sometimes achieved by explicit methods.

\begin{remark*}
  We will use the terminology that some subsets of spaces of smooth functions, such as $\cI$, are defined using strata in the corresponding jet spaces. The strata we talk about are subsets of a finitely dimensional jet space, or a multijet space. These sets, like $\cI$, are defined as sets
  of functions whose jet extension misses some given strata. Whereas the actual subsets of smooth functions are infinite dimensional, as subsets of an infinite dimensional Fr\'echet space, e.g.\ $C^\infty(N,\Omega)$, the multijet transversality theorem allows us to analyse finite dimensionally, and hence conclude that the original subsets in the Fr\'echet space are residual.
\end{remark*}

\subsection{Thom--Boardman stratification}\label{sub:thom_boardman}

%It is known that if $X$ and $Y$ are two smooth compact manifolds, then for a sufficiently generic smooth map $f\colon X\to Y$,
%the image of $X$ is Whitney stratified. \mpar{I was expecting this to be a result in this section, but it does not appear. }
%This is a part of the Thom--Mather theory \cite{Mather_stable}. We recall a few results that are well-known to experts.

Suppose $Z,Y$ be two smooth compact manifolds of dimension $m$ and $m'$. For a map $\phi\colon Z\to Y$, and an integer $r\ge 0$,
we define the set
$\beta^r\phi$, to be the set of points $x\in Z$ such that $\dim D\phi(T_xZ)=\min(\dim Z,\dim Y)-r$, compare \cite[Secion VI.1]{GG}.
This space need not be a manifold, but assume for the moment it is. Then, setting $r_1=r$ and assuming $r>0$,
we can iterate this procedure. Namely, for $r_2\ge 0$
define $\beta^{r_1,r_2}\phi$ to be the set of points $x\in\beta^{r_1}\phi$ such that $D(\phi|_{\beta^{r_1}\phi})$ has rank
$\min(\dim\beta^{r_1}\phi,\dim Y)-r_2$. Thom--Boardman theory deals with these spaces. In particular, there is a formula
for the expected dimension of the sets $\beta^{r_1,\dots,r_\ell}\phi$.
The sets $\beta^{r_1,\dots,r_\ell}\phi$ can be defined, using jet extensions, even if $\beta^{r_1}$ is not smooth, compare \cite[Section VI.5]{GG}. We require that $r_1,\dots,r_{\ell-1}>0$ and $r_\ell\ge 0$.
\begin{definition}[Thom--Boardman map]\label{def:TB}
  A smooth map $\phi\colon Z\to Y$ is called \emph{Thom-Boardman} if all the $\beta^{r_1,\dots,r_\ell} Z$ are smooth of expected dimension.
\end{definition}
The jet transversality theorem implies the following result.

\begin{proposition}[compare \expandafter{\cite[Theorem VI.5.2]{GG}}]\label{prop:res}
  The set of Thom--Boardman maps is residual in $C^\infty(Z,Y)$.
\end{proposition}

Golubitsky-Guillemin~\cite[Theorem VI.5.2]{GG} proved residuality of maps satisfying an extra condition, called the NC-condition. We will use the extended version below.
Our discussion is specialised to the following situation.
Let $X,Y$ be two smooth compact manifolds of dimensions $m+1$ and $m$. Define $\cP(X,Y)\subseteq C^\infty(X,Y)$
by the conditions:
\begin{itemize}
  \item $\Pi$ is onto;
  \item for any $y\in Y$, $\Pi^{-1}(y)\cong [0,1]$,
  \item for any $x\in X$, $D\Pi(x)\colon T_x X\to T_{\Pi(x)}Y$ is onto.
\end{itemize}
In the applications, such $\Pi$ arises naturally from a non-vanishing vector field on $X$ and contraction along trajectories.
It is clear from the definition, that $\cP(X,Y)$ is open. Note that if $\Pi\in\cP$, $Z\subseteq X$ is a submanifold, and $\phi=\Pi|_Z$,
then $\beta^{r_1,\dots,r_\ell}\phi=0$ if at least any of the $r_i$ is greater than $1$. Moreover, if $r_1=\dots=r_\ell$ and $\dim Z=\dim X-k$,
then the expected dimension of $\beta^{r_1,\dots,r_\ell}\phi$ is $\dim Z-k\ell$. While latter observation follows from the general formula presented in \cite[Section VI.5]{GG} or \cite[Section I.2]{AVG}, it is quite instructive to understand case $\ell=1$. Then, for $x\in Z$,
we know that $x\in\beta^1\phi$, if and only if $\ker D\Phi\subseteq T_xZ$. Now $\dim\ker D\Phi=1$, and the codimension of $T_xZ$ in $T_xX$ is $k$. Hence, the condition $\ker D\Phi\subseteq T_xZ$ is of codimension $k$.

For this situation, we use the notation $\beta^{\mathbf{\ell}}\phi$ for $\beta^{1,\dots,1}\phi$, and $\beta^{\mathbf{\ell},0}$ for $\beta^{1,\dots,1,0}\phi$, where the sequence of $1$'s in both definitions has length $1$.
We adopt a specific variant of Definition~\ref{def:TBP}
\begin{definition}\label{def:TBP}
  The map $\Pi\colon X\to Y$ is the \emph{Thom--Boardman projection}, if $\Pi\in \cP(X,Y)$, $\phi:=\Pi|_Z$ is Thom--Boardman, and $\ell\ge 0$, $\beta^{\mathbf{\ell},0}\phi$ is an immersion with normal crossings.
\end{definition}
\begin{lemma}\label{lem:TBP}
  The set of Thom--Boardman projections is residual in $\cP(X,Y)$.
\end{lemma}
\begin{proof}
  Suppose $Z$ and $X$ are compact. The restriction map $C^\infty(X,Y)\to C^\infty(Z,Y)$ is open. As the preimage of a dense subset under an open map is dense, the preimage of a residual set under open map is open. By \cite[Theorem VI.5.2]{GG} (see Proposition~\ref{prop:res}), the set of maps $\phi$ that are Thom--Boardman and satisfy the NC-condition is residual. Hence, the set of map $\Pi\in C^\infty(X,Y)$ whose restriction shares these properties, also is. Next, $\cP(X,Y)$ is open in $C^\infty(X,Y)$, meaning that the set of maps $\Pi\in \cP(X,Y)$ satisfying that are Thom--Boardman and satisfy the NC-condition, is open in $\cP(X,Y)$. But this is exactly the statement of the lemma.
\end{proof}

\section{Immersed Morse functions}\label{sec:immersed_morse}
The general notion of a stratified Morse function (Definition~\ref{def:Morse_function}) can be specialised to the following definition.

\begin{definition}[Immersed Morse function]\label{def:immersed_morse_function}
For a generic immersion $G:N \imra \O$, consider its d-fold point startification $\{\O[d], d\geq 0\}$ of $\O$.
  A function $F\colon\O\to\R$ is called an \emph{immersed Morse function for $G$} if for each $d$ and each critical point $p$ of the restriction $F|_{\O[d]}$.
  \begin{enumerate}[label=(IM-\arabic*)]
    \item\label{item:IM1}  the Hessian of $F|_{\O[d]}$ is nondegenerate at $p$ and
    \item\label{item:IM2} $p$ is not a critical point of $F$
      restricted to the $(d-1)$-fold intersection $Y_{1}\cap\dots\cap \widehat Y_i \cap \dots \cap Y_{d}$, for every $i \in \{1,\dots,d\}$, where $Y_1,\dots, Y_d$ are branches of $\O[d]$ at $p$.
  \end{enumerate}
  If additionally for any two $p\neq p'$ such that $p$ is a critical point of $F|_{\O[d]}$ and $p'$ is a critical point of $F|_{\O[d']}$
  we have $F(p)\neq F(p')$, we call $F$ an \emph{excellent} immersed Morse function.
\end{definition}

\begin{remark*}
  Item~\ref{item:IM1} in Definition~\ref{def:immersed_morse_function} corresponds to condition \ref{item:M1} of Definition~\ref{def:Morse_function}. Item~\ref{item:IM2} corresponds to condition~\ref{item:M2}. In the Introduction, we referred to the pair $(F,G)$ as a Morse immersion but since $G$ is fixed in many constructions, it is important now to focus on the properties of $F$ relative to the fixed stratification $\{\O[d]$\} given by $G$.
\end{remark*}

\begin{convention}
If it is clear from the context, we will call an immersed Morse function simply a Morse function. A point $p$ is called a \emph{critical point of $F$} if $p$ is a critical point of $F|_{\O[d]}$ for some stratum $\O[d]$.
\end{convention}

We can characterise the local behaviour of a Morse function.

\begin{lemma}[Immersed Morse Lemma]\label{lem:immersed_morse}
  Suppose $F \colon \O \to \R$ is an immersed Morse function and let
	$p\in \O[d]$ be a critical point of $F|_{\O[d]}$ of index $h$. There exist local coordinates at $p$, denoted
	\[(x_1,\ldots,x_m,y_{11},\ldots,y_{1k},\ldots,y_{d1},\ldots,y_{dk}),\]
	where $m=n+k-dk$, such that the $j$th branch of $M$
	at $p$ is given, for $j=1,\dots,d$, by $\{y_{j 1}=\dots=y_{j k}=0\}$, and $F$ in these coordinates
	has the form
	\begin{equation}\label{eq:formofimmersed}
	  F(x_1,\ldots,y_{dk})=F(p)-x_1^2-\dots-x_h^2+x_{h+1}^2+\dots+x_m^2+\sum_{j=1}^d y_{j 1}.
	\end{equation}
\end{lemma}

In particular, the Immersed Morse Lemma implies that all critical points of an immersed Morse function are isolated.

\begin{proof}
We begin with a general result, from which we shall deduce the Immersed Morse Lemma.

\begin{lemma}\label{lem:normal}
  Suppose $M^n \subseteq \O^{n+k}$ is a generically immersed manifold, let $p \in \O[d]$, and assume there are local
  coordinates $(x_1,\dots,x_m,y_{11},\dots,y_{dk})$ near $p$ $($with $m=n+k-dk=\dim \O[d])$ such that
    the branches of $M$ through $p$ are given by $Y_j=\{y_{j1}=\dots=y_{jk}=0\}$.
    Assume that there is a smooth function $F\colon\O\to\R$ $($not necessarily Morse$)$ such that
    for each $j=1,\dots,d$ there exists an index $i_j$ with $\frac{\partial F}{\partial y_{ji_j}}(p)\neq 0$.

    Then there exist local coordinates $(\wt{x}_1,\dots,\wt{y}_{dk})$ near $p$ such that $Y_j=\{\wt{y}_{j1}=\dots=\wt{y}_{jk}=0\}$ for $j=1,\dots,d$
    and the function $F$ has the form
    \[F(\wt{x}_1,\dots,\wt{x}_m,\wt{y}_{11},\dots,\wt{y}_{dk})=F_1(\wt{x}_1,\dots,\wt{x}_m)+\sum_{j=1}^d\wt{y}_{j1}\]
    for some function $F_1$.
\end{lemma}

First we show how to prove the Immersed Morse Lemma from Lemma~\ref{lem:normal}.
Take coordinates $(x_1,\dots,y_{dk})$ near $p$ as in Lemma~\ref{lem:normal}.
Note that $\frac{\partial F}{\partial x_i}(p)=0$ for all $i=1,\dots,m$ for otherwise $p$ is not a critical point of $\O[d]$ (the coordinates
$(x_1,\dots,x_m)$ are local coordinates on $\O[d]$).
By item~\ref{item:IM2} of Definition~\ref{def:immersed_morse_function} we claim that for any $j=1,\dots,d$, there exists an index $i_j$ such that
$\frac{\partial F}{\partial y_{ji_j}}(p)\neq 0$. Indeed, if such an index does not exist for given $j$, the function $F$
restricted to $Y_1\cap\dots\cap Y_{j-1}\cap Y_{j+1}\cap\cdots\cap Y_d$ would have a critical
point at $p$.

By applying Lemma~\ref{lem:normal} we obtain that, after changing coordinates, $F(x_1,\dots,y_{dk})=F_1(x_1,\dots,x_m)+\sum_{j=1}^d y_{d1}$.
Clearly $F_1=F|_{\O[d]}$ so by item~\ref{item:IM1} of Definition~\ref{def:immersed_morse_function}, $F_1$ has a nondegenerate critical point at $p$. We apply the standard Morse Lemma to $F_1$ (see e.g.~\cite{Mi1}) to obtain a
change of coordinates of $x_1,\ldots,x_m$ to turn $F_1$ into a sum of quadratic terms. This concludes the proof of the Immersed Morse Lemma, modulo the proof of Lemma~\ref{lem:normal}, which give next.
\end{proof}

\begin{proof}[Proof of Lemma~\ref{lem:normal}]
  The proof uses induction in the next claim.  The precise statement we prove is the following.

  \emph{Claim.} For each $j=1,\dots,d$, there are local coordinates $(\wt{x}_1,\dots,\wt{y}_{dk})$ near $p$ such that $Y_j=\{\wt{y}_{j1}=\dots=\wt{y}_{jk}=0\}$
  and the function $F$ has the form
  \[F(\wt{x}_1,\dots,\wt{x}_m,\wt{y}_{11},\dots,\wt{y}_{dk})=F_1(\wt{x}_1,\dots,\wt{x}_m,\wt{y}_{11},\dots,\wt{y}_{jk})+\sum_{j'=j+1}^d\wt{y}_{j'1}.\]
    for some function $F_1$.

  Note that for $j=d$ the claim is trivial (take $x_i = \wt{x}_i$, $y_i = \wt{y}_i$ and $F_1=F$), while for $j=0$ the claim corresponds precisely to the statement of Lemma~\ref{lem:normal}. Now to prove the claim
  for $j-1$, assuming the claim for $j$, we use the inverse function theorem.
  For simplicity assume that variables for which the claim for $j$ holds are without tildes and the variables that we strive to define
  are going to have the tilde sign.

  First number the variables in such a way that $\frac{\partial F}{\partial y_{j1}}(p)\neq 0$.
  Define
  \[\wt{y}_{j1}=F_1(x_1,\dots,x_m,y_{11},\dots,y_{jk})-F_1(x_1,\dots,x_m,y_{11},\dots,y_{j-1,k},0,\dots,0).\]
  With this variable and $\wt{F}_1=F_1(x_1,\dots,x_m,y_{11},y_{j-1,k},0,\dots,0)$ we have
  \[F_1(x_1,\dots,x_m,y_{11},\dots,y_{jk})=\wt{F}_1(x_1,\dots,y_{j-1,k})+\wt{y}_{j1}.\]
  Therefore we need to make sure that the change of variables
  \begin{equation}\label{eq:local_variable_change}(x_1,\dots,x_m,y_{11},\dots,y_{jk})\mapsto (x_1,\dots,y_{j-1,k},\wt{y}_{j1},y_{j2},\dots,y_{jk})\end{equation}
  is a local diffeomorphism that preserves strata.

  Verification that the change \eqref{eq:local_variable_change} is a local diffeomorphism is a straightforward application of the inverse
  function theorem together with the observation that $\frac{\partial F}{\partial y_{j1}}(p)\neq 0$. Clearly the image of the set
  $\{y_{i1}=\dots=y_{ik}=0\}$ for $i<j$ is again the set $\{y_{i1}=\dots=y_{ik}=0\}$ after the change \eqref{eq:local_variable_change}. Therefore
  the change of variables preserves strata $Y_1,\dots,Y_{j-1}$. The strata $Y_{j+1},\dots,Y_d$ are unaffected. To see that $Y_j$ is preserved,
  note that locally near $p$ we have $y_{j1}=\dots=y_{jk}=0$ if and only if $\wt{y}_{j1}=y_{j2}=\dots=y_{jk}=0$. This uses again that $\frac{\partial F}{\partial y_{j1}}(p)\neq 0$.

  The change \eqref{eq:local_variable_change} provides the necessary induction step.
\end{proof}

\begin{theorem}[Density of Morse functions]\label{thm:density}
For every generic immersion $G:N \looparrowright \O$, there exists an immersed Morse function for $G$. Furthermore the set of immersed Morse functions is open and dense in $C^2(\Omega;\R)$.
\end{theorem}

While the theorem admits a direct proof using the transversality theorem, we will deduce it from Corollary~\ref{cor:paths}, which will be proven in Subsection~\ref{sub:immersed_cerf}.
The proof does not use material from Sections \ref{section:grim-vector-fields} -- \ref{sec:integrate}.
%\npar{Added this, I hope it is true...}{Rephrased slightly}

\section{Gradient like immersed (grim) vector fields}\label{section:grim-vector-fields}

The gradient of a Morse function is a vector field with specific properties. A gradient-like vector field with respect to an immersed Morse function $F$ has many of the key properties of the gradient of a Morse function, but is modified in the neighbourhood of the critical points so that the critical points of $F$ on all the strata are also critical points of the vector field.  %Without modification it would only be the ordinary critical points of depth zero that are also critical points of the vector field.
It is technically extremely useful to have the notion of gradient-like vector fields. They perform the same function as gradients, but are much more flexible.

\begin{definition}[Grim vector field]\label{def:grim}
  A vector field $\xi$ on $\Omega$ is called a \emph{gradient-like immersed vector field} with respect to an immersed Morse function $F$,
  for short a \emph{grim} vector field, if $\xi$ has the following properties.
  \begin{enumerate}[label=(G-\arabic*), series=grim]
    \item\label{item:G1}(Positivity)
      The directional derivative $\partial_{\xi} F$ is nonnegative and vanishes precisely at the critical points of $F$.
    \item\label{item:G2}(Tangency) If $w\in \O[d]$, then $\xi(w)\in T_{w}\O[d]$. That is, $\xi$ is tangent to the stratum.
    \item\label{item:G3} (Local behaviour at singular points) If $p\in \O[d]$ is a critical point of $\xi$, then there are local coordinates centred at $p$
      \[(x_1,\ldots,x_m,y_{11},\ldots,y_{1k},\ldots,y_{d1},\ldots,y_{dk})\]
      (with $m=n+k-dk$), such that $\O[d]$ is given locally by
      \[\{y_{11}=\dots=y_{1k}=0\}\cup\cdots\cup\{y_{d1}=\dots=y_{dk}=0\},\]
      the function $F$ is given by \eqref{eq:formofimmersed}, that is:
      \[ F(x_1,\ldots,y_{dk})=F(p)-x_1^2-\dots-x_h^2+x_{h+1}^2+\dots+x_m^2+\sum_{j=1}^d y_{j 1}\]
      for some $h$ with $0 \leq h \leq m$, and in these coordinates $\xi$ has the form
      \begin{equation}\label{eq:localgrim}
\xi=(-x_1,\ldots,-x_h,x_{h+1},\dots,x_m,\sum_{i=1}^ky_{1i}^2,0,\ldots,0,\sum_{i=1}^ky_{2i}^2,0,\ldots,0,\sum_{i=1}^ky_{di}^2,0,\ldots,0).
      \end{equation}
      In particular the $\frac{\partial}{\partial y_{ji}}$-coordinate of $\xi$ is $\sum_{\ell=1}^ky_{j\ell}^2$ for $i=1$ and is zero for $i=2,\ldots, k$.
  \end{enumerate}
\end{definition}

Note that in the local coordinates around the critical points, the difference between $\nabla F$ and $\xi$ are the terms $\sum_{i=1}^ky_{j i}^2$ in $\xi$, whereas the corresponding term in $\nabla F$ is $1$. The r\^{o}le in $\xi$ of the term $\sum_{i=1}^ky_{j i}^2$ is to arrange that $\xi$ vanishes at each critical point.  We do not have a cubic term in $F$, since if we did that, at critical points on deeper strata, $F$ would also have a critical point considered as a function on $\O$. The linear terms we use mean that critical points on deeper strata are not also critical points of $F\colon \O \to \R$.

\begin{figure}
  %\missingfigure{Trajectories of $\xi$ near a critical point}
  \input{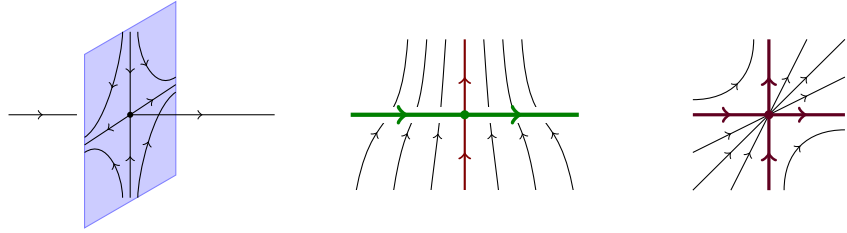}
  \caption{Trajectories of a grim vector field near a critical point. Left picture: a critical point of depth one and codimension two. Middle picture: a critical point of depth one and codimension two (the manifold $Y$ is the horizontal line, the `vertical' trajectories that are drawn go below the line). Right picture: a critical point of depth two and codimension one (the stratified manifold is the union of the central horizontal and vertical lines).}\label{fig:traj}
\end{figure}

\subsection{Existence of grim vector fields}

\begin{theorem}[Existence of grim vector fields]\label{thm:grim_exist}
	Every immersed Morse function $F$ admits a grim vector field. Moreover, for every immersed Morse function $F$ and every grim vector field $\xi$ for it,
	there exists a Riemannian metric on $\Omega$ such that $\xi=\nabla F$ away from an arbitrarily small neighbourhood of the set of critical points
	of $F$.
\end{theorem}

\begin{remark*}
	The conditions \eqref{eq:formofimmersed} and \eqref{eq:localgrim}  exclude the possibility that $\xi=\nabla F$ near a critical point of $F$
	restricted to each stratum.
\end{remark*}

\begin{proof}
  The proof is similar to the corresponding proof in~\cite[Section 3]{BP}.
	Let $p_1,\dots,p_s$ be the critical points of $F$. For each $i\in\{1,\dots,s\}$, choose an open neighbourhood $U_i'$ of $p_i$
	such that $F$ has the required special form of Definition~\ref{def:immersed_morse_function}.
	Choose $U_i$ to be a neighbourhood of $p_i$ such that $\ol{U}_i \subseteq U_i'$, and let $V$ be
	an open subset of $\O$ such that $V\cup\bigcup U_i'=\O$ and $U_j\cap V=\emptyset$ for any $i\in\{1,\dots,s\}$. On each of the $U_j'$ we
	choose a metric $g_i$ in which the local coordinate system of \eqref{eq:formofimmersed} is orthonormal. Define $\xi_i$ on $U_i$ by
	\eqref{eq:localgrim}.
	
	Let us now define the metric on $V$. Let $u\in V$ be a point in the $d$--th stratum. Choose a coordinate system
	$\{x_1,\ldots,x_m,y_{11},\ldots,y_{dk}\}$ (with $m=n+k-dk$ as usual)
	on an open set $U_u$
	containing $u$ and contained in $V$ such that the $j$--th branch of $\O[d]$ is given by $y_{j1}=\cdots=y_{jk}=0$. As $u\in V$,
	$u$ is not a critical point of $F|_{\O[d]}$. Therefore, there exists an index $i\in\{1,\ldots,m\}$ such that $\frac{\partial F}{\partial x_i}(u)\neq 0$.
	Without loss of generality we suppose it is $x_1$ and $\epsilon\in\{-1,1\}$ is such that
	$\epsilon\frac{\partial F}{\partial x_1}(u)>0$. Shrink $U_u$ if necessary
	in order to guarantee that $\epsilon\frac{\partial F}{\partial x_1}>\delta_u>0$ for some $\delta_u>0$ at all points in $U_u$.
	Now for each $u'\in U_u$ there exists a positive definite (and in particular invertible) symmetric matrix $A(u')$ such that the first row, and by symmetry also the first column,
	is \[\left(\epsilon\frac{\partial F}{\partial x_1}(u'),\ldots,\epsilon\frac{\partial F}{\partial y_{dk}}(u')\right).\]
The sign of the entries other than the first is not controlled, but by choosing the remaining entries of the matrix judiciously we can arrange for the entire matrix to be positive definite.
	We can also assume that the coefficients of $A(u')$ depend smoothly on $u'$. We define the metric $g_u$ on $U_u$ such that $A(u')$ is the matrix of the scalar product
	of vectors in $T_{u'}\Omega$ in the chosen coordinate system.

	The gradient of $F$ in the metric $g_u$ satisfies by definition that $\nabla F(u')^TA(u')\nu=dF(\nu)(u')$
	for all $\nu\in T_{u'}\Omega$ and for all $u' \in U_u$. As the first row and column of $A(u')$ is $\epsilon$
	times the vector
	of partial derivatives of $F$, it follows that $\nabla F(u')^T=(\epsilon,0,\dots,0)=
	\epsilon\frac{\partial}{\partial x_1}$. Hence, $\nabla F(u')$ is tangent
	to all branches of $M$ passing through $u'$. Define $\xi_u := \nabla F$
	in $U_u$. Note that $\partial_{\xi_u}\nabla F>0$ on the whole of $U_u$.

	Cover $\Omega$ by the sets $U_i'$ for $i=1,\dots,s$ and $U_u$ for $u\in V$.
	Let $\phi_1,\ldots,\phi_s,\{\phi_u\}_{u\in V}$ be a partition of unity subordinate
	to this covering. We define $\xi=\sum\phi_i\xi_i+\sum\phi_u\xi_u$ and $g=\sum\phi_i g_i+\sum\phi_u g_u$. From the construction we see that
	\begin{itemize}
	  \item $\xi$ has the local form \eqref{eq:localgrim} in each of the $U_i$, because $\xi=\xi_i$ in $U_i$;
	  \item if $w\in \O[d]$, then $\xi\in T_w \O[d]$, because all the $\xi_i$ and $\xi_u$
	    had this property and belonging to $T_w \O[d]$ is preserved under the linear
	    combination operation;
	  \item $\partial_{\xi}F\ge 0$ with equalities precisely at the critical points of $F$.
	    This follows from the fact that $\partial_{\xi_i}F\ge 0$ and $\partial_{\xi_u}F>0$.
	\end{itemize}
	This shows that $\xi$ is a grim vector field for $F$. Moreover, away from
	the neighbourhood $\bigcup_i U'_i$ of the critical points of $F$, we have $\nabla F=\xi$.
\end{proof}

An argument analogous to the proof of Theorem~\ref{thm:grim_exist} can be used to prove the following result.
Its rough meaning is that for a given Morse function $F$ we can extend a vector field $\eta$
from the submanifold $M$ to a grim vector field on the whole of $\O$.
The formulation is given in the form that best suits applications in
Part~\ref{part:pathlifting} below.

\begin{proposition}\label{prop:grim_extend}
  Let $F\colon\O\to\R$ be an immersed Morse function.
  Suppose $U\subseteq M$ is an open subset of $M$ contained entirely in the first stratum of $M$
  $($if $M$ is embedded we can take $U=M)$.
  Let $\eta$ be a vector field on $U$ that is tangent to $M$ and which is a gradient-like vector field for $F|_U$. Then for any open subset $V\subseteq U$ such that
  $\ol{V}\subseteq U$ there exists a grim vector field $\xi$ for $F$ such that $\xi|_V=\eta$.
  If $M$ is embedded and $U=M$, we can take $V=M$.
\end{proposition}

\begin{remark}\label{rem:grim_extend}
  Proposition~\ref{prop:grim_extend} is a generalisation of the following classical result. If $f\colon M\to\R$ is a Morse function, $U$ is an open set and $\eta$ is a vector field
  that is gradient-like for $f$ on $U$, then for any open subset $V$ such that $\ol{V}\subseteq U$, there exists a gradient-like vector field $\eta'$ on $M$ for $f$ with $\eta'|_V=\eta|_V$. This follows from Proposition~\ref{prop:grim_extend} by taking $M=\Omega$.
\end{remark}

\begin{proof}[Proof of Proposition~\ref{prop:grim_extend}]
The idea of the proof is first to extend the vector field $\eta$ to an open subset of $\O$
containing $V$ and glue it with some grim vector field on $\O$.

  Let $p_1,\dots,p_s$ be the critical points of $F$ contained in $U$. For each such point~$p_i$ there exist an open set $U_{p_i}$ and
  coordinates $x_1,\dots,x_n,y_1,\dots,y_k$
  such that
  \begin{itemize}
    \item $M$ is given by $\{y_1=\dots=y_k=0\}$;
    \item the vector field $\eta$ on $U_{p_i}\cap M$ is equal to $(-x_1,\dots,-x_h,x_{h+1},\dots,x_n)$ for some $h$. Note that $\eta$ has $n$ coordinates, and not $n+k$, because $\eta$ is a vector field
      on $M$ and not on~$\O$;
    \item with $h$ from the previous item, we have
  $F=-x_1^2-\dots-x_h^2+x_{h+1}^2+\dots+x_n^2+y_1+F(p_i)$.
  \end{itemize}
  The second item follows from the fact that $\eta$ is gradient like on $\O$.
 We extend the vector field to the whole of $U_{p_i}$ via
\[\xi_{p_i}=(-x_1,\dots,-x_h,x_{h+1},\dots,x_n,\sum_{\ell=1}^k y_\ell^2,0,\dots,0).\]

  Suppose $w\in U$ is not a critical point of $F$.
  Choose a neighbourhood $U_w$ of $w$ in $\Omega$
  such that it
  does not contain any critical points of $F$.
  By the rectifiability lemma for vector fields, see e.g.~\cite[Sections~7~and~32]{Arnold_ODE},
  on shrinking $U_w$ if needed, we may assume that there are coordinates $(x_1,\dots,x_n,y_1,\dots,y_k)$ on $U_w$
  such that $M\cap U_w$ is given by $\{y_1=\dots=y_k=0\}$ and
  $\eta=\frac{\partial}{\partial x_1}$ on $M\cap U_w$.

  We define $\xi_w=\frac{\partial}{\partial x_1}$. Note that $\partial_\eta F(w')>0$
  for $w'\in U_w\cap M$. Therefore, shrinking further $U_w$ if necessary, we may
  and shall assume that $\partial_{\xi_w}F(w')>0$ for all $w'\in U_w$.
  Finally, let $\xi_\O$ be any grim vector field for $F$.

  The union of the sets $\{U_w\}$ and $\{U_{p_i}\}$ covers the closure $\ol{V}\subseteq M$.
  As $\ol{V}$ is compact, there exists an open set $U_V\subseteq\O$ containing $\ol{V}$
  and whose closure is contained in the union of $U_w$ and $U_{p_i}$. Define $U_\O=\O\setminus\ol{U}_V$. Then $U_\O$ is disjoint from $\ol{V}$. Also $U_\O$, together with the sets
  $U_w$ and $U_{p_i}$, cover the whole of $\O$. We use a partition of unity
  to glue the vector fields $\xi_w$ on $U_w$, $\xi_{p_i}$ on $U_{p_i}$, and $\xi_\O$
  on $U_\O$, to a vector field $\xi$.

  Every critical point of $F$ belongs to precisely one subset of the cover.
  Hence, the vector field $\xi$ satisfies \ref{item:G3} near each critical point of $F$
  (note that \ref{item:G3} is usually not preserved under taking a convex combination
  of two vector fields, so an extra argument was needed).
\end{proof}

In a similar spirit we can prove the following result.

\begin{proposition}\label{prop:regulartogrim}
  Suppose $F\colon \O\to\R$ is an immersed Morse function and $\xi$ is a vector field on $\O$ satisfying $\partial_\xi F\ge 0$, with equality
  precisely at critical points of $F$ $($including critical points of $F$ restricted to deeper strata$)$, and such that $\xi$ is tangent to all the strata.

  Then there exists a grim vector field $\xih$ agreeing with $\xi$ everywhere except perhaps small neighbourhoods of the critical points of $F$.
\end{proposition}

\begin{proof}
  We sketch a quick argument. By the assumptions, $\xi$ already satisfies \ref{item:G1} and \ref{item:G2}
  from Definition~\ref{def:grim}. We need to take care of \ref{item:G3};
  more specifically, we need to make sure that $\xi$ satisfies \eqref{eq:localgrim}
  in some local coordinates near each of the critical points of $F$.

  To guarantee this, take $p\in\O$ to be a critical point
  of $F$ and assume $p$ lies on the $d$-th stratum.
  By the Immersed Morse Lemma~\ref{lem:immersed_morse},
  we may choose a neighbourhood $U_p$ of $p$
   such that there are coordinates $\{x_1,\dots,x_n,y_{11},\dots,y_{dk}\}$ with the property that
  the $j$-th branch is given by $\{y_{j 1}=\dots=y_{j k}=0\}$ and $F$ has the form $-x_1^2-\dots-x_h^2+x_{h+1}^2+\dots+x_m^2+y_{11}+\dots+y_{d1}+F(p)$. Here $h$ is the index of $p$
  and $m=n+k-dk$.
  Keeping in mind \eqref{eq:localgrim}, define a vector field on $U_p$:
  \[\xi_p=(-x_1,\dots,-x_h,x_{h+1},\dots,x_m,\sum_i y_{1i}^2,0,\dots,0,\sum_i y_{2i}^2,0,\dots,0).\]
  Since both $F$ and $\xi_p$ are given by explicit formulae, we
  may promptly verify that $\partial_{\xi_p}F\ge 0$ everywhere on $U_p$
  with equality only at $p$.

  We will now glue $\xi$ with $\xi_p$. Let $\phi\colon\O\to[0,1]$ be a bump function equal to $1$ in a small neighbourhood of $p$ and vanishing away from $U$. Replace $\xi$ by $\xih=(1-\phi)\xi+\phi\xi_p$. With this definition, $\xih$ still satisfies \ref{item:G1} and \ref{item:G2}, because
  both conditions are preserved under taking convex linear combination of vector fields.
  Moreover, $\xih$ satisfies \ref{item:G3} near $p$.
  We perform the same procedure for all critical points. The new vector field satisfies axiom \ref{item:G3}
  at all critical points.
\end{proof}

\subsection{Pull-backs}\label{sub:pull_back}
Let $G\colon N\to \O$ be a generic immersion, let $M=G(N)$ and let $F\colon\O\to\R$ be an immersed Morse function for $(\O,M)$. Choose a grim vector field $\xi$ on $\O$ for $F$.
Define $f\colon N\to\R$ by $f(u)=F \circ G(u)$. This is a Morse function on $N$; we refer to it as the `pull-back' of $F$. We ask whether we can find a vector field $\eta$ on $N$ that is a `pull-back' of $\xi$, and which is a gradient-like vector field for $N$.

As our first approximation we could try $\eta_0(u)=(DG(u))^{-1}(\xi)$, where $DG$ denotes the derivative. Since $\xi$ is tangent to $M$ and $DG(u)$ is injective, the vector field $\eta_0$ is well-defined. Moreover, unless $p=G(q)$ is a critical
point of $F$, $\eta_0(u)\neq 0$ and $\partial_{\eta_0} f>0$. However, if $p$ is a critical point of $F$ on the stratum $\O[d]$ with $d\ge 2$, then $\xi$ vanishes
at $p$ and so $\eta_0$ vanishes at each point of the preimage of $p$, whereas $G^{-1}(p)$ consists of precisely $d$ points, none of them being a critical point of $f$
by item~\ref{item:IM2} in Definition~\ref{def:immersed_morse_function}. Therefore, $\eta_0$ is not a gradient-like vector field for $f$ and some alterations are needed. These are described in the next lemma.

\begin{lemma}[Pull-back Lemma]\label{lem:pull_back_lemma}
  For any open subset $U$ of $N$ containing all the preimages $G^{-1}(p)$ of critical points of $F$ at depth $d$ at least $2$, there exists
  a vector field $\eta$ on $N$ such that:
  \begin{enumerate}[label=(P-\arabic*)]
    \item $\eta=\eta_0$ away from $U$;\label{lem:equal}
    \item $\eta$ is gradient-like for $f$.
  \end{enumerate}
\end{lemma}
\begin{proof}
Suppose $p\in \O[d]$, $d\ge 2$, is a critical point of $F$.
There are local coordinates $(x_1,\dots,y_{dk})$ in $\O$ near $p$ such that $\xi$ has the form
\eqref{eq:localgrim}. The point $p$
has precisely $d$ preimages under $G$. Let $q$ be one such preimage.
Renumbering coordinates if needed, we may and will assume that $q$ corresponds to the branch of $M$ through $p$ given by $y_{d1}=\dots=y_{dk}=0$.

The local coordinates near $p$ induce local coordinates in a neighbourhood $U_q$ of $q$: these are
\[(x_1,\dots,x_m,y_{11},\dots,y_{d-1,k}).\]
In these coordinates we have
\begin{align*}
  f&=-x_1^2-\dots-x_h^2+x_{h+1}^2+\dots+x_m^2+y_{11}+\dots+y_{d-1,1}\\
\eta_0&= (-x_1,\dots,-x_h,x_{h+1},\dots,x_m,\sum y_{1i}^2,0,\dots,\sum y_{d-1,i}^2,0,\dots).
\end{align*}
Here, as usual, $h$ denotes the index of $p$ and $m=n+k-dk$.

Shrink $U_q$ if needed so that the closure of $U_q$ is contained in $U$ and the sets $U_q$ corresponding to different critical points
are pairwise disjoint. Choose another neighbourhood $U'_q$ of $q$ such that $\ol{U'}_q \subseteq U_q$. Consider a bump function
$\phi_q\colon N\to [0,1]$ supported on~$U_q$ and equal to~$1$ on~$U'_q$.
For $u\in U_q$ we define
\[\eta(u)=\eta_0(u)+\phi\cdot\sum_{j=1}^{d-1}\smfrac{\partial}{\partial y_{j1}}.\]
For $u \in N \sm \cup U_q$ define $\eta(u) = \eta_0(u)$.
We have $\smfrac{\partial}{\partial y_{j1}}f=1$ for all $j=1,\dots,d-1$,
so the new vector field satisfies $\partial_{\eta}f>0$ everywhere on $U_q$.

By construction, $\partial_\eta f>0$ in a neighbourhood of preimages of all critical
points of $F$ at depth $2$ or more. As mentioned above, this condition was the
only obstruction for $\eta$ to be a gradient-like vector field.
\end{proof}

\subsection{Handle attachments}

Now we study the changes in the topology of the level sets while crossing a critical level
of an immersed Morse function $F \colon \O \to \R$. Recall that $\dim N =n$ and $\dim \O = n+k$. Therefore $\O[d]$ has dimension $n+k-dk$.

\begin{proposition}\label{prop:immersedhandleattachments}
  Suppose $p\in F$ is an index $h$  immersed critical point at the $d$-th stratum, and suppose that $F^{-1}(F(p))$ has no other critical points. Choose $\varepsilon>0$ small
	enough so that $F$ restricted to $F^{-1}(F(p)-\varepsilon,F(p)+\varepsilon)$ has $p$ as its only critical point. For $j=0,1,2,\ldots$ denote
	by $\O[j]^\pm$ the intersection of the stratum $\O[j]$ with $F^{-1}(F(p)\pm\varepsilon)$.
	\begin{itemize}
		\item If $j>d$, then $\O[j]^+$ is diffeomorphic to $\O[j]^-$;
		\item $\O[d]^+$ arises from $\O[d]^-$ by a surgery on an embedded $S^{h-1} \times D^{m-h+1}$, that is a surgery of index equal to the index of $p$, where $m=\dim\O[d] = n+k-dk$.
	\end{itemize}
\end{proposition}
\begin{proof}
  The first part follows because on $\O[j]\cap F^{-1}(F(p)-\varepsilon,F(p)+\varepsilon)$ the vector field $\xi$ has no critical points. The second part is completely
  standard.
\end{proof}

\begin{remark*}
  If $j<d$, then $\O[j]^+$ usually differs from $\O[j]^{-}$. A precise description of the change uses a generalisation of the `rising water principle'; see \cite[Section 6.2]{GS}.
  Since we do not need it in the paper, we shall not dive further into this problem.
  We aim, however, to return to the rising water principle in a subsequent paper.% \cite{BP_future}.
\end{remark*}

We will need a generalisation of \cite[Lemma~4.7]{Mil65} and \cite[Lemma~3.10]{BP}, enabling us to realise an isotopy from the identity map of a level set to some other self-diffeomorphism of that level set as the flow of a gradient-like immersed vector field.  Recall that the \emph{flow} of a vector field $\xi$ on $\O$ is the 1-parameter family of diffeomorphisms of $\O$ generated by, or induced by, $\xi$. See \cite[pages~10--11]{Mi1} or \cite[Section 7]{Arnold_ODE}.

Suppose that $M$ is an immersed image in $\O$, that $F\colon \O \to \R$ is a Morse function and that $\xi$ is a grim vector field inducing a flow $\Xi_s$. Assume that $a,b\in\R$ are such that there are no critical points of $F$ in $F^{-1}[a,b]$. Use the flow $\Xi_s$ of $\xi$ to identify $F^{-1}[a,b]$ with $F^{-1}(a)\times[a,b]$, as follows.  For a point $z\in F^{-1}[a,b]$ we let $s_-$ and $s_+$ be the real numbers such that $F(\Xi_{s_-}(z))=a$ and
$F(\Xi_{s_+}(z))=b$. We map $z\in F^{-1}[a,b]$ to the pair $(\Xi_{s_-}(z),\frac{s_+b-s_-a}{s_+-s_-})\in F^{-1}(a)\times[a,b]$.   The inverse of this identification induces a diffeomorphism
\[\varphi \colon F^{-1}(a) = F^{-1}(a) \times \{b\} \to F^{-1}(b).\]

\begin{lemma}[Isotopy Insertion]\label{lem:isoinject}
  Let $F \colon \O \to \R$, $\xi$, $\varphi$, and $a,b \in \R$ be as above.
 Let $\Psi$ be a strata preserving self-diffeomorphism of $F^{-1}(a)$
  and suppose $\Psi$ is isotopic to the identity through strata-preserving
  diffeomorphisms.

  There exists a new grim vector field $\wt{\xi}$, agreeing with $\xi$
  away from $F^{-1}(a,b)$, such that the flow of $\wt{\xi}$ induces a strata-preserving diffeomorphism \[\wt{h} \colon (\O,M)\cap F^{-1}(a)  \to (\O,M)\cap F^{-1}(b)\] $($as described above$)$ such that
  $\wt{\varphi} = \varphi \circ \Psi.$
\end{lemma}

\begin{proof}
	We follow the proof of \cite[Lemma~4.7]{Mil65}.
	As described above, the flow $\Xi_s$ of $\xi$ yields an identification of $F^{-1}(a)\times[a,b]$ with $F^{-1}[a,b]$.
	Denote it by $\Phi\colon F^{-1}(a)\times[a,b]\to F^{-1}[a,b]$, so that $\Phi|_{F^{-1}(a) \times \{b\}} =\varphi \colon F^{-1}(a) \times \{b\} = F^{-1}(a) \to F^{-1}(b)$.
Let $\Psi_s\colon F^{-1}(a)\to F^{-1}(a)$, $s\in[a,b]$ be an isotopy such that $\Psi_a$ is the identity of $F^{-1}(a)$,
	and $\Psi_b=\Psi$. We require that $\Psi_s$ does not depend on $s$ for $s$ close to $a$ and $b$ and that $\Psi_s$ preserves the strata
	of $M$. Then $\Psi_s$ defines a map $\ol{\Psi}\colon F^{-1}(a)\times[a,b]\to F^{-1}(a)\times[a,b]$ by the formula
	$(w,s)\mapsto(\Psi_s(w),s)$.

	Define the new vector field $\wt{\xi}=D(\Phi\circ\ol{\Psi}\circ\Phi^{-1})\xi$
	(here $D$ denotes the derivative of the self-diffeomorphism of $F^{-1}[a,b]$).
	Note that since $\Phi$ and $\ol{\Psi}$ are both strata-preserving, and
	$\xi$ is tangent to the strata, the
	vector field $\wt{\xi}$ is also tangent to the strata of $M$.
	The assumption that
	$\Psi_s$ does not depend on $t$ for $t$ close to $\{a,b\}$ implies that $\xi$ agrees with $\wt{\xi}$ near $F^{-1}(\{a,b\})$, so
	we can extend the vector field $\wt{\xi}$ by $\xi$ to the whole of $\Omega$.

	By the definition of the flow, the flow of $\wt{\xi}$
	induces a diffeomorphism $\wt{\varphi}\colon F^{-1}(a) \to F^{-1}(b)$ with
	$\wt{\varphi} = \varphi \circ \Psi$, as required.
\end{proof}

The following simple result is a very useful feature of grim vector fields. It is a generalisation of Proposition~\ref{prop:immersedhandleattachments}.  It says that if we have an immersed Morse function with no critical points on the zeroth and first strata in the inverse image of an interval, then we obtain a regular homotopy.  We will use this as part of the proof that immersed link concordance implies regular link homotopy.

%	Suppose $\xi$ is a grim vector field for $F$ and $p\in \O[d]$ is a critical point of $\xi$ with $F(p)=c$. Suppose
%	$d>1$ and that $\xi$ has no other critical points in $F^{-1}[c-\varepsilon,c+\varepsilon]$. Then $M\cap F^{-1}[c-\varepsilon,c+\varepsilon]$ is a regular homotopy between $M\cap F^{-1}(c-\varepsilon)$ and $M\cap F^{-1}(c+\varepsilon)$.

%\begin{proof}
  %It is enough to prove the existence of the regular homotopy in a coordinate neighbourhood of $p$.
  %Notice that the function $F$ restricted to the $j$-th branch does not have a critical point in $F^{-1}[c-\varepsilon,c+\varepsilon]$, even at $p$ (see Definition~\ref{def:immersed_morse_function}).
  %Thus the level sets of the function $F$ restricted to the $j$-th branch are isotopic.
  %Therefore passing from $F^{-1}(c-\delta)$ to $F^{-1}(c+\delta)$ induces an isotopy of
  %each of the branches of $M\cap F^{-1}(c)$. Now isotopy of branches implies a regular homotopy of their union.
%\end{proof}

%Here is a consequence of the Crossing Deeper Strata Theorem~\ref{thm:deeperstrata}.

%\begin{corollary}\label{cor:trace}

\begin{theorem}[Crossing Deeper Strata]\label{thm:deeperstrata}
	Suppose $F\colon \O \to[0,1]$ is an immersed Morse function and that $F$ does not have any critical points on the zeroth and first stratum. Then $\O\cong \O_0\times [0,1]$ where $\O_0=F^{-1}(0)$, and there exists a regular homotopy $H_t\colon M_0\to\O_0$,  where $M_0:=M\cap\O_0$, such that $M$ is the trace
	of the regular homotopy $H_t$. That is, $M=\ol{H}(M_0\times[0,1])$ where $\ol{H}(w,t)=(H_t(w),t)$.
\end{theorem}

\begin{proof}
	Since $F$ does not have critical points on the zeroth stratum, $F$ regarded as a Morse function on $\O$ has no critical points at all, so
	$\O \cong \O_0\times[0,1]$. As $M$ is the image of a generic immersion, we have $M=G(N)$ for some $N$. Define
	$f\colon N\to [0,1]$ by $f(w)=F(G(w))$. As~$F$ does not have critical points on the first stratum, $f$ is a Morse function without
	critical points, hence we can identify $N$ with $M_0\times[0,1]$ in such a way that $f$ is the projection onto the second factor. The homotopy
	$H_t$ is constructed as follows: for $x \in M_0$, let  $$H_t(x)=G(x,t) \in \O_0 \times \{t\} \cong \O_0.$$  %\ypar{changed}

	We claim that $H_t$ is a regular homotopy. To see this, we need to show that $DH_t(x)$ is a monomorphism for all $(x,t)\in N\times[0,1]$. To see this, we take $w=(x,t)\in N\times[0,1]$. The map $G$ being an immersion means that $DG(w)$ is a monomorphism. As $f(w)=F(G(w))$,
	$DG$ takes $\ker Df$ to $\ker DF$. Since $Df$ is non-degenerate, $\ker Df$ is of codimension one in $T_w (N\times[0,1])$. If $DG|_{\ker Df}$ has nontrivial kernel at $w$, then $\dim DG(\ker Df(w))<\dim N$. And then $\dim DG(T_w(N\times[0,1]))<\dim N+1$, contradicting the fact
	that $DG$ is non-degenerate. But $DG|_{\ker Df}=DH_t$. So $DH_t$ is a monomorphism.
\end{proof}
%\ypar{MB. This is the newly added part}

The homotopy constructed in Theorem~\ref{thm:deeperstrata} might be not sufficient to prove that concordance implies
regular link homotopy.
In that setting, we will be given an identification of $\O$ with $\O_0 \times [0,1]$ and $N$ as $N_0 \times [0,1]$.  Theorem~\ref{thm:deeperstrata}  gives (possibly) different such product structures, and so does not give the desired regular homotopy.
We have the following improvement of Theorem~\ref{thm:deeperstrata}, under stronger assumptions.

\begin{proposition}\label{prop:crossingforhomotopy}
  Suppose $N=N_0\times[0,1]$, $\Omega=\Omega_0\times[0,1]$, and $G_\tau\colon N\to \Omega$ is a family of immersions such that
  $G_0(N_0\times\{0\})\subseteq \Omega_0\times\{0\}$, $G_0(N_0\times\{1\})\subseteq \Omega_0\times\{1\}$, and $G_\tau$ does not depend on $\tau$
  on $N_0\times\{0,1\}$. Assume that there exists a path
  of functions $F_\tau\colon\Omega\to[0,1]$, $\tau\in[0,1]$, such that $F_\tau$ is Morse $($as a function on $\Omega)$,
  $F_\tau^{-1}(0)=\Omega_0\times\{0\}$, $F_\tau^{-1}(1)=\Omega_0\times\{1\}$,
  $F_0$ is projection onto the second factor, and $F_1$ is an immersed Morse function for $G_1(N)$
  with no critical points on the first stratum.

  Then the maps \[H_i=N_0\to N_0\times\{i\}\xrightarrow{G_0|_{N_0 \times \{i\}}} \Omega\times\{i\}\to \Omega,\] $i=0,1$,
  are regularly homotopic. Moreover, if $N$ is disconnected and for each $\tau \in [0,1]$, we have that $G_\tau$ does not create intersections between different components of $N$, then neither does the regular homotopy between $H_0$ and $H_1$.
\end{proposition}

\begin{proof}
  As $F_0$ has no critical points (as a Morse function on $\Omega$), none of the $F_\tau$ have critical points either.
  Stability of Morse functions (see \cite[Proposition III.2.2]{GG}), lets us find a family $\Psi_\tau\colon \Omega\to\Omega$,
  preserving $\Omega_0\times\{0,1\}$, and a family of strictly increasing maps $\lambda_\tau\colon[0,1]\to[0,1]$, such that
  \[F_\tau = \lambda_\tau\circ F_0\circ \Psi_\tau.\]
  The family $\Psi_\tau$ can be assumed to be the identity on $\Omega_0\times\{0\}$ and to preserve $\Omega_0\times\{1\}$.
  For simplicity of the argument we may and will replace the path $F_\tau$ by $\lambda_\tau^{-1}\circ F_\tau$ so that we have $\lambda_\tau$ being the identity for all $\tau$.

  Now define a new family of immersions, $G'_{\tau} := \Psi_\tau\circ G_\tau$.  Then, $F_0\circ G'_\tau=F_\tau\circ G_\tau$. In particular, since by hypothesis $F_1$ has no critical points of depth one, with respect to the stratification determined by $G_1(N)$, it follows that $F_0$ has no critical points at depth $1$ with respect to the stratification determined by $G'_1(N)$. Thus by Theorem~\ref{thm:deeperstrata}, the composition $F_0\circ G'_1$ induces a product structure on $N$, which we denote by $\psi\colon N\xrightarrow{\cong}  N_0\times[0,1]$.
  Define a family of maps as
  \[H'_t\colon N_0\to N_0\times\{t\}\subseteq N\xrightarrow{G'_1}\Omega_0\times\{t\}=\Omega_0.\]
  %giving a regular homotopy between $H'_0$ and $H'_1$.
  By hypothesis, $G_\tau$ is independent of $\tau$ on $N_0\times\{0,1\}$, and moreover we arranged for $\Psi_\tau$ to act as the identity
  near $\Omega_0\times\{0,1\}$. Hence $H'_0=H_0$ and $H'_1=H_1$, and so $H'_t$ gives a regular homotopy between $H_0$ and $H_1$..

It remains to prove that if $G_\tau$ keeps connected components of $N$ separate, then so does the regular homotopy. Note that $G_\tau$ separates connected
  components if and only if $G'_\tau$ does. Hence in particular $G'_1$ separates connected components. But the homotopy $H_t$ in the proof of Theorem~\ref{thm:deeperstrata} is constructed by tracing along what is called $G(N)$ in the notation of that theorem, the role of which in the present proposition is played by $G'_1(N) = G_1'(N_0 \times [0,1]$.
  Hence  no intersections between different components are created.
\end{proof}

\section{Membranes}\label{sec:hypermembranes}

\subsection{Definitions}
Recall that $M^n \subseteq \Omega^{n+k}$ is a generically immersed manifold, $F\colon\Omega\to\R$ is an immersed Morse function and $\xi$
is a grim vector field for $F$.
In the case of a critical point on $\O\setminus M$, as in classical Morse theory,
the stable and unstable manifolds are discs (at least near the critical point), whose intersections with level sets $F^{-1}(c)$ are spheres whose dimension depends
on the index of the critical point.  If $p\in \O[d]$
is a critical point of $\xi$ and $d>1$,
the linear part of $\xi$ is degenerate at $p$.

Since we know the form of $\xi$
in a neighbourhood of $p$, we can explicitly describe the dynamics of the vector field near $p$. The objects of interest to us are
the set of points that are attracted to the critical point (in classical Morse theory, the stable or descending manifold) and the set of points that are
repelled (in classical Morse theory, this corresponds to the unstable or ascending manifold).

\begin{definition}[Membranes]\label{def:hypermembrane}
  For a critical point $p$ of $\xi$, the ascending (respectively descending) membrane $\Ha(p)$ (respectively $\Hd(p)$) the set of points $w$ such that a trajectory of $\xi$ through $w$ reaches $p$ in the infinite past (respectively in the infinite future).
\end{definition}

The notion is a straightforward generalisation of the membranes of Perron~\cite{Pe} and Sharpe~\cite{Sha}, which were defined for embedded submanifolds; see also \cite{BP}.

In local coordinates \eqref{eq:localgrim} the descending
membrane is given by
\begin{multline*}
\{x_{h+1}=\dots=x_{r}=0=y_{12}=\dots=y_{1k}=y_{22}=\dots=y_{2k}=\dots=y_{dk}\}\cap
\{y_{11}\le 0,\,y_{21}\le 0,\ldots,y_{d1}\le 0\}.\end{multline*}
Similarly, the ascending membrane is given by
\begin{multline*}
\{x_{1}=\dots=x_{h}=0=y_{12}=\dots=y_{1k}=y_{22}=\dots=y_{2k}=\dots=y_{dk}\}\cap \{y_{11}\ge 0,\,y_{21}\ge 0,\ldots,y_{d1}\ge 0\}.
\end{multline*}

The ascending and descending membranes	
  $\Ha(p)$ and $\Hd(p)$ are manifolds-with-corners. In particular they
  are stratified manifolds in the sense of Section~\ref{sec:stratified_manifolds}:
  the $d$-th stratum of $\Ha(p)$ is the intersection of $\Ha(p)$ with $\O[d]$.
  In particular, we will be able to use Stratified General Position Lemma~\ref{disjoint}
  to study intersections of membranes of different critical points.
  As a first step in this direction, we prove the following result.

\begin{lemma}\label{lem:dimension}
  Given a critical point $p$ of $\xi$ of index $h$ and depth $d$, for any $j\le d$ we have $\dim\Hd(p)\cap \O[j]=h+d-j$,
  and $\dim\Ha(p)\cap \O[j]=n-k(d-1)-h+d-j$.
\end{lemma}

\begin{proof}
	The stratum at depth $d$ has dimension $n-k(d-1)$. The stable manifold at the $d$-th stratum has dimension $h$ and the unstable manifold has
	the complementary dimension $n-k(d-1)-h$. This gives the claimed answer when $d=j$.  Now the dimension of the membrane at each stratum increases by one as we go to a more shallow level (i.e.\ decreasing~$j$ by one).
\end{proof}

\subsection{The Morse-Smale condition}\label{sec:morse-smale-condn}
In classical Morse theory, the Morse--Smale condition means that stable and unstable manifolds of different
critical points of a vector field intersect transversely. The generalisation to the case of grim vector fields is a straightforward extension of the
approach of \cite{BP}.

\begin{definition}\label{def:Morse_smale_immersed}
	Let $M\subseteq\Omega$ be an immersed manifold, $F \colon \O \to \R$ a Morse function and $\xi$ a grim vector field for $F$.
	We will say that $\xi$ satisfies the \emph{Morse--Smale condition}
	if for any two critical points $p$, $p'$ of $\xi$,
the ascending membrane of $p$ intersects the descending membrane of $p'$ transversely, in the sense of the intersection of stratified manifolds in the stratified manifold~$\O$.

More specifically, for every $j=0,1,\ldots$, the intersection of membranes $\Ha(p)\cap \O[j]$ and $\Hd(p')\cap \O[j]$ is required to be transverse in $\O[j]$.
\end{definition}

\begin{proposition}
  Every grim vector field can be perturbed outside of critical level sets to a vector field that satisfies the Morse--Smale condition.
  %\npar{Is this really possible. Let $p,q,r$ be critical points with $F(p) < F(q) <F(r)$. Suppose that the membranes of $p$ and $r$ do not intersect transversely. Then is seems to me we are going to have to perturb the vector field in the level set of $q$ in order to make them transverse. I guess what Milnor shows is that to achieve transversality here we perturb in $F^{-1}(F(p),F(r))$, so outside of the level sets of $p$ and $r$. }{Let's not go into details, please. We can perturb the descending membrane of $q$ to make it transverse to the asc. mem. of $p$, and to that between level sets $p$ and $q$. This done, we perturbe the desc. memb. of $r$ between level sets $q,r$, to be transverse to the asc.mem of $p$ and $q$ simultaneously.}
\end{proposition}

\begin{proof}[Sketch of proof]
  The proof is essentially a repetition of the proof given in \cite[Theorem 5.2]{Mil65}. The only difference is that we use the fact that we need to use
  the analogue of \cite[Lemma 5.3]{Mil65} in the stratified category, but this is essentially the statement of Lemma~\ref{disjoint}.
\end{proof}

For Morse--Smale grim vector fields we can calculate the dimension of the intersections of membranes.

\begin{lemma}\label{lem:intersectiondimension}
	Suppose a grim vector field $\xi$ is Morse--Smale and $p,p'$ are two critical points with $F(p) < F(p')$, such that the index of $p$ is $h$ and the index of $p'$ is $h'$.
	Suppose that $p$ lies on the $d$--th stratum and $p'$ lies on the $d'$--th stratum. Then for every $c\in(F(p),F(p'))$ and for depth $j\ge 0$ we have
	\[\dim \Ha(p)\cap\Hd(p')\cap \O[j]\cap F^{-1}(c)= (h'+d')-(h+d)+(k-2)(j-d)-1.\]
\end{lemma}

\begin{proof}
	According to Lemma~\ref{lem:dimension} we have %$\dim\Ha^h(x)=r+i-h$ and $\dim\Hd^h(y)=n-k(j-1)-s+j-h$.
	\[\dim\Ha(p)\cap \O[j]=n-k(d-1)-h+d-j \text{ and } \dim\Hd(p')\cap \O[j]=h'+d'-j.\]
	We also have $\dim \O[j]=n-k(j-1).$ By transversality, the dimension of the intersection in $\O[j]$ is equal to
	\begin{align*}
  \left(h'+d'-j\right)+ &\left(n-k(d-1)-h+d-j\right)-\left(n-k(j-1)\right) \\ =& (h'+d')-(h+d)+(k-2)(j-d).
\end{align*}
	Intersecting $\Ha(p)\cap\Hd(p')\cap \O[j]$ with the level set
	$F^{-1}(c)$ drops the dimension by~$1$.
\end{proof}

The following corollary will be extremely useful.

\begin{corollary}\label{cor:highcodim}
	Suppose the codimension $k\ge 2$. With the notation of Lemma~\ref{lem:intersectiondimension}, if $\Ha(p)\cap\Hd(p')$ is nonempty, then $h'+d'>h+d$.
\end{corollary}

\begin{proof}
	By Lemma~\ref{lem:intersectiondimension}, %we have that
\[\dim \Ha(p)\cap\Hd(p')\cap \O[j]\cap F^{-1}(c)= (h'+d')-(h+d)+(k-2)(j-d)-1.\]
If $\Ha(p)\cap\Hd(p')\neq\emptyset$, then this number is nonnegative for some $j$.
Clearly $j\le d$, because by the tangency condition \ref{item:G2} the depth of a  limit point of a trajectory is never smaller than the depth of a generic point of the trajectory.

Since $k \geq 2$, it follows that $(k-2)(j-d) \le 0$.  Thus
\[ (h'+d')-(h+d) -1 \geq  (h'+d')-(h+d)+(k-2)(j-d)-1 \geq 0, \]
and therefore $h'+d' > h+d$ as claimed.
\end{proof}

Intersections of membranes give rise to obstructions to performing rearrangement and cancellation of critical points.
The contrapositive of Corollary~\ref{cor:highcodim} can be used to show that such intersections do not occur.

\section{Grim neighbourhoods and broken trajectories}\label{section:broken-trajectory-lemma}

\subsection{Broken trajectories}

  Let $F\colon\O\to\R$ be a Morse function and let $\xi$ be a grim vector field for $M$.
Recall that a \emph{trajectory} of $\xi$ is a curve $\gamma\colon\R\to\O$ such that $\frac{d}{dt}\gamma(t)=\xi(\gamma(t))$. The \emph{limit points} of the trajectory are $z:= \lim_{t\to\infty}\gamma(t)\in\O$ (forward limit) and $w:= \lim_{t\to-\infty}\gamma(t)\in\O$ (backward limit).
We say that $\gamma$ reaches $z$ in the \emph{infinite future} and $w$ in the \emph{infinite past}.

\begin{definition}\label{def:zigzags_everywhere}
Let $\xi$ be a smooth vector field in $\Omega$.
  \begin{itemize}
    \item[(a)] A \emph{broken trajectory} of $\xi$ is a finite collection $\gamma_1,\dots,\gamma_s$ of trajectories of $\xi$ such that the forward limit of each of the $\gamma_i$ is the backward limit of $\gamma_{i+1}$ fo $i=1,\dots,s-1$.
    \item[(b)] If $\Omega'$ is the closure of an open set of $\Omega$, a \emph{zigzag trajectory} in $\Omega'$ is a finite collection of trajectories $\gamma_1,\dots,\gamma_s$
      all contained in $\Omega'$ and such that one of the limit points of $\gamma_i$ coincides with one of the limit points of $\gamma_{i+1}$ for every $i = 1,\dots,s-1$; see Figure~\ref{fig:zigzag}.
  \end{itemize}
\end{definition}

\begin{figure}
  \input{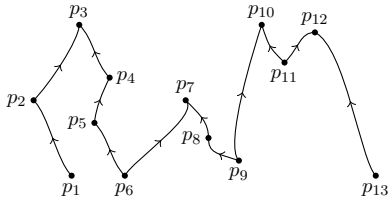}
  \caption{An example of a zigzag trajectory. All points $p_1,\dots,p_s$ are critical points
  of the vector field.}\label{fig:zigzag}
\end{figure}
\begin{definition}\label{def:relation}
  Let $F\colon\O\to\R$ be a Morse function and let $\xi$ be a grim vector field for $F$.
  Let $a<b$ be two non-critical values of $F$.
 For two points $w,z\in F^{-1}([a,b])$ we say that $w\sim_{0} z$ if there exists a $($subset of a$)$ trajectory of $\xi$ union its limit points, contained in $F^{-1}([a,b])$, connecting $w$ and $z$.
  We define  the relation $\sim_{\xi,a,b}$ as the transitive closure of the relation $\sim_{0}$.

  For any point $w\in F^{-1}[a,b]$ we define $\cK_{\xi,a,b}(w)$ to be closure of the set of points $z$ in $F^{-1}[a,b]$ such that $w\sim_{\xi,a,b} z$.
\end{definition}

\begin{remark*}
  We have that $w\sim_{\xi,a,b}z$ if and only if there is a zigzag trajectory of $\xi$ in $F^{-1}[a,b]$ connecting $w$ with $z$.
  %\npar{In the definition of zigzag trajectory, $\O'$ is open. But $F^{-1}[a,b]$  is not open. Should it be $F^{-1}(a,b)$ instead? Or something else?}{Change the property of $\Omega'$}
\end{remark*}

\subsection{Grim neighbourhoods}
In the study of dynamics of vector fields, there is the notion of an index pair; see e.g.~\cite{ConleyZehnder,Salamon}. If $\xi$ is a vector field,
an \emph{index pair} is a pair of sets $(X,L)$ such that $L\subseteq X$ is the exit set for $\xi$. That is, if $\gamma$ is a trajectory of $\xi$,
$\gamma(0)=x\in X$, then there exists $t_0>0$ such that for $t\in(0,t_0)$, if $x\in X\setminus L$ then $\gamma(t)\in X$, and if $x \in L$ then $\gamma(t)\notin X$.
 Index pairs are used to study the behaviour of $\xi$ near critical points.

\begin{definition}[Index triple]
  A triple $(X,L_\iin,L_\oout)$ forms an \emph{index triple} for $\xi$, if $(X,L_\oout)$ forms an index pair for $\xi$, and $(X,L_\iin)$ forms an index pair for $-\xi$.
\end{definition}

The definition means that any trajectory of $\xi$ that intersects $X$ enters $X$ through $L_\iin$ and exits through $L_\oout$.

\begin{example}\label{ex:index_point}
  Suppose $p\in\Omega$ is an isolated critical point of $\xi$. Then  any neighbourhood $U$ of $p$
  contains an index triple that contains~$p$ (see \cite[Theorem 4.3]{Sal2} for a more general statement).
\end{example}

In our applications, we will impose smoothness conditions on the index triple. This leads to a notion of a regular index triple.

\begin{definition}[Regular index triple]\label{def:regular_index_triple}
  A \emph{regular index triple} is a triple $(X,\piin X,\pout X)$ such that $X$ is a manifold-with-corners of codimension $0$ in $\O$, $\piin X,\pout X\subseteq\partial X$
  are manifolds with boundary (of codimension $0$ in $\partial X$), $\ptan  X:=\partial X\setminus (\piin X \cup \pout X)$ is a submanifold, and the corners of $\partial X$
  are at $\piin X\cap \ptan X$ and $\pout X\cap \ptan X$.
  We also require that $\piin X$  and $\pout X$  be disjoint.
\end{definition}

We point out that there is no ambiguity with writing $\bd X$ or $\partial X$ in Definition~\ref{def:regular_index_triple},
since $X$ is a codimension zero submanifold of $\O$, and so $\bd X = \partial X$, i.e.\ the manifold boundary and the point-set boundary/frontier agree.

It follows from the definition  that $\xi$ is everywhere tangent to $\ptan  X$: if not, the trajectory of~$\xi$ through a point in $\ptan X$
would leave $X$ immediately either in future or in the past, so that point would belong to either $\piin X$ or $\pout X$.

\begin{remark}\label{rem:remark}
For an arbitrary vector field, given an isolated critical point $p$, and a neighbourhood $V$ containing $p$ but no other critical points, there need not exist a regular index triple $X$ with $p \in X \subseteq V$.
However, if $p$ is a critical  point of a grim vector field $\xi$, then an index triple exists; see Lemma~\ref{lem:grimneigh} below.
%\mpar{Many remarks now seem to have a star, so they don't get numbered. But a couple don't. Should they all have stars? Or is it precisely the ones that need to be referred to later that keep their numbers?  I'd like to know the system/convention.}
\end{remark}

\begin{definition}[Grim neighbourhood]\label{def:grimneigh}
  Suppose $F$ is an immersed Morse function and $\xi$ is a grim vector field for $F$. A \emph{grim neighbourhood} of a set $C$
  is a regular triple index $(X,\piin X,\pout X)$ such that $C\subseteq X$, $X\setminus C$ has no critical points, and each of $\piin X,\pout X$ belong
  to a single level set of $F$.
\end{definition}

\begin{figure}
	\input{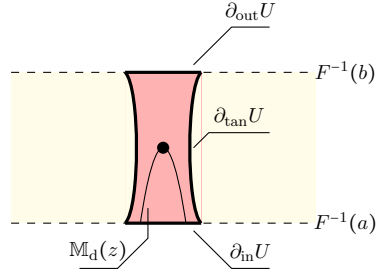}
	\caption{A grim neighbourhood.}\label{fig:grim}
\end{figure}

In most cases, we will consider a grim neighbourhood for an isolated critical point of $\xi$.
  We make the following observation.

  \begin{lemma}\label{lem:grim_observation}
    Let $U_C$ be a grim neighbourhood of a closed subset $C$ with $\partial_{\oout}U_C\subseteq F^{-1}(b)$ and $\partial_{\iin}U_C\subseteq F^{-1}(a)$ for some $a,b$ $($necessarily $a<b)$.

    Let $a',b'\in\R$ be such that $C\subseteq F^{-1}(a',b')$ and $a<a'$, $b'<b$. Then
    $U_C\cap F^{-1}(a',b')$ is also a grim neighbourhood of $C$.
  \end{lemma}
The proof of Lemma~\ref{lem:grim_observation} is completely straightforward and will be omitted. We pass to the construction of a grim neighbourhood.

\begin{lemma}\label{lem:grimneigh}
	Suppose $p\in\O$ is a critical point of $F$ and $a,b$ are such that there is no other critical point in $F^{-1}[a,b]$.
	Then for any open subset $V\subseteq F^{-1}(a)$ containing $\Hd(p)\cap F^{-1}(a)$ there exists a grim neighbourhood $U\subseteq\O$ of $p$ such that
	$\ol{V}=\partial_{\iin} U\subseteq F^{-1}(a)$ and $\partial_{\oout}(U)\subseteq F^{-1}(b)$.
\end{lemma}

\begin{proof}
  Let $W$ be the complement of $V$ in $F^{-1}(a)$. Define $U=F^{-1}(a,b)\setminus\Xi_{(a,b)}(W)$,
  where we recall that $\Xi_{(a,b)}(W)$ is the set of points in $F^{-1}(a,b)$ that are reached from $W$ by the flow of $\xi$.
  The verification that $U$ satisfies the desired conditions is straightforward.
\end{proof}

\begin{remark*}
  The same argument shows that if $V\subseteq F^{-1}(b)$ contains $\Ha(z)\cap F^{-1}(b)$, then there is a grim neighbourhood $U$ of $z$ such that $\partial_{\oout} U=\ol{V}$.
\end{remark*}

We can now state a result on grim neighbourhoods.

\begin{lemma}\label{lem:grimneighexist}
%  \textcolor{red}{This part is false. Luckily, it is never used.}
%  \textcolor{green}{Suppose $w\in F^{-1}[a,b]$.
%  If a  grim neighbourhood $U''$ in $F^{-1}[a,b]$ contains $w$,
%  then the closure of $U''$ contains $\cK_{\xi,a,b}(w)$.}
  For any open subset $V\subseteq F^{-1}[a,b]$ with $\cK_{\xi,a,b}(w) \subseteq V$, there exists a grim neighbourhood $V'$ of $\cK_{\xi,a,b}(w)$
  in $F^{-1}[a,b]$ such that $V'\subseteq V$.
\end{lemma}

\begin{proof}
  In the proof we will use a weak version of the Broken Trajectory Lemma~\ref{lem:limiting_lemma}, which appears below i.e.\ in the presence of the function $F$.
  This formulation of the Broken Trajectory Lemma is well-known, and is commonly used in the proof that $\partial^2=0$
  in Morse homology; compare e.g.\ \cite[Proof of Theorem 3.1]{Salamon}. It does not depend on Lemma~\ref{lem:grimneighexist}, so we do not
  have a circular dependence of lemmas.

  Suppose $\cK_{\xi,a,b}(w)$ contains critical points $p_1,\dots,p_r$. If $r=0$, that is, there are no critical points, then $\cK_{\xi,a,b}(w)$
  is just the intersection of $F^{-1}[a,b]$ with the trajectory through $w$. In that case, we let $w_0$ be the intersection of that
  trajectory with $F^{-1}(a)$. For any $m>0$, we choose $W_m$ to be the ball in $F^{-1}(a)$, with center $w_0$ and radius $1/m$, and $V_m$
  to be $\Xi_{[a,b]}W_m$, the set of points in $F^{-1}([a,b])$ that can be reached from $W_m$ by the flow of $\xi$.
  Then  $\bigcap_m V_m=\cK_{\xi,a,b}(w)$, so for large $m$, by compactness, $V_m\subseteq V$.

  For the rest of the proof, assume $r>0$.
  Choose a sequence of open sets $V_m\subseteq F^{-1}[a,b]$ containing $p_1,\dots,p_r$ such that $V_m\supset V_{m+1}$ and $\bigcap V_m=\{p_1,\dots,p_r\}$.
  For instance, one could take for $V_m$ the union of open balls of radius $1/m$ centred at the $p_i$ in some metric.
  For each $V_m$ set
  \[U_{m}=\bigcup_{z\in V_m} \cK_{\xi,a,b}(z).\]
  As $w$ is connected to at least one critical point by a trajectory of $\xi$, $w\in U_m$ for all $m$. From the construction
  of $U_m$, we also have
  $U_m\supseteq U_{m+1}$.
  We claim that $U_m$ is open, and that $\bigcap U_m=\cK_{\xi,a,b}(w)$.
 Before we prove openness, we make an observation.

  \begin{lemma}\label{lem:shelter}
    For $m$ sufficiently large, $U_m$ does not contain any critical point other than $p_1,\dots,p_r$.
  \end{lemma}

  \begin{proof}
    Suppose for contradiction that $x_m\in U_m$ is a critical point of $\xi$ different than $p_1,\dots,p_r$. Suppose a trajectory of $\xi$
    from (or to) $x_m$ hits a point $y_m\in V_m$. In fact, $x_m$ could be connected to yet another critical point of $\xi$, in that case,
    we choose $x_m$ to be a point connected by a trajectory (not a zigzag trajectory) to a point in $V_m$. On passing to a subsequence,
    the points $y_m$ converge to a point $y$, which is necessarily one of the points $p_1,\dots,p_r$. On the other hand, the points $x_m$
    are critical points of $\xi$ different than $p_1,\dots,p_r$. The number of critical points is finite, so the sequence $x_m$, up to passing to a subsequence, is constant. Call it $x$.
    By the Broken Trajectory Lemma~\ref{lem:limiting_lemma}, there is a broken trajectory from $x$ to $y$. But this means that there is a zigzag
    trajectory from $w$ to $x$, contradicting the assumption that $p_1,\dots,p_r$ were all of the critical points in $\cK_{\xi,a,b}(w)$.
  \end{proof}

 Next, we prove openness.

  \begin{lemma}\label{lem:grim_open}
    For sufficiently large $m$, the subset $U_m$ is open.
  \end{lemma}

  \begin{proof}
    Let $x\in U_m$. We want to show that an open subset of $x$ is contained in $U_m$. The easy case is that $x$ is a critical point of $\xi$. Then $x$ is one of the $p_1,\dots,p_r$ (by Lemma~\ref{lem:shelter}),
    and the open subset $V_m$, which contains $x$, is in $U_m$.
    %\npar{I didn't follow this. Isn't $V_m$ already a subset of $U_m$ be definition of $U_m$?}{Rephrased. Is it better?}

    Suppose $x$ is not a critical point of $\xi$. Then $x$ is connected by a zigzag trajectory to a point $v\in V_m$. If this trajectory is a zigzag trajectory, and
    not just a trajectory, choose a part of it $\gamma$
    that passes through $x$ and hits a critical point, say $p_1$. Either way, $\gamma$ must enter $V_m$, so $\gamma$ connects $x$ to some point $v\in V_m$. This means that there is a non-critical point $v\in V$ and a trajectory of $\xi$ connecting $x$ to $v$. The time to travel from $x$ to $v$ is finite, so
    by continuous dependence of solutions on the initial conditions, any point sufficiently close to $x$ can be connected to a point in $V_m$
    by a trajectory of $\xi$, and hence also lies in $U_m$. Hence $U_m$ is open, as desired.
  \end{proof}

  \begin{lemma}
We have that $\bigcap U_m=\cK_{\xi,a,b}(w)$.
       \end{lemma}
  \begin{proof}
     The inclusion $\cK_{\xi,a,b}(w)\subseteq U_m$ follows from the definition of $U_m$. Therefore, we have to prove that
  $\bigcap U_m\subseteq \cK_{\xi,a,b}(w)$.

  To this end, take $w'\in\bigcap U_m$. By definition, it is connected by a zigzag trajectory to a sequence of points $v_m\in V_m$.
  Arguing as in the proof of Lemma~\ref{lem:grim_open}, we show that $w'$ is connected by an actual trajectory of $\xi$ with a point $v_m$.
  The points $v_m$ have a converging subsequence, let $v$ be the limit. It is necessarily one of the critical points $p_1,\dots,p_r$.
  The trajectories connecting $w'$ to $v_m$ limit into
  a broken trajectory connecting $w'$ to $v$. This means that $w'$ is connected to one of the points $p_1,\dots,p_r$. But these points
  belong to $\cK_{\xi,a,b}(w)$. Hence $w'$ also belongs to $\cK_{\xi,a,b}(w)$.
We have shown that $\bigcap U_m=\cK_{\xi,a,b}(w)$.
\end{proof}

  \emph{Continuing the proof of Lemma~\ref{lem:grimneighexist}}, we aim to deduce that one of the $U_m$ must be contained in $V$. This follows from compactness of $\cK_{\xi,a,b}$ by the following classical argument. Consider the cover of $F^{-1}[a,b]$ by $V$ and $X_m:=F^{-1}[a,b]\setminus\ol{U}_m$.
  This is an open cover of a compact set, so it has a finite subcover. The subcover must contain some sets $X_{m_1},\dots,X_{m_\ell}$ and possibly $V$. Suppose $m_1<\dots<m_\ell$. Then $X_{m_1}\subseteq\dots\subseteq X_{m_\ell}$, therefore we may assume that the cover contains only one set $X_m$. It necessarily contains $V$, then. That is to say, for some $m$, $X_m\cup V=F^{-1}[a,b]$. Suppose $y\in F^{-1}[a,b]$ is such that
  $y\in\ol{U}_m$ but $y\notin V$. Then  $y\notin X_m\cup V$, which is a contradiction. This means that $\ol{U}_m\subseteq V$, so we just take $V'=U_m$.
  This completes the proof of Lemma~\ref{lem:grimneighexist}.
  %\mpar{Completely rephrased.}
  %Define the functions $\cK_{\xi,a,b}$: $f_m(x)$ is the maximal radius of a ball with center $x$ and contained in $U_m$. Each of $f_m$ is a continuous (in fact, it is Lipschitz with Lipschitz constant $1$) function. Moreover, $f_m$ converge monotonously to $0$. The set $\cK_{\xi,a,b}(w)$, as a closed subset of a compact space $F^{-1}[a,b]$ is compact. By Dini's theorem, the convergence is uniform. Next, we can always choose $\varepsilon>0$, such that $V$ is contained in the set of all points at distance at most $\varepsilon$ of $\cK_{\xi,a,b}$ (the pillow of $\cK_{\xi,a,b}$). If $m$ is large enough so that $f_m(x)<\varepsilon$for all $x\in\cK_{\xi,a,b}$, then $U_m\subseteq V$.\ypar{This is not ideal, but maybe it is enough?}
\end{proof}

\subsection{The Broken Trajectory Lemma}

The notion of a broken trajectory is connected with the slogan that the closure of the space of trajectories is obtained by adding the broken trajectories. For gradient-like vector fields of Morse functions, the result is standard; see \cite{Salamon}.
As we aim to use the result for the proof of Vector Field Integration Lemma below, we do not use Morse functions, only vector fields. The assumptions are therefore
somewhat more complicated than usual.

\begin{lemma}[Broken Trajectory]\label{lem:limiting_lemma}
  Assume $\xi$ is a vector field on $\Omega$, and $\Omega$ is presented as a disjoint union of three sets $\Omega_-\cup\Omega_0\cup\Omega_+$,
  where $\Omega_-$ and $\Omega_+$ are open, and $\Omega_0$ is compact.
  Suppose that
  \begin{enumerate}[label=(BT-\arabic*)]
    \item\label{item:BT_finite} $\xi$ has finitely many critical points in $\Omega_0$ and they are all isolated. Moreover, there are
      no critical points on $\ol{\Omega}_0\setminus\Omega_0$;
    \item\label{item:BT_crits} There is a pairwise disjoint family of subsets $\{X_p\}$, with $X_p\subseteq\Omega_0$, of pairwise disjoint
      regular index triples for all the critical points $p\in\Omega_0$;
      %Each critical point $p\in\Omega_0$ of $\xi$ admits a descending family of regular index triples;
    \item\label{item:BT_limitset} Any trajectory of $\xi$ through a point $z\in\Omega_0$ has either a forward limit at a critical
      point $p\in\Omega_0$, or it leaves in finite time to $\Omega_+$;
      %\npar{Once they have left to $\O_+$, do they stay there, or can they return. If we mean that they cannot ever return to $\O_0$, then we should say so.}{Added a condition. Surprisingly, it existed before but was commented out.}
    \item\label{item:BT_backwardlimitset} Any trajectory of $\xi$ through a point $z\in\Omega_0$ has either a backward limit in a critical
      point $p\in\Omega_0$, or it leaves in finite time in the past to $\Omega_-$;
    %\item\label{item:BT_nojump} A trajectory in $\Omega_+$ stays in $\Omega_+$ forever in the future; a trajectory in $\Omega_-$
    %  stays in $\Omega_-$ forever in the past.
    \item\label{item:BT_forward} The set $\Omega_+$ $($respectively $\Omega_-)$ is forward invariant, respectively backward invariant. That is, a trajectory staying in $\Omega_+$, respectively $\Omega_-$, stays forever in the future in $\Omega_+$, respectively forever in the past in $\Omega_-$.\label{item:BT_nojump}.
  \end{enumerate}
  Suppose $\{z_\ell\}$ and $\{w_\ell\}$ are sequences of points in $\O_0$ and that for every $\ell$ there is a trajectory of $\xi$ passing through $z_\ell$ and then through $w_\ell$.
  Assume $\lim_{\ell\to\infty} z_\ell=z$ and $\lim_{\ell\to\infty} w_\ell=w$ and $z,w \in \Int \O$.
  Then either there exists a $($possibly broken$)$ trajectory of $\xi$ passing through $z$ and then through $w$,
  or a $($non-constant$)$ broken trajectory connecting a critical point of $\xi$ in $\Omega_0$ with itself.
  %
  %, or
  %the trajectory of $\xi$ through $z$ hits $\Omega_+$ in the future and the trajectory of $\xi$ through $w$ hits $\Omega_-$
  %in the infinite past.
\end{lemma}
\begin{proof}
  Assume first that $z,w$ are not critical points of $\xi$. The case that at least one of $z,w$ is a critical point requires only minor adjustments to the argument below, and is left to the reader.

  By \ref{item:BT_finite}, we let $p_1,\dots,p_m$ be critical points of $\xi$ in $\Omega_0$. Choose a pairwise disjoint regular index triples
  $(X_1,\pout X_1,\piin X_1)$, \dots, $(X_m,\pout X_m, \piin X_m)$ for $p_1,\dots,p_m$, which exist by \ref{item:BT_crits}. Note that each of $X_1,\dots,X_m$
  is open and has compact closure. %\npar{I changed this to say that each of these sets is open with compact closure. Is that what you meant? }{It is okay right now.}
   We assume also that $z,w$ are disjoint from the closures of $X_1,\dots,X_m$; in particular, so are the points $z_\ell$ and $w_\ell$.

  We consider two simple cases.

  First, we let $T_\ell$ be the time it takes to go from $z_\ell$  to $w_\ell$. If $T_\ell$ has a bounded subsequence, on passing to a subsequence we have that $T_\ell\to T$ for some $T \in \R$.
  By continuous dependence of solutions on the initial conditions, we conclude that
  the trajectory of $\xi$ starting from $z$ reaches $w$ after time $T$.  This concludes the proof.

  The second simple case is that the trajectory through $z$ hits $\Omega_+$. Then  nearby trajectories also hit $\Omega_+$
  in finite time. Once a trajectory enters $\Omega_+$, it stays there forever by \ref{item:BT_forward}. Thus, we get a contradiction with contradicting the fact that the trajectories connect $z_\ell$ to $w_\ell$.

  The last case is that $T\to\infty$, and the trajectory through $z$ stays forever in $\Omega_0$. By \ref{item:BT_limitset}, it
  has to hit a critical point, say $p_{i_1}$. That is to say, at some moment, say $t_1$,
  it enters $X_{i_1}$ through, say $v_1\in \piin X_{i_1}$ and stays
  in $X_{i_1}$ forever on. Note that the trajectory hits the interior of $X_{i_1}$ after finite time. That is to say, there is
  a sequence of points $v_{1,\ell}\in \piin X_{i_1}$, converging to $v_1$, such that the trajectory through $z_{\ell}$ hits $v_{1,\ell}$.
  %and the time it takes to get from $z_{\ell}$ to $v_{1,\ell}$ is close to $t_1$.

  We claim that the trajectory through $v_{1,\ell}$ hits $\pout X_{i_1}$ for $\ell$ sufficiently large. Otherwise, this trajectory stays
  forever in $X_{i_1}$, hence $w_\ell\in X_{i_1}$, so $w\in\ol{X}_{i_1}$, which is a contradiction. Let $w_{1,\ell}\in \pout X_{i_1}$ be the point of the first
  exit of the trajectory of $\xi$ through $v_{1,\ell}$. Up to passing to a subsequence, we may and will assume that $w_{1,\ell}$ converge
  to a point $y_1$.

  The backward limit of the trajectory through $y_1$ has to be $p_{i_1}$. Otherwise, the trajectory through $y_1$ exits (in the past)
  $X_{i_1}$ through a point $v'_1$. As $y_1$ is the limit of $w_{1,\ell}$, and $\xi$ connects $v_{1,\ell}$ to $w_{1,\ell}$,
  we infer that $v'_1$ is the limit of $v_{1,\ell}$. But the limit of $v_{1,\ell}$ was $v_1$, the trajectory through $v_1$ hits $p_{i_1}$
  (does not reach $y_1$, in particular); the contradiction shows that the backward limit of $y_1$ is $p_{i_1}$. Compare Figure~\ref{fig:broken_traj}.

  \begin{figure}
    \input{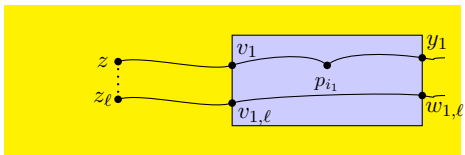}
    \caption{Notation of the proof of the Broken Trajectory Lemma~\ref{lem:limiting_lemma}.}\label{fig:broken_traj}
  \end{figure}

  We start the inductive process, by looking at the forward behaviour of the trajectory through $y_1$. To see this, recall that $w_{1,\ell}$
  converges to $y_1$. The
  trajectory through $w_{1,\ell}$ hits $z_\ell$ in the past by construction, so hits $w_\ell$ in the future. If the time required
  to reach $w_\ell$ from $w_{1,\ell}$ is bounded, then the trajectory through $y_1$ reaches $w$ in finite time. We declare the limit trajectory to be the union of a trajectory from $z$ through $v_1$ to $p_{i_1}$, and the trajectory from $p_{i_1}$ through $y_1$ to $w$.

  If the time is infinite, we repeat the procedure, i.e.\ we find the index $i_2$ such that the trajectory through $y_1$ hits $p_{i_2}$.
  We assume that $v_2$ is the point of last entry of this trajectory through $y_1$, and we let $v_{2,\ell}\in \piin X_{i_2}$
  be the sequence of points on the trajectory through $w_{1,\ell}$ converging to $v_2$. We let $w_{2,\ell}\in \pout X_{i_2}$ be
  the points on the same trajectory, when they exit $X_{i_2}$, and $y_2$ be the limit of a subsequence of $w_{2,\ell}$. As above,
  we show that the backward trajectory through $y_2$ hits $p_{i_2}$.

  We repeat the process to obtain a sequence of indices $i_1,\dots,$, points $v_1,v_2,\dots\in \piin X_{i_1},\piin X_{i_2},\dots$ and
  points $y_1,\dots\in \pout X_{i_1}, \pout X_{i_2}, \dots$. The point $y_i$ is the limit of points $w_{i,\ell}$ as $\ell\to\infty$,
  and the trajectory through the original point $z_\ell$ passes through all the points $w_{1,\ell},w_{2,\ell},\dots$.

  There are two possibilities: either the process stops after a finite time, that is, the trajectory through the point $y_m$
  reaches $w$ in finite time. Then  there is a trajectory from $z$ through $v_1$ to  $p_{i_1}$, from $p_{i_1}$ through $y_1$ and $v_2$ to $p_{i_2}$, and so on, and
  the trajectory from $p_{i_m}$ through $y_m$ to $w$. This is the desired broken trajectory.

  The other option is that the process does not stop. As the number of singular points is finite,
  there have to be repetitions among the indices $i_1,i_2,\dots$. Suppose $i_s=i_{s'}$ and $s'>s$. Then
  there is a broken trajectory that:
  \begin{itemize}
    \item starts at $p_{i_s}$ and leaves $X_{i_s}$ through $y_s$;
    \item enters $X_{i_{s+1}}$ through $v_{s+1}$ and hits $p_{i_{s+1}}$;
    \item resumes from $p_{i_{s+1}}$ and leaves $X_{i_{s+1}}$ through $y_{s+1}$;
    \item continues this behaviour until it enters $X_{i_{s'}}$ through $v_{s'}$ and hits $p_{i_{s'}}$.
  \end{itemize}
  This means that we have created a circular trajectory, as in the last clause of the lemma's statement.
\end{proof}

\section{Integrating grim vector fields}\label{sec:integrate}

The purpose of this section is to prove the following result, called the Vector Field Integration Theorem. It says that
under some natural assumptions, a vector field can be turned into a Morse function.  In the upcoming theorem and proof, $\O_{\pm}$ are open subsets, and $\partial \O_{\pm}$ denotes the point-set boundary, or frontier, namely the closure minus the interior.

\begin{theorem}[Vector Field Integration]\label{thm:integrate}
  %% add double statement for $F$.
  Assume there are two $($possibly empty$)$ disjoint open subsets $\O_-$ and $\O_+$ of $\O$ and suppose $F\colon\ol{\O}_-\sqcup\ol{\O}_+\to\R$
  is an immersed Morse function such that $F|_{\partial \O_-}=a$ and $F_{\partial \O_+}=b$ for some real numbers $a<b$. Moreover suppose $F$ restricts to $F\colon\O_-\to(-\infty,a)$ and  $F\colon\O_+\to(b,\infty)$; see Figure~\ref{fig:integr1}.
  \begin{figure}
    \input{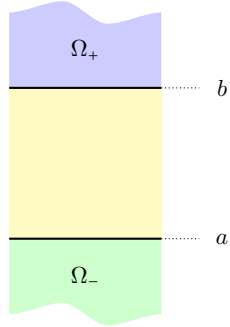}
    \caption{Notation of Vector Field Integration Theorem~\ref{thm:integrate}.}\label{fig:integr1}
  \end{figure}
  Suppose $\xi$ is a vector field on $\Omega$ that is a grim vector field for $F$ on $\O_-\sqcup\O_+$
  satisfying the following conditions $($reminiscent of \ref{item:G2} and \ref{item:G3} of Definition~\ref{def:grim}$)$:
  \begin{itemize}
    \item (compare \ref{item:G2}) If $w\in \O[d]$, then $\xi(w)\in T_{w}\O[d]$. That is, $\xi$ is tangent to all the strata.
    \item (compare \ref{item:G3}) If $p\in \O[d]$ is a critical point of $\xi$, then there are local coordinates centred at $p$
      \[(x_1,\ldots,x_m,y_{11},\ldots,y_{1k},\ldots,y_{d1},\ldots,y_{dk})\]
      $($with $m=n+k-kd)$, such that $\O[d]$ is given locally by the intersection of $d$ strata:
      \[\{y_{11}=\dots=y_{1k}=0\}\cap\ldots\cap\{y_{d1}=\dots=y_{dk}=0\},\]
      such that $\xi$ has the form
      \[\xi=(-x_1,\ldots,-x_h,x_{h+1},\dots,x_m,\sum_{j=1}^ky_{1i}^2,0,\ldots,0,\sum_{j=1}^ky_{2i}^2,0,\ldots,0,\sum_{j=1}^ky_{di}^2,0,\ldots,0).\]
  \end{itemize}
  Suppose additionally that $\xi$ satisfies the three conditions below.
  \begin{enumerate}[label=(G-\arabic*), resume=grim]
    \item\label{item:G4} For each trajectory $\gamma\colon\R\to\Omega$ of $\xi$, either the limit $\lim_{s\to \infty}\gamma(s)$ $($respectively $\lim_{s\to -\infty}\gamma(s))$
      exists and belongs to $\O\setminus\ol{\O_-\cup\O_+}$,
      or $\gamma$ enters $\O_+$ in the future $($respectively, enters $\O_-$ in the past$)$ and stays there forever.
      %\npar{Added, and stays there.}{Added: forever for better rhythm of the phrase.}
    \item\label{item:G5} There are no broken circular trajectories, that is there is no collection of trajectories $\{\gamma_1,\ldots,\gamma_{n+1}\}$ with $\gamma_{n+1}=\gamma_1$
      such that $\lim_{s\to\infty}\gamma_{i}(s)=\lim_{s\to-\infty}\gamma_{i+1}(s)$ for $i=1,\ldots,n$.
    \item\label{item:G6} For each point $w\in\partial\O_-$ $($respectively $w\in\partial\O_+)$ the vector field $\xi$ points out of $\O_-$ $($respectively into $\O_+)$. In particular, there are no critical points on $\partial \O_- \cup \partial_+ \O$.
  \end{enumerate}
  Then $F\colon\ol{\O}_-\sqcup\ol{\O}_+\to\R$ extends to an immersed Morse function $F\colon\Omega\to\R$ such that $\xi$ is a grim vector field for $F$.
\end{theorem}

This generalises the vector field integration lemma in \cite[Section 3.1]{BP}.  We remark that there was a mistake in that proof, because
the function constructed in \cite{BP} was not necessarily continuous: it was smooth along trajectories, but could have jumps at points
such that a trajectory through that point hits the boundary of $V_i$. This is a rather technical mistake and can be fixed by carefully adding a smoothing
function.
We provide a new proof here, both extending to the immersed case and fixing the previous error.

\begin{proof}
  If there are no critical points of $\xi$ in $\O\setminus(\O_-\cup\O_+)$, the definition of $F$ is straightforward and
  follows \cite[Proof of Assertion 5, page 54]{Mil65}.
  Namely, take $w\in\O\setminus(\O_-\cup\O_+)$. By \ref{item:G4} the trajectory of $\xi$ through $w$ enters $\O_+$ in the future and
  $\O_-$ in the past. Let $T_+$ be the time it takes to go from $w$ to $\O_+$ and $T_-$ the time it takes
  to go from $\O_-$ to $w$. We set
\[F(w)=a+\frac{T_-}{T_-+T_+}(b-a).\]
The function $F$ is smooth except possibly at $\partial\O_-$ and $\partial\O_+$, but using a standard argument (similar to the
argument used in \cite[Proof of Assertion 5, page 54]{Mil65}) we can make it smooth on $\O$.

\smallskip
Next, suppose there are critical points of $\xi$ in $\O\setminus(\O_+\cup\O_-)$.
  Choose a partial order of all the critical points of $\xi$ in $\O\setminus(\O_+\cup\O_-)$ by requiring that $p\leq p'$
  if there exists a (possibly broken) trajectory of $\xi$ starting from $p$ and ending up in $p'$.
  The condition \ref{item:G5} says $p \leq p'$ is indeed a partial order.  Every partial order has an associated strict partial order, where we write $p < p'$ if $p \leq p'$ and $p \neq p'$.
  %\npar{added so that we can also use the notation $<$ later.}{great!}

  We will first construct a function $F$ near critical points and then use the flow of $\xi$ to extend it over the whole of $\O$.
  Define a function $F$ first on critical points of $\xi$. The only condition on the values of $F$ are that $F(p) \leq F(p')$
  whenever $p \leq p'$ and $a<F(p)<b$. Next, each critical point $p$ of $\xi$ has a neighbourhood in which the vector field has
  the form \eqref{eq:localgrim}. We choose pairwise disjoint neighbourhoods $U_p$ and impose the form
  of $F$ on $U_p$ to be given by \eqref{eq:formofimmersed}. Shrink $U_p$ in such a way that
  if $p<p'$, then $\sup_{w\in U_p} F(w)<\inf_{w'\in U_{p'}} F(w')$.

  Having defined $F$ on $\bigcup U_p$, we may speak of grim neighbourhoods of $p$ (as long as they are in $U_p$).

  \begin{lemma}\label{lem:noreturn}
    Suppose $p$ is a critical point of $\xi$.
    There exists a grim neighbourhood $V_p\subseteq U_p$ such that if a trajectory of $\xi$ leaves $V_p$ then it never returns to $V_p$ again.
  \end{lemma}

  \begin{proof}
    A grim neighbourhood $V_p$ exists by Lemma~\ref{lem:grimneigh} applied to $F|_{U_p}$. We shrink it until it satisfies the required conditions.
    First, fix $\varepsilon_p>0$ to be such that $F^{-1}(F(p)-\varepsilon_p)\cap\Hd(p)$ and $F^{-1}(F(p)+\varepsilon_p)\cap\Ha(p)$ are separated from the boundary of $U_p$ and each trajectory
    of $\xi$ that hits $p$ in the future (respectively in the past) hits $F^{-1}(F(p)-\varepsilon_p)\cap\Hd(p)$ (respectively hits $F^{-1}(F(p)+\varepsilon_p)$);
    see Figure~\ref{fig:choice}.
    \begin{figure}
      \input{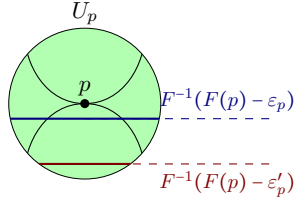}
      \caption{Choice of $\varepsilon_p$ in the proof of Lemma~\ref{lem:noreturn}. While $\varepsilon_p$ satisfies
      the conditions spelled out in the proof, $\varepsilon_p'$ is too large.}\label{fig:choice}
    \end{figure}
  This condition is nontrivial, because for the moment $F$ is only defined inside $U_p$. With this choice of $\varepsilon_p$ we let $\mathcal{V}_p$ be the family
  of all grim neighbourhoods of $p$ whose `in' boundary is on $F^{-1}(F(p)-\varepsilon_p)$ and whose `out' boundary is on $F^{-1}(F(p)+\varepsilon_p)$.
  Note that $\bigcap_{V_p\in\mathcal{V}_p} V_p$ is the union of the ascending and descending membrane of $p$ intersected with $F^{-1}(F(p)-\varepsilon_p,F(p)+\varepsilon_p)$.
  Observe that if $V,V'\in\mathcal{V}_p$, then $V\cap V'\in\mathcal{V}_p$. As $\O$ is second countable, a standard argument shows
  that there exists a nested family $V_1\supset V_2\supset \dots$ of grim neighbourhoods in $\mathcal{V}_p$ such that $\bigcap_{\ell} V_\ell=(\Ha(p)\cup\Hd(p))\cap F^{-1}(F(p)-\varepsilon_p,F(p)+\varepsilon_p)$

  Choose such a nested family. We will show that in this nested family, there is a grim neighbourhood $V_p$ such that if a trajectory of $\xi$ leaves $V_p$ then it never returns to $V_p$ again.
  Suppose for a contradiction that for every $\ell=1,2,\dots$, there exists a trajectory of $\xi$ leaving $V_\ell$ and entering $V_\ell$ again. Let $w_\ell$ be
  the first exit point and let $z_\ell$ be the next entrance point; see Figure~\ref{fig:broken}. We have that $w_\ell\in F^{-1}(F(p)+\varepsilon_p)$ and $z_\ell\in F^{-1}(F(p)-\varepsilon_p)$. By passing  to a subsequence we may and will assume that $z_\ell\to z$ and $w_\ell\to w$ for some $z,w$ i.e.\ the subsequences converge and their limits are $z$ and $w$ respectively. Now by assumption $z\in \Hd(p)$ and $w\in\Ha(p)$. In particular,
  there is a broken trajectory of $\xi$ starting at $z$ and ending at $w$. On the other hand, by the Broken Trajectory Lemma~\ref{lem:limiting_lemma} there exists a possibly broken trajectory of $\xi$ starting at $w$ and ending at $z$, or a circular trajectory.
  In the second case we get an immediate contradiction with \ref{item:G5}. For the first case,
    taking the union of these two broken trajectories we obtain a circular broken
    trajectory, contradicting~\ref{item:G5}. This completes the proof of the lemma.
    \begin{figure}
      \input{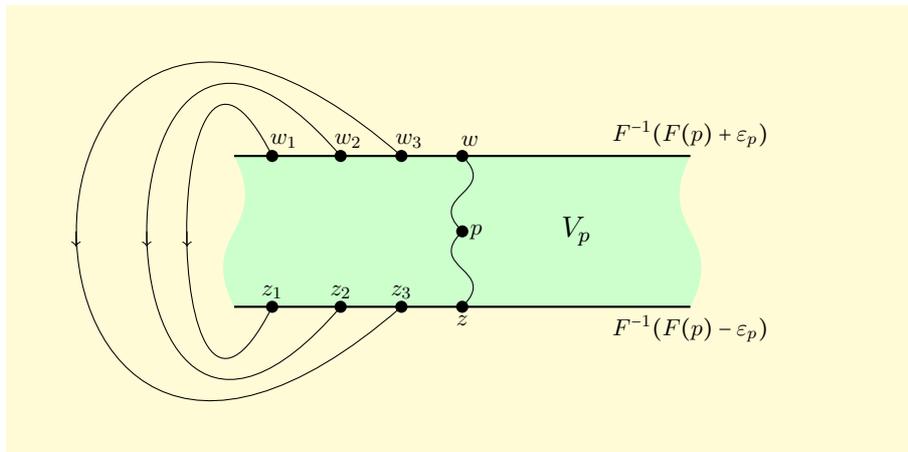}
      \caption{Existence of points $z_\ell,w_\ell$ and their limit in the proof of Lemma~\ref{lem:noreturn}.}\label{fig:broken}
    \end{figure}
  \end{proof}
As a next step towards the proof of the Vector Field Integration Theorem~\ref{thm:integrate}, we prove the following result.
\begin{lemma}\label{lem:noreturn2}
  Suppose $p<p'$. Then, on shrinking neighbourhoods $V_p$ and $V_{p'}$,
  we may assume that no trajectory starting from $V_{p'}$ hits
  $V_p$ in the future.
\end{lemma}

\begin{proof}
  Suppose such neighbourhoods do not exist. Acting as in the proof of Lemma~\ref{lem:noreturn}, we create a trajectory starting at $p'$ and hitting $p$ in the infinite future (the reader might imagine that the points $w_i$ of the proof of Lemma~\ref{lem:noreturn} belong $\pout V_{p'}$, while points $z_j$ belong to $\piin V_p$). As $p<p'$, there exists a possibly broken trajectory from $p$ to $p'$. Connecting these two broken trajectories, we obtain a circular broken trajectory, contradicting~\ref{item:G5}.
\end{proof}

Continuing the proof of the Vector Field Integration Theorem~\ref{thm:integrate}, we enumerate the critical points of $F$ by $p_1,\dots,p_S$ and denote by $c^-_i$ and $c^+_i$ the minimum and maximum of $F$ on $V_{p_i}$. We have $c^-_i<c^+_i \in \R$. We may and will assume that the critical points are ordered in such a way that $c^+_i<c^-_{i+1} \in \R$ for all~$i$.

We need to set up some more terminology. We will say that $F$ is defined \emph{up to the level set $a'$} if there is a set $\O_-'\subseteq\O$ such that $F$ is defined on $\O_-'$, $\partial\O_-'=F^{-1}(a')$ is the exit set of $\O_-'$, and $\xi$ is grim for $F$ on $\O_-'$.
If $F$ is defined up to a level set $a'$ such that $\xi$ has no critical points on $\O\setminus(\O_+\cup\O_-)$, we can extend $F$ to the whole of $\O$ using the argument from the beginning of the proof.

\smallskip
Suppose $F$ is defined up to the level set $a'$, where $a'\in[c^+_i,c^-_{i+1}]$. This assumption means in particular that
  $V_{p_i}\subseteq\O'_-$. We have two cases. Either $a'=c^-_{i+1}$ or $a'<c^-_{i+1}$. Assume first $a'<c^-_{i+1}$. Let $A=F^{-1}(a')$. Define subsets $A_{i+1},\dots,A_S \subseteq A$  by the condition that a trajectory of $\xi$ starting from $A_{\ell}$ hits $\ol{V}_{p_\ell}$ before hitting other neighbourhoods. Let $C_{\ell}$ be the union of the trajectories of $\xi$ starting from $A_{\ell}$ and before they hit $\ol{V}_{z_\ell}$; see Figure~\ref{fig:cyllinders_everywhere}.
  Also let $T_\ell\colon A_{\ell}\to(0,\infty)$ be the function measuring the shortest time taken to get from the given point in $A_{\ell}$
  to some point in~$V_{\ell}$.

  \begin{figure}
    \input{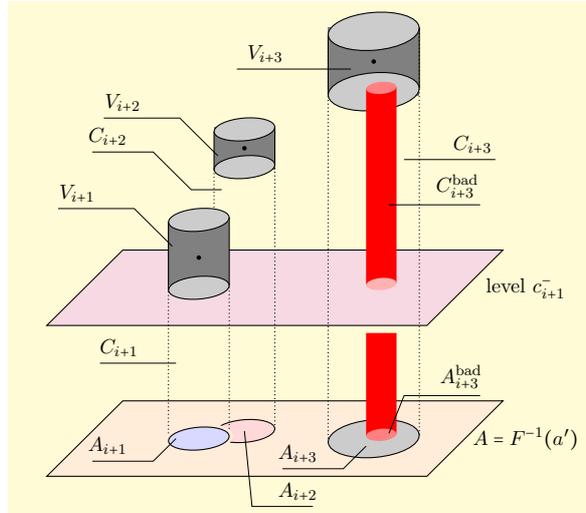}
    \caption{Proof of the Vector Field Integration Theorem~\ref{thm:integrate}.}\label{fig:cyllinders_everywhere}
  \end{figure}

  First, observe that $A_{i+1}$ is closed, because a trajectory starting from $A$ and hitting $\ol{V}_{i+1}$ does not hit any other critical points by the ordering
  of critical points. In particular $C_{i+1}$ is closed. Rescale the vector field $\xi$ by a function supported in a neighbourhood of $C_{i+1}$ (but disjoint from $A$ and $V_{i+1}$) in such a way that $T_{i+1}$ is constant on $A_{i+1}$ and equal to $c^-_{i+1}-a$. Redefine the functions $T_{\ell}$, for each $\ell \geq i+1$, with respect to the new vector field.

  Our goal is to define the value $F(w)$ for a point $w$ using the time required to go from $A$ to~$w$. The problem with doing precisely this is that some neighbourhoods $\ol{V}_p$ may be reached too early. To remedy this we need to slow down the vector field $\xi$ along some trajectories. To this end,
  for $\ell>i+1$ we let $A_{\ell}^{\bbad}$ be the set of points $w\in A_{\ell}$ such that $T_{\ell}(w)\le c^+_{i+1}$. Suppose $\ell_0$ is the smallest index for which $A_{\ell}^{\bbad}$ is  nonempty. Note that the closure of $A_{\ell}^{\bbad}$ is disjoint from the union $A_{i+1}\cup\dots\cup A_{\ell_0-1}$ (this union is a closed set). Define the cylinder $C_{\ell}^{\bbad}$ to be the union of all trajectories of $\xi$ starting from $A_{\ell}^{\bbad}$ before they reach $\ol{V}_{\ell}$.
  Rescale $\xi$ in the cylinder $C_{\ell}^{\bbad}$ in such a way that $\xi$ runs slower in $C_{\ell}^{\bbad}$ and so making $A_{\ell}^{\bbad}$ empty. By induction
  we arrive at the situation in which all the sets $A_{\ell}^{\bbad}$ are empty.

  Now let $\O_{i+1}$ be the set of points reached by $\xi$ in time less than or equal to $c^-_{i+1}-a'$. By construction, $\xi$ is nonvanishing
  on $\O_{i+1}$. For a point $w\in\O_{i+1}$ that is reached from $A$ in time~$t_w$ we set
  \[F(w)=a'+t_w.\]
  It is clear that $F$ is a continuous function, and it can be made smooth near $A$ by carefully rescaling the vector field $\xi$. In this way, we have defined
  $F$ up to the level set $c^-_{i+1}$.

  \smallskip
  Now suppose that $F$ is defined up to the level set $a'=c^-_{i+1}$. Define sets $A_{\ell}$, $C_{\ell}$ analogously, as well as the function $T_{\ell}$. Alter the vector field $\xi$ near the boundary of $V_{p_{i+1}}$
  to guarantee that $\partial_\xi F=1$ near the boundary of $V_{p_{i+1}}$. This is a technical step, which simplifies the construction later.

  For $\ell=i+2,\dots$ we define $A_{\ell}^{\bbad}$ as the set
  of points $w\in A_\ell$ such that $T_\ell(w)\le c^+_{i+1}-c^-_{i+1}$  and we perform an analogous rescaling as above to make $A_{\ell}^{\bbad}$ empty.

  Define $\O_{i+1}^0$ to be the union of the set of points in $\O$ reached from $A\setminus V_{p_{i+1}}$ by $\xi$ in time less or equal to $c^+_{i+1}-c^-_{i+1}$.
  For a point $w\in\O_{i+1}$ that is reached from $A\setminus V_{p_{i+1}}$ in time $t_w$ we set
  \[F(w)=a'+t_w.\]
  As $\partial_\xi F=1$ near the boundary of $V_{p_{i+1}}$,  the function $F$ is well-defined and continuous, and we can make it smooth by carefully altering $F$ near $\partial V_{p_{i+1}}$. An argument, following the ideas  presented in \cite[Proof of Assertion 5, page 54]{Mil65}, is presented as Lemma~\ref{lem:piecewise_smooth} below.
  In this way we have defined $F$ up to the level set $c^+_{i+1}$.

\smallskip
The above procedure in the proof of Vector Field Integration Theorem~\ref{thm:integrate} was to increase gradually the level sets up to which $F$ is defined. This was the key technical difficulty of the proof. At the final stage of the proof, we need to specify the starting level set.
If $\O_-$ is nonempty, the assumptions of the proposition tell us that $F$ is defined up to the level set $a$.
If $\O_-$ is empty, then the critical point $p_1$ is necessarily an index $0$ critical point. We set $\O_{-}=V_{p_1}$, and then $F$ is well-defined
up to the level set $c^+_0$.

A small technical problem arises: the extension we construct has not been guaranteed yet to be smooth near the level sets of $F$ corresponding to the values $c_i^{\pm}$: it is a priori only continuous and piecewise smooth.
The proof of the Vector Field Integration Theorem~\ref{thm:integrate} is completed using the following general result, which upgrades the continuous function we have just defined to a smooth function.

\begin{lemma}\label{lem:piecewise_smooth}
  Assume that $F\colon \Omega\to\R$ is a continuous function with a smooth level set $Z:=F^{-1}(c)$. Set $\Omega_-=F^{-1}(-\infty,c]$ and
  $\Omega_+=F^{-1}[c,\infty)$. Suppose $F$ is smooth on $\Omega_-$ and $\Omega_+$ and has no critical points near $\Omega_-$.

  Assume that $\xi$ is a smooth vector field such that $\partial_\xi F>0$ in a neighbourhood of $Z$ in $\Omega_-$ and $\partial_\xi F>0$
  in a neighbourhood of $Z$ in $\Omega_+$.

  Then, for any $\varepsilon>0$, there exists a smooth function $F_\varepsilon\colon\Omega\to\R$ such that $\partial_\xi F_\varepsilon>0$, $\|F-F_\varepsilon\|<\varepsilon$, and such that we have $F(z)=F_\varepsilon(z)$ whenever $|F(z)-c|>\varepsilon$.
\end{lemma}
Applying Lemma~\ref{lem:piecewise_smooth} to the function $F$ constructed above with $c$ ranging over $c_{i}^\pm$ and $\varepsilon$ choosen
in such a way that the intervals $[c_i^{\bullet}-\varepsilon,c_i^{\bullet}+\varepsilon]$, $\bullet=\pm$, do not overlap,
completes the proof of the Vector Field Integration Theorem~\ref{thm:integrate}.
\end{proof}

\begin{proof}[Sketch of proof of Lemma~\ref{lem:piecewise_smooth}]
  Choose $\delta>0$ sufficiently small, so that $F$ has no critical points on $F^{-1}[c-\delta,c+\delta]$. Let $Z_-=F^{-1}(c-\delta)$,
  $Z_+=F^{-1}(c+\delta)$. From the assumptions we deduce that the flow of $\xi$ flows from $Z_-$ to $Z_+$.

  Multiply the vector field $\xi$ by a smooth positive function in such a way that $\partial_\xi F\equiv 1$ near $Z_-$ and $Z_+$ and for any point $z\in Z_-$
  the time to reach $Z_+$ is equal to $2\delta$. We set
  \[F_\delta(z)=\begin{cases} F(z)&\colon F(z)\notin [c-\delta,c+\delta] \\ c-\delta+t(z)&\colon F(z)\in [c-\delta,c+\delta],\end{cases}\]
  where $t(z)$ is the time to reach $z$ starting from the unique point $z_-\in Z_-$ such that the trajectory of $\xi$ through $z_-$ hits $z$.

  The function $F$ is smooth on $F^{-1}(c-\delta,c+\delta)$.
  As the time to reach $Z_+$ from $Z_+$ is equal to $2\delta$, we infer that $F_\delta$ is continuous on $F^{-1}(c\pm\delta)$, in particular,
  it is continuous everywhere. Next, by construction, $\partial_\xi F_\delta\equiv 1$ on $F^{-1}[c-\delta,c+\delta]$. Since we assumed
  that $\partial_\xi F\equiv 1$ near $Z_\pm$, the function $F_\delta$ agrees with $F$ on a neighbourhood of $Z_\pm$. That is to say, $F_\delta$ is smooth.

  Next, for any $z\in F^{-1}[c-\delta,c+\delta]$ we have $F_\delta(z),F(z)\in[c-\delta,c+\delta]$. Hence $|F_\delta(z)-F(z)|<2\delta$.
  The statement is proved by setting $\varepsilon=\delta/2$.
\end{proof}
\begin{remark*}
  During the proof of the Vector Field Integration Theorem~\ref{thm:integrate} we were multiplying the vector field $\xi$ by a positive function. Note that this
  does not affect the veracity of the final statement. If $\xih$ is the `corrected' vector field, $\xih=\phi\xi$ for some $\phi\colon\O\to\R_{>0}$ such that
  $\phi\equiv 1$ in a neighbourhood of each of the critical points of $\xi$, then $\xih$ is a grim vector field for $F$ if and only if $\xi$ is.
\end{remark*}

\part{Families of Morse functions.}\label{part:just_paths}

Before we prove the Path Lifting Theorem~\ref{thm:path_lifting},
which is the main technical tool of the present paper, we need to investigate paths of Morse functions in detail. In this part, Part~\ref{part:just_paths}, we study one-parameter families of Morse functions.  Classical theorems on rearrangement and cancellations of critical points
of a Morse function are generalised to the case of immersed Morse functions and phrased in the language of a one-parameter family of functions
(see Sections~\ref{sec:rear} and~\ref{sec:cancel}). In the case
of cancellations, these function cease to be Morse functions for one value of the parameter. This motivates the study of one-parameter families
in the spirit of Cerf. We recall his approach in Section~\ref{sec:cerf}.

Now we introduce a few conventions used throughout Part~\ref{part:just_paths} and~\ref{part:pathlifting}.
We consider a smooth compact manifold $N$ of dimension $n$ is immersed into an ambient manifold $\Omega$ of dimension $n+k$. The immersion
is usually denoted by $G$, possibly with some decorations (e.g.\ subscripts). Both $N$ and $\Omega$
are allowed to have boundary, in which case we require that $G$ maps the boundary, and only the boundary, to the boundary. Moreover, any Morse function
is assumed to be locally constant on the boundary, while vector fields are transverse to the boundary.

Immersed Morse functions on $\Omega$ are denoted by capital $F$,
whereas Morse functions on $N$ are denoted by lowercase $f$. Usually we have $f=F\circ G$, but the precise formula can vary, and it is always
specified explicitly.
A grim vector field on $\Omega$ is denoted by $\xi$. A vector field on $N$ is denoted by $\eta$.

A path of functions is usually denoted by $F_\tau$.
The moments on the path are usually denoted by Greek letters, like $F_{\tau}$ for $\tau\in[\alpha,\beta]$ as opposed to level sets of functions, which are denoted
by Latin letters, like $F^{-1}[a,b]$.

The path $F_{\tau}$ is \emph{supported} on a subset $U\subseteq\O$, if $F_{\tau}|_{\O\setminus U}$ does not depend on $\tau$.
The \emph{concatenation} of paths $F^1_{\tau}$ and $F^2_{\tau}$ is denoted by $F^1_{\tau} \cdot F^2_{\tau}$.

In Part~\ref{part:just_paths}, we focus on immersed Morse functions on $\O$, where $G\colon N\to\Omega$ is an immersion and $N$ has all
connected components of the same dimension. Necessary changes if this is not the case are given in Section~\ref{sec:multidim}.
\section{Regular paths}\label{sec:cerf}

The aim of this section is to introduce regular paths. We specify generic conditions that guarantee nice behaviour of a function or a map,
as well as conditions on paths of objects in which the mildest possible degeneracies occur.
We begin by recalling the stratification of the space $\cF:=C^\infty(N,\R)$ in the spirit of \cite{Cerf}. In particular, we introduce
the notion of a \emph{regular path of functions}, also called an $\cF^1$-path.

Next, we use analogous methods to stratify the space of functions $C^\infty(\Omega,\R)$ (with a fixed immersed submanifold in $\Omega$)
and of smooth maps $N \to \Omega$, in particular to study regular
(in a suitable sense) paths of immersions. Next, we combine the latter two notions. First, in the presence of a function $F\colon\Omega\to\R$, we study the set of maps $G\colon N\to\Omega$, such that $F$ is an immersed Morse function with respect to $M:=G(N)$. Finally, we study
pairs $(F,G)$, where $F\colon\Omega\to\R$ and $G\colon N\to\Omega$.
We specify what it means for a path of such pairs $(F_\tau,G_\tau)$ to be regular.
In Subsection~\ref{sub:neat}, we discuss the necessary changes in the theory when $\Omega$ and $N$ have boundary.

%\color{blue!50!black}
%MB's explanation. Ordered vs. excellent. Excellent means the critical values are distinct. Ordered means excellent and
%having critical points arranged in order corresponding to the indices. Being excellent or ordered is an open condition: if we have such function, then neighboring functions satisfy the same condition. However, there is a distinction. Ordered functions are not dense, while excellent are. In fact, if we take a non-excellent ordered Morse function,
%all functions close to it, will also be non-excellent.
%
%In the study of stratifications, we do not use the word ordered. This condition appears later.
%\color{black}

\subsection{Ambient Cerf theory}\label{sub:ambient_cerf}
Cerf's theory \cite{Cerf} deals with paths of functions, that is with smooth maps $f_{\bullet}\colon[0,1]\to \mathcal{F}$, where
\[\mathcal{F}:=C^\infty(N,\R).\]
Here $N$ is either a smooth, closed manifold, or a smooth compact cobordism $(N;\bd_0 N,\bd_1 N)$ with a decomposition of the boundary $\bd N = \bd_0 N \sqcup \bd_1 N$. In the latter case we impose neatness conditions, specified in Subsection~\ref{sub:neat} below. Now we recall the parts of this theory that are needed in the proof of Path Lifting Theorem~\ref{thm:path_lifting}.

Recall that a Morse function is called \emph{excellent} if the function mapping critical points to critical values is injective.  The subspace of excellent Morse functions is open and dense in $\mathcal{F}$, which follows from Theorem~\ref{thm:density} and an easy argument to see that the excellent condition is open and dense.

We will analyse the function space $\mathcal{F}$ with a special emphasis on paths in $\cF$. The space $\cF$ has a decomposition \cite{Cerf,Mather_II}, of which we need only a part.
We describe subsets $\cF^0$, $\cF^1_{\alpha}$, and  $\cF^1_{\beta}$ in $\cF$.
\begin{itemize}
\item $\cF^0$ consists of all excellent Morse functions, meaning that the critical values of distinct critical points are distinct.
\item $\cF^1_{\alpha}$ are functions which in addition to nondegenerate critical points have exactly one cubic singularity (that is, an $A_2$ singularity)
  with local normal form
\[x_1^3-x_2^2 -\dots-x_j^2+x_{j+1}^2+\dots+x_n^2.\]
All critical values are assumed to be distinct.
\item $\cF^1_{\beta}$ consists of all Morse functions with exactly two critical points having the same value.
\end{itemize}

 Let $\cF^1:=\cF^1_{\alpha}\cup \cF^1_{\beta}$. Then $\cF^0$ and $\cF^0\cup \cF^1$ are open in $\cF$ by \cite[I,2.1]{Cerf}.
 When we want to emphasise the domain we write $\cF(N)$, $\cF^0(N)$, and $\cF^1(N)$.
 We make the following definition.

\begin{definition}\label{def:regular_path}
 A path $f_{\bullet} \colon [0,1]\to \cF^0\cup \cF^1$ is called an {\em $\cF^1$-path} or a \emph{regular path}
 if $f_\tau\in \cF^1$ for only finitely many time parameters $\tau_i$, $f_0,f_1\in\cF^0$, and if $f_{\bullet}$ intersects $\mathcal{F}^1$ transversely.
 \end{definition}

 The next result follows from the fact that $\cF^1$ is a codimension one subspace of $\mathcal{F}$; its proof can be found in
 \cite[II.3.2, III.1.3, III.2.3]{Cerf}.

 \begin{lemma}\label{lem:regular_path}
  Each path with endpoints in $\cF^0$ can be perturbed rel.\ boundary to an $\cF^1$-path. In particular, if
  $f_0,f_1\in\cF^0$, then there is an $\cF^1$-path $f_{\tau}$, $\tau\in[0,1]$ connecting $f_0$ and $f_1$.
\end{lemma}

The study of $\cF^1$-paths can be divided into the study of paths that intersect $\cF^1$ at precisely one time value.

\begin{definition}\label{def:paths-of-xx}
  An $\cF^1$-path $f_{\tau}$, $\tau\in[0,1]$  is called a \emph{path of birth}, respectively a \emph{path of death}, or a \emph{path of rearrangement},
  if $f_\tau\in\cF^1$ at precisely one time value $\tau_0$, and:
  \begin{itemize}
    \item \emph{birth:} $f_{\tau_0}\in\cF^1_\alpha$ and $f_\tau$ has two more critical points for $\tau>\tau_0$ than for $\tau<\tau_0$;
    \item \emph{death:} $f_{\tau_0}\in\cF^1_\alpha$ and $f_\tau$ has two fewer critical points for $\tau>\tau_0$ than for $\tau<\tau_0$;
    \item \emph{rearrangement:} $f_{\tau_0}\in\cF^1_\beta$.
  \end{itemize}
\end{definition}

There is a local normal form for each of these three paths. That is, any path of birth, death or rearrangement can be changed by homotopy keeping
the endpoints fixed to a path which crosses the stratum $\cF^1$ in a specific way. These forms are called \emph{elementary paths of birth, death, or rearrangement}; see \cite[Part III]{Cerf}. The description of normal forms will be useful when discussing lifting of elementary paths. These normal forms will be discussed in Section~\ref{subsection:path-birth} for the birth case, and Section~\ref{sec:path_of_death} for the death case.
For the case of paths of rearrangements, we will reformulate Cerf's method using the notion of $\xi$-paths in Section~\ref{sec:xi_path}, because vector fields allow us greater control when constructing regular homotopies.

While phrasing and proving Path Lifting Theorem~\ref{thm:path_lifting}, it will be much easier (and enough for all applications so far)
to lift a path up to suitably defined homotopy. This motivates the following definition.

\begin{definition}\label{defn:homotopy-of-paths}
  Let $f_{\tau}$ and $\wt{f}_\tau$ be $\cF^1$-paths. %\npar{Reorganised this definition into an enumerate.}{great!}
\begin{enumerate}[label=(\roman*)]
  \item  Suppose that $f_0 = \wt{f}_0$ and $f_1 = \wt{f}_1$. Then $f_{\tau}$ and $\wt{f}_\tau$ are \emph{$\cF^1$-homotopic} if they are homotopic rel.\ endpoints ($\tau=0,1$) through $\cF^1$-paths in $\cF_0 \cup \cF_1$.
  \item   We say that $f_\tau$ and $\wt{f}_\tau$ are \emph{lax homotopic} if there is a homotopy $h_{\sigma,\tau}$
 % \ypar{I'm trying to be consistent, and write $h_{\sigma,\tau}$ for the homotopy.}
  of paths such that for every $\sigma$ the path $\tau\mapsto h_{\sigma,\tau}$ is an $\cF^1$-path, $h_{0,\tau}=f_\tau$, $h_{1,\tau}=\wt{f}_\tau$
  and the paths $\sigma\mapsto h_{\sigma,0}$, $\sigma\mapsto h_{\sigma,1}$ belong to $\cF^0$.
\item Suppose $f_\tau,\wt{f}_\tau$ are $\cF^1$-paths and $g_\sigma$, $\sigma\in[0,1]$ is a path in $\cF^0$ (i.e.\ a path of excellent Morse functions) such that $g_0=f_0$ and $g_1=\wt{f}_0$. We say that $f_\tau$ and $\wt{f}_\tau$ are \emph{lax-homotopic over $g_\sigma$} if there exists
a lax homotopy $h_{\sigma,\tau}$ as in the previous item, with the additional property that $h_{\sigma,0}=g_\sigma$.
  \item  If $f_{\tau}$ and $\wt{f}_\tau$ are lax homotopic over a constant path $g_\sigma\equiv f_0=\wt{f}_0$,
    then the paths $f_\tau$ and $\wt{f}_\tau$ are called \emph{left-homotopic}.
\end{enumerate}
\end{definition}
%\ypar{MB: I've added a new terminology here. Previous statements were not precise enough and we were close to claiming that two paths in $\cF^0$ with the same end-points are homotopic. I've also commented out the lemma below.}

\begin{remark}\label{rem:weak_homotopy_revisited}
  Consider a pair of $\cF^1$-paths $f_{\tau}$ and $\wt{f}_\tau$  with the same endpoints $f_0 = \wt{f}_0$ and $f_1 = \wt{f}_1$. Then $f_{\tau}$ and $\wt{f}_\tau$  are lax homotopic if and only if they are $\cF^1$-homotopic.

We describe the proof of the harder, forward direction.
  If $f_\tau,\wt{f}_\tau$ are lax homotopic, and $h_{\sigma,\tau}$ is the corresponding homotopy, then replacing $\tau\mapsto h_{\sigma,\tau}$
  by the concatenation of paths $\sigma'\mapsto h_{\sigma',0}$, $\sigma'\in[0,\sigma]$, $\tau\mapsto h_{\sigma,\tau}$ and the inverse of the path
  $\sigma'\mapsto h_{\sigma',1}$, $\sigma'\in[0,\sigma]$, yields an $\cF^1$-path for all $\sigma \in [0,1]$, with endpoints $f_0$ and $f_1$. Hence it is straightforward to construct an $\cF^1$-homotopy.
  %This proves the following lemma.
  %Any two pairs of paths with the same endpoints are lax homotopic if and only if they are homotopic through $\cF^1$-paths.
\end{remark}

%
%\begin{comment}
%\begin{lemma}\label{lemma:lax-homotopic paths-same-endpoints-are-F1-homotopic}
%Consider a pair of $\cF^1$-paths $f_{\tau}$ and $\wt{f}_\tau$  with the same endpoints $f_0 = \wt{f}_0$ and $f_1 = \wt{f}_1$. Then $f_{\tau}$ and $\wt{f}_\tau$  are lax homotopic if and only if they are homotopic through $\cF^1$-paths.
%\end{lemma}
%
%  Despite this lemma, we introduce the notion of lax homotopy because it is frequently convenient to speak about homotopies of paths where the endpoints (especially the endpoint at $\tau=1$)  are not necessarily fixed.
%\end{comment}

\subsection{Paths of immersions}\label{sub:paths_of_immersions}
We now begin the discussion of the immersed case, which will stretch over the next few subsections.
We will consider an $n$-dimensional compact manifold
$N$ and a compact manifold $\Omega$. The aim of this subsection is to study singularities that may occur on a generic
path of immersions from $N$ to $\Omega$. Most of the results are not new; we refer e.g.\ to \cite[Section 3.2]{Ekholm}.

\begin{definition}
  The space $\cI(N,\Omega)$, in short $\cI$, is by definition the space of $C^\infty$ immersions from $N$ to $\Omega$. If $N$ and $\Omega$ have boundary,
  we assume the immersion to be neat and generic on the boundary (Definition~\ref{defn:neat-and-very-neat} below).
\end{definition}

We do not discuss the case where $N$ and $\Omega$ have boundary in detail until Subsection~\ref{sub:neat}.
Note that $\cI(N,\Omega)$ is open in $C^\infty(N,\Omega)$; see \cite[Lemma II.5.5]{GG}.
Define \[\cI^0=\cI^0(N,\Omega)\]
to be the space of generic immersions (see Definition~\ref{def:generic-immersion}).
By Lemma~\ref{lem:generic-immersion}, $\cI^0$ is open and dense in $\cI$.

Our aim is to study one-parameter deformations of a generic immersion. Our approach is essentially that of Ekholm \cite[Section 3]{Ekholm},
except that we do not have restrictions on dimensions: these restrictions are not needed for \cite[Section 3]{Ekholm}, rather they are used
in later sections of that work.

In the next definition we define a preliminary subspace of immersions, containing all immersions where genericity fails at depth $d$ for at least one point. After the definition, we will proceed to consider the points where genericity fails more than the minimal possible amount.

\begin{definition} \label{def:cI_1}
  Let $d\ge 2$.
  The space $\wt{\cI}^1_{d}$ is by definition the subspace of immersions $G\in\cI$ for which there exists
  a $d$-tuple $u_1,\dots,u_{d}$  of distinct points of $N$ that are all mapped by the immersion to the same point $w:=G(u_1)=\dots=G(u_{d})$,
  such that there are pairwise disjoint open subsets $U_1,\dots,U_{d}$ of $N$, with $U_j$ containing $u_j$,
  with the properties:
  \begin{enumerate}[label=(CI-\arabic*)]
    \item $G(U_1),\dots,G(U_{d-1})$ intersect transversely at $w$  (compare the beginning of Section~\ref{sec:transverse});
    \label{item:CI1}
    \item with $M_1=G(U_1)\cap\dots\cap G(U_{d-1})$, $M_2=G(U_{d})$, $T_wM_1+T_wM_2$ is a positive codimension
      subspace in $T_w\Omega$.\label{item:CI3}
  \end{enumerate}
\end{definition}

For $d \geq 2$ we will consider subspaces $\cI^2_{\alpha,d}$, $\cI^2_{\beta,d}$, and $\cI^2_{\gamma,d,d'}$ of $\wt{\cI}^1_d$, whose
definition is given below.
We will then set
\[\cI^2_{\alpha}:=\bigcup_{d}\cI^2_{\alpha,d},\, \cI^2_{\beta}=\bigcup_{d}\cI^2_{\beta,d},\,\cI^2_{\gamma}=\bigcup_{d,d'}\cI^2_{\gamma,d,d'}.\]
Then with \[\cI^2:=\cI^2_\alpha\cup\cI^2_\beta\cup\cI^2_\gamma,\]
we note that $\cI^2 \subseteq \wt{\cI}^1$ and define
\[\cI^1=\wt{\cI^1}\setminus \cI^2.\]
The idea of the definition of $\cI^1$ is that we remove from $\wt{\cI}^1$ all the immersions, namely those in $\cI^2$ for which the failure of genericity is worse than the minimal possible failure, so that a generic path lies in $\cI^1$.

Now we give the definitions of $\cI^2_{\alpha,d}$, $\cI^2_{\beta,d}$, and $\cI^2_{\gamma,d,d'}$.

\begin{itemize}
  \item (\emph{extra branch passing through tangency}) The subspace $\cI^2_{\alpha,d} \subseteq \wt{\cI}^1_d$ consists of maps such that
    there exists a set of points $u_1,\dots,u_{d}$
    such that $G(u_1)=\dots=G(u_{d})$,
    satisfying items~\ref{item:CI1} and~\ref{item:CI3} of Definition~\ref{def:cI_1}, but
    there is another point $u_{d+1}$ mapped to $G(u_1)$.
  \item (\emph{higher tangency}) The subspace $\cI^2_{\beta,d}  \subseteq \wt{\cI}^1_d$ consists of maps for which there is a set of points $u_1,\dots,u_d$
    satisfying items~\ref{item:CI1}, and~\ref{item:CI3} of Definition~\ref{def:cI_1},
    and $T_wM_1+T_wM_2$ is of codimension at least~2 in $T_w\O$.
  \item (\emph{two tangencies at the same time}) For $d' \geq 2$, the subspace $\cI^2_{\gamma,d,d'} \subseteq \wt{\cI}^1_d$ is the set of maps for which there is a set of points $u_1,\dots,u_d$ satisfying~\ref{item:CI1} and~\ref{item:CI3},
    as well as another set of points $u'_{1},\dots,u'_{d'}$ with $G(u'_1)=\dots=G(u'_{d'})\neq G(u_1)$ satisfying
    \ref{item:CI1} and~\ref{item:CI3}.
    %\item $\cI^2_\delta$ is the set of maps for which there is a set of points satisfying~\ref{item:CI1}, \ref{item:CI2} and~\ref{item:CI3},
    %  but there is a permutation of branches $U_1,\dots,U_{i_1+i_2}$ and possibly different integers $i'_1,i'_2\ge 1$ with $i'_1+i'_2=i_1+i_2$,
    %  such that items~\ref{item:CI1}, \ref{item:CI2} and~\ref{item:CI3} are satisfied for $i'_1$ and $i'_2$ and this extra pair.
\end{itemize}

The following result echoes \cite[Lemma 3.4]{Ekholm}.

\begin{proposition}\label{prop:gen_path}
  Suppose $G_\tau\colon N\hookrightarrow\Omega$, $\tau\in[0,1]$, is a path in $\cI$ such that $G_0,G_1\in\cI^0$.
  Then there exists a path $G'_\tau$, arbitrarily close $($in the $C^\infty$ topology$)$ to $G_\tau$, with $G'_0=G_0$ and $G'_1=G_1$,
  such that $G'_\tau$ lies in $\cI^0$ for all but finitely many values of $\tau$, and for these values we have that $G_\tau\in\cI^1$.
\end{proposition}

In the proof of  Proposition~\ref{prop:gen_path}, we will heavily use  multijet transversality from Subsection~\ref{sub:multijet}.
As a step towards Proposition~\ref{prop:gen_path}, we give the following result.

\begin{lemma}\label{lem:avoids_0}
  Suppose $r>0$, $i_1,i_2\ge 1$ are such that $(i_1-1)k\le n$ and $(i_2-1)k\le n$.
  We consider the subspace of $J^1_{i_1+i_2}(N,\O)$ determined by the condition corresponding to maps $G\colon N\to\Omega$ such that there are pairwise distinct points $u_1,\dots,u_{i_1+i_2}\in N$ such that
  $G(u_1)=G(u_2)=\dots=G(u_{i_1+i_2})$, and there exist open subsets $U_1,\dots,U_{i_1+i_2}$ of $N$, for which the images $G(U_1),\dots,G(U_{i_1})$
  are transverse at $w:=G(u_1)$, the images $G(U_{i_1+1}),\dots,G(U_{i_1+i_2})$ are transverse at $w$, and
  \[\dim(T_w(G(U_1)\cap\dots\cap G(U_{i_1}))+T_w(G(U_{i_1+1})\cap\dots\cap G(U_{i_1+i_2})))\le\dim\Omega-r.\]
  The subspace of the $(i_1+i_2)$-fold 1-multijet space $J^1_{i_1+i_2}(N,\O)$ has codimension
  \[(i_1+i_2-1)(n+k)+r(n-(i_1+i_2-1)k+r)\] in  $J^1_{i_1+i_2}(N,\O)$.
\end{lemma}

\begin{proof}
  The condition that $G(u_1)=G(u_2)=\dots=G(u_{i_1+i_2})$ is of codimension $(i_1+i_2-1)\dim\Omega = (i_1 +i_2-1)(n+k)$, which accounts for the first summand.
  The condition that $G(U_1),\dots,G(U_{i_1})$ are transverse at $w$ is open. Define \[M_1:=G(U_1)\cap\dots\cap G(U_{i_1}).\]
  By transversality, near $w$, this is a manifold of dimension $s_1:=n-(i_1-1)k$.

  Analogously, transversality of $G(U_{i_1+1}),\dots,G(U_{i_1+i_2})$ is an open condition; we set
  \[M_2 := G(U_{i_1+1})\cap\dots\cap G(U_{i_1+i_2}).\]
  This is a manifold (near $w$) of dimension $s_2:=n-(i_2-1)k$.
  Set $V_1=T_wM_1$ and $V_2=T_wM_2$.
  The condition that $V_1+V_2$ forms a codimension~$r$ subspace of $T_w\Omega$ can be computed as follows.
  Choose a basis $e_{i1},\dots,e_{is_i}$ of $V_i$. We have two cases.
  \begin{itemize}
    \item if $\ell=n+k-r-s_1\ge 0$, then we choose the vectors $e_{11},\dots,e_{1s_1},e_{21},\dots,e_{2\ell}$ generically. They span a subspace of dimension $n+k-r$. The remaining $s_2-\ell$ vectors, $e_{2,\ell+1},\dots e_{2s_2}$ are all compelled to stay the codimension $r$ subspace, which gives $r(s_1+s_2-(n+k-r))=r(n-(i_1+i_2-1)k+r)$ conditions.
    \item if $\ell<0$, we can choose only the vectors $e_{11},\dots,e_{1,n+k-r}$ points generically so that they span a subspace of dimension $n+k-r$. There remains $s_1-(n+k-r)+s_2=s_2-\ell$ vectors $e_{1,n+k-r+1},\dots,e_{1s_1}$, $e_{21},\dots,e_{2s_2}$, which are all forced to belong to the codimension $r$ subspace. This gives the same amount of conditions as the first case.
  \end{itemize}
  This accounts for the second summand.
  Altogether, we have $(i_1+i_2-1)(n+k)+r(n-(i_1+i_2-1)k+r)$ conditions as desired.
\end{proof}

\begin{proof}[Proof of Proposition~\ref{prop:gen_path}]
  Choose $d\ge 2$.
  If $kd>n+1$, then the condition that there exist $d$ points $u_1,\dots,u_d$ such that $G(u_1)=\dots=G(u_d)$
  is of codimension $(d-1)(n+k)>nd+1$
  in the multijet space $J^0_d(N,\Omega)$. Hence the stratum  $\cI^{1}_{d}$ is missed by
  a (one-dimensional) path of immersions by Corollary~\ref{cor:multijet}~(ii).

Hence, for the rest of the proof, we suppose $d \geq 2$ and $kd\le n+1$.  The strategy is as follows.
Every path of immersions can be assumed to lie in $\cI^0 \cup \wt{\cI}^1$. We then consider the codimension of the subspaces $\cI^2_{\alpha,d}$, $\cI^2_{\beta,d}$, and $\cI^2_{\gamma,d,d'}$, and we check that the complements are each residual, so each subspace can be avoided by a generic path of immersions.  Then  we use that a countable intersection of residual sets (given by these three sets, for all values of $d$ and $d'$) is residual (Lemma~\ref{lem:countable-int-residual-is-residual}).

  \begin{lemma}\label{lem:on_alpha}
    The conditions on $\cI^2_{\alpha,d}$ define a subspace of codimension equal to
    $n(d+1)+k+1$ in $J^1_{d+1}(N,\Omega)$. In particular, $\cI^2_{\alpha,d}$ is
    missed by a generic $k$-parameter family of immersions from $N$ to $\Omega$.
  \end{lemma}

  \begin{proof}
    By Lemma~\ref{lem:avoids_0} with $r=1$, $i_1=d-1$, $i_2=1$, we get $dn+1$ conditions. The extra point $u_{d+1}$ mapped to $G(u_1)$
    gives $(n+k)$ more conditions. This gives $n(d+1)+k+1$ conditions. These conditions involve $i_1+i_2+1 = d+1$ points of $N$,
    so we end up with a codimension $n(d+1)+k+1$ subspace of $J^1_{d+1}(N,\Omega)$. Since $\dim N^{(d+1)}=n(d+1)$,
    the space $\cI^2_{\alpha,d}$ is avoided by a generic $k$-parameter family of immersions by Corollary~\ref{cor:multijet}~(iii).
  \end{proof}

  \begin{lemma}\label{lem:on_beta}
    The conditions on $\cI^2_{\beta,d}$ define a subspace of codimension greater than or equal to $nd+4$
    in $J^1_{d}(N,\Omega)$. A generic $3$-parameter family of immersions from $N$ to $\Omega$ avoids $\cI^2_{\beta,d}$.
  \end{lemma}

  \begin{proof}
    We apply Lemma~\ref{lem:avoids_0} for $i_1=d-1$ and $i_2=1$.
    As $r\ge 2$, this gives codimension $nd+(r-1)(n-(d-1)k)+r^2>nd+r^2\ge nd+4$  by Corollary~\ref{cor:multijet}~(iii).
    %By Theorem~\ref{thm:multijet} with $s=d$ and $\dim X = \dim N =n$,
  \end{proof}

  Next, we deal with the case of independent tangencies. For this, we assume that $d,d' \geq 2$ are such that both $kd\le n+1$ and $kd'\le n+1$.

  \begin{lemma}\label{lem:on_gamma}
    The conditions on $\cI^2_{\gamma,d,d'}$ define a subspace of codimension greater than $n(d+d')+2$ in $J^1_{d+d'}(N,\Omega)$. In particular, a generic one-parameter
    family avoids $\cI^2_{\gamma,d,d'}$.
  \end{lemma}

  \begin{proof}
  Conditions~\ref{item:CI1}, and~\ref{item:CI3} give $nd+1$ conditions involving points $u_1,\dots,u_d$
  and, independently, $nd'+1$ conditions involving points $u'_1,\dots,u'_{d'}$. Altogether,
  we get $n(d+d')+2$ conditions in $J^1_{d+d'}(N,\Omega)$. We conclude by Corollary~\ref{cor:multijet}~(ii).
  \end{proof}

  We can now finish the proof of Proposition~\ref{prop:gen_path}.
  Lemmas~\ref{lem:on_alpha}, \ref{lem:on_beta}, and~\ref{lem:on_gamma} show that the sets of one-parameter families
  missing $\cI^2_{\alpha,d}$, $\cI^2_{\beta,d}$ and $\cI^2_{\gamma,d,d'}$ form a residual set in the set of all smooth immersions
  from $N$ to $\Omega$, so a residual set in the space of all smooth immersions from $N$ to $\Omega$ by
  openness of $\cI(N,\Omega)$ in $C^\infty(N,\Omega)$.
  The intersection of finitely many residual sets is residual by Lemma~\ref{lem:countable-int-residual-is-residual}, which gives the statement of Proposition~\ref{prop:gen_path}.  % This concludes the proof of Proposition~\ref{prop:gen_path}.
\end{proof}

\begin{definition}
A path $G_\tau$, $\tau\in[0,1]$ of immersions from $N$ to $\Omega$ is called \emph{regular} if $G_0,G_1\in\cI^0$ and $G_\tau\in\cI^0$
for all but finitely many values of $\tau$, and if whenever $G_\tau$ is not in $\cI^0$, then it lies in $\cI^1$.
\end{definition}

Proposition~\ref{prop:gen_path} can be rephrased in the following way.

\begin{corollary}\label{cor:gen_path}
  Every path $G_\tau$ with endpoints generic immersions can be perturbed rel.\ endpoints to a regular path of immersions.
\end{corollary}

We conclude this subsection with the following observation. Suppose $G_\tau$ is a regular path such that for $\tau\neq\frac12$, $G_\tau$
is a generic immersion, and for $\tau=\frac12$, $G_\tau$ has a tangency at the $d$-th stratum. The topology of $G_\tau(N)$ changes: in particular
the $d$-th stratum undergoes a surgery while crossing the value $\tau=\frac12$. A precise description can be deduced from \cite[Lemma 3.5]{Ekholm}, which can be easily generalised to the case when $\dim N=n$ and $\dim\Omega=n+k$ (as mentioned above, Ekholm's paper \cite{Ekholm}
has specific dimension constraints, but the proofs in his Section~3 work without these restrictions). %As we do not use this type of results in the paper, we do not go into details.

\begin{remark*}
  Note that a regular homotopy is a path of immersions, but it is not the case that every regular homotopy is a regular path of immersions.
  %\npar{MP: added this, since it seemed potentially confusing.}{MB: Deleted the word `unfortunately'. Life would be dull if any regular homotopy were a regular path of immersions.}
\end{remark*}

\subsection{Paths of immersed Morse functions}\label{sub:immersed_cerf}

In this subsection, we consider a fixed map $G\colon N\to\Omega$, which we assume to be a generic immersion.
We set $M=G(N)$ and define the strata $\O[0],\O[1],\dots$ as in Subsection~\ref{sub:immersed_morse_intro}.
Consider the space of maps $\cA := C^\infty(\Omega;\R)$.

\begin{definition}\label{def:spaces_cA}
Let $d \geq 0$.
  \begin{itemize}
    \item (\emph{non Morse critical point}) The subspace $\cA_{\alpha,d}\subseteq\cA$ is by definition the space of functions $F\in\cA$ such that there exists
      a point $p\in \O[d]$ with $Df_d(p)=0$ (with $f_d:=F|_{\O[d]}$)  and such that $D^2f_d(p)$ has nontrivial kernel.
      Inside $\cA_{\alpha,d}$ we specify the following subsets.
      \begin{itemize}
	  \item (\emph{non-degenerate non-Morse singularity}) Define the subspace $\cA^1_{\alpha,d}\subseteq \cA_{\alpha,d}$  by the conditions that $\dim\ker D^2f_d(p)=1$, and for $v\in\ker D^2f_d(p)$, $v\neq 0$, we have $D^3f_d(p)(v,v,v)\neq 0$.
	  \item (\emph{degenerate non-Morse singularity}) Define the subspace $\cA^2_{\alpha,d} \subseteq \cA_{\alpha,d}$ by the condition that $\dim\ker D^2f_d(p)\ge 2$, or for all $v\in\ker D^2f_d(p)$,
	  we have $D^3f_d(p)(v,v,v)=0$.
  \end{itemize}
  \item (\emph{two critical points at the same level}) The subspace $\cA^1_{\beta,d_1,d_2}\subseteq\cA$ is the set of maps such that there are two critical points $p_1$ of $F|_{\O[d_1]}$
    and $p_2$ of $F|_{\O[d_2]}$ such that $F(p_1)=F(p_2)$, but no third critical point has this property. The subspace $\cA^2_{\beta,d_1,d_2,d_3}$
    is the set of maps such that there exists $p_i$, critical points of $F|_{\O[d_i]}$, with $F(p_1)=F(p_2)=F(p_3)$.
  \item (\emph{extra branch passing through a critical point}) For each $s$ with $1 \leq s <d$, define the subspace $\cA_{\gamma,d,s}\subseteq\cA$ to be the subspace of smooth functions such that there is a critical point $p\in \O[d]$ of $F|_{\O[d]}$ and there are
      branches $Y_1,\dots,Y_s$, through $p$ such that $F$ restricted to $Y_1\cap\dots\cap Y_s$ has a critical point at $p$.
    \item (\emph{two events at the same time}) The subspace $\cA_{\omega}$ is the subspace of maps that satisfy independently two conditions
    from the above list, e.g. two non-Morse critical points, two critical points at the same level etc.
\end{itemize}
  We let
  \[\cA_\alpha=\bigcup_d\cA_{\alpha,d},\, \cA_{\beta}=\bigcup_{d_1,d_2}\cA_{\beta,d_1,d_2},\text{ and } \cA_\gamma=\bigcup_{d,s}\cA_{\gamma,d,s}.\]
  Furthermore we set
  \[\cA^1_\alpha=\bigcup_d\cA^i_{\alpha,d};\;\; \cA^2_\alpha=\bigcup_d\cA^i_{\alpha,d}; \;\; \cA^1_{\beta}=\bigcup_{d_1,d_2}\cA^1_{\beta,d_1,d_2} \text{ and } \cA^2_{\beta}=\bigcup_{d_1,d_2,d_3}\cA^i_{\beta,d_1,d_2,d_3},\]
%  \npar{I think we have to define separately for $i=1,2$ in the $\beta$ case. I changed it in the $\alpha$ case too, to match.}{I'm not sure if I understand. Is there anything I shoud be doing with this? MP: no, just approving of what is done (or at least accepting it).}
   so that $\cA_\alpha=\cA_\alpha^1\cup\cA_\alpha^2$, and $\cA_\beta=\cA^1_\beta\cup\cA^2_\beta$.
  The set $\cA^0$ is the complement
  \[\cA^0 := \cA \sm (\cA_\alpha\cup\cA_\beta\cup\cA_\gamma).\]
\end{definition}

  Note that the set $\cA^0$ consists precisely of the excellent immersed Morse functions in the sense of Definition~\ref{def:immersed_morse_function}.
  The subspace $\cA_\alpha$ corresponds to the failure of condition~\ref{item:IM1}. The subspace $\cA_{\beta}$ corresponds to the non-excellent Morse functions,
  while $\cA_{\gamma}$ is the failure of the~\ref{item:IM2} condition.

  \begin{lemma}\label{lem:codim_cA}~
  \begin{enumerate}[label=(\roman*)]
    \item The set $\cA^1_{\alpha,d}$ is defined, via the jet extension map, using a union of $($finitely many$)$ strata of codimension at least $n+k+1$ in $J^3(\Omega;\R)$.
    \item The set $\cA^2_{\alpha,d}$ is defined using a union of strata of codimension at least $n+k+2$ in $J^3(\Omega;\R)$.
    \item   The set $\cA^1_{\beta,d_1,d_2}$ is defined by a stratum of codimension $2(n+k)+1$ in $J^1_2(\Omega;\R)$,
      while $\cA^2_{\beta,d_1,d_2,d_3}$ is defined by a stratum of codimension $3(n+k)+2$ in $J^1_3(\Omega;\R)$,
    \item   The set $\cA_{\gamma,d,s}$ is defined by a stratum of codimension at least $n+2k$ in $J^1(\Omega;\R)$.
  \end{enumerate}
\end{lemma}

\begin{proof}
  The stratum $\cA_{\alpha,d}$ is defined by the following conditions: a point $p$ belongs to $\O[d]$ (codimension equal to $\codim \O[d]\subseteq\Omega$),
  the derivative $DF(p)$ vanishes at $T_p \O[d]$ (codimension equal to $\dim \O[d]$).
  Now $\det D^2f_d(p)=0$ is one more condition. Altogether, there are $1+\dim\Omega = n+k+1$ conditions.
  Next, within $\cA_{\alpha,d}$ we specify $\cA^1_{\alpha,d}$ by open conditions (codimension zero). The set $\cA^2_{\alpha,d}$ is
  the complement of $\cA^1_{\alpha,d}$ in $\cA_{\alpha,d}$ and it is a union of two strata (one is when $\dim\ker D^2f_d(p)>1$, the
  other is when $D^3f_d(p)(v,v,v)=0$). Each of these two strata is defined using one extra condition, so the codimension is one higher, namely $n+k+2$.

  The stratum $\cA^1_{\beta,d_1,d_2}$ is defined by $p_1\in \O[d_1]$ and $DF(p_1)$ vanishes on $T_{p_1}\O[d_1]$ (codimension equal to $\dim\Omega$
  as above). The same codimension applies for $p_2\in \O[d_2]$ being a critical point of $F|_{\O[d_2]}$. Next, there is an additional
  condition $F(p_1)=F(p_2)$, so we get $2\dim\Omega + 1 = 2(n+k) +1$ conditions. The same reasoning gives the codimension
  of the stratum $\cA^2_{\beta,d_1,d_2,d_3}$.

  Finally, consider $\cA_{\gamma,d,s}$. For $p\in \O[d]$, the condition that $DF$ vanishes on $T_p(Y_1\cap\dots\cap Y_s)$
  specifies a stratum of codimension $\dim(Y_1\cap\dots\cap Y_s)=\dim \O[d]+k(d-s)\ge \dim \O[d]+k$. Together with $p\in \O[d]$ ($\dim\Omega-\dim \O[d]$ conditions),
  we have codimension at least $n+2k$.
\end{proof}
We deal with $\cA_{\omega}$. We use the following principle. %\ypar{Added a discussion of multiple events.}

\begin{lemma}[Principle of independent singularities]\label{lem:A_omega}
  Let $X,Y$ be two smooth manifolds.
  Suppose two properties of a function define subspaces $\cA_1 \subseteq J^{a_1}_{b_1}(X,Y)$, and $\cA_2 \subseteq J^{a_2}_{b_2}(X,Y)$. A simultaneous occurrence of these two properties defines a subspace $\cA_{12}$ of $J^{\max(a_1,a_2)}_{b_1+b_2}(X,Y)$ of codimension $\codim \cA_1+\codim\cA_2$. In particular, if two properties $\cA_1,\cA_2$ have codimension at least $b_1\dim\Omega+1$ and $b_2\dim\Omega+1$ respectively,
  then a generic 1-parameter family avoids the situation when two properties appear simultaneously.
\end{lemma}
\begin{proof}
  Say the first property depends on the values of functions and differentials (up to ordeer $a_1$) at points $p_1,\dots,p_{b_1}$,
  and the total number of conditions is $\codim\cA^1$
  conditions on them. Suppose analogously, for the second property, that it involves $\codim\cA^2$ conditions on the values of functions
  and differentials at another $b_2$-tuple of points, $p'_{1},\dots,p'_{b_2}$. The simultaneous occurrence means that we simply merge the conditions, so that we have $\codim\cA^1+\codim\cA^2$ conditions on a $(b_1+b_2)$-tuple of points $p_1,\dots,p_{b_1},p'_1,\dots,p'_{b_2}$.

  The second part follows from Corollary~\ref{cor:multijet}(ii).
\end{proof}
We point out that multijet transversality is adapted to the situation, where the two tuples of points $p_1,\dots,p_{b_1}$ and $p'_1,\dots,p'_{b_2}$ in the proof of Lemma~\ref{lem:A_omega} do not coincide. That is, the two properties are really considered as independent.

\begin{corollary}\label{cor:paths}
  The set of functions $F\in\cA$ avoiding $\cA_\alpha\cup\cA_\beta\cup\cA_\gamma$ is open and dense.
  The set of paths of functions $F_\tau\in\cA$, $\tau\in[0,1]$ avoiding $\cA^2_\alpha$, $\cA^2_\beta$, $\cA_\gamma$ and $\cA_\omega$, and
  intersecting $\cA^1_\alpha$ and $\cA^1_\beta$ in finitely many values of $\tau \in (0,1)$, is open and dense.

  In particular, any two functions $F_0,F_1\in \cA\setminus\cA_{\alpha}\cup\cA_\beta\cup\cA_\gamma$ can be connected by such a path.
\end{corollary}

\begin{proof}
  The parameter counting argument of Lemma~\ref{lem:codim_cA} and the Multijet Transversality Theorem~\ref{thm:multijet}, used via Corollary~\ref{cor:multijet}, provide us with density both of functions and of paths, where for $\cA_\omega$ we use the independence principle enunciated in Lemma~\ref{lem:A_omega}.
  As $\cA_\alpha$
  and $\cA_\gamma$ involve only a single point (so can be handled by the Jet Transversality Theorem), we have openness. However, the
  set of functions avoiding $\cA_\beta$ is clearly open, as well as the set of paths having finitely many rearrangements.
\end{proof}

\begin{definition}
We call a path of functions as in the corollary an \emph{$\cA^1$-regular path} of functions on $\O$.
\end{definition}

We remark that Corollary~\ref{cor:paths} implies Theorem~\ref{thm:density}, so we have now provided the proof of the latter theorem, as promised.

\subsection{Paths of immersions relative to a function}\label{sub:paths_relative}
Now we consider the situation where $F\colon\Omega\to\R$ is fixed and $F$ is Morse (in the classical sense) as a function on $\Omega$. Suppose
$G\colon N\looparrowright \Omega$ is an immersion.

\begin{definition}\label{def:another_F_word}
  The map $G$ is called \emph{$F$-regular} if $G$ is a generic immersion (i.e.\ $G\in\cI^0$), and $F$ is immersed Morse with respect to $M=G(N)$.
  %The map $G$ is called
  %\emph{almost $F$-regular} if $G$ is a generic immersion, $F$ is Morse when restricted to the first stratum of $M$ and for any critical point $p$ of $F|_{\O[d]}$ with $d>1$,
  %and for any branch $Y_i$ through $p$, $F|_{Y_i}$ has no critical point at $p$.
\end{definition}

%In other words, almost $F$-regular immersions are those for which the behaviour of critical points at depth $1$ is controlled. In light
%of Theorem~\ref{thm:deeperstrata}, critical points at depth~$1$ are crucial in the proof that concordance implies regular homotopy.

We let $\cR^0$ denote the set of all $F$-regular immersions. %We let $\cR^{+}$ be the set of almost regular immersions.
We now consider the complement of $\cR^0$ in $\cR:= C^{\infty}(N,\O)$.
The following definition is similar in spirit to Definition~\ref{def:spaces_cA}, except
that in Definition~\ref{def:spaces_cA} we discussed the subspaces of $C^\infty(\Omega,\R)$, while here we consider the
analogous conditions, but on the space $\cR$ of smooth maps from $N$ to $\Omega$.
\begin{definition}\ \label{def:spaces_cR}
Let $\cR := C^{\infty}(N,\O)$ and fix an excellent Morse function $F \colon \O \to \R$.
  \begin{itemize}
    \item (\emph{a non-Morse singularity at the stratum of depth $d$}) The subspace $\cR_{\alpha,d}(F) \subseteq \cR$ of smooth maps is by  definition the space of maps $G\in\cI^0$ such that, with $M:=G(N)$ and  $f_d=F|_{\O[d]}$,  there is  a point $p\in \O[d]$ with $Df_d(p)=0$, $f_d=F|_{\O[d]}$,  and with $D^2f_d(p)$ degenerate.
      Inside $\cR_{\alpha,d}(F)$ we specify the following subspaces.
      \begin{itemize}
	\item (\emph{non-degenerate case}) The subspace $\cR^1_{\alpha,d}(F) \subseteq \cI^0 \subseteq \cR_{\alpha,d}(F)$ consists of those maps that additionally satisfy that $D^2f_d(p)$ has one-dimensional kernel spanned by $v$ and $D^3f_d(p)(v,v,v)\neq 0$.
	\item (\emph{degenerate case}) The subspace $\cR^2_{\alpha,d}(F) \subseteq \cR_{\alpha,d}(F)$ consists of those maps that additionally satisfy that $D^2f_d(p)$ has either at least two-dimensional kernel, or it is a one dimensional kernel
	  and $D^3f_d(p)(v,v,v)=0$ for all $v\in\ker D^2f_d(p)$.
      \end{itemize}
    \item (\emph{two critical points on the same level set}) The subspace $\cR^1_{\beta,d_1,d_2}(F)  \subseteq \cI^0 \subseteq  \cR$ is the set of maps $G\in\cI^0$ such that with $M=G(N)$,
      there are two critical points $p_1$ of $F|_{\O[d_1]}$
      and $p_2$ of $F|_{\O[d_2]}$ such that $F(p_1)=F(p_2)$, but no other critical points have critical value $F(p_1)$. The subspace $\cR^2_{\beta,d_1,d_2,d_3}(F)$ is the set of maps such that $p_i\in \O[d_i]$ is a critical point of $F|_{\O[d_i]}$, $i=1,2,3$ and $F(p_1)=F(p_2)=F(p_3)$;
    \item (\emph{extra branch passing through critical points}) The subspace $\cR_{\gamma,d,s}(F)  \subseteq \cR$ is the space of maps $G\in\cI^0$
      such that there is a critical point $p\in \O[d]$ of $F|_{\O[d]}$ and there are
      branches $Y_1,\dots,Y_s$, $s<d$, through $p$ such that $F$ restricted to $Y_1\cap\dots\cap Y_s$ has a critical point at $p$.
    \item (\emph{failure to be a generic immersion}) The subspace $\cR_{\delta,d}  \subseteq \cR$ is the space of maps having a $d$-fold non-transversality point (similar to $\wt{\cI}^1_d$). We let $w\in\Omega$ be a non-transversality
      point of $G$ and let $u_1,\dots,u_d$ be the preimages of $W$ and let $U_1,\dots,U_d$ be pairwise disjoint open neighbourhoods in $N$
      of $u_1,\dots,u_d$. We specify two subspaces of $\cR_{\delta,d}$.
      \begin{itemize}
	\item (\emph{simple non-transversality point}) The subspace $\cR^1_{\delta,d}(F) \subseteq \cR_{\delta,d}$ is the space of maps in $\cI^1$, where for any subset $(i_1,\dots,i_r)$
	  of $(1,\dots,d)$ for which $M':=G(U_{i_1})\cap\dots\cap G(U_{i_r})$ is a transverse intersection, and $\dim M'>0$,
	  the restriction $F|_{M'}$ has nonvanishing derivative at $w$.
	\item (\emph{degenerate non-transversality point}) The subspace $\cR^2_{\delta,d}(F) \subseteq \cR_{\delta,d}$ is the space of maps $G$ where either $G\in\cI^2$, or $F|_{M'}$ has a critical point at $w$ for some transverse intersection $M'$ defined as in the previous item.
	\item (\emph{two events happening at the same time}) The subspace $\cR^2_\omega(F)$ is the subspace of maps that satisfy independently two conditions
    from the above list, e.g. two non-Morse critical points, two critical points at the same level etc.
      \end{itemize}
  \end{itemize}
We let
\[\cR^i_\alpha(F) := \bigcup_d\cR^i_{\alpha,d}(F),\,\, \cR^1_{\beta}(F):= \bigcup_{d_1,d_2}\cR^1_{\beta,d_1,d_2}(F), \text{ and } \cR_\gamma(F)=\bigcup_{d,s}\cR_{\gamma,d,s}(F).\]
We also define $\cR^i_\delta(F)=\bigcup_{d}\cR^i_{\delta,d}(F)$ for $i=1,2$, and $\cR^2_\beta(F)=\bigcup_{d_1,d_2,d_3}\cR^2_{\beta,d_1,d_2,d_3}(F)$. Set
\[\cR^1(F)=\cR^1_{\alpha}(F)\cup\cR^1_\beta(F)\cup\cR^1_\delta(F) \text{ and }\cR^2(F)=\cR^2_{\alpha}(F)\cup\cR^2_\beta(F)\cup\cR_\gamma(F)\cup\cR^2_\delta(F).\]
\end{definition}

  Similarly to the case of the $\cA$ spaces, the subspace $\cR_\alpha(F)$ corresponds to the failure of condition~\ref{item:IM1} of $F$ as an immersed Morse function
  on $M=G(N)$. The subspace $\cR_{\beta}$ corresponds to non-excellent Morse functions,
  while $\cR_{\gamma}$ is the failure of the~\ref{item:IM2} condition. Finally, $\cR_\delta$ corresponds to maps $G$ that are
  not immersions.
  The following statement has essentially the same proof as for the $\cA$ spaces.

 \begin{lemma}\label{lem:cR_function}
    The set of maps $G$ that avoid $\cR^1(F)\cup\cR^2(F)$ is open-dense.
  \end{lemma}

  The 1-parameter variant of Lemma~\ref{lem:cR_function} holds as well, but the $\cR^1$ stratum cannot be avoided in general.

\begin{lemma}\label{lem:cR_path}
  Suppose $k>1$. Suppose $G_\tau$ is a path in the space of immersions $\cI=\cI(N,\Omega)\subseteq C^\infty(N,\Omega)$. Then $G_\tau$ can be perturbed in such a way that $G_\tau$ avoids $\cR^2(F)$, avoids $\cI^2$,
  and hits $\cR^1(F)\cup\cI^1$ only at finitely many points. Moreover, if the original path was such that $F$ is immersed Morse
  with respect to $G_0(N)$ and $G_1(N)$, then we may assume that the perturbation  fixes $G_0$ and $G_1$.
\end{lemma}

\begin{proof}
The parameter counting argument for $\cR_\alpha,\cR_\beta$, and~$\cR_\gamma$ is the same as in the case of the
corresponding spaces $\cA$, so the set of maps missing $\cR_\alpha^2\cup\cR_\gamma$ is residual. For $\cR_\delta$ note that the conditions defining $\cR^1_{\delta}$ form an open-dense subset
in the space $\cI^1$: indeed, non-vanishing of the derivative of $F$ at each of $M'$ is an open-dense condition.
The case of $\cR_\omega$ is handled by Lemma~\ref{lem:A_omega}, because the principle remains the same.
The intersection of a residual set with an open-dense subset is residual, hence dense.
\end{proof}

Motivated by Lemma~\ref{lem:cR_path}, we introduce the following definition.

\begin{definition}\label{def:F_path}
  Suppose $F$ is an excellent Morse function on $\Omega$. A path $G_\tau\in\cI$, $\tau\in[0,1]$, of
  immersions from $N$ to $\Omega$ is called
  an \emph{$F$-regular path} if $G_0,G_1$ are generic immersions, $F$ is immersed Morse with respect to $G_0(N),G_1(N)$,
  and the path $G_\tau$ is transverse to $\cR^1(F)$ avoiding $\cR^2(F)$.
\end{definition}

With this definition, Lemma~\ref{lem:cR_path} allows us to state that any path of immersions can be perturbed to an $F$-regular path.
Note that a function belonging to an $F$-regular path of functions need not be itself an $F$-regular function.

\subsection{Double paths}\label{sub:double_path}

We now merge the notion of a path of Morse functions and the path of immersions into the notion of a double path.
A \emph{double path} is a pair of paths $(F_\tau,G_\tau)$, $\tau\in[0,1]$, such that $F_{\tau} \colon \O \to \R$ is an $\cF^1$-path (Definition~\ref{def:regular_path}), and for each $\tau$, $G_\tau\colon N\to\Omega$ is an immersion.

We will specify regularity conditions on a double path, which guarantee that for all but finitely many $\tau\in[0,1]$, $F_\tau$ is an immersed Morse function with respect to the pair $(\Omega,G_\tau(N))$, while for those finitely many parameter values where it is not, $F_\tau$ has mildest possible singularities.

The logic behind our construction is that we will usually fix an $\cF^1$-regular path $F_\tau\colon\Omega\to\R$,
and perturb $G_\tau$ as in Subsection~\ref{sub:paths_relative}. We will make precise what it means for
$G_\tau$ to be an $F_\tau$-regular path of functions, where $F_\tau$ changes with $\tau$ too.

To begin the plan sketched above, fix an $\cF^1$-path of functions $F_\tau \colon \O \to \R$ on $\Omega$.  Let $\cD := C^\infty(N\times[0,1],\Omega)$. We think of a path $G_\tau$ as an element of $\cD$ (by currying).
To specify generic regularity conditions, one defines subspaces $\cD_{\alpha}$, $\cD_{\beta}$, $\cD_{\gamma}$, and $\cD_{\delta}$ of $\cD$, with respect to the fixed path $F_\tau$, by analogy with the spaces $\cR_\alpha$, $\cR_\beta$, $\cR_\gamma$, and~$\cR_\delta$.
  There is one extra subspace needed, $\cD^2_{\varepsilon}$, which takes care of points where $F_\tau$ fails to be Morse at the same time as $G_\tau$ fails to be generic.
Here are the precise definitions.

\begin{definition}\ \label{def:spaces_cD}
Fix an $\cF^1$-path $F_{\tau}$.
\begin{itemize}
    \item (\emph{non-Morse singularity at higher depth}) The subspace $\cD_{\alpha,d}  \subseteq \cD $ of smooth paths is by  definition the space of paths $G_\tau$ such that there exists $\tau_1$ with $F_{\tau_1}$ a Morse function and  $G_{\tau_1} \in \cR_{\alpha,d}(F_{\tau_1})$.
      Inside $\cD_{\alpha,d}$ we specify the following subspaces.
      \begin{itemize}
	\item (\emph{non-degenerate non-Morse critical point}) The subspace $\cD^1_{\alpha,d} \subseteq  \cD_{\alpha,d}$ consists of those paths such that  there exists $\tau_1$ with  $F_{\tau_1}$ a Morse function and $G_{\tau_1} \in \cR_{\alpha,d}^1(F_{\tau_1})$.
	\item (\emph{degenerate non-Morse critical point}) The subspace $\cD^2_{\alpha,d} \subseteq \cD_{\alpha,d}$ consists of those paths $G_\tau$ such that there exists $\tau_1$ with  $F_{\tau_1}$ a Morse function and  $G_{\tau_1} \in \cR_{\alpha,d}^2(F_{\tau_1})$.
%
% that additionally satisfy that $D^2f_d(p)$ has either at least two-dimensional kernel, or it is a one dimensional kernel
%	  and $D^3f_d(p)(v,v,v)=0$ for all $v\in\ker D^2f_d(p)$.
      \end{itemize}
    \item (\emph{two critical points on the same level set}) The subspace $\cD^1_{\beta,d_1,d_2}  \subseteq  \cD$ is the set of paths $G_\tau$ such that there exists $\tau_1 \in [0,1]$ with  $F_{\tau_1}$ a Morse function and  $G_{\tau_1} \in \cR^1_{\beta,d_1,d_2}(F_{\tau_1})$.
    \item (\emph{three or more critical points on the same level set}) The subspace $\cD^2_{\beta,d_1,d_2,d_3}  \subseteq  \cD$ is the set of paths $G_\tau$ such that there exists $\tau_1 \in [0,1]$ with  $F_{\tau_1}$ a Morse function and  $G_{\tau_1} \in \cR^2_{\beta,d_1,d_2,d_3}(F_{\tau_1})$.
    \item (\emph{extra branch passing through a critical point}) For $s<d$, the subspace $\cD_{\gamma,d,s}  \subseteq \cD$ is the space of paths $G_\tau$
      such that there exists $\tau_1 \in [0,1]$ with  $F_{\tau_1}$ a Morse function and $G_{\tau_1} \in \cR_{\gamma,d,s}(F_{\tau_1})$.
    \item (\emph{failure to be a generic immersion}) The subspace $\cD_{\delta,d}  \subseteq \cD$ is the space of paths $G_\tau$
      such that there exists $\tau_1 \in [0,1]$ with  $F_{\tau_1}$ a Morse function and  $G_{\tau_1} \in \cR_{\delta,d}(F_{\tau_1})$.
  We specify two subspaces of $\cD_{\delta,d}$.
      \begin{itemize}
	\item The subspace $\cD^1_{\delta,d} \subseteq \cD_{\delta,d}$ is the space of paths $G_\tau$
      such that there exists $\tau_1 \in [0,1]$ with  $F_{\tau_1}$ a Morse function and $G_{\tau_1} \in \cR^1_{\delta,d}(F_{\tau_1})$.
	\item The subspace $\cD^2_{\delta,d} \subseteq \cD_{\delta,d}$ is the space of paths $G_\tau$
      such that there exists $\tau_1 \in [0,1]$ with  $F_{\tau_1}$ a Morse function and  $G_{\tau_1} \in \cR^1_{\delta,d}(F_{\tau_1})$.
      \end{itemize}
\item (\emph{failure of $F_\tau$ to be excellent Morse}) The subspace $\cD^2_{\varepsilon}$ is the set of paths $G_\tau$ for which there exists $\tau_1 \in [0,1]$ such that $F_{\tau_1}$ is not Morse and $G_{\tau_1}$ is not a generic immersion.
    \item (\emph{two events at the same time}) The subspace $\cD_{\omega}$ is the subspace of paths such that for
      some $\tau_1\in[0,1]$, $G_{\tau_1}\in \cR_\omega(F_{\tau_1})$.
  \end{itemize}
\end{definition}

We let
\[\cD^i_\alpha  := \bigcup_d\cD^i_{\alpha,d} ,\,\, \cD^1_{\beta} := \bigcup_{d_1,d_2}\cD^1_{\beta,d_1,d_2} ,  \cD_\gamma  =\bigcup_{d,s}\cD_{\gamma,d,s},\text{ and } \cD^i_\delta = \bigcup_{d}\cD^i_{\delta,d}.\]
Set also $\cD^2_\beta:=\bigcup_{d_1,d_2,d_3}\cD^2_{\beta,d_1,d_2,d_3}$.
\[\cD^1 =\cD^1_{\alpha} \cup\cD^1_\beta \cup\cD^1_\delta  \text{ and }\cD^2 =\cD^2_{\alpha} \cup\cD^1_\beta\cup\cD_\gamma \cup\cD^2_\delta \cup \cD^2_{\varepsilon}\cup \cD_\omega.\]

\begin{proposition}\label{prop:perturb_to_regular}
Fix an $\cF^1$-path of functions $F_\tau \colon \O \to \R$.
  For any path $G_\tau$ of immersions, there is a $C^\infty$-close regular path of immersions $G_\tau'$ such that $G'_\tau$ misses $\cD^2$, and hits
  $\cD^1 $ only for finitely many values of $\tau$.
\end{proposition}

\begin{proof}[Sketch of proof]
The dimension counting argument for the spaces $\cD^i_{\bullet}$, for $\bullet \in \{\alpha, \beta, \gamma, \delta\}$, is analogous to the dimension counting argument for the $\cA$ spaces, given in Lemma~\ref{lem:codim_cA} and Corollary~\ref{cor:paths}; as usual this relies upon the Multijet Transverality Theorem~\ref{thm:multijet}.
It follows that the path $G_\tau$ can be perturbed to avoid $\cD^2_{\alpha} \cup \cD^2_{\gamma} \cup \cD^2_{\delta}$ and to intersect
$\cD^1$ at finitely many values. Lemma~\ref{lem:A_omega} can be applied to show that the path $G_\tau$ can also be perturbed
to avoid $D_\omega$.
It remains to deal with $\cD^2_{\varepsilon}$. Suppose $\tau_1,\dots,\tau_r$ are precisely the values for which $F_{\tau_i}$ is not Morse.  The condition $\cD^2_{\varepsilon}$ means that $G_{\tau_i}$ is not a generic immersion for one of these values. Since being a generic immersion is open and dense, we can perturb $G_{\tau}$ to arrange that $G_{\tau_i}$ is a generic immersion for $i=1,\dots,r$. After this perturbation, $G_\tau$ avoids~$\cD^2_{\varepsilon}$.
\end{proof}

\begin{definition}\label{def:regular_dp}
  A double path is called \emph{regular} if:
  \begin{itemize}
    \item $F_0$ is immersed Morse with respect to $G_0(N)$, and $F_1$ is immersed Morse with respect to $G_1(N)$.
    \item $F_\tau$ is an $\cF^1$-path of Morse functions, when regarded as a path of functions on $\Omega$;
    \item $G_\tau$ is a regular path of immersions;
    \item $G_\tau$ intersects $\cD^1 $ at finitely many points and misses $\cD^2 $.
  \end{itemize}
\end{definition}

Proposition~\ref{prop:perturb_to_regular} can be rephrased in the following manner.

\begin{corollary}\label{cor:perturb_to_regular}
  If $(F_\tau,G_\tau)$ is a double path, then there exists a regular double path $(F'_\tau,G'_\tau)$ arbitrarily close to it. Moreover, if
  $G_0$ is a generic immersion and $F_0$ is immersed Morse with respect to $G_0(N)$, then we can take $F'_0=F_0$, and $G'_0=G_0$. %\ypar{Added this.}
\end{corollary}

\begin{proof}
  First perturb $F_\tau$ to be an $\cF^1$-path (Definition~\ref{def:regular_path}), and then use Proposition~\ref{prop:perturb_to_regular} to perturb $G_\tau$.
\end{proof}
\begin{example}
  Suppose $(F_\tau,G_\tau)$ is such that $F_\tau$ is a $\cF^1$-path and $G_\tau=G_0$ is an regular immersion. Then $(F_\tau,G_\tau)$ is a regular double path.
\end{example}

\begin{example}
  Assume $(F_\tau,G_\tau)$ is such that $F_\tau=F_0$ and $G_\tau$ is a regular $F_0$-path in the sense of Definition~\ref{def:another_F_word}. Then $(F_\tau,G_\tau)$ is a regular double path.
\end{example}

We need the following statement.

\begin{lemma}\label{lemma:already-F1-path}
  Suppose $(F_\tau,G_\tau)$ is a regular double path. Then $F_\tau\circ G_\tau$ is an $\cF^1$-path of functions on $N$.
\end{lemma}

\begin{proof}
  Take $u\in N$ and choose $U\subseteq N$, a neighbourhood of $u$, such that $G_\tau|_U$ is an embedding.
  With this choice, it is clear that $u$ is a critical point of $F_\tau\circ G_\tau$ if and only if $G_\tau(u)$
  is a critical point of $F_\tau$ restricted to $U$. As $G_\tau$ omits $\cD_\gamma$, if $F_\tau|_{G_\tau(U)}$ has a critical point on $G_\tau(U)$, $G_\tau(u)$ must belong to the first stratum. This shows that the critical points of $F_\tau\circ G_\tau$ correspond to critical points of $F_\tau$ on the first stratum.

  The correspondence goes further.
  Morse critical points of $F_\tau\circ G_\tau$ correspond to Morse critical points of $F_\tau$ on the first stratum, and births/deaths of pairs of critical points of $F_\tau\circ G_\tau$ correspond to births/deaths of pairs of critical points of $F_\tau$ on the first stratum. That is, the double path missing $\cD^2_{\alpha,1}$ implies that the non-Morse critical points of $F_\tau\circ G_\tau$ are precisely births and deaths. A refinement of the argument, showing that $F_\tau\circ G_\tau$ is actually transverse to $\cF^1$ if $G_\tau$ is transverse to $\cD^1_{\alpha,1}$, is left to the reader.

  Transversality of $G_\tau$ to $\cD_{\beta,1,1}$ means that there are finitely many rearrangements along the path $F_\tau\circ G_\tau$.
In other words, $F_\tau\circ G_\tau$ is a path of functions with finitely many births, deaths and rearrangements.
  For all but finitely many values of $\tau$, $F_\tau\circ G_\tau$ is excellent Morse, for the remaining values
  it belongs to $\cF^1$.
\end{proof}

\subsection{Neat paths for manifolds with boundary}\label{sub:neat}

So far, the discussion of path of maps from $\Omega$ to $\R$ and from $N$ to $\Omega$ was done for $N$ and $\Omega$ closed. In our applications,
we allow for $N$ and $\Omega$ to be manifolds with boundary, provided the maps behave nicely near the boundary.
In this short subsection, we make this intuition precise. The bottom line is that all the results that are proved about paths of functions on closed manifolds carry over
to the case of compact manifolds with boundary, as long as suitable neatness conditions are preserved.

We recall the definition of neat immersions and extend the definition to \emph{very neat immersions}.

\begin{definition}\label{defn:neat-and-very-neat}
An immersion $G\colon N\looparrowright\O$ is \emph{neat} if
$G^{-1} (\bd \O) = \bd N$, and there exist collar neighbourhoods of $\bd N$ and $\bd \O$ such that in these coordinates, for $x \in \bd N$ and $t \in [0,1]$ we have that $G(x,t) = (G(x,0),t)$.
%We say moreover that $G\colon N\looparrowright\O$ is \emph{very neat} if $G$ is neat, and $G|_{\bd N}\colon\bd N\looparrowright\bd\O$ is a generic immersion.
%MP: This is part of the definition of generic earlier, so it's a bit confusing to have it here too.
\end{definition}

\begin{definition}
  A path $G_\tau$, $\tau\in[0,1]$ of immersions is \emph{very neat} if $G_0$ is neat, the restriction $G|_{\bd N}\colon\bd N\looparrowright\bd\O$ is a generic immersion, and if there is a neighbourhood $U$ of $\bd N$ in $N$, and a neighbourhood
  $W$ of $\bd\O$ in $\O$, such that $U\cong \bd N\times[0,1)$, $W\cong\bd\O\times[0,1)$, and for each $\tau \in [0,1]$ and $(x,t) \in N \times [0,1]$, we have that $G_\tau(x,t)=(G_0(x,0),t)$. That is, we require the path to be independent of $\tau$ near the boundary of $N$.
\end{definition}

Any neat path of immersions can be perturbed to a regular, neat path of immersions.  In fact,
being neat implies that $G_\tau$ is a regular immersion on $U$ for each $\tau$. Next, regularity
is local, so we can choose a smaller product neighbourhood $U'$ of $\bd N$
and a perturbation of $G_\tau$ to a regular path of immersion, such that
the perturbation is supported away from $U'$. Then  the newly created path is still a neat path, with $U'$ replacing $U$ and $G_0(U')$ replacing $W$.
%\ypar{Changed generic to regular here to match earlier terminology.}

Suppose $\O$ is a manifold whose boundary is a union $\bd_-\O\sqcup\bd_+\O$. A function $F\colon\O\to[0,1]$ is called \emph{neat},
if $F^{-1}(0)=\bd_-\O$, $F^{-1}(1)=\bd_+\O$ and $F$ has no critical points in a neighbourhood $W$
of $\bd\O$.  In the immersed case we assume moreover that $M$ is a neatly immersed manifold (i.e.\ the image of a neat immersion)
and that $F$ has no critical points near $\bd\O$ on each stratum.

\begin{definition}\label{def:neat_path_F}
  A path of functions $F_\tau\colon\O\to\R$, $\tau\in[0,1]$, is called \emph{neat} if,
  for any $\tau$, $F_\tau^{-1}(0)=\bd_-\O$, $F_\tau^{-1}(1)=\bd_+\O$, and there is an open subset $W_\tau \subseteq\O$ containing $\bd\O$
  such $F_\tau|_W=F_0$ and $F_0$ has no critical points on $W_\tau$.
\end{definition}

\begin{proposition}\label{prop:perturb-to-neat-path}
  Any neat path of functions can be perturbed to a neat $\cF^1$-path, possibly at the expense of shrinking the open subset $W$.
\end{proposition}

This follows as above, because the property of being an $\cF^1$-path is local.
We now pass to double paths of functions.

\begin{definition}
  A double path $(F_\tau,G_\tau)$ is called \emph{neat} if $F_\tau$ is neat and $G_\tau$ is very neat.
\end{definition}

\begin{proposition}
  If $(F_\tau,G_\tau)$ is a neat double path, there exists a perturbation, fixing the endpoints at $\tau =0,1$, to a neat and regular double path.
\end{proposition}

\begin{proof}
  The proof combines the arguments in this subsection. First, by Proposition~\ref{prop:perturb-to-neat-path}
we can perturb $F_\tau$ to a neat $\cF^1$-path at the expense of shrinking a neighbourhood of $\bd\O$
on which $F_\tau$ is independent of $\tau$. Next, we perturb  $G_\tau$ to a very neat path of immersions such that $(F_\tau,G_\tau)$ is a regular double path, as in Proposition~\ref{prop:perturb_to_regular}, using the local argument once again to avoid changing $G_{\tau}$ near $\bd N$. %\ypar{added some refs here.}
\end{proof}

\section{Rearrangement for immersed Morse functions}\label{sec:rear}

Motivated by Cerf theory, we state the rearrangement theorem under conditions which guarantee disjointness of the appropriate ascending and descending membranes. The rearrangement theorem creates a suitable path of functions in which a rearrangement occurs.
Despite this formulation, the proof of the rearrangement theorem is
a relatively straightforward generalisation of an analogous result for the embedded case~\cite{BP}. This in turn was
proven in a similar way to the corresponding result in Milnor's book \cite{Mil65}. We begin by introducing some notation.

\begin{theorem}[Rearrangement]\label{thm:grim_rearrangement}
	Let $M$ be an $n$-dimensional immersed manifold of codimension $k$ in $\O$.
	Let $F\colon\O\to\R$ be a Morse function and let $p$ be a critical point of $F$. Let $\xi$ be
	a grim vector field for $F$.
 Let $[a,b]$ be an interval in $\R$ containing $F(p)$. %Let $\cK_p=\cK_{\xi,a,b}(p)$.
	
	%%%%%%%%%%%%%
	
	For every $c\in[a,b]$ and for any open subset $V$ of $F^{-1}[a,b]$ containing $\cK_{\xi,a,b}(p)$, there exists a path of Morse functions $F_\tau$, $\tau\in[0,1]$, supported in $V$,  that moves the critical point $p$ to the level set $c$. More specifically, the following hold for $F_\tau$.
	\begin{itemize}
	 \item $F_0=F$ i.e.\ the path starts at $F$.
	% \item on $F_0^{-1}\{a,b\}$ we have $F_\tau=F_0$; % (the path is constant away from $F_0^{-1}[a,b]$); %% Removed this item since it's coverd by the penultimate one.
	 \item $F_1(p)=c$.
	 \item For every $\tau \in [0,1]$ the critical points of $F_\tau$ are the same as the critical points of $F_0$.
     \item There exists a neighbourhood $U$ of $p$ such that $F_\tau(w)=F_0(w)+F_\tau(p)-F_0(p)$ for each $w\in U$.  In particular the index of $p$ is unchanged.
      \item For every $w\notin V$ and for every $w \in F_0^{-1}\{a,b\}$ we have $F_\tau(w)=F_0(w)$ for all $\tau$. That is, $F_\tau(c) \neq F_0(x)$ implies that $x \in V$.
      \item The vector field $\xi$ is a grim vector field for $F_\tau$ for every $\tau \in [0,1]$.
    \end{itemize}
\end{theorem}

The theorem implies that we can move the value of $p$ under $F_1$ to any given $c \in [a,b]$ by only changing $F$ within a pre-determined neighbourhood of $\cK_{\xi,a,b}(p)$.

\begin{proof}
	The proof is an extension of the argument of Milnor in \cite[Section 4]{Mil65}. 	
	Set \[\O' := F^{-1}[a,b] \text{ and } \cK_p := \cK_{\xi,a,b}(p).\]
We claim that there exists a function $\mu\colon\O'\to[0,1]$
	such that $\mu$ is constant on trajectories of~$\xi$, and moreover $\mu\equiv 0$ outside $V$
	and $\mu\equiv 1$ in a neighbourhood of $\cK_p$.
	To see this, first construct two grim neighbourhoods $U_1,U_2$ of $\cK_{p}$ with the property that $U_2\subseteq U_1\subseteq V$ and the boundary of $U_2$
	intersects the boundary of $U_1$ only at the level sets $F^{-1}(a)$ and $F^{-1}(b)$ and
%	\begin{equation}\label{eq:misterious_intersection} \partial(\ol{U_1}\cap F^{-1}(a))\cap\partial(\ol{U_2}\cap F^{-1}(b))=\emptyset.
%	\end{equation}
\begin{equation}\label{eq:misterious_intersection} \partial(\ol{U}_1\cap F^{-1}(a))\cap\partial(\ol{U}_2\cap F^{-1}(a))= \partial(\ol{U}_1\cap F^{-1}(b))\cap\partial(\ol{U}_2\cap F^{-1}(b)) = \emptyset;
	\end{equation}
	see Figure~\ref{fig:rearr7}.
	The existence of a grim neighbourhood of $\cK_p$ follows directly from Lemma~\ref{lem:grimneighexist}. Note that the assumption of Lemma~\ref{lem:grimneighexist} is satisfied by the properties of $\xi$.
	\begin{figure}
	  \input{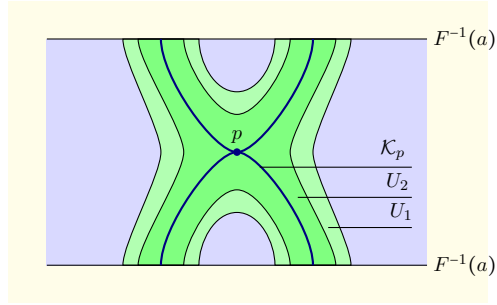}
	  \caption{Proof of the rearrangement theorem.}\label{fig:rearr7}
	\end{figure}
	The function $\mu$ is constructed as follows.
	\begin{itemize}
	  \item Away from $\ol{U}_1$ we set $\mu=0$;
	  \item On $\ol{U}_2$ we set $\mu=1$;
	  \item We first extend $\mu$  on $F^{-1}(a)$ across $F^{-1}(a)\cap(\ol{U}_1\setminus\ol{U}_2)$
	    to a smooth function.
	\end{itemize}
	Finally, take a point $w\in U_1\setminus \ol{U}_2$.
	As there are no critical points of $\xi$ in $U_1\setminus U_2$, and both
	$U_1$ and $U_2$ are $\xi$-invariant, the trajectory of $\xi$ through $w$ hits $F^{-1}(a)$ in the past at a point $z\in F^{-1}(a)\cap (\ol{U}_1\setminus\ol{U}_2)$.
	We declare $\mu(w):=\mu(z)$. This completes the definition of the function $\mu \colon \Omega' \to [0,1]$.

  Now choose an auxiliary smooth function
  $\Psi\colon\R\times[0,1]\times[0,1]\to\R$, written as $\Psi(t,\tau,s)$, such that:
  \begin{enumerate}[label=(O-\arabic*)]
    \item\label{item:O1} %For all $\tau,s$,
      We have $\Psi(t,\tau,s)=t$,  whenever at least one of the following holds:
      \begin{itemize}
	\item  $t\le a$;
	\item $t\ge b$;
	\item $\tau=0$;
	\item $s=0$.
      \end{itemize}
    \item\label{item:O2} The map $t\mapsto \Psi(t,\tau,s)$ is monotone increasing for all $\tau$ and $s$.
    \item\label{item:O3} For all $\tau$, and for sufficiently small $\delta>0$ the function $t\mapsto \Psi(t,\tau,1)$ is linear with derivative $1$ if $t\in [F(p)-\delta,F(p)+\delta]$.
    \item\label{item:O4} $\Psi(F(p),1,1)=c$.
  \end{enumerate}
We define
\[F_\tau(w)=\Psi(F(w),\tau,\mu(w)).\]
We have $\partial_\xi F_\tau(w)=\frac{\partial\Psi}{\partial t}\partial_\xi F+\frac{\partial\Psi}{\partial s}\partial_\xi\mu$. Now
by \ref{item:O2} we have $\frac{\partial\Psi}{\partial t}>0$, and since $\xi$ is grim we know that $\partial_\xi F\ge 0$ (with equality only at the critical points of $F$). By the construction of $\mu$ we have that $\partial_\xi\mu=0$. It follows that $\partial_\xi F_\tau(w)\ge 0$ with equality only at the critical points of~$F$.

Verifying that all the required properties of $F_\tau$ are satisfied is routine. For example, by \ref{item:O1} we see that $F_\tau(w)=F(w)$ whenever $F(w)\notin[a,b]$ or $\tau=0$.
The behaviour of the function $F_\tau$ near $p$ is governed by \ref{item:O3} and the rearrangement is guaranteed by \ref{item:O4}. The last condition in \ref{item:O1},
that is $\Psi(t,\tau,0)=t$ together with the fact that $\mu=0$ outside of $U_2$ implies that we do not change the function $F$
except near $\cK_p$.
\end{proof}

\begin{corollary}\label{cor:used_to_be_a_theorem}
  Let $F\colon\O\to\R$ be a Morse function, let $p,p'$ be two critical points of $F$, and let $a<b$ be two non-critical values such that $a<F(p)<F(p')<b$.
  Suppose there is a grim vector field $\xi$ such that $\cK_{\xi,a,b}(p)$ and $\cK_{\xi,a,b}(p')$ are disjoint. Then there exists
  a path of Morse functions $F_\tau$, $\tau\in[0,1]$, $F_0=F$, such that $\xi$ is a grim vector field for $F_\tau$ for every $\tau$, and such that $F_1(p)>F_1(p')$. The path is supported on an arbitrarily small open neighbourhood $V \subseteq F^{-1}([a,b])$ of $\cK_{\xi,a,b}(p)$ and $\cK_{\xi,a,b}(p')$.
\end{corollary}

\begin{proof}
  Apply the Rearrangement Theorem~\ref{thm:grim_rearrangement} to $p$ with $c=\frac12(F(p')+b)$.
\end{proof}

\begin{remark*}
  If $F_\tau$ is a path constructed in Corollary~\ref{cor:used_to_be_a_theorem}, then
  for every $\tau$ the critical points of $F_\tau$ are the same as the critical points of $F_0$. Furthermore, there exists a neighbourhood $U$ of $p$
  such that $F_\tau(w)=F_0(w)+F_\tau(p)-F_0(p)$ for each $w\in U$.  In particular the index of $p$ is the same for every $\tau \in [0,1]$.
\end{remark*}

As explained in \cite[Section 4]{BP} in the embedded case, in codimension $k>1$ we can always use the local rearrangement theorem to order critical points in such a way that higher index critical points have higher critical value. In the immersed case there is an analogous result, once we have replaced the index $h$ with $h+d$, where $d$ is the depth.
The statement and the proof are essentially the same as in
Milnor's book \cite[Section~4]{Mil65} with a necessary technical modification for the immersed case.
We begin with the following definition.

\begin{definition}\label{def:ordered_function}
  A Morse function $F\colon\O\to\R$ is called an \emph{ordered Morse function} if for any two critical points $p$ and $p'$ of indices $h$ and $h'$ and at depth $d$ and $d'$, respectively, the condition $h'+d'>h+d$ implies that $F(p')>F(p)$.
\end{definition}

\begin{remark*}
  An ordered Morse function is not required to be excellent. Two critical points are allowed to have the same critical value, as long as the
  sum $h+d$ is equal for both points.
\end{remark*}

\begin{theorem}[Global Rearrangement]\label{thm:grim_global_rearrangement}
	Suppose $k\ge 2$. Let $M^n\subseteq\Omega^{n+k}$ be an immersed manifold and let $F\colon\Omega\to\R$ be an immersed Morse function.
	Then there exists a path $F_\tau$ of Morse functions such that $F_0=F$ and $F_1$ is ordered. The path
	can be chosen in such a way that $\xi$ is a grim vector field for $F_{\tau}$ for all $\tau$.

	Furthermore, if $p$ and $p'$ are critical points of $F$ and $h+d=h'+d'$ $($where $h,h'$ are the indices and $d,d'$ are the depths$)$,
	then any prescribed order of $F_1(p)$ and $F_1(p')$ can be achieved. In particular, we can arrange, if desired, for $F_1$ to be excellent.
\end{theorem}

\begin{proof}
	By Corollary~\ref{cor:highcodim} if $h+d\ge h'+d'$ and $\xi$ is Morse--Smale, then $\Ha(p)\cap\Hd(p')=\emptyset$.
	We conclude the proof in the standard way, e.g.\ as in \cite[Section~4]{Mil65}.
\end{proof}

\section{\texorpdfstring{$\xi$}{xi}-paths}\label{sec:xi_path}

A $\xi$-path is, roughly speaking, a family of functions with a common grim vector field $\xi$. The vector field $\xi$ gives us a precise
control on the behaviour of the family of functions. We describe paths on $\Omega$, i.e.\ in the immersed setting.
The case of paths on $N$, i.e.\ on a smooth manifold (where the paths are called $\eta$-paths)
is then given as a specialisation. Both notions of a $\xi$-path and of an $\eta$-path
are used throughout the rest of Parts~\ref{part:just_paths} and~\ref{part:pathlifting}.
%Later we move on to paths on smooth
%manifolds, without any immersion in the story, so-called $\eta$-paths. More results can be proved for $\eta$-paths.
%\subsection{Paths in the immersed theory}
%The notion of a $\xi$-path will be used throughout the rest of Part~\ref{part:pathlifting}.

\subsection{$\xi$-paths for immersed Morse functions}
\begin{definition}\label{def:a_xi_path}
Let $F=F_0 \colon \O \to \R$ be an immersed Morse function and let $\xi$ be a grim vector field with respect to $F_0$.
A path $F_\tau \colon \O \to \R$ of immersed Morse functions is called a \emph{$\xi$-path}
if $\xi$ is a grim vector field for $F_\tau$ for every $\tau$.
\end{definition}

\begin{example}
  The path constructed in Rearrangement Theorem~\ref{thm:grim_global_rearrangement} is a $\xi$-path.
\end{example}

We make the following observation.

\begin{lemma}\label{lem:xi_path_fixes_crits}
  If $F_\tau$ is a $\xi$-path, then a point $p\in\O$ is a critical point of $F_\tau$ if and only if $p$
  is a critical point of $F_0$. The index of $p$ as a critical point of $F_\tau$ does not depend on $\tau$.
\end{lemma}

\begin{proof}
    By  definition $\xi$ vanishes only at critical points of $F_\tau$. Therefore, critical points of $F_\tau$
  are the same for all $\tau$.

  Suppose $p$ is a critical point of $F_\tau$ at depth $d$. By the definition, the index of $p$ is the dimension
  of the stable manifold of $\xi$ restricted to the stratum $\O[d]$. As such, it does not depend on $\tau$.
\end{proof}

\begin{definition}[A special $\xi$-path]\label{def:special_xi_path}
  A $\xi$-path $F_{\tau} \colon \O \to \R$ is \emph{special} if for every critical point $p$ of $\xi$, there exists a neighbourhood
  $U$ of $p$ and a real valued function $c_p \colon [0,1] \to \R$ such that $F_\tau|_U=F_0|_U+c_p(\tau)$ for all $\tau \in [0,1]$.
\end{definition}

\begin{lemma}\label{lem:xi_connects}
  Suppose $F_0$ and $F_1$ are two immersed Morse functions with a common grim vector field $\xi$. Assume that
  $F_0-F_1$ is locally constant in some neighbourhood of the critical points of $\xi$. Then there exists a special $\xi$-path
  $F_\tau$ connecting $F_0$ and $F_1$.
\end{lemma}

\begin{proof}
  Define $F_\tau=\tau F_1+(1-\tau)F_0$. At $\tau = 0,1$ we get $F_0$, $F_1$ respectively.  As $F_0-F_1$ is constant in some neighbourhood of each of the critical points, for each $\tau\in[0,1]$, in these neighbourhoods we can write $F_\tau = F_0 + \tau K_p$ for some $K_p \in \R$ that depends on the critical point $p$.
  % $\alpha F_0 + \beta$ for some $\alpha, \beta \in \R$.
  Therefore $F_\tau$ satisfies the conditions \ref{item:IM1} and~\ref{item:IM2} of Definition~\ref{def:immersed_morse_function} for each $\tau \in [0,1]$. Thus $F_\tau$ is immersed Morse.

  Obviously, $\partial_\xi F_\tau>0$ except at critical points of $F_\tau$. Hence $\xi$ is grim for $F_\tau$.
\end{proof}

A special case of a $\xi$-path is a path of rearrangement as constructed in the proof of the Rearrangement Theorem~\ref{thm:grim_rearrangement}.

\begin{definition}[Elementary $\xi$-path of rearrangement]\label{def:elementary}
  A special $\xi$-path $F_\tau \colon \O \to \R$ is called an \emph{elementary $\xi$-path of rearrangement}
  if there is a unique critical point $p$ such that $c_p(\tau)$ is not identically zero, and
  there exists a unique critical point $q$ such that the function
      $\tau\mapsto F_\tau(p)-F_\tau(q)$ crosses zero as $\tau$ goes from $0$ to $1$. Moreover,
      we assume that for this pair $F_\tau(p)-F_\tau(q)$ has precisely one zero for $\tau\in[0,1]$ and this is a simple zero, i.e.\ at the zero, the derivative with respect to $\tau$ does not vanish.
\end{definition}

The following result is an analogue of Cerf's uniqueness of rearrangement.

\begin{proposition}\label{prop:unique_if_xi}
  Suppose $F_\tau$ and $\wt{F}_\tau$ are two elementary $\xi$-paths of rearrangement
  that move the critical value of $p_-$ above the critical value of another critical point $p_+$
  and $F_0=\wt{F}_0$. Assume that $F_0$, $\wt{F}_0$, $F_1$, and $\wt{F}_1$ are excellent immersed Morse functions.
  Then the paths $F_\tau$ and $\wt{F}_\tau$ are left-homotopic $($Definition~\ref{defn:homotopy-of-paths}$)$.
  %Then there exists a path of immersed Morse
  %functions $H_\sigma$ such that $H_0=F_1$, $H_1=\wt{F}_1$, and there are no rearrangements on $H_\sigma$.
\end{proposition}
\begin{proof}
  In a neighbourhood of $p_+$, the functions $F_\tau$ and $\wt{F}_\tau$ are independent of $\tau$.
  In particular, $F_\tau(p_+)=\wt{F}_\tau(p_+)=F_0(p_+)$.

  Reparametrise $\wt{F}_\tau$ and $F_\tau$ if needed (which can be achieved by a left-homotopy) so that $F_\tau(p_-),\wt{F}_\tau(p_-)<F_0(p_+)$ for $\tau<1/2$ and
  $F_\tau(p_-),\wt{F}_\tau(p_-)>F_0(p_+)$ for $\tau>1/2$. For $\sigma\in[0,1]$, define
  \[H_{\sigma,\tau}(z)=(1-\sigma)F_\tau(z)+\sigma\wt{F}_\tau(z).\]

  We first show that $H_{\sigma,\tau}$ is Morse. We clearly have $\partial_\xi H_{\sigma,\tau}\ge 0$ with
  equality only at critical points of $F_0$. That is to say, the only critical points of $H_{\sigma,\tau}$ are
  those of $F_0$. We need to show that $H_{\sigma,\tau}$ satisfies both items of Definition~\ref{def:immersed_morse_function}
  at these points.
  To this end, recall that
  $F_\tau$ and $\wt{F}_\tau$ are special $\xi$-paths. Hence, for each critical point $p'$ there exists a neighbourhood $U$
  such that $F_\tau$ and $\wt{F}_\tau$ restricted to $U$ differ from $F_0$ by a constant (depending on $\tau$ and on $U$).
  Hence, $H_{\sigma,\tau}$ on $U$ differs from $F_0$ by a constant. The Morse condition on $H_{\sigma,\tau}$ follows
  from the Morse condition on $F_0$. Hence, $H_{\sigma,\tau}$ is Morse.

  Clearly, for all $\sigma \in [0,1]$,  $H_{\sigma,\tau}(p_-)<H_{\sigma,\tau}(p_+)$ for $\tau<1/2$ and $H_{\sigma,\tau}>H_{\sigma,\tau}(p_+)$ for $\tau>1/2$. Therefore,
  a rearrangement occurs at $\tau=1/2$ for each path $\tau\mapsto H_{\sigma,\tau}$ for a fixed value of $\sigma$. Moreover,
  the path $F_\tau$ being an $\cF^1$-path means that the stratum $\cF^1_\beta$ is intersected transversely, which amounts to saying that
  \[\frac{d}{d\tau}(F_\tau(p_-)-F_\tau(p_+))|_{\tau=1/2}>0.\]
  An analogous condition is satisfied by $\wt{F}_\tau$, hence it is also satisfied by $H_{\sigma,\tau}$ for all $\sigma$. That is,
  $\tau \mapsto H_{\sigma,\tau}$ intersects the $\cF^1$-stratum of $C^\infty(\O,\R)$ transversely, as required by Definition~\ref{defn:homotopy-of-paths}.
\end{proof}

Now we discuss the case of two paths of rearrangements relative to different vector fields.

\begin{proposition}\label{prop:uniqueness_without_Cerf}
  Let $F\colon\O\to\R$ be an immersed Morse function, let $p$ be a critical point of ~$F$ with $F(p)=c$.
  Assume that $\xi$ and $\xih$ are grim vector fields for $F$.
  Let $F_\tau$ $($respectively $\widehat{F}_\tau)$, be an elementary $\xi$-path $($respectively, an elementary $\xih$-path$)$
  of rearrangement that moves the point $p$ to the level $c'=F_1(p)$, with $F_0=\widehat{F}_0$.

  Assume also that there exists an interval $[a,b]$ containing $c$ and $c'$, such that
  $\xi$ and $\xih$ are both tangent to $\cK_{\xi,a,b}(p)$ and agree in a neighbourhood of $\cK_{\xi,a,b}(p)$.
Then $F_\tau$ and $\widehat{F}_\tau$ are left-homotopic.
\end{proposition}

\begin{proof}
  Let $U$ be a neighbourhood of $\cK_{\xi,a,b}(p)$ such that $\xi=\wh{\xi}$ in $U$.
  Define $\xi_\sigma=(1-\sigma)\xi+\sigma\xih$. This need not be a grim vector field in general (a convex combination
  of grim vector fields is not necessarily grim), but it is certainly grim in $U$ because there $\xi_{\sigma} = \xi = \wh{\xi}$.

  For each $\sigma$, let $\tau\mapsto F_{\sigma,\tau}$ be
  the path as in the proof of Theorem~\ref{thm:grim_rearrangement}. To construct these paths, we use the same functions $\mu$ and $\Psi$ for
  all $\sigma$:
  note that $\xi_\sigma$ does not depend on $\sigma$ on $U$, so we can indeed use the same $\mu$.
  Note that for this construction we do not need
  $\xi_\sigma$ to be grim everywhere on $\O$, only on $U$, because the path constructed in Theorem~\ref{thm:grim_rearrangement} is supported on~$U$. By construction, all the paths $\tau\mapsto F_{\sigma,\tau}$ are $\cF^1$-paths, so $F_{0,\tau}$ and $F_{1,\tau}$ are left-homotopic.

  The paths $F_\tau$ and $F_{0,\tau}$ need not be equal, but they are left-homotopic by virtue of Proposition~\ref{prop:unique_if_xi}.
  Likewise, the path $\widehat{F}_\tau$ and $F_{1,\tau}$ are left-homotopic for the same reason. Hence, $F_\tau$ and $\widehat{F}_\tau$
  are left-homotopic as well.
\end{proof}

\subsection{Paths on a smooth manifold}
We pass to studying paths on the manifold $N$. We have two possibilities: either we assume that $N$ is closed,
or that $N$ is compact with boundary and the functions on $N$ behave nicely near the boundary. The precise formulation of
nice behaviour was given in Subsection~\ref{sub:neat}. For simplicity, in this subsection we assume that $N$ is closed, leaving
necessary adaptations to the reader.

\begin{definition}\label{def:eta-path}
An \emph{$\eta$-path} is a path of Morse functions $f_\tau \colon N \to \R$ with a common gradient-like vector field $\eta$.
An \emph{$\eta$-path} of rearrangement is an $\eta$-path $f_\tau$ such that precisely two critical points are rearranged along $f_\tau$. %\ypar{MB: Added this definition.}
\end{definition}

Recall that vector fields on $N$ are denoted using letter $\eta$, therefore what was said for $\xi$-paths for $\Omega$, will be now put into the context of $\eta$-paths on $N$.

\begin{lemma}\label{lem:linear_eta_path}
  Suppose $f_0,f_1\colon N \to \R$ are two Morse functions. Assume $\eta$ is a gradient-like vector field for both $f_0$ and $f_1$. Then
  the path $f_\tau=(1-\tau)f_0+\tau f_1$ is an $\eta$-path of Morse functions.
\end{lemma}

\begin{proof}
  The same argument that was used in the proof of Lemma~\ref{lem:xi_path_fixes_crits} shows that $f_0$ and $f_1$ have the same critical points.
  As $\partial_\eta f_0,\partial_\eta f_1>0$ away from the set of critical points, we infer that $\partial_\eta f_\tau>0$ away from the set
  of critical points of $f_0$. In particular, critical points of $f_\tau$ are the same as the critical points of $f_0$.

  It remains to show that $f_\tau$ is Morse.
  Take one critical point $q\in N$. Let $W^s$ and $W^u$ be the tangent spaces at $q$ to the stable and unstable manifolds of $\eta$ respectively. Since $f_0$ and $f_1$ are Morse, $D^2f_0(q)$ and $D^2f_1(q)$ are both negative (respectively: positive)
  definite forms on $W^s$ (respectively: on $W^u$). This shows that $D^2f_\tau(q)=(1-\tau)D^2f_0(q)+\tau D^2f_1(q)$ is negative definite on $W^s$ (respectively:
  positive definite on $W^u$). This means that $D^2f_\tau(q)$ is of full rank. That is, $q$ is a Morse critical point of $f_\tau$. This implies that $f_\tau$
  has only Morse critical points and that $\partial_\eta f_\tau\ge 0$ with equality only at these points.
\end{proof}

As a consequence we prove a variant of Proposition~\ref{prop:uniqueness_without_Cerf} for paths starting from different  points.
%\npar{MB. There is a place for simplification. In fact, we use only a part of Lemma~\ref{lem:uniq1}, namely for left-homotopy. This is the same proof as Proposition~\ref{prop:uniqueness_without_Cerf}, which is never used. If we want to shorten the paper, we could delete Prop.~\ref{prop:uniqueness_without_Cerf} and simplify (slightly) Lemma~\ref{lem:uniq1}. I am slightly against, we gain a little, but Prop.~\ref{prop:uniqueness_without_Cerf} is one of the few uniqueness results for immersed Morse theory.}{MP: I don't mind too much, but if there is something we are not using I think it would be worthwhile to shorten the paper. I somehow doubt anyone is going to find and use this Proposition later.}
%\ypar{MB. I suggest that we leave things as they are now.}

\begin{lemma}\label{lem:uniq1}
  Suppose $\eta$ is a common gradient-like vector field for two functions $f \colon N \to \R$ and $\wt{f} \colon N \to \R$ belonging to the same connected component
  of $\cF^0$. Then any two $\eta$-paths of rearrangements starting from $f$ and $\wt{f}$ that move a critical point $q_-$ above the critical point $q_+$  are lax homotopic.

  Moreover, if $f=\wt{f}$, then the paths are left-homotopic. %\ypar{MB: added this, we use only the left-homotopy afterwards.}
\end{lemma}

\begin{proof}
  Choose two $\eta$-paths $f_\tau$ and $\wt{f}_\tau$ of rearrangements that intersect the $\cF^1_\beta$ stratum at $\tau=1/2$, with $f_0=f$ and $\wt{f}_0 = \wt{f}$.
  Set $h_{\sigma,\tau}=(1-\sigma)f_\tau+\sigma\wt{f}_\tau$. We have proved in Lemma~\ref{lem:linear_eta_path}
   that $\eta$ is gradient-like for $h_{\sigma,\tau}$. We need to show that the paths $\tau\mapsto h_{\sigma,\tau}$ rearrange precisely the critical points $q_-$ and $q_+$ for all $\sigma$.

  As $f$ and $\wt{f}$ belong to the same
  connected component of $\cF^0$, for any two critical points $q_1,q_2$ of $f$ (which are also critical points of $\wt{f}$), we have
  $f(q_1)<f(q_2)$ if and only if $\wt{f}(q_1)<\wt{f}(q_2)$.
  It is then easy to check that if $q_1,q_2$ are two critical points, then:
  \begin{itemize}
    \item if $q_1=q_-$ and $q_2=q_+$, then
      $h_{\sigma,\tau}(q_1)<h_{\sigma,\tau}(q_2)$ for $\tau<1/2$.
      There is equality for $\tau=1/2$ and $h_{\sigma,\tau}(q_1)>h_{\sigma,\tau}(q_2)$ for $\tau>1/2$.
    \item if $\{q_1,q_2\}\neq\{q_-,q_+\}$, then $h_{\sigma,\tau}(q_1)<h_{\sigma,\tau}(q_2)$ if and only if $f(q_1)<f(q_2)$.
  \end{itemize}
  That is to say, for any $\sigma\in[0,1]$, the path $\sigma\mapsto h_{\sigma,\tau}$ is an $\eta$-path of rearrangement intersecting $\cF^1_\beta$ at $\tau=1/2$. As in the proof of Proposition~\ref{prop:unique_if_xi}, we check that if the intersections of $f_\tau$ and $\wt{f}_\tau$ with $\cF^1_\beta$ are transverse, then so is the intersection of $\sigma\mapsto h_{\sigma,\tau}$. It follows that $h_{\sigma,\tau}$ is a lax
  homotopy between the paths $f_\tau$ and $\wt{f}_\tau$.

  Note that the lax homotopy is over the path $h_{\sigma,0}=(1-\sigma)f_0+\sigma\wt{f}_0$. That is, if $f_0=\wt{f}_0$, then, $h_{\sigma,0}=f_0$, proving that $h_{\sigma,\tau}$ is in fact a left-homotopy.
  %\npar{Changed some $\tau$s to $0$s in the last sentence. }{And I've changed $h_{\sigma,0}=0$ to $h_{\sigma,0}=f_0$.}
\end{proof}

Another way of creating lax homotopic paths is the following.

\begin{lemma}\label{lem:cut_and_paste}
  Suppose $f \colon N \to \R$ is a Morse function and $\eta$ is a gradient-like vector field for~$f$. Assume
  that $q_-$ and $q_+$ are two consecutive $($with respect to the value of $f)$ critical points of~$f$. Suppose also $U$ is an open subset containing a neighbourhood of the unstable manifold of $q_-$ up to a level set above $q_+$, as well as a neighbourhood of $q_+$. % (the set $U$ might be disconnected).

  Assume $h_{\sigma,0}$, $\sigma\in[0,1]$, $h_{00}=f$, is a path of excellent Morse functions whose support is disjoint from $U$. Let $\wt{\eta}$ be a gradient-like vector field for $h_{10}$  agreeing with $\eta$ on $U$.

  If $f_\tau$, $\tau\in[0,1]$ is an $\eta$-path of rearrangement starting at $f$ and lifting $q_-$ above $q_+$ and $\wt{f}_{\tau}$ is an $\wt{\eta}$-path of rearrangement
  starting at $h_{10}$ and lifting $q_-$ above $q_+$, then the paths $\tau\mapsto f_{\tau}$ and $\tau\mapsto \wt{f}_{\tau}$ are lax homotopic
  over $h_{\sigma,0}$.
\end{lemma}
\begin{figure}
  \input{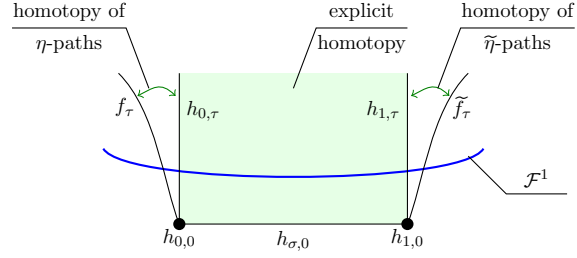}
  \caption{Proof of Lemma~\ref{lem:cut_and_paste}, schematic of paths.}\label{fig:concat_re}
\end{figure}
\begin{proof}
  The aim of the proof is to create a homotopy of paths that is guided by $h_{\sigma,0}$ away from $U$ and and an $\eta$-path inside $U$.
  Such a homotopy
  will be constructed explicitly. In the second step, we will use this homotopy to show that $f_{\tau}$ and $\wt{f}_{\tau}$ are lax homotopic;
  compare Figure~\ref{fig:concat_re}.

  To begin with, choose
  $h_{0,\tau}$ to be an $\eta$-path of rearrangement starting at $h_{00}$, lifting $q_-$ above $q_+$ and supported on $U$.
  Assume that rearrangement occurs at $\tau=1/2$ and $h_{0,\tau}$ intersects $\cF^1_\beta$ transversely.
  For $\tau\in[0,1]$ and $\sigma\in[0,1]$ define
  \begin{equation} \label{eq:wtfst}
    h_{\sigma,\tau}(u) = \begin{cases} h_{0,\tau}(u) & u\in U \\ h_{\sigma,0}(u) & u\notin U.\end{cases}
  \end{equation}
  The definition of $h_{\sigma,\tau}$ coincides with previously defined $h_{0,\tau}$ if $\sigma=0$, and agrees with given $h_{\sigma,0}$
  if $\tau=0$. Indeed, if $\tau=0$, the right-hand side of \eqref{eq:wtfst} is $h_{\sigma,0}$ away from $U$, and $h_{0,0}$ in $U$,
  but in $U$, $h_{0,\sigma}=h_{0,0}$. Hence, the right-hand side is equal to $h_{\sigma,0}$. Furthermore,
  note that the right-hand side of \eqref{eq:wtfst} if $\sigma=0$ is equal to $h_{0,\tau}$ in $U$, and $h_{0,0}$ away from $U$, but
  $h_{0,\tau}$ is supported on $U$. Hence, \eqref{eq:wtfst} indeed extends the definition of $h_{\sigma,\tau}$ from the set $\{\tau\sigma=0\}$.

  The paths $\tau\mapsto h_{0,\tau}$ and $\sigma\mapsto h_{\sigma,0}$ have disjoint supports. This means that $h_{\sigma,\tau}$ is smooth and depends smoothly on $\sigma,\tau$. Next, since $\sigma\mapsto h_{\sigma,0}$ is a path
  of excellent Morse functions, no critical point of $f$ other than $q_\pm$ has critical value in the interval
  $[f(q_-),f(q_+)]$. That is to say, the only rearrangements that occur along $\tau\mapsto h_{\sigma,\tau}$ are at $\tau=1/2$, where the critical values of $q_-$ and $q_+$ agree.
  It follows that $\tau \mapsto h_{\sigma,\tau}$ is a path of rearrangements for each $\sigma$, and $\sigma \mapsto (\tau \mapsto h_{\sigma,\tau})$ is a lax homotopy between
  $\tau \mapsto h_{\tau,0}$ and $\tau \mapsto h_{\tau,1}$, and this is clearly a lax homotopy over $h_{\sigma,0}$.

  In the remainder of the proof, we will show that $\tau\mapsto h_{1,\tau}$ and $\wt{f}_\tau$ are left-homotopic and that $\tau\mapsto h_{0,\tau}$ and $f_\tau$ are left-homotopic. Concatenating these homotopies will give us a lax homotopy over $h_{\sigma,0}$ between $f_\tau$ and $\wt{f}_{\tau}$.

  \smallskip
  \emph{Proving homotopy between $h_{\tau,1}$ and $\wt{f}_{\tau}$.} We aim to show that $h_{1,\tau}$ and $\wt{f}_\tau$ are both $\wt{\eta}$-paths of rearrangement. For $\wt{f}_\tau$, this is one of the assumptions of the lemma. For $h_{1,\tau}$, we note that $\wt{\eta}$ is gradient-like
  for $h_{0,1}$, that is for the starting point. The path $\tau\mapsto h_{1,\tau}$ is supported on $U$, and $\eta$ is gradient-like for $h_{1,\tau}$ in $U$. But on $U$, $\eta=\wt{\eta}$. That is, $\wt{\eta}$ is gradient-like for $h_{1,\tau}$ for all $\tau$, hence it is a $\wt{\eta}$-path. Any two $\wt{\eta}$-paths of rearrangement with the same starting point and rearranging the same pair of critical points are left-homotopic by Lemma~\ref{lem:uniq1}.

\smallskip
\emph{Proving homotopy between $f_{0,\tau}$ and $f_{\tau}$.}
By construction, $\tau\mapsto h_{0,\tau}$ is an $\eta$-path of rearrangement. By the assumptions, $f_\tau$ also is an $\eta$-path.
Two $\eta$-paths of rearrangement rearranging the same pair of critical points are left-homotopic by Lemma~\ref{lem:uniq1}.
\end{proof}

\begin{remark}\label{rem:cut_and_paste}
  The proof of Lemma~\ref{lem:cut_and_paste} implies also that the paths $f_\tau$ and the concatenation of $\wt{f}_\tau$ with $h_{\sigma,0}$
  are left-homotopic. Given the two left homotopies at the end of the proof, the statement is equivalent to saying that the paths $\tau\mapsto h_{0,1}$ and the concatenation of paths
  $h_{\sigma,0}$ and $h_{1,\tau}$ are left-homotopic. But this left homotopy is easily constructed using $h_{\sigma,\tau}$.
\end{remark}

\section{The Cancellation Theorem}\label{sec:cancel}

The cancellation theorem is about the possibility of simplifying a Morse function by cancelling a pair of critical points.
The assumptions are similar to those in the ambient case~\cite[Section 5]{Mil65}: the two critical points must be of consecutive indices
and there should be precisely one trajectory of a Morse--Smale gradient-like vector field connecting the two critical points. In the immersed case, there is one more assumption, namely the two critical points and the unique trajectory between them should all be at the same depth.

In the proof of Path Lifting Theorem~\ref{thm:path_lifting}, we will need only the case of critical points at depth $1$, which was already done in \cite[Theorem 5.1]{BP}. However, we will need a more detailed statement than in \cite{BP} for the purpose of using uniqueness of death. Therefore, we give a more
detailed statement and a more detailed proof.

\begin{theorem}[Cancellation]\label{thm:grimcanc}
  Let $F\colon\O\to\R$ be an immersed Morse function with respect to $M=G(N)$, for $G \colon N\to \Omega$ a generic immersion,  with a grim vector field $\xi$ satisfying the Morse-Smale condition.
  Let $p_-$ and $p_+$ be two critical points of $F$ that are at depth $d$, the index of $p_-$ is equal to $h$, and suppose that the index of
  $p_+$ is equal to $h+1$.

  Suppose there is a single trajectory $\gamma$ of $\xi$ connecting $p_-$ and $p_+$, and that this trajectory belongs to the $d$-th stratum.
  Finally, assume that there are no critical points of $F$ in $F^{-1}[p_-,p_+]$ other than $p_-$ and $p_+$.

  Then the critical points $p_-$ and $p_+$ can be cancelled. That is, there exists a smooth family $\xi_\sigma$, $\sigma\in[0,1]$
  of grim vector fields for $F$, and a path $F_\tau$, $\tau\in[0,2]$ such that:
  \begin{enumerate}[label=(C-\arabic*)]
    \item\label{item:Canc_start} $F_0=F$, $\xi_0=\xi$;
    \item\label{item:Canc_perturb} The vector fields $\xi_\sigma$ are grim vector fields for $F_0$, they all satisfy the Morse--Smale
      conditions, and $\gamma$ is the single trajectory of $\xi_\sigma$ connecting $p_-$ to $p_+$;
    \item\label{item:Canc_coors} The vector field $\xi_1$ is grim for all functions $F_\tau$, $\tau\in[0,1]$, in particular these functions
      have the same critical points;
    \item\label{item:Canc_norear} The path $F_\tau$, $\tau\in[0,1]$ can be chosen to be supported on a predefined grim neighbourhood of $\gamma$;
    \item\label{item:Canc_form}
    In a smaller neighbourhood $U$ of $\gamma$, there exist coordinates $x_1,\dots,x_m$, $y_{11},\dots,y_{dk}$ $($with $m=n+k-dk)$ and a strictly increasing function $\Upsilon \colon \R \to \R$
      such that for $\tau\in[1,2]$, $F_\tau= \Upsilon \circ H_\tau^1$, where
      \[H_\tau^1(x_1,\dots,y_{dk}) = x_1^3-3x_1 + 6(\tau-1) x_1 - x_2^2-\cdots-x_h^2+x_{h+1}^2+\dots+x_m^2+y_{11}+\cdots+y_{d1}.\]
  \end{enumerate}
  In particular, $F_2$ has the critical points $p_-$ and $p_+$ removed.
\end{theorem}
%\npar{I've written `perhaps we should explain what the double path is, as it is not clear here.' I think this refers to us using later that this theorem is quoted later as giving rise to a double path.}{MB: tried to guess your intentions and added a few sentences.}

  In the statement of Theorem~\ref{thm:grimcanc} we specify a path $F_\tau$, $\tau\in[0,2]$. The part $\tau\in[0,1]$ is technical and serves as a preparation for the second part. The actual cancellation occurs along the part of the path with $\tau\in[1,2]$.  Recall that $n+k := \dim \O$, $n = \dim N$, $m = n+k -dk$, $h$ is the index, and $d$ is the depth. %\npar{Added a reminder here. }{Maybe it's not needed to repeat what is $n+k$ etc. It should be clear from the statement.}

  The path $F_{\tau}$, $\tau\in[0,2]$ constructed in Theorem~\ref{thm:grimcanc} (or more generally, the double path $(F_{\tau},G)$ with $G$ constant) depends on many choices. In the absolute case, Proposition~\ref{prop:cerf_milnor} shows that different choices lead to left-homotopic paths.
  Compare Remark~\ref{rem:grimcanc_choices}.

  We also note that with $G_\tau:=G$ for $\tau\in[0,2]$, $G_\tau$ is a regular immersion for all $\tau$, and $F_\tau$ crosses the $\cF^1$ stratum only at one point, when the two critical points $p_-$ and $p_+$ are cancelled. That is $(F_\tau,G_\tau)$ is a regular double path.

\begin{proof}[Proof of Theorem~\ref{thm:grimcanc}]
  The proof relies on several technical lemmas. We first give the proof of the theorem, and then prove the lemmas.
  The first lemma is an analogue of \cite[Assertion 6]{Mil65}. It specifies a local coordinate system. It also takes care of the property~\ref{item:Canc_perturb}.

  \begin{lemma}\label{lem:coor_system}
    There is a neighbourhood $U$ of $\gamma$, and a path $\xi_\sigma$ of grim vector fields for $F$, with $\xi_0=\xi$, satisfying the Morse-Smale condition, such that $\xi_\sigma=\xi$ away from $\gamma$, and there is a coordinate system $(x_1,\dots,x_r,y_{11},\dots,y_{dk})$
    on the whole of $U$, such that the branches are given by $\{y_{i1}=\dots=y_{ik}=0\}$, and
    \begin{equation}\label{eq:xi_canc}
      \xi_1=(v(x_1),-x_2,\dots,-x_h,x_{h+1},\dots,x_r,\sum y_{1i}^2,0,\dots,\sum y_{2i}^2,0,\dots,\sum y_{di}^2,0,\dots).
    \end{equation}
    Here $v$ vanishes for precisely two values: $0$ and $1$, is positive for $v\in(0,1)$, negative elsewhere and $\frac{\partial v}{\partial x_1}$
    is $1$ near $x_1=0$ and $-1$ at $x_1=1$.
  \end{lemma}

  \begin{remark*}
    The function $v$ is not specified here, as it is not specified in \cite{Mil65} either. Careful analysis of the proof of
    \cite[Assertion 6, page 55]{Mil65} reveals that we can find the coordinate system for any predefined function $v$ satisfying
    the above restrictions. The same applies here.
  \end{remark*}
  As mentioned above, we defer the proof of the lemma until the end of the proof of the theorem.
  Note that in the coordinates of the lemma, $p_-=(0,\dots,0)$ and $p_+=(1,0,\dots,0)$.
  The coordinate system constructed in Lemma~\ref{lem:coor_system} is defined in a neighbourhood of $\gamma$, that is, locally.
  In the next lemma, we extend it to a grim vector field, that is, to a semi-local coordinate system.
  \begin{lemma}\label{lem:extend}
    There is a grim neighbourhood $U'$ of $\gamma$, containing $U$, such that the coordinate system of Lemma~\ref{lem:coor_system} can be extended
    over $U'$, and such that $\xi_1$ has still the form \eqref{eq:xi_canc} in $U'$.
  \end{lemma}

  Given the lemma, keeping in mind \ref{item:Canc_form}, we introduce an auxiliary function on $U'$.
  \begin{equation}\label{eq:g_tau_def}
    H(x_1,\dots,y_{dk})=x_1^3-3x_1-x_2^2-\cdots-x_h^2+x_{h+1}^2+\dots+x_m^2+y_{11}+\cdots+y_{d1}.
  \end{equation}

  \begin{lemma}
    On $U'$ we have $\partial_{\xi_1}H\ge 0$ with equality only at $p_-$ and $p_+$.
  \end{lemma}

  \begin{proof}
    The result is immediate from \eqref{eq:xi_canc}.
  \end{proof}

  As cancelling the critical points of $H$ is straightforward, we aim to replace locally $F$ by $H$.
  The result constructs a path $F_\tau$, $\tau\in[0,1]$ with the properties \ref{item:Canc_norear}.

  \begin{lemma}\label{lem:replace}
    There is a strictly increasing function $\Upsilon\colon\R\to\R$, a neighbourhood $U_1$ of $\gamma$, $U_1\subseteq U$ and a $\xi_1$-path of functions $F_\tau$, $\tau\in[0,1]$
    such that $F_0=F$ and $F_1|_{U_1}=\Upsilon\circ H$.
  \end{lemma}
  We defer the proof of the lemma until the end of the proof of the theorem. The rest of the proof of Theorem~\ref{thm:grimcanc}
  is now straightforward.
  Choose a cut-off function $\omega$ supported in $U_1$, equal to $1$ on a neighbourhood $U_2 \subseteq U_1$ of $\gamma$. For $\tau\in[1,2]$ we set
  \[H_\tau^{\omega}(x_1,\dots,y_{dk})=x_1^3-3x_1 + 6(\tau-1)\omega x_1 - x_2^2-\cdots-x_h^2+x_{h+1}^2+\dots+x_m^2+y_{11}+\cdots+y_{d1},\]
  and then
  \[
    F_\tau(z)=\begin{cases}
      F_1(z) & z\notin U_1\\
      \Upsilon\circ H_\tau^{\omega}(z) & z\in U_1.
    \end{cases}
  \]
 A straightforward check using the chain rule and the fact that $\Upsilon$ is strictly increasing shows that $F_2$ has no critical points in $U_2$ (which gives $U$ in the statement of Theorem~\ref{thm:grimcanc}, where $\omega \equiv 1$. In $U_2$,  we also have that $F_\tau= \Upsilon\circ H_\tau^{1}$.     The proof of Theorem~\ref{thm:grimcanc} is complete, modulo Lemmas~\ref{lem:coor_system},  \ref{lem:extend} and~\ref{lem:replace}, whose proofs follow.
\end{proof}

\begin{proof}[Proof of Lemma~\ref{lem:coor_system}]
  Choose local neighbourhoods $U_-,U_+$ of critical points with local Morse coordinates as in Definition~\ref{def:grim}.
  In particular, $\xi$ has the form \eqref{eq:localgrim}. Assume that $U_-$ and $U_+$ are balls in these coordinates, and that $\gamma$ intersects
  $\partial U_-$ only at  one point, which we call~$w_-$. Moreover, assume that $w_+$ is the single point of intersection of $\gamma$ with $\partial U_{+}$.
  Let $V_0\subseteq \partial U_-$ be a neighbourhood of $w_-$ in $\partial U_-$ such that any trajectory starting from $V_0$ eventually hits $U_+$. Define $U_{0}$ to
  be the set of points on these trajectories, that is the set of $x$ such that the trajectory through $x$ hits $V_0$ in the past and $U_{+}$ in the future.

  \begin{figure}
    \input{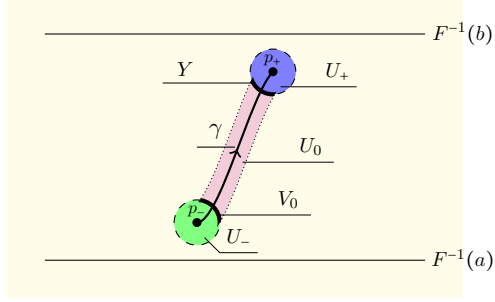}
    \caption{Notation of proof of Cancellation Theorem~\ref{thm:grimcanc}. }\label{fig:grimcanc}
  \end{figure}

  Extend the coordinate system from $U_-$ to $U_{0}$ in the following way. For a point $x\in U_{0}$, a trajectory of $\xi$ through $x$ hits a point $w_x\in V_0$.  Near $w_x$ we already have coordinates from $U_-$. Then we assign
  the coordinates to $x$ using the flow $\xi$ in such a way that $\xi$ has the form \eqref{eq:localgrim} on the whole of $U_-\cup U_{0}$.
  For simplicity, we will assume that $\ol{U}_0\cap U_-$, respectively $\ol{U}_0\cap U_+$ lies on one level set of $F_\tau$.
  Denote this level set $c_-$, respectively $c_+$.

  On the intersection $\ol{U}_{0} \cap \ol{U}_{+}$ we have two coordinate systems. One comes from $U_{0}$ and the other comes from $U_{+}$. If the two
  coordinate systems match (up to possibly shifting the value of $x_1$), then the lemma is proved with $\xi_\sigma\equiv\xi$.
  If not, we need to find a suitable
  isotopy between the two coordinate systems and alter the vector field $\xi$ using the Isotopy Insertion Lemma~\ref{lem:isoinject}
  to induce this isotopy by the flow. The proof builds on the proof of \cite[Assertion 6]{Mil65}, and uses it in the last part.

  Write $Y=\ol{U}_0\cap \ol{U}_+$. The two coordinate systems specify maps $h_0\colon Y\to Z_0$, $h_1\colon Y\to Z_0'$,
  where $Z_0,Z_0'\subseteq \R^{n+k-1}=\R^{h-1}\oplus\R^{m-h}\oplus\R^k\oplus\dots\oplus\R^k$, with $d$ copies of $\R^k$. %\ypar{Added how many copies of $\R^k$. }
   Both $h_0$ and $h_1$ take the unique intersection point of $\gamma$ with $Y$ to $(0,\dots,0)\in \R^{n+k-1}$.

 We aim to show that there is a compactly supported homotopy of $h_1\circ h_0^{-1}$ to a map that is the identity on a neighbourhood of $(0,\dots,0)\in Z_0$.
  To be more precise, denote the first summand of the space $\R^{n+k-1}$ by $L_0^s$ (`$s$' for `stable'), the second by $L_0^u$ (`$u$' for `unstable')
  and all the others by $L_1,\dots,L_d$.
  Set $L_0=L_0^s\oplus L_0^u$.
  Note that the choice of coordinate systems on $U_-$ and $U_+$ involves choosing a numbering of branches. We choose these two numberings
  to be consistent, that is, the flow of $\xi$ maps the $i$-th branch in $U_-$ to the $i$-th branch in $U_+$.

  We can summarise the properties of $h_1\circ h_0^{-1}$ in the following lemma.

  \begin{lemma}\label{lem:isotopy_we_start_with}
    For any subset $\mathfrak{p}$ of $\{1,\dots,d\}$, let $L_\mathfrak{p}=\bigoplus_{i\in\mathfrak{p}}L_i$. Then  for any $\mathfrak{p}$
    the map $h_1\circ h_0^{-1}$:
    \begin{enumerate}[label=(H-\arabic*)]
      \item maps $L_0\oplus L_\mathfrak{p}\to L_0\oplus L_\mathfrak{p}$; \label{item:H-branches}
      \item Moreover, the image of $L_0^u$ in $L_0$ intersects $L_0^s$ transversely and only at $0$.\label{item:H-transverse}
    \end{enumerate}
  \end{lemma}

  \begin{proof}
    Item~\ref{item:H-branches} follows from the fact that $L_0\oplus L_{\mathfrak{p}}$ is the image under $h_0$ of the intersection of branches $\bigcap_{i\notin \mathfrak{p}} Y_i$.
    Item~\ref{item:H-transverse} follows from the fact that the flow
    is Morse--Smale.
  \end{proof}

  The proof of Lemma~\ref{lem:coor_system} consists of subsequent
  simplifications of the composition $h_1\circ h_0^{-1}$, done on
  smaller and smaller neighbourhoods of $(0,\dots,0)\in Z_0$.
  The first result in this direction is reduction to the linear case.

  \begin{lemma}[Reduction to the Linear Case]\label{lem:reduction}
    There exists an open subset $Z_1\ni\{0\}$, $Z_1\subseteq Z_0$,
    and an isotopy $h^1_\tau$ with  $h^1_0=h_1\circ h_0^{-1}$ such that $h^1_\tau$ is independent of $\tau$ in a neighbourhood of the boundary of $\ol{Z}_0$ and $h^1_1=D(h_1\circ h_0^{-1})(0)$ identically on $Z_1$. Moreover for each $\tau$, $h^1_\tau$ satisfies items~\ref{item:H-branches} and~\ref{item:H-transverse}.
  \end{lemma}

  \begin{proof}[Proof of Lemma~\ref{lem:reduction}]
    To simplify the proof, endow $\R^{n+k-1}$ with an inner product. Suppose $V_\ell$ is a ball with radius $1/\ell$ and centre $0$. For sufficiently large $\ell$, the closure of the
    ball $V_\ell$ is contained in $Z_0$. The isotopy will be independent
    of $\tau$ away from $V_\ell$.

    Let $\phi_\ell$ be
    a cut-off function supported on $V_{\ell}$, equal to $1$ on $V_{2\ell}$ and whose derivative is bounded by $\|D\phi_\ell\| \leq 4\ell$. With
    $A:=D(h_1\circ h_0^{-1})(0)$, define
    \[h^1_{\ell,\tau}=A\tau\phi_\ell+(1-\tau\phi_\ell) h_1\circ h_0^{-1}.\]
    Note that for all choices of $\ell$, item~\ref{item:H-branches} is satisfied because it is satisfied by $A$ and $h_1\circ h_0^{-1}$.
    Condition~\ref{item:H-transverse} is open in the $C^1$-norm.
    Thus, we need to show that $h^1_{\ell,\tau}$
    is close to $h_1\circ h_0^{-1}$ in the $C^1$-norm. This is equivalent to saying that $\kappa_\ell:=\phi_\ell(A-h_1\circ h_0^{-1})$ converges
    to $0$ in the $C^1$-norm. By the Taylor formula, there is a constant $c_2<\infty$ such that $\|Ax-h_1\circ h_0^{-1}(x)\|\le c_2 \|x\|^2$.
    Thus, if $\phi_\ell$ is supported on the ball of radius $1/\ell$, we know that the $C^0$-norm of $\kappa_\ell$ is bounded by $c_2/\ell^2$.
    Looking at the derivatives, we note
    \[\| D\kappa_\ell(x)\| =\| D\phi_\ell\cdot(A-h_1\circ h_0^{-1})+\phi_\ell D(A-h_1\circ h_0^{-1})\| \leq \| D\phi_\ell\cdot(A-h_1\circ h_0^{-1})\|+ \|\phi_\ell D(A-h_1\circ h_0^{-1})\|.\]
    The first term is bounded by $4\ell c_2/\ell^2=4c_2/\ell$, because we control the derivative of $\phi_\ell$ via $\|D\phi_\ell\| \leq 4\ell$, and $\|A-h_1\circ h_0^{-1}\| \leq c_2/\ell^2$. The second term is bounded by
    $2c_2/\ell$, because $|\phi_\ell| \leq 1$, and by differentiating the Taylor approximation $\|D(A-h_1\circ h_0^{-1})\| \leq 2c_2\|x\|$, which on the ball of radius $1/\ell$ is bounded by $2c_2/\ell$. Thus altogether $\|D\kappa_\ell(x)\|\le 6c_2/\ell$.

    Thus, for sufficiently large $\ell$, $h^1_{\ell,\tau}$ satisfies~\ref{item:H-transverse}. We then declare $Z_1:=V_{2\ell}$
    and $h^1_\tau:=h^1_{\ell,\tau}$.
  \end{proof}

The next step reduces the linear map to a linear map preserving
  the decomposition of $\R^{n+k-1}$ into $L_0\oplus L_1\oplus\dots\oplus L_d$.

  \begin{lemma}[Reduction to block structure]\label{lem:reduction_block}
    There is an open neighbourhood $Z_2\subseteq Z_1$ of $0$, a path $h^2_\tau$ of isotopies,
    supported on a compact set contained in $Z_1$ and containing $Z_2$
    such that $h^2_0=h^1_1$, and $h^2_\tau$ satisfies items~\ref{item:H-branches} and~\ref{item:H-transverse}. Moreover,
    $h^2_1$ preserves the decomposition $\R^{n+k-1}=L_0\oplus L_1\oplus\dots\oplus L_d$.
  \end{lemma}

  \begin{proof}
    With the decomposition $\R^{n+k-1}=L_0\oplus\dots\oplus L_d$,
    we can write $A_{ij}$, $i,j=0,\dots,d$ for the submatrix
    giving the map from $L_i$ to $L_j$. With this notation, we have
    \begin{itemize}
      \item $A_{ij}$ is zero except if $i=j$ or $j=0$;
      \item $A_{00}$ maps $L_0^u$ to a subspace transverse to $L_0^s$.
    \end{itemize}
%    \npar{Perhaps a sentence on strategy would be good. Is the point that this describes the starting situation, and the aim is to make $A_{ij}=0$ for all $i \neq j$?  We could say so.}{MB: added a sentence.}
    In fact, the first part follows from~\ref{item:H-branches}, the second from~\ref{item:H-transverse}.
    As $A$ is linear, these two conditions are actually equivalent to the conditions~\ref{item:H-branches}  and \ref{item:H-transverse}.
    This form of $A$ is almost what is required
    by Lemma~\ref{lem:reduction_block}, except for the terms $A_{i0}$ that spoil the block structure of~$A$.
    Write $A=A_0+A_1$, where $A_0$ is the sum of diagonal blocks $A_{ii}$ and $A_1$ is the sum of the off-diagonal blocks $A_{i0}$ ($i>0$).

    Define $h^2_\tau(x)=A_0x+(1-\tau\phi(x))A_1x$ for a cut-off function $\phi$. We require that $\phi$ be supported on $Z_5$, and that it is equal to $1$ on a smaller neighbourhood $Z_6$ of $0\in\R^{n+k-1}$
    Note that $h^2_\tau$ is no longer linear (the off-diagonal terms in a matrix representing $h^2_\tau$ depend on $x$. Nevertheless,
    conditions \ref{item:H-branches} and~\ref{item:H-transverse} are easily seen to be satisfied. Indeed,
    condition~\ref{item:H-branches} follows from the block structure. The transversality condition~\ref{item:H-transverse} involves only the $A_{00}$ block, which is the same for $h^2_\tau$ and $A$.
  \end{proof}

  We can now modify the map $A$ to the identity acting block by block. Recall that $Z_5$ is the neighbourhood of $0\in\R^{n+k}$
  such that $h^2_1(x)$ is the linear map given by the block matrix.
  Shrink $Z_5$ if needed to make sure it has a product structure, i.e.\ $Z_5=(Z_5\cap L_0)\times\dots\times (Z_5\cap L_d)$.

  \begin{lemma}[Blockwise modification]\label{lem:blockwise}
    For any $i=0,\dots,d$, there is an isotopy $h_{i,\tau}^3$ of restrictions of $Z_5\cap L_i$, supported
    such that $h_{i,0}^3=A_{ii}$ and $h_{i,1}^3$ is the identity on an open subset $Z_{7i}\subseteq Z_5\cap L_i$.

    The isotopy preserves conditions~\ref{item:H-branches} and~\ref{item:H-transverse}.
  \end{lemma}

  \begin{proof}
    The case $i=0$ corresponds to the absolute case (classical cancellation theorem).
    This is the statement of \cite[Theorem 5.6]{Mil65}. Suppose $i>0$.
    We need to find an isotopy $h_{i,\tau}^3$ such that $h_{i,0}^3$ is $A_{ii}$, $h_{i,1}^3$ on a smaller subset is the identity.
    This is a straightforward problem in linear algebra, which we leave to the reader.
  \end{proof}

  By taking the direct sum of the maps $h_{i,\tau}^3$, we obtain the following result.

  \begin{corollary}
    There is an isotopy $h^3_\tau$ supported on $Z_5$ such that $h_0^3 = h^2_1$ and $h^3_1$ is the identity on $Z_6:=Z_{60}\times\dots\times Z_{6d}$.
    The isotopy preserves conditions~\ref{item:H-branches} and~\ref{item:H-transverse}.
  \end{corollary}
  \emph{Summarizing the proof of Lemma~\ref{lem:coor_system}}, we concatenate the isotopies $h^1_\tau$ (reducing to the linear part), $h^2_\tau$ (reducing to the block part), and $h^3_\tau$ (reducing to the identity), to obtain an isotopy whose starting map is $h_1\circ h_0^{-1}$, and whose end map
  is the identity on $Z_6$. Denote this isotopy by $h_\sigma$. Note that $h_\sigma$ is the identity near the boundary of $Z_0$. We inject this
  isotopy into $U_0$ by modifying the vector field $\xi$. As a result, we obtain a family of vector fields $\xi_\sigma$,
  such that the map induced by $\xi_\sigma$ is $h_\sigma$. The vector field $\xi_\sigma$ is tangent to the branches
  (because of \ref{item:H-branches}), by~\ref{item:H-transverse} it does not create any new trajectories between $p_-$ and $p_+$, and it is Morse--Smale, as desired.

  The end of the proof of Lemma~\ref{lem:coor_system} follows the same lines as the proof of Milnor's cancellation theorem \cite[Page 57]{Mil65}.
  Let $U_6$ be the subset of $U_0$ made out of trajectories that pass through $h_0(Z_6)$.
 The flow of $\xi_1$ induces the identity in the coordinates on $U_-$ and $U_+$. Hence, we can glue these coordinates in such a way that $\xi_1$ has the form \eqref{eq:xi_canc} in these coordinates.
\end{proof}

\begin{remark}\label{rem:coor_system}
 The proof of Lemma~\ref{lem:coor_system} required a modification of the vector field $\xi$ by injecting an isotopy $h_\sigma$
  that is a concatenation of the isotopies $h^1_\tau$, $h^2_\tau$, and $h^3_\tau$. Suppose $\xi_\sigma$ is a vector field
  obtained by injecting only an isotopy $h_\tau$, for $\tau\in[0,\sigma]$. By the property~\ref{item:H-transverse}, which is
  preserved by $h_\tau$, there is precisely one trajectory of $\xi_\sigma$ on the first stratum that connects $p_-$ with $p_+$.   This property will be important when discussing the uniqueness of death.
\end{remark}

We are now in position to prove Lemma~\ref{lem:extend}.

\begin{proof}[Proof of Lemma~\ref{lem:extend}]
  Recall that $\ol{U}_0\cap\ol{U}_+$ belongs to a single level set $F^{-1}(c_+)$. Choose a pair of grim neighbourhoods $X_-$ and $X_+$
  of $\Ha(p_-)$ and $\Hd(p_+)$ respectively, in $F^{-1}[a,b]$, such that $X_-\cap F^{-1}(c_+)$ intersects $X_+\cap F^{-1}(c_+)$ along an open
  subset in $F^{-1}(c_+)$ whose closure is contained in $Z_7$. By shrinking $X_-$ and $X_+$ if needed, we may assume that
  there are coordinates in both $X_-$ and $X_+$, on which $\xi_1$ has the form \eqref{eq:xi_canc}. The flow of $\xi_1$ induces
  the identity in these coordinates by Lemma~\ref{lem:coor_system}, hence the two coordinate systems match together to produce a coordinate
  system on $X_-\cup X_+$.
\end{proof}

We can now prove Lemma~\ref{lem:replace}, and by this conclude the proof of Cancellation Theorem~\ref{thm:grimcanc}.

\begin{proof}[Proof of Lemma~\ref{lem:replace}]
  We use the notation of the proof of Lemma~\ref{lem:extend}.
  In particular, $X_-\cup X_+$ is a grim neighbourhood of $\Ha(p_-)\cup\Hd(p_+)$ in $F^{-1}[a,b]$. Choose $a',b'$ in such a way that
  $F(a)<F(a')<F(p_-)<F(p_+)<F(b')<F(b)$. Let $X_0=(X_+\cup X_-)\cap F^{-1}[a',b']$.   Let $X_1$ be a grim neighbourhood of $\Ha(p_-)\cup\Hd(p_+)$
  in $F^{-1}[a',b']$ such that $\ol{X_1\cap F^{-1}(a')}\subseteq X_0$.
  Such an $X_1$ exists because of Lemma~\ref{lem:grimneighexist}. Define $\mu\colon F^{-1}(a')\to [0,1]$
  to be a cut-off function equal to $1$ on $X_1\cap F^{-1}(a')$ and to $0$ away from $X_0\cap F^{-1}(a')$. Extend $\mu$
  to a $\xi$ invariant function on the whole of $F^{-1}[a,b]$.

  Set $\zeta_-$ to be the infimum of $H$ on $X_0\times[0,1]$ (the last coordinates being for the $\tau$ variable),
  and set $\zeta_+$ to be the supremum of $H$.
  Choose $\Upsilon \colon \R \to \R$ in such a way that $a<\Upsilon(\zeta_-)<\Upsilon(\zeta_+)<b$, and $\Upsilon$ is linear
  near $F(p_-)$ and $F(p_+)$ with derivative $1$.
  On $F^{-1}[a',b']$ we define the function $F_\tau$ via:
  \begin{equation}\label{eq:H_tau_def}
    H_\tau=(1-\tau\mu)F+\tau\mu \Upsilon\circ H.
  \end{equation}
  We have $\partial_{\xi_1} H_\tau=(1-\tau\mu)\partial_{\xi_1} F+\tau\mu\partial_{\xi_1}(\Upsilon\circ H)$. This is clearly a non-negative
  number, because $\partial_{\xi_1} F$ and $\partial_{\xi_1}(\Upsilon\circ H)$ are nonnegative, and positive except at $p_-$ and $p_+$. Moreover, $\mu\equiv 1$ in a neighbourhood of $p_-$ and $p_+$,
  so the path $H_\tau$ is actually a special $\xi_1$-path
  (see Definition~\ref{def:special_xi_path}). Furthermore, $H_\tau$ on $F^{-1}(a')$ is strictly greater than $a$,
  and on $F^{-1}(b')$ is strictly smaller that $b$. This allows us to define a function $F_\tau$ in the following way.

  We set $F_\tau=F$ away from $F^{-1}[a,b]$ and $F_\tau=H_\tau$ on $F^{-1}[a',b']$.
 Suppose $z\in F^{-1}(a,a')$ (the case $z\in F^{-1}(b',b)$ is analogous). As $F$ has no critical points in $(a,a')\cup (b',b)$,
  there are unique points $z_0\in F^{-1}(a)$ and $z_1\in F^{-1}(a')$ such that $\xi_1$ flows from $z_0$ through $z$ to $z_1$.
  Suppose $t_1$ is the time taken to get from $z_0$ to $z_1$ and $t_0$ is the time to get from $z$ to $z_1$. We set
  \[ F_\tau(z) := \frac{t_0}{t_1} F(z_0)+\frac{t_1-t_0}{t_1}H_\tau(z_1).\]
  This interpolation is continuous. There might be non-smooth points at the level sets $a,a',b',b$, but in a presence
  of a vector field, we can smooth them in a standard way; see Lemma~\ref{lem:piecewise_smooth}.
\end{proof}

\section{Paths of birth}\label{subsection:path-birth}
We defined paths of birth in Definition~\ref{def:paths-of-xx}. Now we are going to specify a concrete class of paths of birth,
so-called elementary paths of birth. The construction is due to Cerf \cite[Section III.2.1, p.~66]{Cerf}.
Fix $h=0,\dots,n$, $\varepsilon>0$, and $\tau\in[0,1]$. Define the function $\lbirth_\tau\colon\R^n\to\R$ by
\[\lbirth_\tau(x_1,\dots,x_n)=-x_1^2-\dots-x_h^2+x_{h+1}^2+\dots+x^2_{n-1}+x_n^3-(2\tau-1)\varepsilon x_n.\]
The function $\lbirth_\tau$ is a standard bifurcation of an $A_2$ singularity: it has no critical points
for $\tau<1/2$, it has a single critical point for $\tau=1/2$ and two Morse critical points
of indices $h$ and $h+1$, respectively, for $\tau>1/2$.

Choose a bump function $\omega\colon\R^n\to[0,1]$, identically equal to $1$ in a neighbourhood of $0$ and having
compact support.
Define
\begin{equation}\label{eq:ellomega}
\lbirth_{\omega,\tau}(x_1,\dots,x_n)=-x_1^2-\dots-x_h^2+x_{h+1}^2+\dots+x^2_{n-1}+x_n^3-(2\tau\omega-1)\varepsilon x_n.
\end{equation}
Choose sufficiently small $\varepsilon>0$. Then $\lbirth_{\omega,\tau}$ has no critical points in $\omega^{-1}(0,1)$. In particular, $\lbirth_\tau$ and $\lbirth_{\omega,\tau}$ have the same critical points, and they agree
in the neighbourhood of these critical points; see \cite[Section III.1.1]{Cerf}.

Define $\psi \colon \R^n \to \R^n$ by
\[\psi(x_1,\dots,x_n)=(x_1,\dots,x_{n-1},\lbirth_0(x_1,\dots,x_n)),\]
so that $\lbirth_0\circ\psi^{-1}=x_n$.

\begin{definition}\label{def:birth_on_Rn}
  The path $\birth_\tau=\lbirth_{\omega,\tau}\circ \psi^{-1} \colon \R^n \to \R$ is called the \emph{elementary path of birth on $\R^n$} (of index $h$).
\end{definition}

The path $\birth_\tau$ is a path such that $\birth_0=x_n$ and $\birth_\tau$ acquires a pair of critical points. Moreover, as $\omega$ has compact support, $\birth_\tau=\birth_0$ away from a compact set in $\R^n$. This motivates the following definition;
compare \cite[Section III.1.2]{Cerf}.

\begin{definition}[Elementary path of birth]\label{def:el_path_birth}
  A path $f_\tau\colon N\to\R$ is an \emph{elementary path of birth} at $q_0$ (of index $h$), if there exists an open set $U$ containing $q_0$, and local coordinates $x_1,\dots,x_n$ in $U$ such that
\begin{itemize}
  \item $q_0=(0,\dots,0)$;
  \item $f_\tau=f_0$ away from $U$;
  \item $f_\tau(x_1,\dots,x_n)=\Upsilon\circ\birth_\tau(x_1,\dots,x_n)$
\end{itemize}
for some smooth, increasing function $\Upsilon\colon\R\to\R$.
\end{definition}

\begin{remark}\label{rem:arbitrarily_small}
  By choosing appropriate $\omega$ (with small support) we can
  assume that $\birth_\tau=\birth_0$ away from an arbitrarily small neighbourhood
  of $(0,\dots,0)$.
\end{remark}
We quote two results of Cerf.

\begin{proposition}[see \expandafter{\cite[Proposition III.1.3.1]{Cerf}}]\label{prop:elem_birth}
  Any path of birth $f_\tau$ can be given an arbitrarily small perturbation to another path of birth $\wt{f}_\tau$ such that for some
  $\delta_0>0$, $\wt{f}_\tau$ restricted to $\tau\in[1/2-\delta_0,1/2+\delta_0]$ is an elementary path of birth.
\end{proposition}

In the proof of Proposition~\ref{prop:elem_birth}, Cerf uses the fact that through each point of $\cF^1_\alpha$
one can pass an elementary path of birth. It follows promptly that we can assume that the paths $\wt{f}_\tau$ and $f_\tau$
can be assumed to cross $\cF^1_\alpha$ at the same point. This leads to the following statement.

\begin{corollary}\label{cor:elem_birth}
  Any path of birth is $\cF^1$-homotopic to an elementary path of birth.
\end{corollary}

The second result shows that paths of birth are, up to lax homotopy, determined by the index of the critical points
that are created, and the connected component of the level set, in which they are constructed.

\begin{proposition}[see \expandafter{\cite[Corollaire III.1.3.2, page 67]{Cerf}}]\label{prop:unicite_naiss}
  Suppose two paths of birth $\wt{f}_\tau$ and $f_\tau$ have $f_0$ and $\wt{f}_0$ in the same connected component of $\cF^0$
  and $f_1$ and $\wt{f}_1$ in the same connected component of $\cF^0$. Suppose the two paths create critical points
  $q_0$, respectively $\wt{q}_0$, on the same component of the fixed level set $f_{1/2}^{-1}(c)$, and $q_0$ and $\wt{q}_0$ are of
  the same index. Then $f_\tau$ and $\wt{f}_\tau$ are lax homotopic.
\end{proposition}

Our statement slightly differs from Cerf's result, hence we indicate how to use Cerf's result to prove Proposition~\ref{prop:unicite_naiss}.

\begin{proof}
  Suppose $g_0$ and $g_1$ are the two points of intersection of $f_\tau$ and $\wt{f}_\tau$ with $\cF^1$. By Corollary~\ref{cor:elem_birth},
  there exists elementary paths of birth $f_{0,\tau}$ and $f_{1,\tau}$, passing through $g_0$, respectively $g_1$, and such
  that $f_{0,\tau}$ and $f_\tau$ are $\cF^1$-homotopic and $f_{1,\tau}$ and $\wt{f}_\tau$ are $\cF^1$-homotopic.   It is enough to show that $f_{0,\tau}$ and $f_{1,\tau}$ are lax homotopic. This follows
  promptly from \cite[Corollaire III.1.3.2]{Cerf}.
\end{proof}

\section{Paths of death}\label{sec:path_of_death}

Throughout Section~\ref{sec:path_of_death} we will assume that $f\colon N\to\R$ is a Morse function, and that $q_-$ and $q_+$ are critical points of indices $h_-$ and $h_+=h_-+1$ respectively. Set $a=f(q_-)$, $b=f(q_+)$.  Assume that $a<b$ and there are no critical points of $f$ with critical values in $(a,b)$.  Fix $c\in (a,b)$.

First, we recall the definition of Cerf's ascending and descending caps (what he calls, \emph{en fran\c{c}ais, nappes ascendantes et descendantes}). These are an abstract form of (parts of) the ascending and descending manifolds of critical points. Next, we recast the statements of Cerf in the language of vector fields.

\subsection{Paths of death via ascending and descending caps}

This subsection is based on \cite[Section III.2]{Cerf}.

\begin{definition}[Pair of caps]\label{def:caps}
  A \emph{pair of caps} is a pair of smooth embeddings $\varphi_D\colon D^{h_+}\to N$, $\varphi_A\colon D^{n-h_-}\to N$, where $D^{h_+}$ and
  $D^{n-h_-}$ are closed discs with $\varphi_A(0)=q_-$, $\varphi_D(0)=q_+$,
  and the composition $f\circ \varphi_D$  (respectively $f\circ\varphi_A$) is a quadratic function equal to $c$ at the boundary $\bd D^{h_+}$
  (respectively $\bd D^{n-h_-}$) and having a global maximum $b$ at zero (respectively, having a global minimum $a$ at zero).
\end{definition}

A pair of caps shall usually be denoted by $(D,A)$ with $D := \varphi_D(D^{h_+})$, $A := \varphi_A(D^{n-h_-})$.

\begin{definition}[Caps in good position]\label{def:good_caps}
  Pair of caps $(D,A)$ is called a \emph{pair of caps in good position}, in short a \emph{pair of good caps}, if the boundaries of $D$ and $A$ intersect transversely at a single point in $f^{-1}(c)$.
\end{definition}

\begin{example}\label{ex:good_pair_eta}
  Let $\eta$ be a gradient-like vector field for $f$. We set $D=W^s_{q_+}(\eta)\cap f^{-1}[c,b]$, $A=W^u_{q_-}(\eta)\cap f^{-1}[a,c]$, that is,
  $D$ and $A$ are parts of the stable (respectively, unstable) manifolds of $q_+$ (respectively $q_-$).
  Then $(D,A)$ form a pair of caps. Moreover, $(D,A)$ are in a good position
  if and only if $\eta$ satisfies the assumptions of Milnor's Cancellation Theorem,  that is, if $W^s_{q_+}(\eta)$ intersects
  $W^u_{q_-}(\eta)$ transversely along a single trajectory connecting $q_-$ and $q_+$; compare \cite[Theorem 5.4]{Mil65}.
\end{example}

\begin{definition}
  If $\eta$ satisfies the assumptions of Milnor's Cancellation Theorem, the pair $(D,A)$ of Example~\ref{ex:good_pair_eta} is
  called the pair of good caps \emph{associated with $\eta$} and denoted $(D_\eta,A_\eta)$.
\end{definition}

\begin{remark*}
  A gradient vector field $\nabla f$ (for some metric) gives rise to a pair of caps $(D,A)$ and any pair
  of caps arises from a gradient vector field; see \cite[Section III.2.1]{Cerf}. As any gradient-like vector field is
  an actual gradient
  for some choice of metric, we can rephrase Cerf's statement as: `any pair of caps $(D,A)$ can be constructed
  via a suitable gradient-like vector field $\eta$'.
\end{remark*}

Suppose $(D,A)$ is a pair of good caps.
The following is a key notion for constructing elementary paths of death; see \cite[page 70]{Cerf}.

\begin{definition}[Double neighbourhood adapted to $(D,A)$]\label{def:double_nbhd}
  A map $\varphi\colon D^n\to N$ is a \emph{double neighbourhood adapted to a pair of good caps $(D,A)$}, if:
    \begin{itemize}
      \item $U=\varphi(D^n)$ contains $D\cup A$;
      \item with coordinates $x_1,\dots,x_n$ on $U$ induced from coordinates on $D^n$, we have $f\circ\varphi=\Upsilon\circ\birth_1$
	for some increasing function $\Upsilon\colon\R\to\R$;
      \item if $\nabla$ denotes the gradient in $U$ in the metric where $x_1,\dots,x_n$ are orthonormal coordinates, then the caps $D$ and $A$ are given by the ascending and descending manifolds of $\nabla f$ respectively.
    \end{itemize}
\end{definition}

%\begin{remark}\label{rem:unofficial}
%  We drop the adjective `standard' from the definition of $\phi$, because hell if I know where to put it in the description of $\phi$ so that
%  it doesn't sound awkward.
%\end{remark}
\begin{remark*}
Throughout the paper, a double neighbourhood adapted to $(D,A)$ is always chosen with respect to the chosen standard model as described
  in \cite[Section III.2.2]{Cerf}.
\end{remark*}

Suppose $\varphi\colon U\to N$ is a double neighbourhood adapted to $(D,A)$. %We can construct an elementary path
%of death from $\varphi$.

\begin{definition}[compare \expandafter{\cite[Section III.2.3]{Cerf}}]\label{def:el_path_death}
  The path of functions $f_\tau(u)$ given by $f_\tau(u)=f(u)$ if $u\neq U$ and $f_\tau(u)=\Upsilon\circ\birth_{1-\tau}(\varphi^{-1}(u))$ is called
  the \emph{elementary path of death associated with~$\varphi$}.
\end{definition}

\begin{lemma}[\expandafter{\cite[Corollary, page 72]{Cerf}}]\label{ars:two}
    Any path of death $f_\tau$ can be written as a composition $f^1_{\bullet}\cdot \wt{f}_{\bullet}\cdot f^0_{\bullet}$ such that $f^1$ and $f^0$ are paths of excellent Morse functions and $\wt{f}_\tau$ is an elementary path of death
    supported on an arbitrarily small neighbourhood of some pair $(D,A)$ of caps in good position.
\end{lemma}

\begin{lemma}[Homotopy of paths of death]\label{lem:homotopy_of_death}
  Suppose $\varphi_\sigma$, $\sigma\in[0,1]$ is a smooth family of maps from $D^n\to N$ satisfying the conditions of Definition~\ref{def:double_nbhd}. Let $\tau\mapsto h_{\sigma,\tau}$ be the path
  of Definition~\ref{def:el_path_death} for $\varphi_\sigma$ starting from a given path $h_{\sigma,0}=g_\sigma$ in $\cF^0$. Then $h_{\sigma,\tau}$ is a left homotopy between paths $h_{0,\tau}$ and $h_{1,\tau}$.
\end{lemma}
\begin{proof}
  First, it is routine to see that $h_{\sigma,\tau}$ depend smoothly on the parameters $\sigma$ and $\tau$.
  By construction, any path $\tau\mapsto h_{\sigma,\tau}$ crosses $\cF^1$ transversely, and at precisely one parameter value of $\tau$ (not equal
  to $0$ and $1$). This indicates that $\tau\mapsto h_{\sigma,\tau}$ is an $\cF^1$-path. Moreover,
  the endpoints $h_{\sigma,1}$ all belong to $\cF^0$, so they are in the same connected
  component to $\cF^0$. This shows that $h_{\sigma,\tau}$ is a homotopy between $h_{0,\tau}$ and $h_{1,\tau}$ as desired.
\end{proof}

Let $\cP$ denote the subspace of all double neighbourhoods adapted to a pair of good caps $(D,A)$ with the topology given by uniform convergence.
Let $\cP{(D,A)}$  be the subspace of
all double neighbourhoods adapted to a fixed pair of good caps $(D,A)$ (Cerf denotes this space by $\cP'$).
The following result of Cerf is essential to prove uniqueness of deaths.

\begin{lemma}[see \expandafter{\cite[page 70, item 3c]{Cerf}}]\label{ars:three}
  The space $\cP(D,A)$ is nonempty and path-connected. If $\cN$ denotes the space of all pairs of good caps $(D,A)$,
  then the assignment $\cP\to\cN$ is a locally trivial fibration with fibres $\cP(D,A)$.
\end{lemma}

Combining Lemma~\ref{lem:homotopy_of_death} with Lemma~\ref{ars:three}, we obtain the following result.

\begin{theorem}[Uniqueness of Death, geometric version]\label{thm:uni_death_one}
  Assume $g_\sigma$ is a path of excellent Morse functions.
  Suppose $\varphi_{D,\sigma}\colon D^{h_+}\to N$ and $\varphi_{A,\sigma}\colon D^{n-h_-}\to N$ are smooth families of functions $($with $\sigma\in[0,1])$ such that $D_\sigma:=\varphi_{D,\sigma}(D^{h_+})$
  and $A_\sigma:=\varphi_{A,\sigma}(D^{n-h_0})$ form a pair of good caps for $g_\sigma$ all $\sigma\in[0,1]$. Suppose $\varphi_0$ and $\varphi_1$ are two double neighbourhoods adapted to $(D_0,A_0)$ and $(D_1,A_1)$, respectively.

  Let $h_{0,\tau}$ and $h_{1,\tau}$ be the paths of death associated with $\varphi_0$ and $\varphi_1$ as in Definition~\ref{def:el_path_death},
  where $h_{0,0}=g_0$ and $h_{1,0}=g_1$. Then the paths $\tau\mapsto h_{0,\tau}$ and $\tau\mapsto h_{1,\tau}$ are lax homotopic over $g_\sigma$.
\end{theorem}

\begin{proof}
  The assumptions imply that there is a path $\zeta_\sigma$ in the space $\cN$, connecting pairs of good caps $(D_\sigma,A_\sigma)$, the path
  being given by $(\varphi_{D,\sigma},\varphi_{A,\sigma})$. By Lemma~\ref{ars:three}, there is a lift of $\zeta_\sigma$ to a path
  $\wt{\zeta}_\sigma$ in $\cP$, whose endpoints are $\varphi_0$ and $\varphi_1$. By Lemma~\ref{lem:homotopy_of_death} the paths
  constructed using $\varphi_0$ and $\varphi_1$ are lax homotopic over $g_\sigma$.
\end{proof}

\begin{remark*}
  Cerf's statement on unicity of death \cite[Proposition III.2.4.4]{Cerf} has slightly different statement, because Cerf uses extra assumptions
  on dimension and simple connectivity of the spaces to guarantee that the rather complicated assumptions of Theorem~\ref{thm:uni_death_one} are satisfied. But our proof follows his proof very closely. %\ypar{Added contribution to Cerf. MP: attribution?}
\end{remark*}

\begin{comment}
The condition that $(D_0,A_0)$ and $(D_1,A_1)$ can be connected by a family of pairs of good caps can be rephrased in a more tractable way;
see \cite[Lemma III.2.4.2]{Cerf}. % the lemma at the bottom of page 252/72.

\begin{lemma}\label{lem:isotopy_of_caps}
  Let $c$ be such that $f(q_-)<c<f(q_+)$. Let $(D_0,A_0)$, $(D_1,A_1)$ be two pairs of good caps. Write $D_i^c=D_i\cap f^{-1}(c)$, $A_i^c=A_i\cap f^{-1}(c)$, $i=0,1$.
  If there exists an isotopy $D_{\sigma}^c,A_\sigma^c$ such that, for all $\sigma$, $D_\sigma^c$ intersects $A_\sigma^c$ transversely
  at a single point, then $(D_0,A_0)$ and $(D_1,A_1)$ are isotopic through pairs of good caps.
\end{lemma}
\end{comment}

\subsection{Paths of death via vector fields}
Our Cancellation Theorem~\ref{thm:grimcanc} (where the absolute case is obtained by setting $d=0$) creates a certain path of functions
along which two critical points are destroyed. The initial piece of data is a choice of a gradient-like vector field. We aim to combine
the two approaches.
%\ypar{Somehow the title of this subsection disappeared. We had only Sub. 3.6.1.}

Assume $f\colon N\to\R$ is a Morse function, $q_-$ and $q_+$ are two critical points of indices $h$ and $h+1$ respectively, with
no critical points in $f^{-1}[f(q_-),f(q_+)]$ apart from $q_-,q_+$, and $\eta$ is a Morse--Smale
vector field, gradient-like for $f$, such that there exists a unique trajectory of $\eta$ connecting $q_-$ and $q_+$. Let
$f_\tau$, $\tau\in[0,2]$, be the path of functions from Theorem~\ref{thm:grimcanc}, and let $\eta_\sigma$, for $\sigma\in[0,1]$, be the path of vector fields
from that theorem: that is $\eta_0$ is the original vector field, and $\eta_1$ is the vector field obtained after applying the modifications required
by Lemma~\ref{lem:coor_system}; compare Remark~\ref{rem:coor_system}.
Recall that in the absolute case we use the notation $f,\eta$, instead of $F,\xi$.

\begin{proposition}\label{prop:cerf_milnor}
  %Choose a level set $c$ between $f(q_-)$ and $f(q_+)$.
  Let $(D_\sigma,A_\sigma)$ be the good pair of caps constructed via $\eta_\sigma$ as in Example~\ref{ex:good_pair_eta}.
  Let $\wt{f}_\tau$ be an elementary path of death with respect to the good pair of caps $(D_0,A_0)$ with $\wt{f}_0=f_0$. Then $\wt{f}_\tau$ is left-homotopic
  to $f_\tau$.
\end{proposition}

\begin{proof}
  Let $\tau\mapsto h_{\sigma,\tau}$ be the elementary path of death starting from $f_0$
  with respect to the good pair of caps $(D_\sigma,A_\sigma)$.
  By Lemma~\ref{lem:coor_system} (and Remark~\ref{rem:coor_system}), the assumptions of Theorem~\ref{thm:uni_death_one} are
  satisfied, so $h_{0,\tau}$ and $h_{1,\tau}$ are lax homotopic over $g_\sigma:=h_{\sigma,0}$. As $h_{\sigma,0}=f_0$, we conclude that $h_{0,\tau}$ and $h_{1,\tau}$ are left-homotopic. Also note that $\wt{f}_\tau=h_{0,\tau}$.
  %\npar{Theorem~\ref{thm:uni_death_one} just says lax. Should this say lax instead of left?}{Fixed} Note that $\wt{f}_\tau=h_{0,\tau}$.

  We aim to prove that $f_\tau$ and $h_{1,\tau}$ are left-homotopic. Note that Lemma~\ref{lem:coor_system} in
  the proof of Theorem~\ref{thm:grimcanc} constructs an explicit coordinate system and the sets $X_-$ and $X_+$
  that are a double neighbourhood adapted to $(D_1,A_1)$. %\ypar{Note that we are more careful with modified version.}

  In this neighbourhood, the path $f_{1+\tau}$, $\tau\in[0,1]$, is -- by construction --
  an elementary path of death; compare item~\ref{item:Canc_form} of
  Theorem~\ref{thm:grimcanc}.  By Theorem~\ref{thm:uni_death_one}, $f_{1+\tau}$, for $\tau\in[0,1]$, is lax homotopic to $h_{0,\tau}$ over $f_{\tau}$, $\tau\in[0,1]$.

  This implies that the paths $\wt{f}_\tau$ and $f_{1+\tau}$ are lax homotopic over $f_{\tau}$. That is, there exists a homotopy $\wt{h}_{\sigma,\tau}$ such that $\wt{h}_{\sigma,0}=f_{\sigma}$, $\wt{h}_{0,\tau}=\wt{f}_\tau$, and $\wt{h}_{1,\tau}=f_{1+\tau}$. Reparametrizing this homotopy to
  \[h'_{\sigma,\tau}=\begin{cases}
    \wt{h}_{2\sigma\tau,0} & \tau\le \frac12\\
    \wt{h}_{\sigma,2\tau-1} & \tau\ge \frac12
  \end{cases}
  \]
  yields a left-homotopy between $\wt{f}_\tau$ and $f_{\tau}$ (the latter path being for $\tau\in[0,2]$).
\end{proof}

\begin{remark}\label{rem:grimcanc_choices}
  The construction of the path in the proof of Theorem~\ref{thm:grimcanc} involved many choices, like the choice of isotopy $h_\tau$,
  of neighbourhoods $X_-,X_+$, etc. Proposition~\ref{prop:cerf_milnor} shows in particular that (at least in the absolute case) the choices
  do not matter, up to left-homotopy of paths.
\end{remark}

From Proposition~\ref{prop:cerf_milnor} we give another variant of Uniqueness of Death Theorem~\ref{thm:uni_death_one},
which is phrased in terms of vector fields and their stable/unstable manifolds only.

\begin{lemma}[Uniqueness of Death, Vector Field version, common starting function]\label{lem:uniqueness_of_death}
  Suppose $\eta_\sigma$, $\sigma\in[0,1]$ is a family of gradient-like vector fields for $f$ satisfying the Morse--Smale
  conditions. Assume that for all $\sigma$, $\eta_\sigma$ has a single trajectory connecting critical points~$q_-$ and~$q_+$.

  The paths of death from Theorem~\ref{thm:grimcanc}, constructed using $\eta_0$ and $\eta_1$ and cancelling $q_-$ and $q_+$, are left-homotopic.
\end{lemma}
\begin{proof}
  Let $\tau\mapsto h_{\sigma,\tau}$ denote the path of death from Theorem~\ref{thm:grimcanc} constructed with $\eta_\sigma$.
  Let $(D_\sigma,A_\sigma)$ be the pair of good caps constructed from $\eta_\sigma$. Let $\tau\mapsto \wt{h}_{\sigma,\tau}$
  be Cerf's path of death constructed with $(D_\sigma,A_\sigma)$.
  By Proposition~\ref{prop:cerf_milnor} the paths $h_{0,\tau}$ and $\wt{h}_{0,\tau}$, as well as the paths $h_{1,\tau}$ and $\wt{h}_{1,\tau}$, are left-homotopic. By the Uniqueness of Death Theorem~\ref{thm:uni_death_one}, the paths $\wt{h}_{0,\tau}$ and $\wt{h}_{1,\tau}$
  are left-homotopic.
\end{proof}

Lemma~\ref{lem:uniqueness_of_death} admits a refinement if $\eta_0$ and $\eta_1$ are gradient-like vector fields for two distinct Morse functions belonging to the same connected component of $\cF^0$. The precise formulation might sound artificial, and it is far from the most general statement, but it gives us a ready-to-use criterion to be applied in later sections.
The argument resembles the
proof of Lemma~\ref{lem:cut_and_paste}.  We are again using a cut-and-paste argument to create a homotopy.

\begin{lemma}\label{lem:uniqueness_after_rear}
  Suppose $g_\sigma$, $\sigma\in[0,1]$ is a path of excellent Morse functions, $q_-,q_+$ are two critical points of each of $g_\sigma$, and $\eta$ is a gradient-like vector field for $f_0$, such that $q_-$ and $q_+$ are in cancelling position.
  Write $D$ for the unstable manifold of $q_+$ intersected with $g_0^{-1}[g_0(q_-),g_0(q_+)]$, and $A$ for the stable manifold of $q_-$ intersected with $f_0^{-1}[f_0(q_-),f_0(q_+)]$.

  Suppose there exists a neighbourhood $U$ of $D\cup A$ such that $g_{\sigma}|_U=g_0|_U$ for all $\sigma$. Assume that $\eta'$ is a gradient-like
  vector field for $g_1$ such that $\eta'|_U=\eta|_U$.

  Then any two Milnor paths of death starting from $g_0$ with $\eta$ and starting from $g_1$ with $\eta'$ are lax homotopic over $g_\sigma$.
\end{lemma}

\begin{proof}
  By Proposition~\ref{prop:cerf_milnor} the paths from Theorem~\ref{thm:grimcanc} are left-homotopic
  to Cerf paths of death associated to good caps for $\eta$ and $\eta'$, respectively. Therefore, it is enough to prove that the Cerf
  paths of death associated with good pair of caps constructed with $\eta$ and $\eta'$ are lax homotopic over $g_\sigma$.

  Now, the good pair of caps associated with $\eta$ and $\eta'$
  are isotopic: this is precisely the part of $D,A$ cut out by a level set between $g_0(q_-)$ and $g_0(q_+)$. We conclude by Theorem~\ref{thm:uni_death_one}.
\end{proof}

Finally, we have a statement on uniqueness of death for a family of functions with a common vector field.

\begin{lemma}\label{lem:uni_common}
  Suppose $h_{\sigma,0}$, $\sigma\in[0,1]$ is a path of excellent Morse functions and $\eta$ is a Morse--Smale gradient-like vector field for all $h_{\sigma,0}$.
  Assume there exists precisely one trajectory of $\eta$ connecting two critical points $q_-$ and $q_+$ and that $h_{\sigma,0}$ is independent of $\sigma$ near $\cK_-\cap \cK_+$, where $\cK_-$, $\cK_+$ are the intersections of the unstable manifold of $q_-$ $($respectively, the stable manifold of $q_+)$ with the level sets $h_{0,0}^{-1}[h_{0,0}(q_-),h_{0,0}(q_+)]$. The paths of death starting with $h_{0,0}$ and $h_{1,0}$ and constructed with $\eta$, cancelling this pair,  are lax-homotopic over $h_{\sigma,0}$.
\end{lemma}
\begin{proof}
  Let these two paths of death, starting from $h_{0,0}$ and $h_{1,0}$ be called $g_{\tau}$ and $\wt{g}_\tau$ respectively.
  Let $U$ be an open set containing $\cK_-\cup\cK_+$ on which $h_{\sigma,0}$ is equal to $f_{0,0}$. Let $U_1$ be another neighbourhood of $\cK_-\cup\cK_+$
  with $\ol{U}_1\subseteq  U$. Construct a path of death $h_{0,\tau}$ starting from $h_{0,0}$, cancelling $q_-$ with $q_+$, with guiding vector field $\eta$, and such that the path is supported on $U_1$. Also let $\phi\colon N\to\R$ be a cut-off function, supported on $U$ and equal to $1$ on $U_1$.
  For $\sigma\in[0,1]$ define:
  \[h_{\sigma,\tau}=h_{0,\tau}\phi+(1-\phi)h_{\sigma,0}.\]
  Away from $U$, $h_{\sigma,\tau}=h_{\sigma,0}$. On $U_1$, $h_{\sigma,\tau}=h_{0,\tau}$. On $U\setminus U_1$, $h_{0,\tau}=h_{0,0}=h_{\sigma,0}$, so $h_{\sigma,\tau}=h_{0,0}$.
  In particular $\sigma\mapsto h_{\sigma,\tau}$ is an $\cF^1$-homotopy of paths of death.

  The path $h_{1,\tau}$ is a path of death, supported on $U_1$, with starting point $h_{1,0}$ and constructed using $\eta$. Any two $\eta$ paths
  with the same starting point are left-homotopic. Thus, $h_{1,\tau}$ is left-homotopic to $\wt{g}_\tau$. Likewise, $h_{0,\tau}$ is left-homotopic to $g_\tau$. Hence, $g_\tau$ and $\wt{g}_\tau$ are lax homotopic over $h_{\sigma,0}$. See Figure~\ref{fig:uni_common} for the notation.
  %\npar{I found this proof hard to follow still, with rather confusing notation. It might help if the two paths we want to show something about are called $f_{0,\tau}$ and $f_{1,\tau}$? Also maybe  a picture showing the abstract idea of the proof?}{Changed the notation.}
\end{proof}
  \begin{figure}
    \input{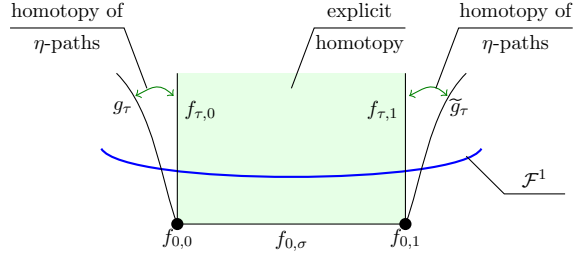}
    \caption{Schematic of the proof of Lemma~\ref{lem:uni_common}.}\label{fig:uni_common}
  \end{figure}
\part{Path Lifting}\label{part:pathlifting}

This part gives the main technical tool needed to prove the Concordance Implies Regular Homotopy Theorem~\ref{thm:concordance},
namely the Path Lifting Theorem~\ref{thm:path_lifting}. First we need some more advanced setup, namely
%The key technical tool in Part~\ref{part:pathlifting} is the Path Lifting Theorem~\ref{thm:path_lifting}, which we prove
%in Section~\ref{sec:path_lifting_proof}, using
%
the lifting lemmas from Sections~\ref{sec:elementary_lift}, \ref{sec:lifting-paths-of-birth}, \ref{sec:lifting-paths-of-rearangment}, and \ref{sec:lifting-paths-of-death}.

The Path Lifting Theorem is one of the key novelties of this paper. To state it, suppose that
$N$ is a compact manifold, and let $f_\tau\colon N\to \R$, $\tau\in[0,1]$ be an $\cF^1$-path of functions. The precise definition of an $\cF^1$-path was given in Definition~\ref{def:regular_path}.
Let $G\colon N\to\O$ be a generic immersion
and suppose $F_0\colon\O\to\R$ is such that $F_0\circ G=f_0$. We ask whether there exists a regular path of immersions $G_\tau \colon N \to \O$ with $G_0=G$,
and an $\cF^1$-path functions $F_\tau\colon\O\to\R$ such that $F_\tau\circ G_\tau=f_\tau$ and $F_\tau$ has no new critical points away from $G_\tau(N)$.
The goal of this part is to prove Path Lifting Theorem~\ref{thm:path_lifting}, which states that this is possible. An important step in the proof is the \emph{finger move}.  We delay a detailed description of the finger move until Part~\ref{part:finger}, since this is a significant detour which takes around 30 pages.

Throughout Part~\ref{part:just_paths}, we assume that $N$ has all
connected components of the same dimension. Necessary changes if this is not the case are given in Section~\ref{sec:multidim} at the end of this part.
\section{Generalities on path lifting}\label{sec:elementary_lift}

In this section we
make precise the notion of lifting a path of functions, and
address the problem of lifting elementary paths.

\begin{definition}\label{def:lifting}
  Suppose $f_{\tau} \colon N \to \R$, $\tau\in[0,1]$ is an $\cF^1$-path of functions on $N$, as in  Definition~\ref{def:regular_path}.  Let $G\colon N\looparrowright\O$ be a generic immersion
  with $M=G(N)$.
  Finally, let $F\colon\O\to\R$ be an immersed Morse function such that $F\circ G=f_0$.

  Suppose we are given a path $(F_\tau,G_\tau)$ with $\tau\in[0,1]$, $F_\tau\colon\O\to\R$ and $G_\tau\colon N\to\O$ such that $F_0=F$, $G_0=G$, and $(F_\tau,G_\tau)$ is a regular double path (see Definition~\ref{def:regular_dp}) with the property that $F_\tau \colon \Omega \to \R$ is Morse
  for all $\tau \in [0,1]$ (not necessarily immersed Morse).
  Set $\wt{f}_\tau=F_\tau\circ G_\tau$.

\begin{enumerate}
  \item  We say that the pair $(F_{\tau},G_\tau)$ \emph{lifts the family $f_{\tau}$} if $f_\tau=\wt{f}_\tau$.
  \item  We say that $(F_\tau,G_\tau)$ \emph{weakly lifts the family $f_\tau$} if $f_1 = \wt{f}_1$ and  $\wt{f}_\tau$ is $\cF^1$-homotopic to $f_\tau$.
\end{enumerate}
\end{definition}

It is a key property of the lift that the function $F_{\tau}$ is Morse for all $\tau\in[0,1]$, that is, it
does not acquire any critical points in the top stratum. It is not hard to lift a path $f_{\tau}$ at the cost of creating critical points away from $N$: take $G_\tau = G$ for all $\tau$, choose a path of functions on $\O$ extending the given function on $N$, and perturb it to be generic
using the methods of Subsection~\ref{sub:immersed_cerf}. But lifting this way, we would not have
  any control on the critical points which might appear on the zeroth stratum, so this would not be a very useful approach.
  %But we want to control the critical points on the top stratum, so we do not wish to do this.

We begin with Lemma~\ref{lem:lift_morse}, which lifts a path of Morse functions with no rearrangements.
Then we prove lemmas on lifting rearrangements, births, and deaths. These lemmas give us an inductive argument to prove the main theorem, namely the
Path Lifting Theorem~\ref{thm:path_lifting}.

\begin{lemma}[Lifting paths of excellent Morse functions]\label{lem:lift_morse}\
  \begin{itemize}
    \item[(a)]
      Assume $N$ and $\O$ are closed.
      Suppose $f_{\tau}$ is a path of excellent Morse functions on
      $N$, $G_0\colon N\to\O$ is a generic immersion, and $F_0\colon\O\to\R$ is an immersed
      Morse function for $M=G_0(N)$, such that $F_0 \circ G_0 = f_0$.

      Then there exists a family $\Psi_\tau$ of self-diffeomorphisms of $N$, and a family $\Upsilon_\tau$ of orientation preserving self-diffeomorphisms of $\R$,
      such that the composition
      \[F_\tau=\Upsilon_\tau\circ F_0\circ G_0\circ\Psi_\tau\]
      defines a lift of the path $f_{\tau}$, that is, for $G_\tau=G_0\circ\Psi_\tau$ we have $f_\tau=F_\tau\circ G_\tau$.

    \item[(b)]
      If $N$ and $\O$ are compact, $f_\tau$ is a neat path, $F_0$ is neat and $G_0$ is neat, then the statement holds, with
      $(F_\tau,G_{\tau})$ also being neat.
  \end{itemize}
\end{lemma}
\begin{proof}
 % Suppose first $f_\tau$ is path of excellent Morse functions.
  Assume first that $N$ and $\Omega$ are closed.
  By stability of Morse functions~\cite[Proposition III.2.2]{GG}, there is a family $\Psi_\tau$ of self-diffeomorphisms of $N$
and a family $\Upsilon_\tau\colon\R\to\R$ of strictly increasing functions such that $f_\tau=\Upsilon_\tau\circ f_0\circ\Psi_\tau$. We define $G_\tau=G_0\circ\Psi_\tau$
and $F_\tau=\Upsilon_\tau\circ F_0$. It is elementary to check that $F_\tau$ satisfies all the needed assumptions.

If $N$ and $\Omega$ have boundary, the argument of \cite{GG} provides us with $\Psi_\tau$ being the identity near the boundary of $N$
and $\Upsilon_\tau$ being the identity near $0$. This means that the path $G_\tau$ is independent of $\tau$ near the boundary of $N$,
while $F_\tau$ is independent of $\tau$ near the boundary of $\O$. That is to say, the double  path $(F_\tau,G_{\tau)}$ is neat.
% $F_\tau$ and $G_\tau$ are neat.
\end{proof}

\begin{remark}\label{rem:lift_morse}
Lemma~\ref{lem:lift_morse} not only says that the path $f_{\tau}$ can be lifted, but also gives a very precise recipe for a lift. For example, it follows
from the form of $F_\tau$ and $G_\tau$ that the image $M=G_\tau(N)$ is independent of $\tau$, only the parametrisation changes. In particular,
$G_\tau$ is a generic immersion for all $\tau$ and $F_\tau$ is a path of immersed Morse functions. In particular $(F_\tau,G_\tau)$
is a regular double path. %\ypar{Added this sentence, it is important.}
Moreover,
if $\xi$ is a grim vector field for $F_0$, then one readily checks that $\xi$ is a grim vector field for the whole family $F_\tau$.
\end{remark}

The following corollary
will be useful. %\ypar{Added a corollary, I thought it was written somewhere.}

\begin{corollary}\label{cor:left_homotopy}
  Suppose $f_\tau$ is an $\cF^1$-path of functions, and $(F_\tau,G_\tau)$ are such that $f_\tau$ and  $\wt{f}_\tau:=F_\tau\circ G_\tau$ are left-homotopic.
  Then there exists a regular double path $(F'_\tau,G'_\tau)$ that weakly lifts $f_\tau$.
\end{corollary}
\begin{proof}
  For technical reasons, we apply a small perturbation to $(F_\tau,G_\tau)$, denoted
  by $(\wt{F}_\tau,\wt{G}_\tau)$ so that the latter forms a regular double path; see Corollary~\ref{cor:perturb_to_regular}.
  As $G_0$ is a generic immersion, and $F_0$ is an immersed Morse function with respect to $G_0(N)$, we may assume that
  the perturbation fixes $(F_0,G_0)$. The condition that $(\wt{F}_\tau,\wt{G}_\tau)$ is a \emph{small perturbation} means that
  we insist that $\wt{F}_\tau\circ\wt{G}_\tau$ and $F_\tau\circ G_\tau$ be left-homotopic. From this it follows that $\wt{F}_\tau\circ\wt{G}_\tau$ and $f_\tau$ are left-homotopic. Let $\wt{h}_{\sigma,\tau}$ be such left-homotopy. That is to say, $\wt{h}_{0,\tau}=\wt{F}_\tau\circ\wt{G}_\tau$,
  while $\wt{h}_{1,\tau}=f_\tau$. Denote $g_\tau=\wt{h}_{\tau,1}$;
  see Figure~\ref{fig:concat_lift}.
  \begin{figure}
    \input{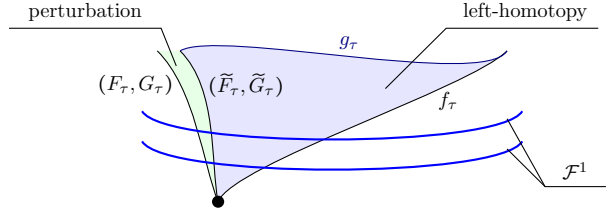}
    \caption{Schematic of the proof of Corollary~\ref{cor:left_homotopy}.}\label{fig:concat_lift}
  \end{figure}
  %\npar{Do you want to hyphenate left homotopy in the label in the figure?}{I'm not quite sure what you mean. I made the lines a bit longer. The length is governed by the parameters $(-2,0)$ and $(2.0,0)$ in line 13 and 17 of the source.} \mpar{In Figure~\ref{fig:concat_lift} it says left homotopy and I thought you might want it to say left-homotopy for consistency. I have done this now. }
  Since $(\wt{F}_\tau,\wt{G}_\tau)$ is a regular double path,
  $\wt{G}_1$ is a generic immersion and $\wt{F}_1$ is an immersed Morse function with respect to $\wt{G}_1$. Lemma~\ref{lem:lift_morse}
  allows us to find a lift of $g_\tau$ starting from $(\wt{F}_1,\wt{G}_1)$. We denote this lift by $(\wt{F}_{1+\tau},\wt{G}_{1+\tau})$, $\tau\in[0,1]$. By construction, $\wt{G}_{1+\tau}$ is a generic immersion for all $\tau$, and $\wt{F}_{1+\tau}$ is an immersed Morse function. That is,
  $(\wt{F}_{1+\tau},\wt{G}_{1+\tau})$ is a regular double path. We concatenate now the paths $(\wt{F}_\tau,\wt{G}_\tau)$, $\tau\in[0,1]$
  and $(\wt{F}_{1+\tau},\wt{G}_{1+\tau})$, to obtain a path $(\wt{F}_{\tau},\wt{G}_\tau)$, $\tau\in[0,2]$.
  The path is a regular double path (as a concatenation of two regular double paths), with $\wt{F}_2\circ\wt{G}_2=g_1=f_1$.

  After rescaling the time, we obtain a path $(F'_\tau,G'_\tau)$, with $F'_\tau=\wt{F}_{2\tau}$, $G'_\tau=\wt{G}_{2\tau}$. We claim that $f'_\tau:=F'_\tau\circ G'_\tau$ is $\cF^1$-homotopic to $f_\tau$. The homotopy is quite easy to construct explicitly. We define it as a concatenation
  of the homotopy $\wt{h}_{\sigma,\tau}$ and the path $g_\sigma$.
  \[h_{\sigma,\tau}=\begin{cases}
      \wt{h}_{\sigma,2\tau} & \tau\le \frac12\\
      g_{\sigma + (2\tau-1)(1-\sigma)} & \tau\ge\frac12.
    \end{cases}
  \]
  It is routine to check that $\wt{h}_{\sigma,\tau}$ being a left-homotopy implies that $h_{\sigma,\tau}$ is an $\cF^1$-homotopy.
\end{proof}
Motivated by Corollary~\ref{cor:left_homotopy} we introduce a piece of terminology.
\begin{definition}\label{defn:completed}
  A path $(F_\tau,G_\tau)$, $\tau\in[0,\ell]$ can be \emph{promoted to a weak lift of $f_\tau$}, if (up to a possible perturbation
  of $(F_\tau,G_\tau)$ rel.\ starting point, there exists a regular double path $(F_{\ell+\tau},G_{\ell+\tau})$, $\tau\in[0,1]$
  such that the whole path $(F_{(\ell+1)\tau},G_{(\ell+1)\tau})$ is a weak lift of $f_\tau$.
\end{definition}

Corollary~\ref{cor:left_homotopy} admits the following generalisation.

\begin{lemma}\label{lem:lax_to_lift}
  Suppose $f_\tau$, $\tau\in[0,1]$ is an $\cF^1$ path, while $(F_\tau,G_\tau)$, for $\tau\in[0,\ell]$ and $\ell>1$, is a regular double path such that:
  \begin{itemize}
    \item $g_\sigma:=F_\sigma\circ G_\sigma$, $\sigma\in[0,\ell-1]$ is an $\cF^0$-path;
    \item $\wt{f}_\tau:=F_{\ell-1+\tau}\circ G_{\ell-1+\tau}$ is lax homotopic to $f_\tau$ over $g_\tau$.
  \end{itemize}
  Then there exists a path $(F_{\ell+\tau},G_{\ell+\tau})$, $\tau\in[0,1]$, such that $f_\tau$ and $F_{(\ell+1)\tau}\circ G_{(\ell+1)\tau}$ are homotopic rel.\ endpoints as $\cF^1$-paths. That is, $(F_\tau,G_\tau)$ can be promoted to a weak lift of $f_\tau$.
\end{lemma}
\begin{proof}
  For simplicity, assume that $\ell=2$. Let $h_{\sigma,\tau}$ be a lax homotopy over $g_\tau$ between $f_\tau$ and $\wt{f}_\tau$. Set $\wt{g}_\sigma=h_{1-\sigma,1}$. This is an $\cF^0$-path, moreover, $\wt{g}_0=\wt{f}_1=F_2\circ G_2$ and $\wt{g}_1=f_1$. Lift this path using Lemma~\ref{lem:lift_morse} to a regular double path $(F_{2+\tau},G_{2+\tau})$ with $F_{2+\tau} \circ G_{2+\tau}= \wt{f}_{\tau}$.  We claim that $F_{3\tau}\circ G_{3\tau}$ is homotopic to $f_\tau$.
  This is routine. We set
  \[\wt{h}_{\sigma,\tau}=\begin{cases}
      g_{3\sigma\cdot \tau} & \tau\in[0,1/3]\\
      h_{\sigma,3\tau-1} & \tau\in[1/3,2/3]\\
      \wt{g}_{3\sigma\tau-3\tau+1} &\tau\in[2/3,1].
    \end{cases}
  \]
  With this choice $h_{0,\tau}$ is a reparametrisation of $f_\tau$, while $h_{1,\tau}$ is a concatenation of $g_\tau$, $\wt{f}_\tau$ and $\wt{g}_\tau$.  That is, $h_{1,\tau}=F_{3\tau}\circ G_{3\tau}$ up to reparametrisation.
\end{proof}

\section{Lifting paths of birth}\label{sec:lifting-paths-of-birth}
Throughout Section~\ref{sec:lifting-paths-of-birth}, we assume that $N$ and $\Omega$ are closed. The proof can be easily adapted to the case that
 $N$ and $\O$ have nonempty boundary, provided that $f_{\tau}$ is a neat path, $F_0$ is neat,
and $G_0$ is neat.
The goal of this section is to prove the following result.

\begin{lemma}[Lifting paths of birth]\label{lem:lift_birth}
  Suppose $f_{\tau} \colon N \to \R$ is a path of birth of index $h$.
  Let $G_0\colon N\to \O$ be a generic immersion and let $F_0\colon\O\to\R$ be an immersed Morse function such that
  $F_0\circ G_0=f_0$.

  Then there exists a path $F_{\tau} \colon \O \to \R$ and
  a path $G_\tau$ of generic immersions, such that:
  \begin{itemize}
    \item the path $F_{\tau}$ is excellent Morse except for $\tau=1/2$;
    \item $F_\tau$ has a birth at the first stratum for $\tau=1/2$;
    \item $(F_\tau,G_\tau)$ weakly lifts the path $f_\tau$;
    \item $G_\tau$ is as close to $G_0$ as we please $($for positive distance$)$.
  \end{itemize}
\end{lemma}

\begin{proof}
  The idea of the proof is to perturb the map $G$, if needed, so that the point at which the birth occurs is on the first stratum.
  This constructs the path $G_\tau$. The next part creates the birth for the immersed Morse function $F$. Finally, we invoke Uniqueness of Birth
  Proposition~\ref{prop:unicite_naiss} to show that the path we construct indeed lifts the original path. Except for Lemma~\ref{lem:birth_no_crit} below, the result is straightforward. The details follow.

  Without loss of generality we may assume that the birth occurs at $\tau=1/2$. Let $q\in N$ be the point at which birth occurs and let $c=f_{1/2}(q)$ be the level set of $q$. Suppose $\delta>0$ is such that,
  for $\tau\in[1/2-\delta,1/2)$,  no critical point of $f_\tau$ has critical value $c$. Such a $\delta$ exists, because $f_{1/2}$ has $q$ as the only critical point at the level set $c$. The path $f_{\tau}$ for $\tau\in[0,1/2-\delta]$ is a path of excellent Morse functions, which we can lift by Lemma~\ref{lem:lift_morse}. It is enough to lift the restricted
  path $f_\tau$ for $\tau\ge 1/2-\delta$. The interval $[1/2-\delta,1/2]$ can be stretched to $[0,1/2]$, so the path
  $f_\tau$ for $\tau\in[1/2-\delta,1]$ can be reparametrised to a path on $[0,1]$ agreeing with the old path on $[1/2,1]$.
  Therefore, without losing generality, we can assume that $f_\tau$ has no critical points at the level set $c$ for $\tau<1/2$.

  The subset of $C^\infty(N,\Omega)$ consisting of maps that take $q$ to the first stratum is clearly open-dense. Hence, we perturb $G_0$ in such a way that $q$ is mapped to the first stratum.
  This perturbation can be done via a family of generic immersions. That is, let
  $G_\tau$, $\tau\in[0,1/4]$, be a path of generic immersions, supported near $p$ and
  such that $p:=G_{1/4}(q)$ belongs to the first stratum. By Lemma~\ref{lem:cR_path}, we may and will assume that $G_\tau$ is $F_0$-regular (see Definition~\ref{def:another_F_word}).
  Extend the path $G_\tau$ by setting $G_\tau=G_{1/4}$ for $\tau>1/4$. %Set $F_\tau=F_0$ for $\tau\in[0,1/2]$.

  Choose a neighbourhood $U$ of $p$ with coordinates $x_1,\dots,x_n,y_1,\dots,y_k$ such that $G_{1/4}(N)\cap U=\{y_1=\dots=y_k=0\}$. Note that $F_{0}$ restricted to $G_{1/4}(N)$ does not have a critical point at~$p$. Therefore there is an index $i=1,\dots,n$, such that $\frac{\partial F_{0}}{\partial x_i}(p)\neq 0$. Without loss of generality we assume that $i=n$.

  Consider the map
  \[\psi(x_1,\dots,x_n,y_1,\dots,y_k)=(x_1,\dots,x_{n-1},F_{0}(x_1,\dots,x_n,y_1,\dots,y_k),y_1,\dots,y_k).\]
  This map is a local diffeomorphism since the derivative is invertible. By composing the coordinate chart with the local inverse map, and shrinking $U$ if necessary to where the inverse is defined,
  we obtain new coordinates on $U$. By a mild abuse of notation we still denote these coordinates by $(x_1,\dots,y_k)$, but we have that $F_{0}\equiv x_n$.

  Set $F_\tau=F_0$ for $\tau\le 1/4$, so that for the time $\tau\in[0,1/4]$ we modify (perturb) $G$ and for $\tau>1/4$ we modify $F$.
  Choose a bump function $\phi$ supported on $U$, with $\phi\equiv 1$
  in a smaller neighbourhood $U'$ of $p$.
  For $\tau\in[1/4,3/4]$, define
  \[F_\tau=F_{0}+\birth_{2\tau-1/2}(x_1,\dots,x_n)-x_n+(\tau-1/4)\varepsilon\phi y_1,\]
  where $\birth_\tau$ is the elementary path of birth of index $h$
  as in Definition~\ref{def:el_path_birth} such that $\birth_{\tau}=\birth_0$ away from $U'$. The r\^ole of the last
  term in the definition of $F_\tau$ is to make sure that the function $F_\tau$ has no new critical points when regarded as a function on
  $\Omega$. The real number $\varepsilon>0$ is sufficiently
  small and will be determined shortly.

  \begin{lemma}\label{lem:birth_no_crit}
    For sufficiently small $\varepsilon>0$ and $\delta>0$, the function $F_\tau$ for $\tau\in[0,1/2+\delta]$ has no critical
    points on the zeroth stratum in $U'$.
  \end{lemma}

  \begin{proof}
    Let $U''\subseteq U'$ be an open subset containing $p$ such that $\phi\equiv 1$ on $U''$. We choose $\delta>0$ by the condition
    that all the critical points of the function $\birth_{2\tau-1/2}(x_1,\dots,x_n)$ belong to $U''$ for $1/4 \leq \tau\le 1/2+\delta$. Such
    $\delta$ exists, because $\birth_{2\tau-1/2}$ creates a critical point at $p \in U''$ for $\tau=1/2$ and the critical points $p_-,p_+$
    of $\birth_\tau$ are at points $(0,\dots,0,\pm \alpha(\tau))$, where $\alpha(\tau)$ depends smoothly on $\tau$.

    This means that for $1/4 \le \tau\le 1/2+\delta$,
    $\birth_{2\tau-1/2}$ has no critical points in $\ol{U'}\setminus U''$. Choose a Riemannian metric. Then  there is a constant $c>0$
    such that $\|D\birth_\tau(z)\|>c$ for all $z\in\ol{U'}\setminus U''$. Set $C$ to be the upper bound on $\ol{U'}\setminus U''$
    for $\|D\phi y_1\|$. If $\varepsilon<\frac{c}{2C}$, then
    \[\|DF_\tau\|> c-\smfrac{3}{4} \varepsilon C > c - \smfrac{3cC}{8C} = \smfrac{5}{8} c > \smfrac12c\]
     on $\ol{U'}\setminus U''$, so the derivative is nonzero there. Next, on $U''$,
    the derivative with respect to $y_1$ is positive, so there are no critical points on $U''$. Thus there are no critical points on the zeroth stratum of $U'$, as desired.
  \end{proof}

  \begin{corollary}
    The function $F_\tau$ is immersed Morse for $\tau\in[0,1/2+\delta]\setminus\{1/2\}$. For $\tau>1/2$ it has two more critical
    points on the first stratum than for $\tau <1/2$.
  \end{corollary}

We continue with the proof of Lemma~\ref{lem:lift_birth}.
  Reparametrise $F_\tau$ to $F_{\vartheta(\tau)}$ where the diffeomorphism $\vartheta\colon[0,1/2+\delta]\to[0,1]$ has positive derivative
  and maps $1/4$ to $1/4$ and $1/2$ to $1/2$. The reparametrized path still acquires the new critical points at $\tau=1/2$, but the path is extended up to $[0,1]$ (and not
  on $[0,1/2+\delta]$).

  Define $\wt{f}_\tau=F_{\vartheta(\tau)}\circ G_\tau$. Note that $\wt{f}_\tau$, for $\tau\in[0,1/4]$, is arbitrary close to $f_0$, in particular~$\wt{f}_\tau$
  belongs to the same stratum of $\cF^0$ as $f_0$. Next, $\wt{f}_\tau$ for $\tau\in[1/4,1]$ acquires a critical point at $\tau=1/2$, that is,
  it crosses the stratum $\cF^1$ at $\tau=1/2$. The crossing is transverse, the path $\birth_\tau$ crosses the stratum transversely and
  the reparametrisation $\vartheta$ we used has non-vanishing derivative at $\tau=1/2$. In particular, $\tau\mapsto\wt{f}_\tau$ is a path of birth.

  Now the birth on $\wt{f}_\tau$ occurs at the same point of $N$ as the birth for $f_\tau$. Moreover, the indices of the critical points created
  by $\wt{f}_\tau$ and $f_\tau$ are the same. We invoke Proposition~\ref{prop:unicite_naiss} to see that $\wt{f}_\tau$ and $f_\tau$ are
  left-homotopic. By Corollary~\ref{cor:left_homotopy}, the double path $(F_{\vartheta(\tau)},G_\tau)$ can be promoted to a weak lift of $f_\tau$.
  %\ypar{I added a ready-to-use Corollary~\ref{cor:left_homotopy}.}
%
%  is lax homotopic to $f_\bullet$, and hence by Lemma~\ref{lemma:lax-homotopic paths-same-endpoints-are-F1-homotopic} is $\cF^1$-homotopic. %\npar{MP: I thought this was still a bit unclear so I tried to rewrite it. To get from the left homotopy to the lax homotopy after adding the path going back to the correct endpoint, I think it's easiest if we use the path given by the left homotopy, and not some random one. Also I added the lemma cited earlier, because we need this fact a few times. }{It is much better now}
%  Using Lemma~\ref{lem:lift_morse}, we may lift $f'_\bullet$ to a path $(F'_\tau,G'_\tau)$, where $F'_0=F_1$ and $G'_0=G_1$. In particular,
%  $F'_1\circ G'_1=f_1$.
%The concatenation of $(F_\bullet,G_\bullet)$ and $(F'_\bullet,G'_\bullet)$ is a weak lift of $f_\tau$.}
\end{proof}

\section{Lifting paths of rearrangement}\label{sec:lifting-paths-of-rearangment}
Next, we lift paths of rearrangement. Throughout Section~\ref{sec:lifting-paths-of-rearangment} we
 will assume that $k=\dim\O-\dim N>1$.
Remark~\ref{rem:codim_one_lift_obs} explains the obstruction to lifting if $k=1$, so the assumption is necessary.

\begin{lemma}[Lifting paths of rearrangement]\label{lem:lift_rearr}
  Let $f_0\colon N\to\R$ be a Morse function and let $\eta$ be a gradient-like vector field for $f_0$. Assume $\eta$ is Morse--Smale.
  Let $q_-,q_+$ be two critical points of $f_0$ with $\ind q_-\ge\ind q_+$ and $f_0(q_-)=a<b=f_0(q_+)$. Suppose
  there are no other critical points of $f_0$ in $f_0^{-1}[a,b]$.
  Assume that there exists a generic immersion $G_0\colon N\to \O$ and an immersed excellent Morse function $F_0\colon\O\to\R$
  such that $f_0=F_0\circ G_0$.

  If $k>1$, then an elementary $\eta$-path $f_\tau$, $\tau\in[0,1]$, moving the critical point $q_-$ above the critical point $q_+$,
  can be weakly lifted to a regular double path $(F_\tau,G_\tau)$ such that $F_\tau$ is Morse on the zeroth and first stratum for all $\tau$, and $F_\tau \circ G_\tau$ is an $\eta$-path $($Definition~\ref{def:eta-path}$)$.

  The path $F_\tau$ is supported on $F_0^{-1}[a-\varepsilon,b+\varepsilon]$, while $G_\tau$ is supported on $f_0^{-1}[a-\varepsilon,b+\varepsilon]$, where $\varepsilon>0$ can be chosen as small as we please. Moreover, if $f_\tau$ is neat, $G_0$ is very neat, and $F_0$ is neat,
  then $(F_\tau,G_\tau)$ is neat.
\end{lemma}

We recall that the condition that $f_\tau$ is elementary means
that $\tau\mapsto f_\tau(q_-)-f_\tau(q_+)$ has precisely one zero in the interval $[0,1]$ and this is a simple zero; compare Definition~\ref{def:elementary}. This amounts to saying that $f_\tau$ intersects $\cF^1_\beta$ transversely at a single point.

In general, not all maps $G_\tau$ will be generic immersions, in particular, some new self-intersections of the components of $N$ can be created.
In the follow-up to Lemma~\ref{lem:lift_rearr}, we control the double points created by $G_\tau$.

\begin{addendum}\label{lem:lift_rearr_immersion}
  The path $G_\tau$ of Lemma~\ref{lem:lift_rearr} can be constructed to have the following properties.
  \begin{enumerate}[label=(IR-\arabic*)]
    \item if $G_0$ is an embedding, then $G_\tau$ is an embedding for all $\tau$;\label{item:ir_embedding}
    %\item if $k\ge 3$, then $G_\tau$ is a generic immersion for all $\tau$, in particular, no new critical points are created;\label{item:ir_large_dimension}
    \item if $N=N_1\sqcup\dots\sqcup N_\ell$ and the images $N_i$ and $N_j$ $($for all $i\neq j)$ under $G_0$ are disjoint, then $G_\tau(N_i)\cap G_\tau(N_j)=\emptyset$ for all $\tau$ whenever $i\neq j$.\label{item:ir_disjoint_image}
    \item if $k=2$ and $n\le 3$, then $G_\tau$ is a generic immersion for all $\tau$. \label{item:ir3}
  \end{enumerate}
\end{addendum}

The proof of Lemma~\ref{lem:lift_rearr} takes the remainder of Section~\ref{sec:lifting-paths-of-rearangment}.
Addendum~\ref{lem:lift_rearr_immersion} is proved alongside Lemma~\ref{lem:lift_rearr}.
We first introduce the notation, and list two extra conditions \ref{item:LR1} and~\ref{item:LR2}. Then we prove the lemma assuming~\ref{item:LR2}.
Next, we arrange the function $F_0$ and the embedding $G_0$
in such a way that \ref{item:LR1} holds. Then  we show how to improve a pair $(F_0,G_0)$ satisfying \ref{item:LR1}
so that \ref{item:LR2} is satisfied. A summary of the entire proof is given in Subsection~\ref{sub:summary_rearr}.

\subsection{Notation of the proof of Lemma~\ref{lem:lift_rearr}}

Let $h_-:= \ind q_-$, $h_+:=\ind q_+$. Set $p_-:=G_0(q_-)$ and $p_+:=G_0(q_+)$.
Then $p_-,p_+$ are critical points of $F_0$ of indices $h_-,h_+$, respectively. Both of $p_-$ and $p_+$ are at
depth $1$. In particular, there are neighbourhoods of $q_-$ and $q_+$, such that $G_0$ maps these neighbourhoods
to the first stratum of $M=G_0(N)$. For $\varepsilon>0$ sufficiently small,
$G_0$ maps
$\cK_{\eta,a-\varepsilon,a+\varepsilon}(q_-)$ and $\cK_{\eta,b-\varepsilon,b+\varepsilon}(q_+)$ (refer to Definition~\ref{def:relation} for
the construction of $\cK$)
to the first stratum. Shrink $\varepsilon$ if necessary to ensure that $F_0$ has no critical points
with critical values in $[a-\varepsilon,a+\varepsilon]\cup[b-\varepsilon,b+\varepsilon]$ other
than $p_-$ and $p_+$.

Denote $\cK_- := \cK_{\eta,a,b+\varepsilon}(q_-)$ and $\cK_+ := \cK_{\eta,a-\varepsilon,b}(q_+)$. With this notation, $\cK_-$ is
the unstable manifold of $q_-$ and $\cK_+$ is the stable manifold of $q_+$. The dimensions are $\dim\cK_-=n-h_-$, $\dim\cK_+=h_+$.
By the Morse--Smale condition on $\eta$, and since
$h_-\ge h_+$, it follows that $\cK_-$ and $\cK_+$ are disjoint.
Set \begin{equation}\label{defn-mathcal-L}
      \cL_-=G_0(\cK_-) \text{ and }\cL_+=G_0(\cK_+).
    \end{equation}

\begin{figure}
  \input{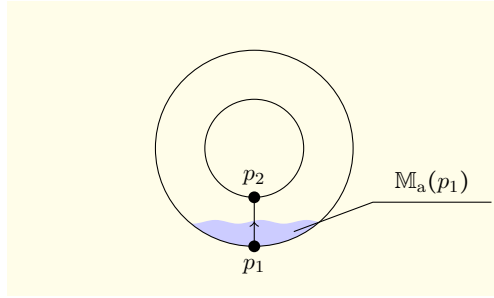}
  \caption{Lifting paths of rearrangements. Here $M$ is a union of two circles in $\R^2$. Two critical points $p_1$ and $p_2$
  are connected by a trajectory staying in the zeroth stratum. The point $p_2$ cannot be moved below $p_1$.}\label{fig:codim_one_obs}
\end{figure}

\begin{remark}\label{rem:codim_one_lift_obs}
  If $k=1$ and $\ind p_+=\ind p_-$, there can be a trajectory of $\xi$ from $p_-$ to $p_+$ in the zeroth stratum, even
  though there are no trajectories in the first. Such a zero-stratum trajectory is an obstruction to lifting rearrangements in codimension one.  Hence our assumption that $k \geq 2$; compare Figure~\ref{fig:codim_one_obs}.
\end{remark}

\subsection{Proof of Lemma~\ref{lem:lift_rearr} under extra conditions}\label{sub:extra_cond}

We first list extra conditions
%\npar{MB. Technically, we use LR2 in the proof, and LR1 is needed to guarantee LR2. Furthermore, we destroy LR1 while fighting for LR2, so we need to rephrase the structure of the proof. MP: Ok, I'll let you write this into the file.  Is this done now?}{Yes, it's done}
which simplify the proof of Lemma~\ref{lem:lift_rearr}.
\begin{enumerate}[label=(LR-\arabic*)]
  \item\label{item:LR1}
    The points $p_-=G_0(q_-)$ and $p_+=G_0(q_+)$ are the only critical points of $F_0$ at depth $0$ or $1$ in $F_0^{-1}[a-\varepsilon,b+\varepsilon]$;
  \item\label{item:LR2} Either $\cL_-$ or $\cL_+$ belong to the first stratum of $G_0(N)$, i.e.\ at least one of $\cL_-$ or $\cL_+$ misses the singular image of $G_0$.
\end{enumerate}
Condition~\ref{item:LR2} will be used to construct a lift, while Condition~\ref{item:LR1} is used to guarantee Condition~\ref{item:LR2}.
\begin{lemma}\label{lem:extra_cond}
  Suppose $(F_0,G_0)$ satisfies \ref{item:LR2}.
  Then there exists a path $(F_\tau,G_\tau)$ with $G_\tau=G_0$ for all $\tau \in [0,1]$
  that weakly lifts $f_\tau$.
\end{lemma}

\begin{proof}
  Suppose $\cL_-$ belongs to the first stratum. As the first stratum is open, it follows that there exists an open set $V\subseteq N$
  containing $\cK_-$ such that $G_0(V)$ belongs to the first stratum. Choose a grim vector field $\xi$ on $\Omega$ for $F_0$,
  which is equal to $DG_0(\eta)$ in a neighbourhood of~$\cL_-$ in $M=G_0(N)$; compare Proposition~\ref{prop:grim_extend}.
  Assume that $\xi$ is immersed Morse--Smale. Let $\wt{\eta}$ be the pull-back of $\xi$.

  There might be critical points of $F_0$ in $F_0^{-1}[a,b]$ other than $p_-$ and $p_+$. %\ypar{Added the procedure, it was missing}
  If that is the case, take a critical point $p\in F_0^{-1}[a,b]$ distinct from $p_-$ and $p_+$. Assume it has $h_p+d_p\le h_-+1$ and it has the smallest value of $F_0$ among critical points with this property. The Morse--Smale condition implies that there are no broken trajectories in $F_0^{-1}[a,F_0(p)]$
  connecting $p_-$ with $p$, so we rearrange using $\xi$, applying Theorem~\ref{thm:grim_rearrangement}, so that the critical value of $p$ is smaller than $a$. Using this procedure,
  we move all critical points with $h_p+d_p\le h_-+1$ below the level set $a$. We denote by $(F_\sigma,G_\sigma)$, $\sigma\in[0,1]$,
  the corresponding $\xi$-path,
  noting that it is supported away from $\cL_-$ and away from all critical points of $F_0\circ G_0$. By construction, $G_\sigma=G_0$. On shrinking $V$ if needed, we may and will assume that the path is supported away from $G_0(V)$.

  We claim that for $F_1$, $\cK_{\xi,a,b}(p_-)$ and $\cK_{\xi,a,b}(p_+)$ are disjoint. Clearly,
  By Corollary~\ref{cor:highcodim}, as $h_{p_+}\le h_{p_-}$,
  implies that there are no broken trajectories from $p_-$ to $p_+$. But also, if there exists a critical point $p$ in $F_1^{-1}[a,b]$
  and a trajectory from $p$ to $p_+$, then $h_p+d_p\le h_{p_+}+1$, also by Corollary~\ref{cor:highcodim}. But $h_{p_+}+1\le h_{p_-}+1$
  and all critical points with $h_p+d_p\le h_{p_-}+1$ have been moved away from $F_1^{-1}[a,b]$. Thus, no critical point is connected
  with $p_+$, so indeed, $\cK_{\xi,a,b}(p_-)\cap \cK_{\xi,a,b}(p_+)=\emptyset$.

  We use the Rearrangement Theorem~\ref{thm:grim_rearrangement} to find
  an elementary $\xi$-path $F_{1+\tau}$, $\tau\in[0,1]$ of Morse functions that lifts $p_-$ above $p_+$. This path can be chosen to be
  supported on an open set $U\subseteq\Omega$ such that $U\cap M\subseteq G_0(V)$. In particular, $\wt{f}_\tau:=F_{1+\tau}\circ G_{1+\tau}$, $\tau\in[0,1]$
  can be supported on $V$.

  The path $\sigma\mapsto g_\sigma:=F_\sigma\circ G_\sigma$, $\sigma\in[0,1]$ is a path of excellent Morse functions.
  We are in position to use Lemma~\ref{lem:cut_and_paste}. As an outcome, $\wt{f}_\tau$ and $f_\tau$, $\tau\in[0,1]$, are lax homotopic
  over $h_{\sigma,0}$.
  This is precisely the situation of Lemma~\ref{lem:lax_to_lift} with $\ell=2$. That is, we can promote $(F_\tau,G_\tau)$
  to a weak lift of $f_\tau$.

  If $\cL_+$ belongs to the first stratum, the proof is analogous. The vector field $\xi$ is assumed  to extend $DG_0(\eta)$ from an open subset
  of $M$ containing $G_0(\cK_+)$ and we construct a $\xi$-path that moves the critical point $p_+$ below $p_-$.
  This completes the  proof of Lemma~\ref{lem:lift_rearr}.
\end{proof}

Suppose $G_0$ satisfies~\ref{item:LR2}. The lift $(F_\tau,G_\tau)$ we have constructed
has the following properties:
\begin{itemize}
  \item $G_\tau$ is a generic immersion for all $\tau$;
  \item the image $G_\tau(N)$ does not change; in fact, the only moment we change $G_\tau$ is when we complete the path $\wt{f}_\tau$
    to a weak lift. Corollary~\ref{cor:left_homotopy} constructs this completion in such a way that $G_\tau(N)$ changes by reparametrisation of $N$;
  \item the function $F_\tau$ is immersed Morse for all $\tau$. The vector field $\xi$ used in the construction of the rearrangement path
    is grim for all $F_\tau$; compare Remark~\ref{rem:lift_morse}.
\end{itemize}

\subsection{Enforcing condition~\ref{item:LR1}}\label{sub:LR1}

We return to assuming the hypotheses of Lemma~\ref{lem:lift_rearr}, and we arrange for \ref{item:LR1} to hold.

\begin{lemma}\label{lem:moving}
  There exists a path $F_\tau\colon\O\to\R$ of immersed Morse functions such that:
  \begin{itemize}
    \item $F_\tau(z)=F_0(z)$ if $F_0(z)<a-2\varepsilon$ or $F_0(z)>b+2\varepsilon$;
    \item $F_\tau=F_0$ in a neighbourhood of $\cL_-$ and $\cL_+$;
    \item the only critical points of $F_1$ with critical values in $[a-\varepsilon,b+\varepsilon]$ are $p_-$ and $p_+$, except if $k=2$ and $h_-=h_+$, in which case there is also potentially a finite number of critical points at depth $d>1$, for which the sum of the index and the depth is equal to $h_-+1$;
    \item Suppose that $N=N_1\sqcup\cdots \sqcup N_\ell$ and the images $G_0(N_1),\dots,G_0(N_\ell)$ are pairwise disjoint.
    If $p_-$ and $p_+$ belong to different connected components, then all the other critical points can be moved so that their critical values do not lie in $[a-\varepsilon,b+\varepsilon]$.
    If $p_+$ and $p_-$ belong to the same component of $G_0(N)$, then the only critical
      points from the previous item that cannot in general be moved away from $[a-\varepsilon,b+\varepsilon]$
      are those that lie on the same component of $G_0(N)$ as $p_-$ and $p_+$;
    \item The composition $F_\tau\circ G_0$ is a path of excellent Morse functions on $N$.
  \end{itemize}
\end{lemma}

Note that this lemma does indeed arrange for \ref{item:LR1} to hold. In the case $k=2$ and $h_- = h_+$, the critical points that can occur with critical values in $[a-\varepsilon,b+\varepsilon]$ have depth at least 2, and so are not relevant to \ref{item:LR1}.

\begin{proof}[Proof of Lemma~\ref{lem:moving}]
Note that  $\cL_-$ and $\cL_+$ (defined in~\eqref{defn-mathcal-L}) are stratified manifolds of dimensions $n-h_-$ and $h_+$, respectively. Assuming that $G_0$ is generic, $\cL_-\cap \cL_+=\emptyset$.
  Choose a grim vector field $\xi$ for $F_0$ which is Morse--Smale and also satisfies the following extra condition.
  \begin{cond}
    The membranes of all the critical points of $\xi$ in $F_0^{-1}[a,b]$ different from $p_-$ and $p_+$ are transverse to $\cL_-$ and $\cL_+$.
  \end{cond}
  This can be achieved in the same way as the Morse--Smale condition; see Subsection~\ref{sec:morse-smale-condn}.
  Using this condition we prove the following result.
  \begin{lemma}\label{lem:dim_asc}
    Let $p\neq p_-,p_+$ be a critical point of $F_0$ in $F_0^{-1}[a-\varepsilon,b+\varepsilon]$ of index $h_p$ and depth $d_p>0$. If $h_p+d_p< h_-+1$, then the descending membrane of $p$
    is disjoint from $\cL_-$. If $h_p+(k-1)(d_p-1)> h_+$, then the ascending membrane is disjoint from $\cL_+$.
  \end{lemma}

  \begin{remark*}
If $d_p=0$, the membranes of $p$ belong to the zeroth stratum, so they are automatically disjoint from $\cL_-$ and $\cL_+$.
  \end{remark*}

  \begin{proof}
    By Lemma~\ref{lem:dimension},
    $\dim \Hd(p)\cap \O[j]=h_p+d_p-j$ and $\dim \Ha(p)\cap \O[j]=n-k(d_p-1)-h_p+d_p-j$.
    For each $j >0$, ($\cL_- \cap \O[0] = \emptyset$ by construction), we have
    $\dim \cL_-\cap \O[j]=n-h_--k(j-1)$. Then
    $\Hd(p)\cap \cL_-\cap \O[j]$ is empty if:
    \begin{equation}\label{eq:empty1}h_p+d_p-j+n-h_--k(j-1)<n-k(j-1) = \dim \O[j].\end{equation}
    Transforming this formula yields
    \[h_p + d_p <h_-+j.\]
    This holds for all $j>0$, since we assume that
    $h_p+d_p<h_-+1$.
    Thus $\Hd(p) \cap \cL_- \cap \O[j] = \emptyset$ for all $j \geq 0$.
    Hence, since $M = \cup_{j \geq 0} \O[j]$,  we have that $\Hd(p)\cap \cL_- = \emptyset$, as required.

    The argument for the $\cL_+$ is analogous. Again, $\cL_+ \cap \O[0] = \emptyset$. As $\dim \cL_+\cap \O[j]=h_+-k(j-1)$,
    $\Ha(p)\cap \cL_+\cap \O[j]=\emptyset$ if
    \begin{equation}\label{eq:empty2}n-k(d_p-1)-h_p+d_p-j+h_+-k(j-1)< n-k(j-1).\end{equation}
    This translates into $h_+<h_p+k(d_p-1)-(d_p-j)$. This holds for all $j>0$
    if $h_+<h_p+(k-1)(d_p-1)$. Thus under this condition  $\Ha(p)\cap \cL_+\cap \O[j]=\emptyset$ for all $j \geq 0$, and hence $\Ha(p)\cap \cL_+=\emptyset$.
  \end{proof}

  \begin{remark}\label{rem:strict_explain}
    The intersection of $\cL_-$ with the descending membrane $\Hd(p)$ need not consist of trajectories of $\xi$, so it can be nonempty even if the expected
    dimension is zero. Therefore, we get strict inequalities in Lemma~\ref{lem:dim_asc}.
  \end{remark}
  Continuing the proof of Lemma~\ref{lem:moving}, we introduce a useful piece of terminology.
  \begin{definition}\label{def:safe_rearrangement}
    A $\xi$-path of rearrangement $F_\tau$ is \emph{safe} if the support of $F_\tau$ is disjoint from~$\cL_-\cup\cL_+$.
  \end{definition}

  Continuing the proof of Lemma~\ref{lem:moving}, we state and prove the following result.
  \begin{lemma}\label{lem:moving_step_1}
    There exists an $F_\tau$ path of rearrangements that is a concatenation of safe $\xi$-path of rearrangements, supported on $F_0^{-1}[a-\varepsilon,b+\varepsilon]$ such that for any two pairs of critical points of $F_\tau$, $p,p'$ with $p,p'\in F_0^{-1}[a+\varepsilon,b-\varepsilon]$,
    we have $h_p+d_p\le h_{p'}+d_{p'}$ whenever $F_1(p)<F_1(p')$.

    Furthermore, the only critical points between $p_-$ and $p_+$ necessarily have $h_-+d_-\le h_p+d_p\le h_++d_+$.
  \end{lemma}
  \begin{remark*}
    If $h_-+d_->h_++d_+$, then the statement of the lemma means that after the rearrangements are made,
    there are no critical points between $p_-$ and $p_+$
  \end{remark*}
  \begin{proof}
    We present an algorithm for performing the necessary rearrangements. We take two consecutive critical points $p,p'$, that is points such that $F_0(p)<F_0(p')$ and such that there are no critical points between $p$ and $p'$.
    We claim that at least one of the situations occur.
    \begin{enumerate}[label=(SR-\arabic*)]
      \item $h_p+d_p \le h_{p'}+d_{p'}$;\label{item:SR1}
      \item $h_p+d_p>h_{p'}+d_{p'}$ and at least one of $d_p,d_{p'}$ is zero;\label{item:SR2}
      \item $h_p+d_p>h_{p'}+d_{p'}$ and $h_{p'}+d_{p'}<h_-+1$; \label{item:SR3}
      \item $h_p+d_p>h_{p'}+d_{p'}$ and $h_p+(k-1)(d_p-1)>h_+$.\label{item:SR4}
    \end{enumerate}
    The claim is proved by contradiction. If none of the four cases occurs, we have (i) $h_p+d_p>h_{p'}+d_{p'}$, (ii) $h_{p'}+d_{p'}\ge h_-+1$, and (iii) $h_p+(k-1)(d_p-1)\le h_+$.  We also have that (iv) $h_- \geq h_+$,  by the assumptions of Lemma~\ref{lem:lift_rearr}, and $d_p >0$, $d_{p'} >0$ by \ref{item:SR2}. Combining these yields:
  \[h_p + d_p = h_p + (d_p-1) +1 \stackrel{(k \geq 2)}{\leq} h_p + (k-1)(d_p-1) + 1 \stackrel{(iii)}{\leq} h_+ +1 \stackrel{(iv)}{\leq} h_- +1 \stackrel{(ii)}{\leq} h_{p'} + d_{p'} \stackrel{(i)}{<} h_p + d_p.\]
  The first inequality used that $d_p \geq 1$.
  Since this is absurd, we deduce that the above list is indeed exhaustive, as asserted.

  We describe now a local algorithm deciding on a case-by-case basis which type of rearrangement should be performed for two consecutive points.

  In case~\ref{item:SR1}, there is no need to rearrange.
  In case~\ref{item:SR2}, the membranes of the critical point at depth zero are disjoint from $\cL_-$ and $\cL_+$, so we can rearrange by moving $p$ above $p'$ (if $d_p=0$), or $p'$ below $p$ (if $d_{p'}=0$). If $d_p=d_{p'}=0$, we can perform either of the two moves. Note that in this move, we rearrange only a depth zero critical point, so that if $p$ or $p'$ happen to be either of $p_-$ or $p_+$, then they are not moved.

  In case~\ref{item:SR3}, the dimension counting argument of Lemma~\ref{lem:dim_asc} ensures that the descending membrane of $p'$ is disjoint from the set $\cL_-$. Therefore, we can perform safe rearrangements moving $p'$ below $p$, unless $p'=p_-$ or $p'=p_+$.

  If $p'=p_-$, then there is a critical point below $p_-$, with $h_p+d_p>h_{p_-}+d_{p_-}$. Originally, the function $F_0$ has no critical points in $F_0^{-1}[a-\varepsilon,a]$. If during the procedure we are now describing, some critical point lands below $p$, then it must have
  $h_p+d_p<h_{p_-}+d_{p_-}$.  Therefore, the only situation where the move is forbidden is if $p'=p_+$ is the top critical point.

  Similarly, in case~\ref{item:SR4}, the ascending membrane of $p$ is disjoint from $\cL_+$, so we can safely move $p$ above $p'$. The discussion
  of special cases is analogous: we cannot perform the move of case~\ref{item:SR4} only if $p=p_-$.

  Having specified the local algorithm, we apply it first for all the critical points between $p_-$ and $p_+$.
  In this way, using a bubble sort, we arrange critical points in ascending order of $h_p+d_p$. Next, we move to the case when $p=p_-$. If the critical point $p'$ directly above $p_-$ has $h_p+d_p > h_{p'}+d_{p'}$, then, as $h_p+d_p=h_-+1$, we land in case~\ref{item:SR3}. Therefore, we can move $p'$ below $p_-$. Inductively,
  we move all critical points below $p_-$, for which $h_{p'}+d_{p'}<h_-+1$.

  Next, suppose that $p'=p_+$ and $p$ is immediately below $p_+$ with $h_p+d_p>h_{p'}+d_{p'}$, and $p\neq p_-$. If $d_p=0$, we fit into case~\ref{item:SR2}, we lift $p$ above $p'$. If $d_p\neq 0$, then, as $k\ge 2$, $h_p+(k-1)(d_p-1)\ge h_p+d_p-1$, so we are in case~\ref{item:SR4}. We can lift the critical point~$p$ above~$p_+$.

  Therefore, the only critical points that can remain between $p_-$ and $p_+$ are those with $h_p+d_p\ge h_-+d_-$ (because all others can
  be safely pushed below $p_-$ or above $p_+$) and, simultaneously, $h_p+d_p\le h_++d_+$.
  \end{proof}

  We study the critical points that remain between $p_-$ and $p_+$ in greater detail. These are precisely the points at which the local algorithm
  fails to move away from the region between $p_-$ and $p_+$. The only situation when this can happen
  is if  $h_-+d_-=h_++d_+$.

  \begin{corollary}\label{cor:moving_step_2}
    Suppose $h_-+d_-=h_++d_+$. Then all the critical points $p$ with $d_p=0$ can be moved above the critical point $p_+$. Moreover, if $k\ge 3$,
    all the critical points can be moved.
  \end{corollary}
  \begin{proof}
    We refine the part of the proof of Lemma~\ref{lem:moving_step_1}. First take two critical points $p$ and $p'$ between $p_-$ and $p_+$
    such that $p'$ is right above $p$ and $p'$ is allowed to be $p_+$. Necessarily, $h_p+d_p=h_{p'}+d_{p'}$ (now all critical points between $p$ and $p'$ satisfy this). If $d_p=0$ and $d_{p'}>0$, then we safely move $p$ above $p'$ as in case~\ref{item:SR2} of the proof of Lemma~\ref{lem:moving_step_1}. Iterating this process, we safely rearrange all critical points with $d_p=0$ by moving them above $p_+$. This completes the first part of the corollary.

    Next, suppose $k\ge 3$ and that we have a critical point $p$ right below $p_+$. Suppose it is not $p_-$.
    By the assumptions, $h_p+d_p=h_++d_+$. As all critical points in the zeroth stratum have been moved above $p_+$, we necessarily have $d_p>0$.
    Since $p\neq p_-$, we have $d_p\neq 1$ too, so $d_p\ge 2$. Then
    \[h_p+(k-1)(d_p-1)=h_p+(k-2)(d_p-1)+d_p-1>h_p+d_p-1=h_++d_+-1 > h_+.\]
    Thus, we are in case~\ref{item:SR4}, so we can safely lift $p$ above $p_+$.
  \end{proof}

  Corollary~\ref{cor:moving_step_2} takes care of the first three items of Lemma~\ref{lem:moving}.

  We now refine the above procedure to prove the fourth item of Lemma~\ref{lem:moving}. Suppose $G_0(N_1),\dots,G_0(N_\ell)$ are disjoint, and $p_-:=G_0(q_-)$
  belongs to $G_0(N_1)$. This is relevant only if $h_-=h_+$, because this condition is part of the statement of the fourth item. Suppose $p$ is a critical point in $F_0^{-1}[a-\varepsilon,b+\varepsilon]$ at depth $d_p\ge 2$ and with
  $h_p+d_p=h_-+1$, and suppose there are no critical points with critical values in $(F(p_-),F(p))$.
  Assume that $p$ belongs, say, to $G_0(N_2)$. By the dimension condition, there are no trajectories between $p_-$ and $p$
  that belong to the zeroth stratum. On the other hand, $\cL_-$ belongs to a different component of $G_0(N)$ than the part of
  the membrane of $p$ belonging to the first and deeper strata (here we use the fact that $\cL_-\subseteq G_0(N_1)$ and $p\in G_0(N_2)$).
  Hence, the descending membrane of $p$ is disjoint from $\cL_-$, so we can safely move $p$ below $p_-$. This proves the fourth item.

  %Analogously, if $h_p+d_p=h_++1$ and $p$ belongs to a different connected component of $G_0(N)$ than $p_+$, we use the same argument to show that it can be moved safely above the critical level set of $p_+$.

  Concatenating all the $\xi$-paths of rearrangement used in these two procedures (first organizing all the critical points between $p_-$ and $p_+$ and then second moving most of their critical values out of the region between the values of $p_-$ and $p_+$) creates a $\xi$-path of
  rearrangement $F_\tau$ with support disjoint from $\cL_-\cup\cL_+$, i.e.\ a safe path. The only critical points of $F_1$ in $F_1^{-1}(a,b)$ are those with $h_p+d_p=h_-+1=h_++1$ and $d_p>0$.
  If $\varepsilon>0$ is small, $F_1^{-1}[a-\varepsilon,b+\varepsilon]$
  also contains only these points.
  We have not moved any of the critical points at depth 1, and so $F_\tau \circ G_0$ is a path of excellent Morse functions, as required by the last item of the lemma.
  This concludes the proof of Lemma~\ref{lem:moving}.
  \end{proof}

    The path constructed in Lemma~\ref{lem:moving} does not change $G_0$, so in particular no new double points of $M=G_0(N)$ are
  created. Moreover, $F_\tau$ is a path of immersed Morse functions, without births or deaths on the deeper strata.

We have done enough to obtain \ref{item:ir_embedding} of Addendum~\ref{lem:lift_rearr_immersion}.
If $G_0$ is an embedding then \ref{item:LR2} is automatically satisfied, and so no further operations are necessary. The procedure described in Subsection~\ref{sub:summary_rearr} will give rise to a regular double path $(F_\tau,G_\tau)$ with $G_\tau$ an embedding for all $\tau$.

\subsection{Enforcing condition~\ref{item:LR2}: the regular homotopy}\label{sub:LR2_move}

The next result, ensuring that it is possible to change the map $G_0$ by a path of immersions
in such a way that $G_1$ maps $\cK_-$ to the first stratum, is the most difficult step
in the proof of Lifting Rearrangement Lemma~\ref{lem:lift_rearr}. The precise construction of the isotopy $G_\tau$ stretches across Subsections \ref{sub:LR2_move}, \ref{sub:f_ing_immersion}, \ref{sub:LR2_gen}, and \ref{sub:lowdim}. The main part of the proof occurs in this subsection. In Subsection~\ref{sub:f_ing_immersion} we show that we can arrange for $G_\tau$ to be an immersion for all $\tau$. In Subsection~\ref{sub:LR2_gen} we show how to modify the construction to arrange for \ref{item:ir_disjoint_image} to hold. In Subsection~\ref{sub:lowdim} we address~\ref{item:ir3}. Then in Subsection~\ref{sub:summary_rearr} we summarise the entire proof.

\begin{proposition}\label{prop:create_isotopy}
  Suppose $(F_0,G_0)$ satisfies condition~\ref{item:LR1}, and $\eta$ is gradient-like for $F_0\circ G_0$.
  There exists a generic path of immersions $G_\tau\colon N\to\R$,
  $\tau\in[0,1]$, such that $G_\tau$ is a regular $F_0$-path $($see Definition~\ref{def:F_path}$)$,
  $F_0\circ G_\tau$ is an $\eta$-path $($Definition~\ref{def:eta-path}$)$ of
  excellent Morse functions, and $G_1(\cK_-) \cap F_0^{-1}([a,b])$ is contained in the first stratum of $G_1(N)$.
\end{proposition}

\begin{proof}[Proof of Proposition~\ref{prop:create_isotopy} modulo Subsection ~\ref{sub:f_ing_immersion}]  We will prove the lemma modulo a key technical lemma, that will be proven in the following subsection.
  Define
  \begin{equation} \label{eqn:defn-some-submanifolds}
    \Omega'=F_0^{-1}[a-\varepsilon,b+\varepsilon] \text{ and }
    N'=G_0^{-1}(\Omega')=f_0^{-1}[a-\varepsilon,b+\varepsilon].
  \end{equation}

  The idea is to push the set $G_0(\cK_-)\cap \O[2]$ above the level set of $F_0^{-1}(b)$,
  so that $G_0(\cK_-)$ does not intersect $\O[2]$ below $F_0^{-1}(b)$.   Hence also $G_0(\cK_-)$ will not intersect $\O[d]$ for $d \geq 3$, because $\O[d] \subseteq \overline{\O[2]}$.
  The push will be guided by a vector field which is
  not necessarily tangent to $G_0(N)$.
  %The approach is conceptually similar to the proof of Rearrangement Theorem~\ref{thm:grim_rearrangement}.\
  Choose a vector field on $\Omega'$, which is a gradient vector field for $F_0$ as \emph{an ordinary Morse function}. This means that
  the vector field need not be tangent to $M$. To emphasise the difference between grim vector fields and this vector field,
  we denote the latter by $\nabla F_0$. Note that $F_0$, regarded as an ordinary Morse function,
  has no critical points on $\Omega'$ (by Condition~\ref{item:LR1}), therefore $\nabla F_0$ does not vanish in $\Omega'$.

  We will assume that $\nabla F_0$ is generic. The precise meaning of the genericity
  condition will be specified in few places, especially in Subsection~\ref{sub:LR2_gen}.
  The next result, Lemma~\ref{lem:projection_vs_gradient}, gives a suitable framework for genericity: instead of perturbing
  the Riemannian metric defining $\nabla F_0$, we can perturb the induced projection. This allows us to use transversality
  arguments in a more flexible way. %\ypar{Rephrased this}

  A gradient-like vector field $\nabla F_0$ induces a projection from $\Omega'$ onto $Y:=F_0^{-1}(a+\varepsilon)$:
  a point $z\in\Omega'$ is mapped to the unique point $y\in Y$ lying on the same trajectory of $\nabla F_0$.
  We denote this projection by $\Pi \colon \O' \to Y$. For future use we give the following standard result.

  \begin{lemma}\label{lem:projection_vs_gradient}
    Suppose $\Pi'\colon\Omega'\to Y$ is another projection, and $\Pi'$ is sufficiently close $($in the $C^1$-norm$)$ to $\Pi$.
    Then there is a gradient-like vector field for $F_0$ inducing $\Pi'$.
  \end{lemma}

  \begin{proof}
    For any $c\in[a+\varepsilon,b+\varepsilon]$ the restriction $\Pi|_{F_0^{-1}(c)}$ is a local diffeomorphism, because $\nabla F_0$ is transverse to the level sets of $F_0$. This property is open, hence
    if $\Pi'\in C^\infty(\Omega',Y)$ is sufficiently close to $\Pi$, then $\Pi'$ restricted to $F_0^{-1}(c)$, for $c\in[a+\varepsilon,b+\varepsilon]$, is a local diffeomorphism as well. %\npar{Is this supposed to say that the vector field is transverse to level sets, or that $\Pi'$ is? It's a bit ambiguous at the moment. }{Fixed.}

    Choose a metric on $\Omega'$. For each point $z\in\Omega'$, let  $\xi(z) \in \ker \Pi'$ be the unit vector such that $\partial_{\xi(z)} F_0(z) >0$. (We know that $\partial_{\xi(z)} F_0(z) \neq 0$ because $\xi(z)$ is transverse to the level set $F_0^{-1}(F_0(z))$.)  We choose
    the orientation of $\xi(z)$ in such a way that $\partial_{\xi(z)}F_0(z)>0$. By construction, $\xi(z)$ is a gradient-like vector field
    for $F_0$. Moreover, since $\xi(z)\in\ker D\Pi'(z)$ for all $z$, the trajectories of $\xi(z)$ are contained in the fibres of $\Pi'$.
  \end{proof}

  We use this lemma to specify the first of several genericity conditions imposed on $\nabla F_0$.

\begin{lemma}\label{lem:generic_Pi}
  The set of maps $\Pi$ such that for each sign $\pm$,  $(\Pi\circ G_0)^{-1}(\Pi\circ G_0(q_\pm)) = \{q_{\pm}\}$,
is open dense in the set of all smooth maps from $\Omega'$ to $Y$.
\end{lemma}

\begin{proof}
  %The first part follows from Thom-Mather theory, in fact, a generic map $\Pi$ is \emph{stratified}; see ??. For the second part,
  The double point set of $\Pi\circ G_0$ is of codimension $k-1$, hence it may be avoided generically by the zero-dimensional manifold $\{q_-,q_+\}$.
It follows from transversality that the given set of maps is open-dense.
\end{proof}

  It follows that there exists a nonempty open neighbourhood $U\subseteq N$ of $q_-$
  with the property that a trajectory of $\nabla F_0$ through any point in $G_0(u)$ for $u\in U$, does not hit $G_0(N)$. Decrease $\varepsilon$, if necessary to ensure that
  \begin{equation}\label{eq:ckminus}
    \cK_-\cap f^{-1}[a,a+2\varepsilon]\subseteq U.
  \end{equation}
  Define $Z\subseteq f^{-1}(a+\varepsilon)$ to be that subset of $z \in f^{-1}(a + \varepsilon)$ that satisfy Condition~\ref{cond:on_Z} below.
  \begin{cond}\label{cond:on_Z}
    Whenever $w \in N$ lies on a trajectory of $\eta$ through $z$ and $a+\varepsilon \le f(w)< a+2\varepsilon$, then the trajectory of $\nabla F_0$ through $G_0(w)$ only   intersects $G_0(N)$ in $\Omega'$ at $G_0(w)$.
  \end{cond}
  We refer to the schematic of $Z$ in Figure~\ref{fig:zeta_def}.
 The conditions defining $Z$ are open, hence $Z$ is an open subset of $f^{-1}(a+\varepsilon)$. Note that \eqref{eq:ckminus} implies also
  that $\cK_-\cap f^{-1}[a+\varepsilon]\subseteq Z$. In particular $Z$ is nonempty.

  As $\cK_-$ is disjoint from $\cK_+$, there exists an open subset $Z_0\subseteq Z$ (`open' meaning open in $f^{-1}(a+\varepsilon)$)
  containing $\cK_-\cap f^{-1}(a+\varepsilon)$ and disjoint from $\cK_+$.

  \begin{figure}
    \input{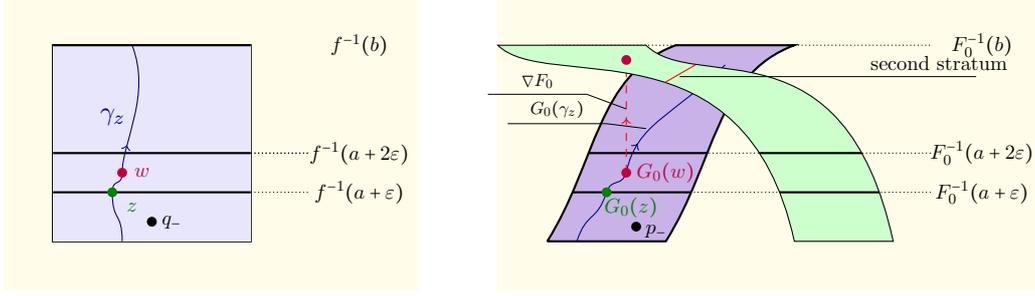}
    \caption{The point $z$ does not belong to $Z$. In fact, the trajectory $\gamma_z$ of $\eta$, in $N$ passes through $w$ (left picture), and the trajectory of $\nabla F_0$ through $G_0(w)$, in $G_0(N) \subseteq \O$, reaches a point on $G_0(N)$ (right picture).}\label{fig:zeta_def}
  \end{figure}

  Define $R\colon\R\to\R$ to be a smooth increasing function such that $R(x)=x$ for $x\notin[a+\varepsilon,b+\varepsilon]$ and such that
  $R$ maps:
  \begin{itemize}
    \item $[a+\varepsilon,a+2\varepsilon]$ onto $[a+\varepsilon,b]$;
    \item $[a+2\varepsilon,b]$ onto $[b,b+\varepsilon/2]$;
    \item $[b,b+\varepsilon]$ onto $[b+\varepsilon/2,b+\varepsilon]$.
  \end{itemize}
  The function $R$ will be our \emph{rescaling function}, denoting the height of the lift. Note that, by construction
  \begin{equation}\label{eq:Rx_and_x}
    R(x)\ge x\textrm{ for all }x\in\R.
  \end{equation}
  We define a new function, which is called the \emph{lift function} $P\colon\Omega'\times[0,1]\to\Omega'$.
  For $\theta\in[0,1]$, and $z\in\Omega'$, we define $P(z,\theta)$ to be the unique point in $\Omega'$ belonging to the
  same trajectory of $\nabla F_0$ as $z$, such that
  \begin{equation}\label{eq:Pz_def}
    F_0(P(z,\theta))=F_0(z)+\theta(R(F_0(z))-F_0(z)).
  \end{equation}
  Note that the function $P\colon\Omega'\times[0,1]\to\Omega'$ is a smooth function, because of the smooth dependence of solutions of an ODE
  on initial conditions.

  We aim to define $G_\tau$ in such a way that a point $u\in\cK_-$ is mapped to $P(G_0(u),\tau)$, i.e.\ it is moved up
  from the level set $f(u)$ to the level set of $(1-\tau)f(u)+\tau R(f(u))$.
  For the map to be smooth, we need to also lift points near $\cK_-$. To this end, we need
  to define a cut-off function
  \[\mu \colon f^{-1}[a+\varepsilon,b+\varepsilon]\to [0,1].\]
  In the construction, we will adjust the function $\mu$ in such a way that
  \begin{itemize}
    \item[(i)]  $\eta$  is a gradient-like vector field for $F_0\circ G_\tau$ for all $\tau\in[0,1]$; and
    \item[(ii)]    $G_\tau$ is an immersion for all $\tau \in [0,1]$.
  \end{itemize}
We address the first property next. The  second property will be recorded below as Lemma~\ref{lemma:G-tau-immersion}, the proof of which we will defer to Subsection~\ref{sub:f_ing_immersion}.

  To prove that $\eta$ is gradient-like for $F_0\circ G_\tau$, we will need the following properties
  of $\mu$:
  \begin{enumerate}[label=($\mu$-\arabic*)]
    \item $\partial_\eta\mu\ge 0$, and equality holds only at critical points of $f$ and at places where the value of $\mu$ lies in $\{0,1\}$;\label{item:mu_0}
    %\npar{added equality at critical points. I don't think it's iff because there are places where $\mu$ is locally constant. }{Rephrased slightly}
    \item the support of $\mu$ intersected with $f^{-1}(a+\varepsilon)$ is contained in $Z_0$;\label{item:mu_1}
    \item $\mu$ is equal to $1$ on $\cK_-\cap f^{-1}[a+2\varepsilon,b+\varepsilon]$ and vanishes on $\cK_+$.\label{item:mu_2}
  \end{enumerate}
 The next lemma shows that such a function exists.

 \begin{lemma}\label{lemma:mu-exists}
   There exists a smooth function $\mu \colon f^{-1}[a+\varepsilon,b+\varepsilon]\to [0,1]$ satisfying \ref{item:mu_0}, \ref{item:mu_1}, and \ref{item:mu_2}.
 \end{lemma}

\begin{proof}
   Take a smooth function $\mu_0\colon f_0^{-1}(a+\varepsilon)\to [0,1]$ supported on $Z_0$
  and equal to $1$ on $\cK_-\cap f^{-1}(a+\varepsilon)\subseteq Z_0$. Define $\O_0' := F_0^{-1}[a+2\varepsilon,b+\varepsilon]$ and  $N_0' := G_0^{-1}(\O_0')$.

  Let $\nu$ be a smooth non-decreasing function from $\R_{\ge 0}$ to $[0,1]$
  equal to $0$ for $t \in \R$ close to $0$ and equal to $1$ for $t>t_0$,
  where $t_0$ is the minimal time needed to reach $f_0^{-1}(a+2\varepsilon)$
  by going from $Z_0$ via $\eta$.

  Then define $\mu\colon N'_0\to\R_{\ge 0}$ via $\mu(z)=\nu(t)\mu_0(z_0)$
  for any $z\in N'_0$ if the trajectory of $\eta$ through $z$
  hits the level set $a+\varepsilon$ at the point $z_0$ and the time to go from $z_0$ to $z$ is equal to $t$.
  With this construction, $\mu$ has all the required properties.
  \end{proof}

  We will impose one more condition on $\mu$ in Subsection~\ref{sub:f_ing_immersion}, and show that $\mu$ can be assumed to also satisfy this.
  %We will give the final definition of $\mu$ satisfying this extra condition in \eqref{eqn:final-defn-of-mu}.
%  \npar{added this sentence, and made the previous lemma. }{commented out the sentence}

  Now we define a path of maps $G_\tau\colon N\to\O$. The definition is as follows. Take $u\in N$. Set $G_\tau(u)=G_0(u)$ if $f_0(u)\notin N'_0$.
  If $f_0(u)\in N'_0$, we set
  \begin{equation}\label{eq:Gtau_def}
    G_\tau(u)=P(G_0(u),\tau\mu(u)).
  \end{equation}
  Compare Figure~\ref{fig:liftG}.
  \begin{figure}
    \input{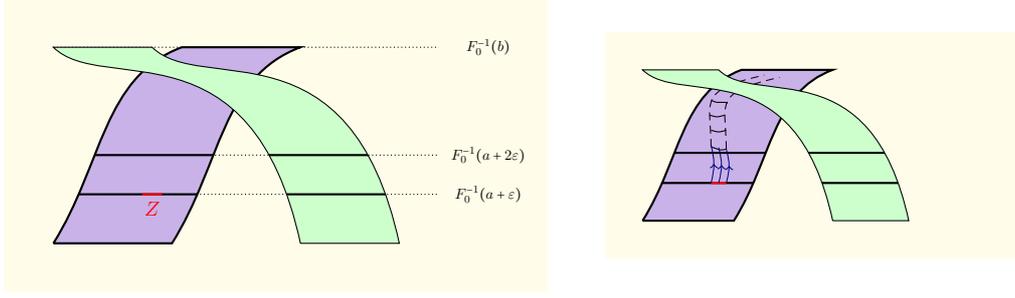}
    \caption{A schematic of $G_\tau$. Points on a trajectory through $Z$ are lifted up by flowing along $\nabla F_0$.}\label{fig:liftG}
  \end{figure}
  We list the properties of $G_\tau$.
  \begin{itemize}
    \item $G_\tau(u)$ depends smoothly  on $\tau$ and $u$, in particular $G_\tau$ is a homotopy between the maps $G_0$ and $G_1$;
    \item $G_\tau(u)=G_0(u)$ on $\cK_+$ by \ref{item:mu_2}. Moreover, if a trajectory of $\eta$ through $u$ misses $Z$, then $G_\tau(u)=G_0(u)$;
    \item $G_\tau=G_0$ above the level set $f_0^{-1}(b+\varepsilon)$ and
      below the level set $f_0^{-1}(a+\varepsilon)$.
  \end{itemize}
  Set
  \[\wt{f}_\tau=F_0\circ G_\tau \colon N \to \R.\]

  \begin{lemma}\label{lem:wtf_is_eta}
    The family $\wt{f}_\tau \colon N \to \R$ is an $\eta$-path of excellent Morse functions.
  \end{lemma}

  \begin{proof}
    It is clear from the construction that $\wt{f}_\tau$ does not depend on $\tau$ for any of the critical points of $f_0$. In fact, the only
    critical point of $f_0$ in $N'_0$ is $q_+$, and we have that $\mu(q_+)=0$.

    We prove that $\partial_\eta f_\tau\ge 0$ with equality only at critical points of $f_\tau$.
    Given a point $u\in N$, \eqref{eq:Pz_def} implies that $P(G_0(u),\tau\mu(u))$ is at the level set
    $F_0(G_0(u))+\tau\mu(u)(R(F_0(G_0(u)))-F_0(G_0(u)))$.
    Noting that $F_0(G_0(u))=f_0(u)$ and $f_\tau(u)=F_0(G_\tau(u))$ we see that
    \eqref{eq:Pz_def} and \eqref{eq:Gtau_def} imply:
    \begin{equation}\label{eq:dmu}
      f_\tau(u)=f_0(u)+\tau \mu(u)(R(f_0(u))-f_0(u)).
    \end{equation}
    Applying $\partial_\eta$ to both sides of \eqref{eq:dmu} we obtain.
    \begin{equation}\label{eq:mega_partial_derivative}
      \begin{split}
	\partial_\eta f_\tau(u)=&\partial_\eta f_0(u)+\\
	&+\tau(R(f_0(u))-f_0(u))\partial_\eta\mu+(R'(f_0(u))\partial_\eta f_0(u)-\partial_\eta f_0(u))\tau\mu=\\
	=&\underbrace{(1-\tau\mu)\partial_\eta f_0(u)}_{(1)}+\underbrace{\tau(R(f_0(u))-f_0(u))\partial_\eta\mu}_{(2)}+
	\underbrace{\tau\mu R'(f_0(u))\partial_\eta f_0(u)}_{(3)}.
      \end{split}
    \end{equation}
    We study the terms (1), (2), and (3).
    \begin{itemize}
      \item[(1)] As $\eta$ is a gradient-like vector field and $\mu(u)\in[0,1]$, term (1) is non-negative. It is zero only if $\mu(u)=\tau=1$
	or $u$ is a critical point of $f$.
      \item[(2)] By construction, $\partial_\eta\mu\ge 0$, with equality at the critical points, and also $R(f_0(u))-f_0(u)\ge 0$ (compare \eqref{eq:Rx_and_x}), so (2) is non-negative.
      \item[(3)] The derivative $R'$ is positive. The term (3) is non-negative. It is positive unless $\tau\mu=0$ or $u$ is a
	critical point of $f_0$.
    \end{itemize}
    From the list, since $\tau \mu$  cannot be $0$ and $1$ simultaneously, it is clear that $\partial f_\tau(u)\ge 0$ with equality exactly at the critical points of $f_0$.
  \end{proof}

  Continuing the proof of Proposition~\ref{prop:create_isotopy}, we have the
  following result.

  \begin{lemma}\label{lem:reduce_to_1}
    $G_1(\cK_-)\cap F_0^{-1}[a,b]$ is disjoint from the set of the double points of $G_1(N)$.
  \end{lemma}

  \begin{proof}
    Take $w\in\cK_-\cap f_0^{-1}[a,b]$. Consider the following three cases.
    \begin{itemize}
      \item If $f_0(w)\le a+\varepsilon$, then $G_1(w)=G_0(w)$ belongs to the first stratum, because $\cK_-\cap f_0^{-1}[a,a+\varepsilon]\subseteq U$.
      \item Suppose $f_0(w)\in[a+\varepsilon,a+2\varepsilon]$. Let $u\in f_0^{-1}(a+\varepsilon)$ belong to the same trajectory of $\eta$
	as $w$. Then $u\in Z$ and by Condition~\ref{cond:on_Z}, $G_1(w)$ belongs to the first stratum.
      \item If $f_0(w)>a+2\varepsilon$, then $\mu(w)=1$, so by \eqref{eq:dmu} we have $f_\tau(w)=(1-\tau)f_0(w) + R(f_0(w))>b$.
	So $G_1(w)\notin F_0^{-1}[a,b]$. \qedhere
    \end{itemize}
  \end{proof}

 We will need one more result, whose proof is deferred to later subsections.

 \begin{lemma}\label{lemma:G-tau-immersion}
There exists a choice of $\mu$ such that $G_\tau$ is an immersion for all $\tau \in [0,1]$.
  \end{lemma}

We will prove Lemma~\ref{lemma:G-tau-immersion} in Subsection \ref{sub:f_ing_immersion}. Modulo this,   Lemmas~\ref{lem:wtf_is_eta},~\ref{lem:reduce_to_1}, and ~\ref{lemma:G-tau-immersion} conclude the proof of Proposition~\ref{prop:create_isotopy}.
\end{proof}

 %
% \begin{proposition}\label{prop:disjoint-first-time}
%  Suppose $N=N_1\sqcup\dots\sqcup N_\ell$ and $G_0(N_1),\dots,G_0(N_\ell)$ are pairwise disjoint. Then there exists
%  a vector field $\nabla F_0$ such that $G_\tau(N_1),\dots,G_\tau(N_\ell)$ are pairwise disjoint for all $\tau\in[0,1]$.
%\end{proposition}
%
%We will prove Proposition~\ref{prop:disjoint-first-time} in Subsection~\ref{sub:LR2_gen}.
%
%
  %, modulo the proof of Lemma~\ref{lemma:G-tau-immersion} in Subsection~\ref{sub:f_ing_immersion}.

In the next subsection, we will check that $G_\tau$ is an immersion for all $\tau$, proving Lemma~\ref{lemma:G-tau-immersion}. This done, we
will perturb $G_\tau$ in such a way that $(F_\tau,G_\tau)$ becomes a regular double path (see Definition~\ref{def:regular_dp}).
The construction will be further tweaked in Subsections~\ref{sub:LR2_gen} and \ref{sub:lowdim}, in order to arrange for \ref{item:ir_disjoint_image} and \ref{item:ir3} to hold.
A summary of the entire proof of Lemma~\ref{lem:lift_rearr} will be given in Subsection~\ref{sub:summary_rearr}.

\subsection{Enforcing condition~\ref{item:LR2}: proving that $G_\tau$ is an immersion}\label{sub:f_ing_immersion}
We will now prove that if $\nabla F_0$ is sufficiently generic, and for a careful choice of $\mu$, then $G_\tau$ is an immersion for all~$\tau$.

Recall that $\Omega'=F_0^{-1}[a-\varepsilon,b+\varepsilon]$. Set $Y=F_0^{-1}(a+\varepsilon)$. As explained in Subsection~\ref{sub:LR2_move}, the
vector field $\nabla F_0$ induces a projection $\Pi\colon \Omega'\to Y$.
We begin with the following observation.

\begin{lemma}\label{lem:PiGtau}
  The map $\Pi\circ G_\tau\colon N\to Y$ does not depend on $\tau$.
\end{lemma}

\begin{proof}
  For any point $u\in N$, by the construction of $G_\tau$, the points $G_\tau(u)$, $\tau\in[0,1]$ belong to the same trajectory of $\nabla F_0$. Therefore, they are mapped by
  $\Pi$ to the same point.
\end{proof}

\begin{corollary}\label{cor:immersion}
  Let $u\in N$ and let $z=G_0(u)$. Suppose that $\nabla F_0$ is not
  tangent at $z$ to the branch of $G_0(N)$ containing $G_0(u)$. Then $G_\tau$ is an immersion near $u$ for all $\tau$.
\end{corollary}

\begin{proof}
By the chain rule, $D(\Pi \circ G_0) = D\Pi(G_{0}(u)) \circ DG_0(u)$.
  As $G_0$ is an immersion, $DG_0(u)$ has trivial kernel.
  The differential $D\Pi(z)$ has kernel spanned by $\nabla F_0(z)$.  By assumption, it follows that the kernel of $D\Pi(z)$ intersects the image
  of $DG_0(u)$ only at $0$. Hence $D(\Pi\circ G_0)(u)$ has trivial kernel. By Lemma~\ref{lem:PiGtau}, $D(\Pi\circ G_\tau)(u)$
  has trivial kernel for all $\tau$. Since $D(\Pi \circ G_\tau) = D\Pi(G_{\tau}(u)) \circ DG_\tau(u)$, it follows that $DG_\tau(u)$ has trivial kernel for all $\tau$, and hence by continuity of the derivative, $G_\tau$ is an immersion near $u$.
\end{proof}

Thus we only need to study points at which $\nabla F_0$ is tangent to $G_0(N)$ at $z$. We show that generically, the set of such points $z \in G_0(N)$  is itself a manifold.
In proving genericity, it is most convenient to perturb the immersion $G_0$.

\begin{lemma}\label{lem:tangent_generic}
  There is an open-dense subset of maps $G_0$ such that the set of points $u \in N$, with the property that $\nabla F_0$ is tangent to the branch
  of $M$ with tangent vectors $DG_0(u)(T_u N) \subseteq T_{G_0(u)}\O$, forms a smooth manifold of dimension $n-k$, missing $q_-$ and $q_+$.
\end{lemma}

\begin{proof}
  The condition $\nabla F_0(G_0(u))\subseteq \Iim DG_0(u)$ defines a smooth submanifold in the 1-fold jet space $J^1(N,\Omega)$
  of maps from $N$ to $\Omega$; compare Subsection~\ref{sub:paths_of_immersions} for the relevant definitions. The codimension is equal to $k$. This can be seen as follows: $DG_0(u)$ is a codimension $k$ subspace and $\nabla F_0(G_0(u))$ is a 1-dimensional vector. The condition that a vector
  fits into a codimension $k$ space is of codimension $k$.
  By the Thom transversality theorem (compare Theorem~\ref{thm:multijet} for $s=1$) the set of maps $G_0$ transverse to this stratum is open-dense.
  If, possibly after a perturbation, $G_0$ is transverse to this stratum, then the set of points in $N$ such that the jet extension of $G_0$ hits the stratum,
  is a smooth submanifold of the same codimension, that is, $k$.
\end{proof}

\begin{remark*}
%  \npar{I have the feeling this remark should be part of the proof of the previous lemma, and we should explain what $\Sigma$ is and how its dimension is computed.}{This remark is never used and puts the previous lemma in a broader context. Once we have the codimension count, we do not need to use it. I rephrased slightly.}
  We can understand Lemma~\ref{lem:tangent_generic} in the broader context of the Thom-Boardman stratification of Subsection~\ref{sub:thom_boardman}.
  The set of points $u$ where $\nabla F_0(G_0(u))$ is tangent to $G_0(N)$ at $G_0(u)$ actually corresponds to the Thom-Boardman stratum $\beta^1$ of the map $\Pi\circ G_0$.  In particular, the dimension of this set can be alternatively computed using e.g.\ \cite[Section I.2]{AVG}.
\end{remark*}

\begin{figure}
  \input{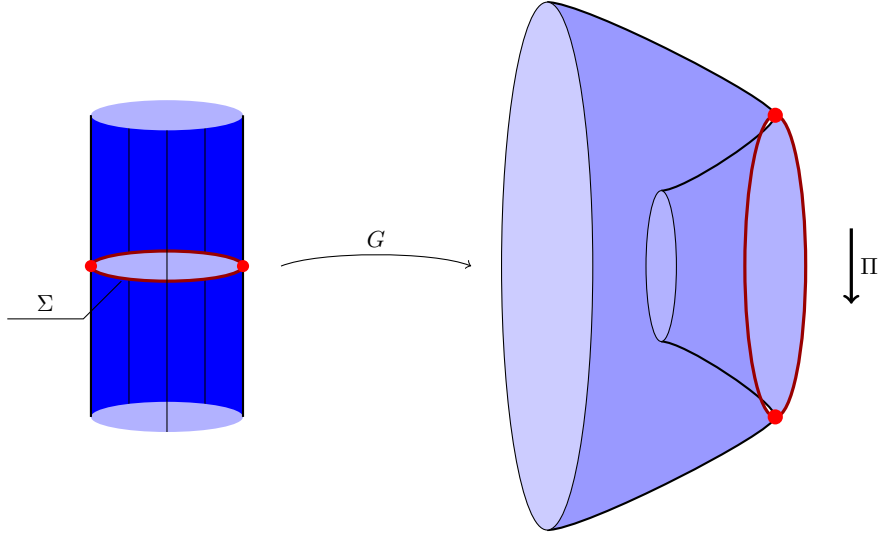}
  \caption{The set $\Sigma$ and its stratification. The thickened points on $\Sigma$ represent the second level of the Thom-Boardman stratification with respect to the map $\Pi\circ G$.}\label{fig:sigma}
\end{figure}

Perturb $G_0$ so that it satisfies the genericity condition spelled out in Lemma~\ref{lem:tangent_generic}. Let $\Sigma\subseteq N$
be the set of points $u$ such that the image of $DG_0(u)$ contains $\nabla F_0(G_0(u))$; see Figure~\ref{fig:sigma}. As critical points of $f_0$ form a finite set, by further perturbing $G_0$, we may and shall assume that
$\Sigma$ misses all critical points of $f_0$.
For each point $u$ in $\Sigma$, there is a unique
vector $\upsilon_u\in T_uN$ such that $DG_0(u)(\upsilon_u)=\nabla F_0$. Hence we obtain a section $\upsilon \colon \Sigma \to TN$ of $TN|_{\Sigma}$.
 We make a trivial observation for future use.

\begin{lemma}\label{lem:positive_multiple}
  If $\upsilon_u$ is parallel to $\eta(u)$,
  then $\upsilon_u$ is not a negative multiple of $\eta(u)$.
\end{lemma}

\begin{proof}
We have $f_0 = F_0 \circ G_0$,  $\upsilon_u \in T_u N$, and $DG_0(\upsilon_u) = \nabla F_0$.
Thus
\begin{equation}\label{eqn:directional-derivatives}
  \partial_{\upsilon} f_0(u) \stackrel{(1)}{=} \partial_{D G_0(\upsilon)} F_0(G_0(u)) \stackrel{(2)}{=} \partial_{\nabla F_0} F_0(G_0(u)) \stackrel{(3)}{=}\|(\nabla F_0)(G_0(u))\|^2 \geq 0.
     \end{equation}
     Here, (1) is functoriality of the differential. The second equality (2) is the definition of $\upsilon_u$. The third equality (3) is the definition of the directional
     derivative as the scalar product with the gradient.
  %We know that $\partial_{\nabla F_0} F_0>0$,
     Hence $\partial_{\upsilon}f_0\ge 0$.
Moreover, $\partial_{\eta}f_0\ge 0$, because the vector field $\eta$ is gradient-like for $f_0$. Notice that we assumed that $\Sigma$ does not pass through any of the critical points. Hence, $\partial_{\eta}f_0$ is actually positive.
  If $\upsilon_u$ is a negative multiple of $\eta$, we must have $\partial_{\upsilon}f_0=\partial_{\eta}f_0=0$, which
  can hold only if $u$ is a critical point of $f_0$. But $u\in\Sigma$ is not a critical point.
\end{proof}

The key property of $\upsilon_u$ is that we can translate the condition on $G_\tau$ being an immersion into a condition involving $\mu$
and $\upsilon$.
%\npar{MP: I've removed the subscript $u$ from $\upsilon_u$ when in the subscript of $\partial_{\upsilon_u}$. So these are now $\partial_{\upsilon}$. It seems to me that we often have directional derivatives with respect to some vector field, e.g.\ we have $\partial_\eta$ not far from here, but we don't specify the point typically. }{Agree. I was using $\upsilon_u$ to stress that $\upsilon$ is not defined everywhere, only on $\Sigma$, but we can keep $\upsilon$ if you prefer.}

\begin{lemma}\label{lem:mu_partial}
  Suppose $u\in\Sigma$. Assume that the function $\mu$ satisfies
  \begin{equation}\label{eq:on_mu}
    \partial_{\upsilon}\mu(u)\ge 0.
  \end{equation}
  Then $G_\tau$ is an immersion near $u$ for all $\tau$.
\end{lemma}

\begin{proof}
  We know that $\dim\ker D(\Pi\circ G_0)(u)=1$ is spanned by $\upsilon_u$. Therefore,
  the kernel of $DG_\tau(u)$ can either be trivial, or it can be a one-dimensional subspace spanned by $\upsilon_u$.
  To prove the lemma, we need to make sure
  that $\upsilon_u$ is not in the kernel of $DG_0(u)$. Since $f_\tau=F_0\circ G_\tau$, if $\upsilon_u\in\ker DG_\tau(u)$, then
  $\partial_{\upsilon} f_\tau(u)=0$. Hence, it is enough to show that \eqref{eq:on_mu} implies
  $\partial_{\upsilon} f_\tau(u)>0$.

  To this end, we use \eqref{eq:dmu}. Differentiating both sides of \eqref{eq:dmu} with respect to $\upsilon$, we obtain
  a formula analogous to \eqref{eq:mega_partial_derivative}.
  \[
	\partial_\upsilon f_\tau(u)
	=(1-\tau\mu)\partial_\upsilon f_0(u)+\tau(R(f_0(u))-f_0(u))\partial_\upsilon\mu+
	\tau\mu R'(f_0(u))\partial_\upsilon f_0(u).
      \]
  Note that $\partial_{\upsilon}f_0$ is positive, since it is $\partial_{\nabla F_0}F_0$ by \eqref{eqn:directional-derivatives}.  An analysis akin to the one following~\eqref{eq:dmu} leads to the statement that $\partial_{\upsilon}\mu(u)\ge 0$ implies that $\partial_{\upsilon}f_\tau>0$.
\end{proof}

In light of Lemma~\ref{lem:mu_partial} (and keeping in mind Corollary~\ref{cor:immersion}), we proceed to the main technical result
of this subsection.

\begin{proposition}\label{prop:mu_is_good}
  There exists a smooth function $\mu\colon f_0^{-1}[a+\varepsilon,b+\varepsilon]\to[0,1]$ satisfying
  items \ref{item:mu_0}, \ref{item:mu_1}, and \ref{item:mu_2} above and such that for any $u\in\Sigma$, $\partial_{\upsilon}\mu(u)\ge 0$
  with equality only at points where $\mu(u)=0,1$.
\end{proposition}

\begin{proof}
One might expect this proof to  start with a function $\mu$
  satisfying items~\ref{item:mu_0}, \ref{item:mu_1}, and~\ref{item:mu_2}, and then improving it to satisfy the statement of Proposition~\ref{prop:mu_is_good}. However, there are technical problems near $\mu^{-1}(\{0,1\})$. To avoid them, we will proceed in a slightly indirect way.

  Let $\kappa>0$ be an integer: we will describe how it is determined at the end of the proof. Choose a family $S_1,S_2,\dots,S_\kappa$
  of open subsets of $f_0^{-1}(a+\varepsilon)$ with the properties that $\ol{S}_1 \subseteq Z_0$, $\ol{S}_{i+1} \subseteq S_i$ and $\cK_-\cap f^{-1}(a+\varepsilon)\subseteq S_\kappa$.  For $i=1,\dots,\kappa$, let $T_i$ be the subset of $f^{-1}[a+\varepsilon(1+\frac{i-1}{4\kappa}),b+\varepsilon(1+\frac{\kappa-i-1}{4\kappa})]$ consisting of the points that are reached from $S_i$ by a trajectory of $\eta$.
  By construction
  \[\ol{T}_{i+1}\subseteq T_i.\]
   We note that here we assume $\eta$ has no critical points other than
  $q_+$ in $f^{-1}[a+\varepsilon,b+\frac54\varepsilon]$, and not just in $f^{-1}[a+\varepsilon,b+\varepsilon]$. But this does not pose any problems to arrange, e.g.\ by decreasing $\varepsilon$.

  By Condition~\ref{cond:on_Z}, no trajectory of $\eta$ starting from $Z$ hits $\Sigma$ before reaching the level set $f^{-1}(a+2\varepsilon)$. This is because of Condition~\ref{cond:on_Z}. We assume that the set $S_1$ is chosen in such a way that no trajectory of $\eta$
  starting from $S_1$ hits $\Sigma$ before reaching the level set $f^{-1}(a+\frac32\varepsilon)$.

  We construct a function $\zeta\colon T_1\to[0,\frac{5}{4}]$ with the following properties:
  \begin{enumerate}[label=($\zeta$-\arabic*)]
    \item $\partial_\eta\zeta>0$ on $T_1$;\label{item:zeta_0}
    \item Suppose $u\in T_1$ is connected to $S_1\setminus S_\kappa$
      by a trajectory of $\eta$. Then $\zeta(u)\in[0,\frac14]$;\label{item:zeta_1}
    \item $\zeta$ is greater or equal to $1$ on $\cK_-\cap f^{-1}[a+2\varepsilon,b+\varepsilon]$.\label{item:zeta_2}
  \end{enumerate}
  There is a correspondence between \ref{item:mu_0} and~\ref{item:zeta_0}, \ref{item:mu_1} and~\ref{item:zeta_1}, as well as
  between \ref{item:mu_2} and~\ref{item:zeta_2}.
  The precise connection will be made precise in due course, at the end of the proof; see   \eqref{eqn:final-defn-of-mu} for the final definition of $\mu$ and the proof, using the conditions \ref{item:zeta_0}, \ref{item:zeta_1}, and \ref{item:zeta_2}, that the conditions \ref{item:mu_0}, \ref{item:mu_1}, and \ref{item:mu_2} are satisfied for this definition (recall that we proved in Lemma~\ref{lemma:mu-exists} that a function $\mu$ satisfying these criteria exists, but we need a refined version to guarantee that $G_\tau$ is an immersion for all $\tau$).

  \begin{lemma}
  There exists $\zeta$ exists satisfying \ref{item:zeta_0}, \ref{item:zeta_1}, and~\ref{item:zeta_2}.
  \end{lemma}

  \begin{proof}[Sketch of proof]
  As these conditions resemble strongly the conditions for $\mu$, the proof that  $\zeta$  exists satisfying these properties is a straightforward
  adaptation of the proof of Lemma~\ref{lemma:mu-exists}; we therefore leave the details to the reader.
\end{proof}
%  \npar{MP: This is rather strange, and sounds circular, because the final definition of $\mu$ is in terms of $\zeta$.}{Rephrased.}
%\npar{MP: I still found it rather confusing. I've changed it here to refer to the lemma rather than the construction of $\mu$.}{I've changed the statement to a lemma, so it is more natural to refer to Lemma~\ref{lemma:mu-exists} in the `proof'.}

  Our goal is to improve the function $\zeta$ so that $\partial_{\upsilon}\zeta(u)>0$ for all $u\in \Sigma\cap T_\kappa$, and hence since the $T_i$ are nested this proves that $\partial_{\upsilon}\zeta(u)>0$ for  all points $u$ in $\Sigma \cap T_i$, for all $i$.
   %\npar{Are the $T_i$ nested? It could help to mention that, if true. (I'm not sure as I'm not currently sure what their definition is.) }{Added above that they are nested.}
  We will need the following technical result.

  \begin{lemma}\label{lem:leaving}
    %Choose $i=1,\dots,\kappa$.
    Suppose $X$ is a closed submanifold of $\Sigma$. Suppose $U$ is an open subset of~$X$
    %$X_0$ is a closure of an open set in $X\cap T_i$.
    %Suppose that
    such that
    neither $\upsilon$, nor $\eta$ is tangent to $X$ at any point of $\ol{U}$.
    Then  there exists a regular index triple $V\subseteq N$ $($see Definition~\ref{def:regular_index_triple}$)$ such that $V\cap X=U$ and
      if a trajectory of $\eta$ exits $\ol{V}$ at $v_+$ and then re-enters $\ol{V}$ at $v_-$, then $\zeta(v_-)>\zeta(v_+)$.
  \end{lemma}

  \begin{proof}
    Choose a Riemannian metric on $\Omega$. Let $E$ be the normal bundle for $X$ in $N$.
    Since neither $\upsilon$ nor $\eta$ is tangent to $\ol{U}$, we may and will assume that $\eta$ and $\upsilon$ are orthogonal to $T\ol{U}$.
    As $\upsilon$ is not parallel to $\eta$ either,
    we choose a metric in which $\eta$ is orthogonal to $\upsilon$. In particular, $\eta$ and $\upsilon$
    can be regarded as sections of the normal bundle to $\ol{U}$.
    Let $L$ be the rank~1 subbundle of $E|_{\ol{U}}$ generated by $\eta$.
    Let $L^\perp$ be the orthogonal complement to $L$ in $E|_{\ol{U}}$.
    We have that $\upsilon_u\in L^\perp_u$, a fact which we will not use until
    Lemma~\ref{lem:improving}.
    For $\rho>0$, we let $\wt{D}(\rho)$ denote the disc
    bundle of $L^\perp$ consisting of closed discs of radius $\rho$.
    By compactness of $\ol{U}$, we may and will choose $\rho$ sufficiently small so that $\exp\colon \wt{D}(\rho)\to N$ is a diffeomorphism onto its
    image. Let $\ol{D}(\rho)\subseteq N$ be the image of $\wt{D}(\rho)$ under the $\exp$ map. Clearly, $\ol{D}(\rho)\cap X=\ol{U}$. We define $D(\rho)$ as the interior of $\ol{D}(\rho)$. Note that $\ol{D}(\rho)$ is a codimension one submanifold of $N$.

    We have chosen $\eta$ to be orthogonal to $L^\perp$. This means that $\eta$ is transverse to $\ol{D}(\rho)$ at least at the points in $\ol{U}$. By openness of the transversality condition, we may and will decrease $\rho$ in such a way that $\eta$ is transverse to $\ol{D}(\rho)$ at every point of $\ol{D}(\rho)$.

    We now thicken $\ol{D}(\rho)$ in the direction of $\eta$. Choose $\delta>0$. By compactness of $\ol{D}(\rho)$ we may and will assume that
    any trajectory of $\eta$ connecting two distinct points of $\ol{U}$
    takes at least $3\delta$ to travel.
    Define $\ol{D}(\rho,\delta)$ as the set of points on trajectories of $\eta$ that can be reached from a point $u\in \ol{D}(\rho)$ in time in $[-\delta,\delta]$. With that choice of $\delta$,
    the flow of $\eta$ induces a diffeomorphism $\ol{D}(\rho,\delta)\cong\ol{D}(\rho)\times[-\delta,\delta]$.
    %sufficiently small.

    We declare $V$ to be the interior of $\ol{D}(\rho,\delta)$.
    The entry set $\piin V$ is the set of points that reach $\ol{D}(\rho)$ in time $\delta$. The exit set $\pout V$ is the set of points that are reached from $\ol{D}(\rho)$ in time $\delta$.

    Note that if a trajectory of $\eta$ exits $\pout V$ and then hits $\piin V$, it has to travel for time at least $\delta$ (otherwise there would be a trajectory connecting two points of $\ol{D}(\rho)$ in time less than $3\delta$).
    As $\zeta$ increases along the trajectories,
    if a trajectory leaves $\ol{V}$ at $v_+$ and then hits $\ol{V}$ at $v_-$,
    we have $\zeta(v_-)>\zeta(v_+)$.
  \end{proof}

  \begin{remark*}
    By compactness of $N'$, the condition $\partial_\eta\zeta>0$ can quickly lead to a universal constant $C_{\eta\zeta}$ such that $\zeta(v_-)>\zeta(v_+)+C_{\eta\zeta}$.
  \end{remark*}

  The next lemma is the key step in the proof of Proposition~\ref{prop:mu_is_good}.

  \begin{lemma}\label{lem:improving}
    Suppose $X$ is a closed submanifold of $\Sigma$, and let $U$ be an open subset of $X$
    such that
    neither $\upsilon$ nor $\eta$ is tangent to $X$ at any point of $\ol{U}$.
    Let $V$ be as in Lemma~\ref{lem:leaving}.
    %Assume also that for any $u\in\partial U$, we have $\partial_{\upsilon}\zeta(u)>0$.
    Let $\varepsilon>0$ and $U'$ be an open subset of $U$ such that $\ol{U}'\subseteq U$. Set $V'\subseteq V$
    to be the set obtained by thickening $\ol{U}'$ as in the proof of Lemma~\ref{lem:leaving} with the same parameters
    $\rho$ and $\delta$.

    Then $\zeta$ can be altered to a function $\wt{\zeta}$ such that
    \begin{itemize}
      \item $\partial_{\upsilon}\wt{\zeta}(u)>0$ for all $u\in \ol{U'}$;
      \item $|\wt{\zeta}(u)-\zeta(u)|<\varepsilon$;
      \item $\wt{\zeta}(u)=\zeta(u)$ away from $V$.
    \end{itemize}
  \end{lemma}

    \begin{figure}
      \input{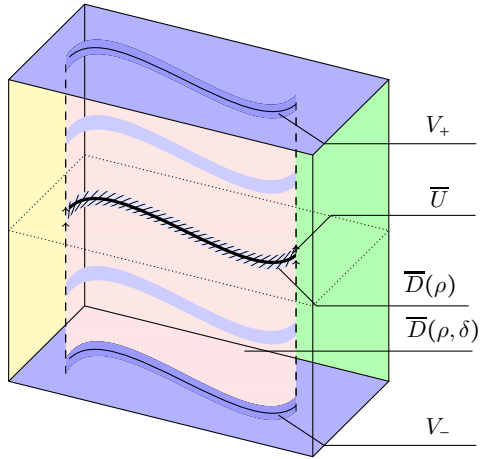}
      \caption{Proof of Lemma~\ref{lem:improving}. Construction of $\ol{D}(\rho)$.}\label{fig:olUp}
    \end{figure}
   % \npar{This figure would look better if the labels were at the end of the lines. They would be more visible that way. I think some of the notation needs updating too, e.g. $\ol{D}(\rho)$ and $\ol{D}(\rho,\delta)$ should be there I think. }{Is it better now? MP: Yes.}

  \begin{proof}
    In the proof of Lemma~\ref{lem:leaving}, we chose a Riemannian metric
    on $N$.
    We will use the same metric in this proof.
    We begin with yet another thickening of $\ol{U}$, this time
    in the direction of $\upsilon_u$. To this end, we choose $\theta_0>0$ with $\theta_0<\rho$, and we allow ourselves to decrease $\theta_0$ further, if need be.
    Consider the line bundle $E$ over $\ol{U}$ spanned by the vector field $\nu_u$. Note that $E\subseteq L^{\perp}$.
    The disc bundle $DE(\theta_0)$ is associated with the line bundle $E$, and the fibres are assumed to be closed and of radius $\theta_0$. In particular, $DE(\theta_0)$ can be regarded as a subbundle of $\wt{D}(\rho)$.
    The $\exp$ map takes $DE(\theta_0)$ diffeomorphically onto its image. We denote this image by $E_\upsilon(\theta_0)$.

    We need to introduce some notation. A point $z$ in $E_\upsilon(\theta_0)$ is the image of a point $(u,\theta(z))\in DE(\theta_0)$, where $u$ is the base and $\theta(z)\in[-\theta_0,\theta_0]$ is the coordinate in the fibre.
    We also let $E_u\subseteq D_\upsilon(\theta_0)$, for $u\in \ol{U}$,
    be the image of the fibre of $DE(\theta_0)$ over $u$ under the $\exp$ map.

  For a point $v\in \ol{D}(\rho)$, we let $v_-,v_+$ denote the entry point and the exit point, respectively, of the trajectory of $\eta$.
  We choose $\delta$ to be small enough so that
  $\zeta(v_-)>\zeta(v_+)+\varepsilon$.
  In the rest of the proof, we will be assuming that $\eta$ inside $V$ is rescaled in such a way that
  a trajectory of $\eta$
  hitting $v_-$ at $t=0$, hits $\ol{D}(\rho)$ at time $t=1/2$ and $v_+$ at time $t=1$. This procedure is needed only to make the formulae more concise.

  Our aim is to construct a family of orientation preserving self-diffeomorphisms $\lambda_{v_-}$ of $[0,1]$, smoothly depending
  on a parameter $v_-\in \piin V$. We require that $\lambda_{v_-}$ be the identity:
  \begin{itemize}
    \item if $v_-$ belongs to the boundary of $\piin V$, that is, if the trajectory of $\eta$ through $v_-$ hits
      $\partial D_\upsilon(\theta_0)$ before leaving $V$;
    \item if the trajectory from $v_-$ hits $\ol{U}$ before leaving $V$.
  \end{itemize}
  Next, suppose $v_-$ passes through a point $v$ in $E_u$, with $u\in \ol{U}'$. Suppose $\theta(v)<\theta_0/2$.
  We set $\lambda_{v_-}$ to be the quadratic (or linear in the degenerate case) map such that $\lambda_{v_-}(0)=0$,
  $\lambda_{v-}(1)=1$ and the condition on $\lambda_{v_-}(1/2)$ that is going to be specified momentarily.
  For $t\in[0,1]$, we set $\kappa(v_-,t)\in N$ to be point on the trajectory of $\eta$ through $v_-$ at time $t$,
  so that $\kappa(v_-,0)=v_-$ and $\kappa(v_-,1)=v_+$. As the trajectory of $\eta$ hits
  $\ol{D}(\rho)$ at time value $1/2$, we have $\kappa(v_-,1/2)\in\ol{D}(\rho)$. Note that the path $t\mapsto\kappa(v_-,\lambda_{v_-}(t))$
  specifies the same trajectory of $\eta$ as $t\mapsto\kappa(v_-,t)$, but after reparametrisation.

  The last condition specifying $\lambda_{v_-}$ is then
  \begin{equation}\label{eq:zeta_mu}
    \zeta(\kappa(v_-,\lambda_{v_-}(1/2)))=\zeta(\kappa(u_-,1/2))+\theta(v)=\frac12+\theta(v),
  \end{equation}
  where the last equality follows from the second item of the bullet list above.
  \begin{figure}
    \input{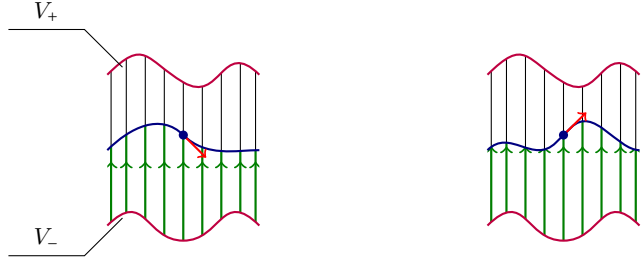}
    \caption{Proof of Lemma~\ref{lem:improving}. The height function is $\zeta$ on the left and $\wt{\zeta}$ on the right. The diffeomorphisms $\lambda$ preserve the trajectories of $\eta$,
    which are represented by vertical lines.}\label{fig:lambda}
  \end{figure}
  The values at $0$, $1/2$, and $1$ specify $\lambda_{v_-}$ uniquely. Furthermore, it is easy to see that a quadratic function $p$ attaining
  values $0$ at $0$ and $1$ at $1$ restricts to a  diffeomorphism of $[0,1]$ if $p(1/2)\in(1/4,1/2)\cup(1/2,3/4)$, where $p(1/2) = 1/2$ is the case corresponding
  to the linear function. To see this, we write the general form of $p$ as $p(x)=\frac{x(x-a)}{1-a}$, with $a\in\R\setminus\{1\}$. Such a function
  is a diffeomorphism of the interval $[0,1]$ if and only if $a\notin[0,2]$.
With $p(1/2)=q$, we have $a=\frac{4q-1}{4q-2}$. We have $a>2$ if $q\in(1/2,3/4)$, and $a<0$ if $q\in(1/4,1/2)$.

  Moreover, $\lambda_{v_-}$ depends smoothly on points $v_-\in\piin V$ such that
  a trajectory through $v_-$ hits $E_\upsilon(\theta_0/2)$, the image under $\exp$ of the subbundle of the disc bundle of $E$
  whose fibres are discs of radius $\theta_0/2$. Note that the condition that $\ol{U}'\subseteq U$ guarantees that
  the definition of $\lambda_{v_-}$ does not contradict the points in the itemized list.

  If $\theta$ is sufficiently close to $0$ (that is, if $\theta_0$ is sufficiently small), then
  \[\smfrac{1}{2}+\theta(v)=\zeta(\kappa(u_-,\lambda_{u_-}(1/2)))+\theta(v)\in(\zeta(\kappa(v_-,0)),\zeta(\kappa(v_-,1))),\]
  If $|\theta(v)|<\frac14$, which can be guaranteed by choosing $\theta_0<\frac14$, then $\lambda_{v_-}$ is indeed a diffeomorphism.

  We have defined $\lambda_{v_-}$ on the boundary of $V_-$ and on all the set of points $v_-$ such that a trajectory of $\eta$ through $v_-$
  hits $E_{\upsilon}(\theta_0)$.
  We extend $\lambda_{v_-}$ smoothly through all points on $V_-$, on which it has not been defined yet. A possible extension is to specify
  a smooth map
  $v_-\mapsto \lambda_{v_-}(1/2)$, taking values in $[1/4,3/4]$ since, as mentioned above, this gives rise to a smooth family
  of diffeomorphisms of $[0,1]$.
  %\begin{remark}\label{rem:improving}
  %  In proving Corollary~\ref{cor:improving} below, we will impose extra conditions on this extension.
  %\end{remark}

  Now for a point $v\in V$ we set
  \[\wt{\zeta}(v)=\zeta(\lambda_{v_-}(t_v)),\]
  where $t_v$ is the time to get from $v_-$ to $v$; see Figure~\ref{fig:lambda}. With this definition, by \eqref{eq:zeta_mu}, for $v\in L_u$
  we have $\wt{\zeta}(v)=\wt{\zeta}(u)+\theta_u$. In particular, $\partial_{\upsilon}\wt{\zeta}(u)=1$.
  Furthermore, as $\zeta(v_-)>\zeta(v_+)-\varepsilon$, and $\zeta(v_-)<\zeta(v),\wt{\zeta}(v)<\zeta(v_+)$,
  we obtain that $|\wt{\zeta}(u)-\zeta(u)|<\varepsilon$ as well.
  \end{proof}
  %\begin{corollary}\label{cor:improving}
  %Suppose $U$ is such as in Lemma~\ref{lem:improving}. If $\partial_{\upsilon}\zeta(u)>0$ for all $u\in\partial U$,
  %then we can arrange that $\partial_{\upsilon}\wt{\zeta}(u)>0$ everywhere on $\ol{U}$.
  %\end{corollary}
  %\begin{proof}
  %  The set $\partial U$ is compact, hence $\partial_{\upsilon}\zeta(u)>c_0$ for all $u\in\partial U$ and some constant $c_0>0$.
  %  We set $U'$ to be the set of points in $U$, where $\partial_{\upsilon}\zeta(u)<\frac{c_0}{2}$.

  %  We improve the function $\zeta$ in $V$ in the same manner as in Lemma~\ref{lem:improving}. As alluded to in Remark~\ref{rem:improving},
  %  we need to be more precise on the way we extend $\lambda_{v_-}$. To this end, choose a cut-off function $\phi\colon V_-\to[0,1]$,
  %  vanishing on $\partial V_-$ and equal to $1$ on $V'_-$. For $v_-$ such that the trajectory of $\eta$ through $v_-$
  %  hits $v\in L_u$ with $u\in \ol{U}\setminus U'$, we define first by analogy to \eqref{eq:zeta_mu}:
  %  \[
  %    \zeta(\lambda^0_{v_-}(1/2))=\zeta(\lambda^0_{u_-}(1/2))+\theta_v,
  %  \]
  %and then $\lambda_{v_-}()$. \textcolor{blue}{Finish!}
  %\end{proof}

  We continue proving Proposition~\ref{prop:mu_is_good}.
  First, consider the map $\Pi\circ G_0\colon T_1\to Y$ (here $\Pi$ is the projection along $\nabla F_0$). As $\dim\ker D\Pi=1$,
  there is a Thom-Boardman stratification of $\Pi\circ G_0$; see Subsection~\ref{sub:thom_boardman}.  We denote
  by $\Sigma_m$ the part of the stratum $\beta^{\mathbf{m}} \Pi$ of this map contained in $T_1$. This is a smooth manifold with boundary. %\ypar{Added this}
  There is an index $m_0$ such that $\Sigma_{m_0}$ is nonempty and all subsequent $\Sigma_i$ are empty
  (there is a universal bound of $m_0$ in terms of the dimensions $n$ and $k$, but we do not need it).
  %We assume
  %that $\Sigma_{m}$ is transverse to $N_{\partial}$ for all $m$.

  By definition,
  $\Pi\circ G_0$ restricted to any $\Sigma_{m}\setminus\Sigma_{m+1}$ is an immersion. That is, $\upsilon_u$
  is not tangent to $\Sigma_{m}\setminus\Sigma_{m+1}$. However, $\eta$ might be tangent to $\Sigma_{m}\setminus\Sigma_{m+1}$.
  Therefore, consider the second projection $\Pi_\eta\colon N'\to f_0^{-1}(a+\varepsilon)$ along the trajectories of $\eta$.
  Let $\Sigma_{m,s}$ stand for the $s$th Thom--Boardman stratum $\beta^\mathbf{s}\Pi_\eta|_{\Sigma_{m}\setminus\Sigma_{m+1}}$.
  By Lemma~\ref{lem:TBP},
  $\Sigma_{m,s}$ is a smooth manifold. %\npar{Are there results being used from elsewhere about this stratification? If so it would be good to reference them.}{I reference smoothness, this is important.}
   Again, for any $m$, there is
  $s_m \in \mathbb{N}$ such that $\Sigma_{m,s_m}$ is nonempty and $\Sigma_{m,i}$ is empty for all $i>s_m$.  We can now finally define
  \[\kappa=1+\sum_m s_m.\]
  We proceed to prove that $\partial_\upsilon\zeta(u)>0$ by induction on $i$. Consider $\Sigma_{m_0,s_{m_0}}$. This is a closed submanifold
  of $T_1$. Let $R\subseteq\Sigma_{m_0,s_{m_0}}$ be the set of points for which
  $\eta$ is parallel to $\upsilon_u$. By Lemma~\ref{lem:positive_multiple},
  on $R$, $\eta$ is a positive
  multiple of $\upsilon_u$. %In particular, since $\partial_\eta\zeta>0$
  In particular, there is a neighbourhood $U_R$ of $R$
  such that $\partial_{\upsilon}\zeta>0$ everywhere on $\ol{U}_R$. We consider
  $\ol{U}:=\Sigma_{m_0,s_{m_0}}\setminus\Int U_R$, letting $U$ be the interior of $\ol{U}$. We let $U'=U\cap T_2$
  and $\varepsilon<\frac{1}{4\kappa}$.
  \begin{figure}
    \input{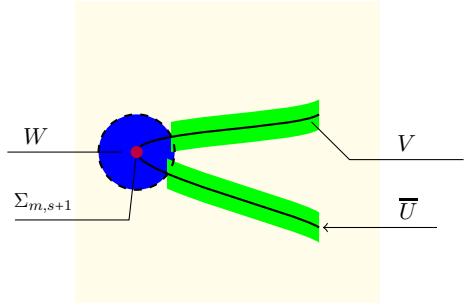}
    \caption{Proof of Proposition~\ref{prop:mu_is_good}. Constructing $\ol{U}$.}\label{fig:eta}
  \end{figure}
  Using Lemma~\ref{lem:improving} we improve $\zeta$ on $U$ in such a way that
  $\partial_\upsilon\zeta(u)>0$ for all $u\in \ol{U}'$. This means that $\partial_{\upsilon}\zeta(u)>0$
  everywhere on $\Sigma_{m_0,s_{m_0}}\cap \ol{T}_2$.

  We proceed with the induction step, inducting on $i$ in $T_i$. Suppose that $\partial_{\upsilon}\zeta(u)>0$
  for $u\in \Sigma_{m,s+1}\cap \ol{T_i}$. We aim to improve $\zeta$ so that this also holds on $\Sigma_{m,s}\cap \ol{T}_{i+1}$.
  The procedure is analogous to the first step.
  Let $R$ be the subset of $\Sigma_{m,s}$ such that $\eta$ is parallel to $\upsilon$. For $u$ in the compact set
  $(R\cup\Sigma_{m,s+1})\cap \ol{T_i}$, we already know that $\partial_{\upsilon}\zeta(u)>0$. In fact, on $R\cap \ol{T_i}$ this follows from Lemma~\ref{lem:positive_multiple},
  while on $\Sigma_{m,s+1}\cap\ol{T_i}$ this is the induction assumption. There exists an open subset $W\subseteq \Sigma_{m,s}$
  containing $R\cup\Sigma_{m,s+1}$ such that $\partial_{\upsilon}\zeta(u)>0$ on the whole of $\ol{W}$.
  We set $U=\Sigma_{m,s}\setminus W$ and $U'=U\cap T_{i+1}$; see Figure~\ref{fig:eta}.
  We apply Lemma~\ref{lem:improving} to obtain that $\partial_{\upsilon}\zeta(u)>0$ for $u \in \ol{U'}$. This means
  that $\partial_{\upsilon}\zeta(u)>0$ on $\ol{W}\cup\ol{U'}$ which contains $\Sigma_{m,s}\cap\ol{T_{i+1}}$.
  This completes the induction step.

  The inductive argument proves the statement until we reach $\Sigma_{m,0}$, that is, the statement
  holds for the whole of $\Sigma_m$. We declare $\Sigma_{m-1,s_m+1}$ to be $\Sigma_m$,
  and we proceed with induction.
  The procedure continues until we conclude that $\partial_\upsilon\zeta>0$ everywhere
  on $\Sigma\cap T_{\kappa}$.

  %\npar{What is going on now? We have defined $\zeta$, now we use it to define $\mu$? I was thinking we could say that, to set the scene for this paragraph.}{Done.}
  The final function $\zeta$ has been constructed. Now, as the final step in the proof, we construct the function $\mu$ using the function $\zeta$. We then show that $\mu$ satisfies the required conditions.
  Recall from Lemma~\ref{lemma:mu-exists} that constructing a function satisfying \ref{item:mu_0}, \ref{item:mu_1}, and \ref{item:mu_2} was relatively straightforward. The key here is that we also need $\partial_{\upsilon} \mu(u) \geq 0$.

  To this end, choose a smooth function $\chi\colon[0,\frac54]\to[0,1]$ with the following properties.
  \begin{enumerate}[label=(CH-\arabic*)]
    \item $\chi$ maps $[0,\frac12]$ to $0$;\label{item:CH1}
    \item $\chi$ maps $[\frac34,\frac54]$ to $1$;\label{item:CH2}
    \item $\chi'(x)>0$ for $x\in(\frac12,\frac34)$.\label{item:CH3}
  \end{enumerate}
  We set
  \begin{equation}\label{eqn:final-defn-of-mu}
    \mu(u)=\chi(\zeta(u)).
  \end{equation}
 % \npar{For the proof below, it was promised above that the conditions on $\zeta$, \ref{item:zeta_0} etc, would be used to see that the conditions on $\mu$, \ref{item:mu_0} etc are satisfied. It would be good if the next couple of paragraphs explain where the $\zeta$ conditions are used. Also the conditions on $\chi$. I made a start on this, but \ref{item:zeta_0} is not invoked yet, nor is the third property of $\chi$. Presumably they should be.}{Done}
  To prove \ref{item:mu_0} we use the property of \ref{item:zeta_0} together with \ref{item:CH3}. The chain rule yields that
  $\partial_\eta\mu>0$ except at the locus where $\chi=\{0,1\}$. At these latter points,  $\partial_\eta\mu=0$.

  The same argument applies
  for the derivative along $\upsilon_u$, proving the last property for $\mu$ in the statement of Proposition~\ref{prop:mu_is_good}.

  Next, note that subsequent applications of Lemma~\ref{lem:improving} will not increase the value of $\zeta$
  by more than $\frac14$ (we apply this lemma $\kappa$ times, and we choose $\varepsilon<\frac{1}{4\kappa}$).
  It follows from \ref{item:zeta_1} that $\zeta$ is less than $\frac12$
  at all points that are connected with $S_1\setminus S_\kappa$
  by a trajectory of $\eta$. Therefore, $\mu=0$ on that set by \ref{item:CH1}. Hence we can extend $\mu$ to the whole of $f^{-1}(a+\varepsilon,b+\varepsilon)$
  by declaring it to be zero away from $S_1$. This proves \ref{item:mu_1} and the second part of~\ref{item:mu_2}.
  The fact~\ref{item:zeta_2} that initially $\zeta\ge 1$ on $\cK_-\cap f^{-1}[a+2\varepsilon,b+\varepsilon]$ implies that after all the alterations we have
  $\zeta\ge\frac34$ on that set. Hence by \ref{item:CH2}, we have that $\chi\circ\zeta=1$, proving \ref{item:mu_2}.  We have proven that \ref{item:mu_0}, \ref{item:mu_1}, and \ref{item:mu_2} hold. Also, we argued that  $\partial_{\upsilon}(\mu)\ge 0$. So $\mu$ satisfies all the required conditions.
 % \npar{In the proposition there is one more condition on $\partial_{\upsilon} \mu(u)$. What about that? }{It is proved after \ref{item:mu_0} and before \ref{item:mu_1}.}
\end{proof}

\subsection{Enforcing condition~\ref{item:LR2}: arranging~\ref{item:ir_disjoint_image}}\label{sub:LR2_gen}

We have shown that for any $\nabla F_0$, the function $\mu$ can be changed in such a way that $G_\tau$ is a regular homotopy.
Now, we will alter $\nabla F_0$ to make sure that item~\ref{item:ir_disjoint_image} of Addendum~\ref{lem:lift_rearr_immersion} is satisfied. Recall that this says that
if $N=N_1\sqcup\dots\sqcup N_\ell$ and the images $N_i$ and $N_j$ $($for all $i\neq j)$ under $G_0$ are disjoint, then $G_\tau(N_i)\cap G_\tau(N_j)=\emptyset$ for all $\tau$ whenever $i\neq j$.

While pushing a point on $G_0(N)$ via the lift map $P$ of Subsection~\ref{sub:LR2_move} we might create new double points. In fact,
we create a self-intersection of $G_\tau$ each time that $P(G_0(u),\tau\mu(u))$ hits another component of $G_0(N)$. In particular,
self-intersections are created only among points that are identified under the map $\Pi$. We will use this principle to prove the next result.

\begin{proposition}\label{prop:disjoint}
  Suppose $N=N_1\sqcup\dots\sqcup N_\ell$ and $G_0(N_1),\dots,G_0(N_\ell)$ are pairwise disjoint. Then there exists
  a vector field $\nabla F_0$ such that $G_\tau(N_1),\dots,G_\tau(N_\ell)$ are pairwise disjoint for every $\tau\in[0,1]$.
\end{proposition}

\begin{proof}
  Choose a grim vector field $\xi$ for $F_0$.
  It is usually not generic, in the sense that the map $\Pi$ induced by $\xi$
  is not Thom-Boardman (see Definition~\ref{def:TB}), in the sense that the Thom-Boardman strata are not smooth of the expected dimension. However, $\xi$ has the following important property.

  \begin{lemma}\label{lem:no_connection}
    No two points $z_i,z_j$ with $z_i \in G_0(N_i)\cap\Omega'$ and $z_j\in G_0(N_j)\cap\Omega'$, $i\neq j$, are connected by a trajectory of $\xi$.
  \end{lemma}

Recall that $\O' = F_0^{-1}(a-\varepsilon,b+\varepsilon)$ as in \eqref{eqn:defn-some-submanifolds}.

  \begin{proof}
    The result relies on property~\ref{item:LR1}. If $z_i$ and $z_j$ are in distinct connected components of $G_0(N)$, the only trajectory of $\xi$ that could possibly connect them has to lie in the zeroth stratum. As $z_i,z_j$ are both on the first stratum or deeper, the only possibility
    that they can be connected by a trajectory in a shallower stratum is that both $z_i$ and $z_j$ are critical points of $F_0$ (this follows from the definition of the grim vector field; see Definition~\ref{def:grim}).

   However,  the fourth item of Lemma~\ref{lem:moving} implies that all critical points of $F_0$ in $\O'$ belong to the same component $G_0(N_i)$.
    This means that even if there is a trajectory in the zeroth stratum, it connects only critical points on the same connected component.
  \end{proof}

  \begin{remark*}
    We could also argue that there are no trajectories on the zeroth stratum for dimensional reasons. However, the present argument
    will also be used in proving an analogous result when lifting paths of death, where a dimension counting argument does not work.
  \end{remark*}

Now suppose that $p_1,\dots,p_r$ are critical points of $F_0$ in $\Omega'$ ($p_-$ and $p_+$ being in this set). The critical point $p_i$  is assumed to belong to $G_0(N_{j_i})$. %There are only possibilities that $j_1=j_2=\dots$, or $r=2$, but we are not going to use this.
  Choose balls $B_1,\dots,B_r$ around these points.

  \begin{lemma}\label{lem:no_connection_2}
    For sufficiently small balls $B_1,\dots,B_r$, there is no trajectory of $\xi$ that passes through $B_i$ and a component $G_0(N_j)$ with $j\neq j_i$. Also,
    there are no trajectories that connect two balls $B_i$ and $B_{i'}$ whenever $j_i\neq j_{i'}$.
  \end{lemma}

  \begin{proof}[Sketch of proof]
    If for every choice of balls there were a trajectory violating one of these conditions, on taking smaller and smaller balls, and passing to a limit we would conclude that
    there is a broken trajectory between two different components of $G_0(N)$. But such a broken trajectory would be composed
    of actual trajectories connecting components of $G_0(N)$ (no such trajectory can terminate at the zeroth stratum, because there are no critical points at depth $0$). Therefore, there would be at least one unbroken trajectory connecting components of $G_0(N)$, contradicting Lemma~\ref{lem:no_connection}.%\npar{This might not be a problem, but is there a chance that the limit could be a broken trajectory?  }{Added a few sentences of explanation.}
  \end{proof}

  Now modify the vector field $\xi$ inside $B_1,\dots,B_p$ so that $\xi$ has no critical points in $\Omega'$.
  That is,  replace the coordinate of $y_{11}$ in \eqref{eq:localgrim} from $\sum y_{1j}^2$ with $\phi+(1-\phi)\sum y_{1j}^2$,
  where $\phi$ is a bump function supported at $B_-$ and equal to $1$ at $p_-$ (respectively supported on $B_+$ and equal to $1$
  at $p_+$). Let $\xi'$ be the replaced vector field. It clearly satisfies the statement of Lemma~\ref{lem:no_connection_2},
  so it also satisfies the statement of Lemma~\ref{lem:no_connection}.

  Now note that the condition in the statement of Lemma~\ref{lem:no_connection} is open. Hence, we may perturb $\xi'$ to $\xi''$,
  in such a way that $\xi''$ satisfies the genericity condition of Lemma~\ref{lem:generic_Pi} and such that no trajectory of $\xi$
  connects two different components of $G_0(N)$. Now $\xi''$ is gradient-like for $F_0$ in $\Omega'$.
  We choose a metric in such a way that $\nabla F_0=\xi''$ (compare Theorem~\ref{thm:grim_exist}).
\end{proof}

By Proposition~\ref{prop:disjoint}, the previous proof with this choice of $\nabla F_0$ gives the desired statement, with the addition that \ref{item:ir_disjoint_image} holds: if $G_0(N_i) \cap G_0(N_j)$ for $i \neq j$, then no new intersections are created between the different components.

\subsection{Condition~\ref{item:LR2} for small dimensions: arranging \ref{item:ir3}}\label{sub:lowdim}
In this subsection,
%As~\ref{item:LR2} has already been addressed if $k\ge 3$,
we focus on the case $k=2$. The first lemma holds for $k > 2$ as well, however, so we state it in general.

\begin{lemma}\label{lem:automatic}
Let $k \geq 2$.
  Suppose $h_->n-k$ or $h_+<k$. Then, for generic $G_0$, Condition~\ref{item:LR2} is automatically satisfied.
\end{lemma}

\begin{proof}
  If $h_->n-k$, the manifold $\cK_-$ has dimension strictly less than $k$. The set of points in $N$ that are mapped to the second stratum
  by $G_0$ has codimension $k$, so if $h_->n-k$, this set is generically avoided by $\cK_-$.
  If $h_+<k$, then $\dim \cK_+<k$, and we repeat the argument.
\end{proof}

\begin{corollary} Suppose $k=2$ and $n \in \{1,2,3\}$. Then condition~\ref{item:LR2} is satisfied for generic $G_0$.
    %\npar{Why is the $n=4$ case considered here? It's not part of \ref{item:ir3}, so maybe it would be better to omit it here.}{Moved to a remark. Maybe not the best strategy. MP: seems fine.}
\end{corollary}
\begin{proof}
  Suppose $n=3$. The only situation when $\dim\cK_+\ge k$ is that $h_+=2,3$. In both cases, we know that $h_-\ge 2,3$, because $h_-\ge h_+$.
  Then $h_->n-k=1$, so we apply~Lemma~\ref{lem:automatic}. The argument for $n=1,2$ is analogous.
\end{proof}

\begin{remark*}
  If $n=4$, Condition~\ref{item:LR2} is satisfied for generic $G_0$ unless $h_-=h_+=2$. Indeed, for $n=4$, $\dim\cK_+\ge 2$, $\dim\cK_-\le 2$ and $\dim\cK_+=\dim\cK_-$ can hold only when $h_-=h_+$.
\end{remark*}

\subsection{Proof of Lifting Rearrangement Lemma~\ref{lem:lift_rearr}}\label{sub:summary_rearr}
We are now in position to prove Lifting of Rearrangement Lemma without assuming \ref{item:LR1} and~\ref{item:LR2} at the beginning.
The proof consists of improving the initial pair $(F_0,G_0)$ so as to assert \ref{item:LR2}. Then, we lift rearrangement. The key
difficulty we address in this subsection is to show that the path $(F_\tau,G_\tau)$ we construct is indeed a weak lift, that is $F_\tau\circ G_\tau$ is left-homotopic to the original path. This relies on patiently controlling each step improving $(F_0,G_0)$ to satisfy
conditions~\ref{item:LR1} and then~\ref{item:LR2}. A schematic of the proof is given in Figure~\ref{fig:concat_rearr}.
\begin{figure}
  \input{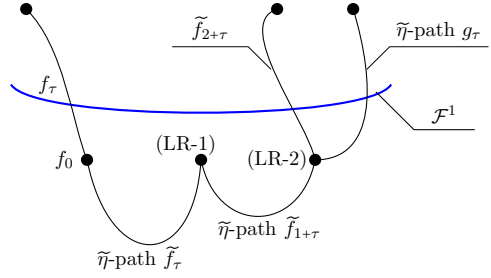}
  \caption{Paths created in the proof of Lifting Rearrangment Lemma~\ref{lem:lift_rearr}.}\label{fig:concat_rearr}
\end{figure}

Consider an immersed Morse function $F_0 \colon \O \to \R$ and a vector field $\eta$ on $N$ that is gradient-like for $f_0=F_0\circ G_0$.
We denote  the original $\eta$-path of rearrangement by $f_\tau$.

At first, we guarantee condition~\ref{item:LR1}.
Lemma~\ref{lem:moving} creates a double path $(F_\tau,G_\tau)$, $\tau\in[0,1]$,
such that $F_\tau$ is a $\xi$-path of rearrangement, $F_\tau$ is supported
away from $\cL_-\cup\cL_+$. Set $\wt{f}_\tau:=F_\tau\circ G_\tau$. Note that $F_\tau$ is not moving any critical points at depth $1$,
hence $\wt{f}_\tau$ is actually a path of excellent Morse functions.
As $\xi$ was constructed independently from $\eta$, we cannot expect that $\eta$ is gradient-like for $\wt{f}_1$.
However, it is so near $\cK_-\cup\cK_+$, which is enough for our purpose, as we now explain.
\begin{lemma}\label{lem:extend_eta}
  There is a vector field $\wt{\eta}$ on $N$ that agrees with $\eta$ near $\cK_-$ and $\cK_+$ and such that $\wt{f}_\tau$
  is a $\wt{\eta}$-path of excellent Morse functions.
\end{lemma}
\begin{proof}
  Choose $U_1,U_2\subseteq N$ to be two open subsets with the property that $U_1\supset\ol{U_2}\supset(\cK_-\cup\cK_+)$ and such that $\wt{f}_\tau$
is supported on $N\setminus U_1$. This is possible because Lemma~\ref{lem:moving} uses only safe rearrangements. Shrink $U_1$
so that $U_1$ contains only $q_-,q_+$ as critical points of $f_0$.

Let $\phi$ be a bump function equal to $1$ on $U_2$ and supported on $U_1$. Let $\eta_\xi$ be the pull-back of $\xi$ via $G_1$, as in Subsection~\ref{sub:pull_back}. Then  $\eta_\xi$ is a gradient-like vector field for $g_\tau$. Also, restricted to $U_1$,
$\eta$ is gradient-like vector field for $\wt{f}_\tau$, because ${\wt{f}_\tau}|_{U_1}={\wt{f}_0}|_{U_1}$. Set $\wt{\eta}=\phi\eta+(1-\phi)\eta_{\xi}$. Then  $\wt{\eta}$ is gradient-like for $\wt{f}_1$ and agrees with $\eta$ on $U_2$.
\end{proof}

Next, Proposition~\ref{prop:create_isotopy} applied to $(F_1,G_1,\wt{\eta})$ creates another double path, which we now denote $(F_{1+\tau},G_{1+\tau})$, $\tau\in[0,1]$,
such that $\wt{f}_{1+\tau}:=F_{1+\tau}\circ G_{1+\tau}$ is a path of excellent Morse functions, $\wt{\eta}$ is gradient-like
for all $\wt{f}_{1+\tau}$, and $\wt{f}_{1+\tau}$ does not depend on $\tau$ near $\cK_-\cup\cK_+$. The main outcome of Proposition~\ref{prop:create_isotopy}
is that $G_2$ satisfies \ref{item:LR2}.

Let $\wt{f}_{2+\tau}$ be the $\wt{\eta}$-path of rearrangement starting from $\wt{f}_2=F_2\circ G_2$.
The pair of functions $(F_2,G_2)$ satisfies~\ref{item:LR2}. Consider the $\wt{\eta}$-path $g_{\tau}$, $\tau\in[0,1]$
that starts at $g_0:=\wt{f}_2$ and lifts the critical point $q_-$ above the critical point $q_+$. By Lemma~\ref{lem:extra_cond},
the path can be weakly lifted.  That is, there exists a double path $(F_{2+\tau},G_{2+\tau})$, $\tau\in[0,1]$, such that $\wt{f}_{2+\tau}:=F_{2+\tau}\circ G_{2+\tau}$ is weakly homotopic to $g_{2+\tau}$.
\begin{lemma}\label{lem:gtau_homo}
  The path $\wt{f}_{2+\tau}$ is lax homotopic to $f_\tau$ over $\wt{f}_\tau$, $\tau\in[0,2]$.
\end{lemma}
\begin{proof}
  As $\wt{f}_{2+\tau}$ is left-homotopic to $g_{\tau}$, it is enough to show that $g_\tau$ is lax homotopic to $f_\tau$ over $\wt{f}_\tau$.
  This is precisely the statement of the lemma. We have set $h_{\sigma,0}=\wt{f}_{2\sigma,0}$ is a path of excellent Morse functions
  supported away from a neighbourhood of $\cK_-$ and $\cK_+$. The path $f_\tau$ is an $\eta$-path, while the path $g_\tau$ is a $\wt{\eta}$-path, and the two vector fields agree in a neighbourhood of $\cK_-\cup\cK_+$.
\end{proof}
\begin{corollary}\label{cor:gtau_homo}
  The path $(F_\tau,G_\tau)$, $\tau\in[0,3]$ can be promoted to a weak lift of $f_\tau$.
\end{corollary}
\begin{proof}
  Lemma~\ref{lem:gtau_homo} implies that the paths $(F_\tau,G_\tau)$, $\tau\in[0,3]$ satisfy the assumptions of Lemma~\ref{lem:lax_to_lift}.
\end{proof}

To conclude the proof of Lifting Rearrangment Lemma~\ref{lem:lift_rearr}, we need to show that $G_\tau$
has the properties stated in Addendum~\ref{lem:lift_rearr_immersion}.
This is done with a case-by-case analysis.
\begin{itemize}
  \item If $G_0$ is an embedding, then~\ref{item:LR2} is always satisfied. Therefore, we do not create self-intersections by
    changing $G_0$ in such a way that \ref{item:ir_embedding} holds; %\ypar{Changed this from (LR2) to (IR1).}
%  \item If $k\ge 3$, then the methods in Subsection~\ref{sub:lr2_k3} allow us to construct $G_\tau$ in such a way that $G_\tau$ is
%    a generic immersion for all $\tau$: in fact, Proposition~\ref{prop:isotopy_k3}, guarantees that $G_\tau$ does not create
%    new self-intersections.
  \item If $N$ is a disjoint union $N_1\sqcup\dots\sqcup N_\ell$ with $G_0(N_i)$ and $G_0(N_j)$  disjoint for each $i \neq j$, then
    we use the methods from Subsection~\ref{sub:LR2_gen} to choose an appropriate vector field $\xi$ so that condition~\ref{item:ir_disjoint_image} is guaranteed.
  \item If $k=2$ and $\dim N \leq 3$, then we use Subsection~\ref{sub:lowdim} to show that~\ref{item:LR2} is automatically satisfied.
\end{itemize}

We have constructed a weak lift of the rearrangement by a regular double path, that does not introduce critical points on the zeroth or first stratum, and we have shown that the conditions in Addendum~\ref{lem:lift_rearr_immersion} can be arranged, when appropriate.  This completes the proof of Lemma~\ref{lem:lift_rearr}.%\npar{MP: added a closing paragraph. }{OK, except for the $\eta$-path.}

\section{Lifting paths of death}\label{sec:lifting-paths-of-death}

We now pass to the most important, and technically most difficult question, namely lifting paths of death.
In the proof of Lemma~\ref{lem:lift_death}, we are going to use the finger move, described in detail in Part~\ref{part:finger}.

\begin{lemma}[Lifting paths of death]\label{lem:lift_death}
  Let $G\colon N\looparrowright \O$ be a generic immersion with $\dim N=n$, $\dim\O=n+k$, and $k\ge 2$.
  Suppose $F\colon\O\to\R$ is an immersed Morse function and $f_\tau\colon N\to\R$, for $\tau\in[0,1]$,  is an
  elementary path of death such that $f_0=F\circ G$.

  Then there exists a regular double path $(F_\tau,G_\tau)$ that weakly lifts $f_\tau$.
  %\npar{Must $F_\tau \circ G_\tau$ also be an $\eta$ path? As in the rearrangement lemma.}{Neither here, nor there. MP: In the rearrangement lemma there is some kind of weaker statement that can be made. Is there such a weaker statement here? }

  Suppose moreover that the level sets of the critical points that are cancelled are $a,b$.
  The path $F_\tau$ is supported on $F_0^{-1}[a-\varepsilon,b+\varepsilon]$, while $G_\tau$ is supported on $f_0^{-1}[a-\varepsilon,b+\varepsilon]$, where $\varepsilon>0$ can be chosen as small as we please. In particular, if $f_\tau$ is neat, $G_0$ is neat and $F_0$ is neat,
  then $(F_\tau,G_\tau)$ is neat.
\end{lemma}
As in the case of the Lifting Paths of Rearrangement Lemma~\ref{lem:lift_rearr}, we give the follow-up result, which specifies additional
properties of $G_\tau$ under certain conditions;

\begin{addendum}\label{lem:lift_death_iso}
  The map $G_\tau$ has the following properties.
  \begin{itemize}
    \item If $k>2$, and $G_0$ is an embedding, then $G_\tau$ is an embedding for all $\tau$.
    \item If $k=2$, and $G_0$ is an embedding, then any new self-intersections are only added  via a finite number of finger moves.
    \item If $k=2$ and $n \in \{1,2\}$, then any new self-intersections are only added via a finite number of finger moves.
        If $n=1$ self-intersections only occur at finitely many values of $\tau \in (0,1)$.
	%\npar{Added an extra sentence about $n=1$.}{Great!}
    \item If $N=N_1\sqcup\dots\sqcup N_\ell$ and $G_0(N_1),\dots,G_0(N_\ell)$ are pairwise disjoint, then for all $\tau \in [0,1]$, we also have that $G_\tau(N_1),\dots,G_\tau(N_\ell)$  are pairwise disjoint.
  \end{itemize}
\end{addendum}

The proof of Lemma~\ref{lem:lift_death} occupies the whole of Section~\ref{sec:lifting-paths-of-death}, except for the finger move, which
occupies the whole of  Part~\ref{part:finger}.
Schematically, the proof is similar to the proof of the Lifting Paths of Rearrangement Lemma~\ref{lem:lift_rearr}. First we set up notation,
then we prove Lemma~\ref{lem:lift_death} under the conditions \ref{item:LD1}--\ref{item:LD4} specified
in Subsection~\ref{sub:lift_death_extra}. The next subsections, \ref{sub:LD1}, \ref{sub:LD0}, \ref{sub:LD2}, \ref{sub:LD34}, and~\ref{sub:finger_mover}, show how to arrange the function $F_0$ and the embedding $G_0$ in such a way that the conditions~\ref{item:LD1}--\ref{item:LD4}
(from Subsection~\ref{sub:lift_death_extra}) are satisfied. Finally, in Subsection~\ref{sub:summary_death},
we give a precise summary of all steps in the proof of Lemma~\ref{lem:lift_death}.

\subsection{Notation of the proof of Lemma~\ref{lem:lift_death}}

We let $q_-,q_+$ denote the critical points of $f$ that are cancelled along $f_\tau$. The indices are $h_-$ and $h_+$ ($h_+=h_-+1$).
  By Assertion~\ref{ars:two},
  it is enough to consider $f_\tau$ that is an elementary path of death from Definition~\ref{def:el_path_death}.
  We choose a double neighbourhood for $f$ and we let $\eta$ be a gradient-like vector field for $f$ such
  that this neighbourhood is adapted to the pair of good caps $(D_\eta,A_\eta)$.  (It may help to refer back to Section~\ref{sec:path_of_death} for the terminology.)

 Let $\gamma$ be the unique trajectory of $\eta$ connecting $q_-$ and $q_+$.
 As $\dim\gamma=1$ and the union of strata $\bigcup_{d \geq 2} \O[d]$ is of codimension~2, we may perturb $G_0$ to a generic immersion
 such that $G_0(\gamma)$ is disjoint from the second stratum and all higher strata. Note that a perturbation can be realised as a path
 $G_\tau$ of $F_0$-regular generic immersions such that $F_0\circ G_\tau$ is a path of an excellent Morse functions.
 Hence, from now on, we will be assuming that $G_0(\gamma)$ is disjoint from $\bigcup_{d \geq 2} \O[d]$.

  Set $a=f(q_-)$, $b=f(q_+)$.
  Choose a neighbourhood $U_0$ of $\gamma$
  such that $G_0$ maps $U_0$ to the first stratum. Write  $\cK_- \subseteq N$, respectively $\cK_+ \subseteq N$ for the unstable manifold of $q_-$, respectively for the stable manifold of $q_+$. We choose $\varepsilon>0$ such
  that $\cK_-\cap f^{-1}[a,a+\varepsilon]$ and $\cK_+\cap f^{-1}[b-\varepsilon,b]$ are contained in $U_0$.
  Note that $\gamma \subseteq \cK_- \cap \cK_+ \cap f^{-1}[a,b]$.
  Define $\cL_-=G_0(\cK_-)$,
  $\cL_+=G_0(\cK_+)$. As usual, we set $p_+=G_0(q_+)$ and $p_-=G_0(q_-)$.

  \subsection{Lifting paths of death under extra conditions}\label{sub:lift_death_extra}

  We will structure the proof using the following conditions on $\eta$, $f$, and $G_0$.
  \begin{enumerate}[label=(LD-\arabic*)]
    \item\label{item:LD1} There are no critical points of $F$ at depth $0$ with critical values in $[a-\varepsilon,b+\varepsilon]$;
    \item\label{item:LD0} The unique trajectory $\gamma$ on $N$, between $q_-$ and $q_+$, is mapped to the first stratum;
    \item\label{item:LD2} The subsets $\cK_-$ and $\cK_+$ of $N$ are mapped to the first stratum of $G_0(N)$;
    \item\label{item:LD3} There exists a grim vector field $\xi$ for $F$ such that the pull-back $\wt{\eta}$ of $\xi$ via $G_0$
      gives rise to a path of death that is left-homotopic to a path of death constructed with $\eta$,
     % \mpar{To be made precise, this is not what we use in the further steps.}
      and moreover there are no critical points of $F$ in $F^{-1}[a-\varepsilon,b+\varepsilon]$ other than $p_-$ and $p_+$;
    \item\label{item:LD4} There are no trajectories of $\xi$ from $p_-$ to $p_+$ other than the image of the trajectory $\gamma$.
  \end{enumerate}

  \begin{remark*}
    \ref{item:LD1} and \ref{item:LD2} are the analogues of \ref{item:LR1} and \ref{item:LR2} respectively; however \ref{item:LD2} requires
    that \emph{both} $\cK_-$ and $\cK_+$ are mapped to the first stratum, whereas \ref{item:LR2} required only one of them.

    A condition on the lack of other critical points appears twice, in \ref{item:LD1} and in \ref{item:LD3}.
    First, before we construct an appropriate
    vector field, we are allowed to perform only safe rearrangements (see the discussion in the proof of Lemma~\ref{lem:dim_asc}). The lack of critical points at depth $0$
    is important in ensuring~\ref{item:LD2}. Once we have a vector field extending $\eta$, we have more control when we  perform rearrangements, and this enables us to
    move all critical points away from between $p_-$ and $p_+$.
    Arranging for \ref{item:LD4} to hold is rather difficult: it will require the finger move, which takes the whole of Part~\ref{part:finger} to construct.
    Removing the other critical points and arranging that there are no trajectories other than $\gamma$ are both needed to apply Cancellation Theorem~\ref{thm:grimcanc}.
  \end{remark*}

  \begin{lemma}[Conditional lifting of paths of death]\label{lem:cond_lift_death}
    Suppose conditions \ref{item:LD3} and \ref{item:LD4} are satisfied. Then there exists a path $F_\tau$ such that $(F_\tau,G_0)$  weakly lifts $f_\tau$.
    %\npar{We want this to be a regular double path, no? Or perhaps we just need to comment in the remark below that we will improve it to a regular double path in the summary section?}{A weak lift is a regular double path by definition, see Definition~\ref{def:lifting}. Corollary~\ref{cor:left_homotopy} ensures that we promote a path to a regular double path. I don't thing changes are needed here.}
  \end{lemma}

  \begin{proof}
    The Cancellation Theorem~\ref{thm:grimcanc} uses a vector field $\xi$ to create a regular double path $(F_\tau,G_\tau)$, $\tau\in[0,2]$ that cancels the pair of critical points  $p_-$ and $p_+$. This is possible since \ref{item:LD4} is satisfied.
   % \npar{This relates to an earlier comment: we need to know what that path is in that previous theorem in order to appeal to it here. It is rather too implicit in Theorem~\ref{thm:grimcanc}, I would say. }{Rephrased} %constructed in Theorem~\ref{thm:grimcanc} using the
    %vector field $\xi$ of conditions \ref{item:LD3} and~\ref{item:LD4}.
    Then the composition $F_\tau\circ G_\tau$ is an $\wt{\eta}$ path of death.
    By \ref{item:LD3}, it is left-homotopic to the original path of death constructed with $\eta$. By Corollary~\ref{cor:left_homotopy},
    the path $(F_\tau,G_\tau)$ can be promoted to a weak lift of the original $\eta$-path.
    %\npar{Shouldn't it also be a regular double path?? That doesn't seem to be mentioned here, nor in the summary, but it is required in the statement of Lemma~\ref{lem:lift_death}.}{Added a discussion after Theorem~\ref{thm:grimcanc}, but maybe it is not sufficient.}
  \end{proof}

  \begin{remark*}
    In the proof of Lemma~\ref{lem:cond_lift_death} we did not use properties~\ref{item:LD1}, \ref{item:LD0}, and~\ref{item:LD2}. However these properties will
    be needed to construct a vector field $\xi$ satisfying~\ref{item:LD3} and~\ref{item:LD4}.
  \end{remark*}

  \subsection{Removing critical points from $[a-\varepsilon,b+\varepsilon]$}\label{sub:LD1}
  In this section we arrange that condition~\ref{item:LD1} holds. The procedure uses the rearrangement theorem and the approach is similar to  the approach of Subsection~\ref{sub:LR1}. In particular, we might need to replace the original vector field $\eta$ by another gradient-like vector field, $\wt{\eta}$, which agrees with the original one near $\cK_-\cup\cK_+$.

  \begin{lemma}\label{lem:rearrange_while_keeping}
    There exists a path of Morse functions $F_\tau\colon\O\to\R$,
    such that $F_\tau=F_0$ on $U_0$, $F_1$ has no critical
    points in $[a-\varepsilon,b+\varepsilon]$ at depth $0$ and
    $F_\tau\circ G_0$ is a path of Morse functions on $N$ with constant
    critical values, and such that the support of $F_\tau\circ G_0$ is disjoint from $\cK_-\cup\cK_+$.
  \end{lemma}
  \begin{proof}
    The proof follows the argument of Lemma~\ref{lem:moving} and discovers what are the depths and indices of the critical points that can and cannot be moved out of the region $F_0^{-1}[a-\varepsilon,b+\varepsilon]$. %\ypar{Added this sentence.}
    The only difference with Lemma~\ref{lem:moving} is that $h_+=h_-+1$, where we recall that $h_\pm$ is the index of $q_\pm$.

    We begin with the case $k=2$, which is the most rigid. Choose a grim vector field $\xi$ for $F$. By Lemma~\ref{lem:dim_asc},
    a critical point $p$ can be safely (see Definition~\ref{def:safe_rearrangement}) moved up by a $\xi$-path if $h_p+d_p>h_++1=h_-+2$ (or, if $d_p=0$), and $p$ can be moved down by a $\xi$-path if $h_p+d_p<h_-+1$. Acting as in the proof of Lemma~\ref{lem:moving}
    we can move all critical points with $h_p+d_p\le h_-$ below~$p_-$ and all critical points with $h_p+d_p\ge h_-+3$ above $p_+$. There remain critical points with $h_p+d_p=h_-+1$ or $h_p+d_p=h_-+2$. These points cannot be safely rearranged in general.

    Suppose a point $p'$ is between $p_-$ and $p_+$ and $d_{p'}=0$. If $h_{p'}=h_-+1$, we can safely move it below $p_-$; here any rearrangement with other critical points is safe and possible, because the only critical points
    have $h_p+d_p\ge h_-+1$. Compare the algorithm for handling \ref{item:SR2} in the proof of Lemma~\ref{lem:moving_step_1} above. On the other hand, if $h_{p'}=h_-+2$, we move it up.

    Concatenating the paths of rearrangement, we construct a $\xi$-path $F_\tau$ supported away from $\cL_-\cup\cL_+$. The only critical points of $F_1$ in $F_1^{-1}(a,b)$ are at depth $2$ or more (recall that the critical points at depth $1$ correspond to critical points of $F_\tau\circ G_\tau$, and since we consider a path of death there are none; see Section~\ref{sec:path_of_death}) %\npar{What about those at depth 1? They are not discounted in the hypothesis, and they have not been rearranged in the previous paragraph. Maybe they cannot be there in an elementary path of death? If so, it would be helpful to point that out. }{Points at depth 1 are just the critical points at of $f_\tau$. They are controlled by the assumptions. Added an explanation.}
    and are such that $h_p+d_p=h_-+1$ or $h_p+d_p=h_-+2$. This concludes the case $k=2$.

    \smallskip
    Suppose $k>2$. Lemma~\ref{lem:dim_asc} allows us to safely move a critical point
$p$ below $p_-$ as long as $h_p+d_p<h_-+1$. If
$h_p+(k-1)(d_p-1)>h_+=h_-+1$, we can safely move a point $p$ above
$p_+$. The only
possibility that a critical point cann be safely moved neither above $p_+$
nor below $p_-$ is if $h_p+d_p\ge h_-+1$ and $h_p+(k-1)(d_p-1)\le
h_+=h_-+1$ simultaneously. These two conditions imply that $h_p+d_p\ge
h_p+(k-1)(d_p-1)$, that is, $d_p\le\frac{k-1}{k-2}$. Since $d_p\ge 2$
and $k\ge 3$ by assumptions, the only possibility is that $d_p=2$ and
$k=3$. In that case, we must have $h_p+d_p=h_-+1$.  %$d_p=2$, $k=3$ and $h_p+d_p=h_-$.
     %\npar{I copied from your email, but the email included this strange extra sentence at the end, currently commented out. It says ``$d_p=2$, $k=3$ and $h_p+d_p=h_-$.''
     %But in this case $h_- = h_p+d_p<h_-+1$, so according to what is written above can't we move $p$ below $p_-$? Should it be $h_p+d_p=h_-+1$?}{yes!}
  \end{proof}

  We remark that the rearrangements performed during the proof of Lemma~\ref{lem:rearrange_while_keeping}, might result in $\eta$
  not being a gradient-like vector field for $f_1:=F_1 \circ G_0$. However, $f_1=f_0$ near $\cK_-\cup\cK_+$. Therefore, $\eta$ is a gradient-like vector field for $f_1$ near $\cK_-\cup\cK_+$. By Proposition~\ref{prop:grim_extend}, there exists a gradient-like vector field $\wt{\eta}$ for $f_1$ agreeing with $\eta$ near $\cK_-\cup\cK_+$.%\ypar{Added this}

  We can refine the result if $N$ is a union of disjoint components, and $G_0$ maps all these components to pairwise disjoint stratified manifolds.

  \begin{lemma}\label{lem:rearrange_to_keep_non_intersecting}
    Suppose $N=N_1\sqcup\dots\sqcup N_\ell$, the images $G_0(N_1),\dots,G_0(N_\ell)$ are pairwise disjoint, and $p_+,p_-$ belong to the component $G_0(N_1)$. Then the path of rearrangement from Lemma~\ref{lem:rearrange_while_keeping} can be chosen to additionally push all critical points away from $F_0^{-1}(a,b)$,
    except for those that belong to $N_1$.
  \end{lemma}

  \begin{remark*}
    Unlike in Lemma~\ref{lem:moving}, we need not consider the case where $p_-$ and $p_+$ are in separate components; the existence of a trajectory in $G_0(N)$ connecting $q_-$ to $q_+$ implies that $p_-$ and $p_+$ belong to the same connected component of $G_0(N)$.
  \end{remark*}

  \begin{proof}[Proof of Lemma~\ref{lem:rearrange_to_keep_non_intersecting}]
    Lemma~\ref{lem:rearrange_while_keeping} allows us to move critical points away from $F_0^{-1}(a,b)$ under certain conditions
    on $h_p+d_p$ depending on the codimension $k$. If $k>3$, we are able to move all critical points, so the statement of Lemma~\ref{lem:rearrange_to_keep_non_intersecting} follows from Lemma~\ref{lem:rearrange_while_keeping}.%\ypar{Added this paragraph}

   Assume $k=2$. By Lemma~\ref{lem:rearrange_while_keeping}, only critical points that cannot be safely moved away from $F_0^{-1}(a,b)$
    are those for which $h_p+d_p=h_-+1$ or $h_p+d_p=h_-+2$, and $d_p\ge 2$.

    The proof in that case   %\npar{The second part of what? What is this referring to?}{Rephrased}
     follows the same lines as the proof of the fourth item of Lemma~\ref{lem:moving}, so here we only give a quick sketch. If a critical point $p$ has $h_p+d_p=h_-+1$,
    its descending membrane is disjoint from the ascending membrane of $p_-$ for dimensional reasons. Next, if $p$ and $p_-$
    belong to different components of $G_0(N)$, the membrane of $p$ belongs to a different component of $G_0(N)$ than $\cL_-=G_0(\cK_-)$.
    Moreover, for dimensional reasons, the descending membrane of $p$ is disjoint from all the ascending membranes of critical points
    between $p_-$ and $p$. Hence,~$p$ can be safely moved below the level set of $p_-$.

    In this way, we can push all critical points with $h_p+d_p=h_-+1$ below the level set of $p_-$, except for those that belong
    to the same component of $G_0(N)$ as $p_-$. Reversing this construction, we push all critical points with $h_p+d_p=h_-+2$ above
    the level set of $p_+$, with the exception of critical points belonging to the same component as $p_+$.

    \smallskip
    For $k=3$ we essentially repeat the argument, except that we only need to move critical points with $h_p+d_p=h_-+1$
     and $d_p=2$.
This completes the proof that \ref{item:LD1} can be arranged, and moreover if the $G_0(N_i)$ are pairwise disjoint, the only critical points remaining in $F^{-1}(a-\varepsilon,b+\varepsilon)$ are on the connected component $G_0(N_1)$.
  \end{proof}

  \subsection{Proving \ref{item:LD0}}\label{sub:LD0}
  In order to prove~\ref{item:LD2} (the analogue of \ref{item:LR2}), we need an intermediate step. This is the place where the present situation  differs from the proof of Proposition~\ref{prop:create_isotopy} used for lifting rearrangements in the previous section.  As in Subsection~\ref{sub:LR2_move},
  we will be working with a gradient-like vector field for $F_0$ as an ordinary Morse function. Again, to distinguish these gradient-like
  vector fields from grim vector fields, we use the notation $\nabla F_0$. Note that any gradient-like vector field is a gradient for some choice of Riemannian metric, so this is not a significant abuse of notation.

  \begin{lemma}\label{lem:vector_for_LD2}
    There exists a vector field $\nabla F_0$, gradient-like for $F_0$,  such that with $\Omega'=F_0^{-1}[a-\varepsilon,b+\varepsilon]$ we have:
    \begin{itemize}
      \item the projection $\Pi$ from $\Omega'$ to the level set $F^{-1}(a-\varepsilon)$ is a Thom-Boardman map $($see Definition~\ref{def:TB}$)$;
      \item suppose $N=N_1\sqcup\dots\sqcup N_\ell$ and that $G_0(N_i)$ are pairwise disjoint for $i=1\dots,\ell$. Then there are open balls near critical points of $F_0$, which we can choose to be as small as we please, such that
	there are no trajectories of $\nabla F_0$ that connect two distinct components $G_0(N_i)$ and $G_0(N_j)$.
    \end{itemize}
  \end{lemma}

  \begin{proof}
    The first item follows from the arguments that have already been used in the proof of Proposition~\ref{prop:create_isotopy}.
    Namely, the map $\Pi$ is defned as the projection onto $F^{-1}(a-\varepsilon)$ along the trajectories of $\nabla F_0$. By Lemma~\ref{lem:projection_vs_gradient}, a small perturbation of $\Pi$ can be realized by a perturbation of $\nabla F_0$ within the family of gradient-like
    vector fields for $F_0$. Furthermore, Proposition~\ref{prop:res} guarantees, that any $\Pi$ can be perturbed to a Thom--Boardman map.
    That is, any gradient-like vector field $\nabla F_0$ can be perturbed to a gradient-like vector field satisfying the first condition.
    %\npar{I do not follow this. The lemma referenced does not mention Thom Boardman at all. I think something is missing here. }{Rephrased.}
    %\npar{Is this more a question about the subset of metrics (that determines $\nabla F_0$)? }{No, vector fields. It is made precise in the statement of Lemma~\ref{lem:generic_Pi}. I do not think we should repeat the argument.}
    The second part is proved in the same way as Lemmas~\ref{lem:no_connection} and~\ref{lem:no_connection_2}.
    %\ypar{Added ref to second lemma here, since that one mentions the small balls. }
  \end{proof}

  We need to add an extra condition regarding the vector field $\nabla F_0$.

    \begin{definition}\label{def:good_position}
      We say that $\gamma$ is in \emph{good position} with respect to $G_0$ if for every $w\in G_0(\gamma)$ the trajectory
      of $\nabla F_0$ through $w$ does not hit $G_0(N) \cap \O'$, except at $w$.
    \end{definition}

    \begin{lemma}\label{lem:good_position}\
      \begin{itemize}
	\item If $k>2$, then for a generic gradient-like vector field $\nabla F_0$, $\gamma$ is in a good position with respect to $G_0$.
	\item If $k=2$, and $F_0$ satisfies \ref{item:LD1}, then
	  there exists a regular $F_0$-path of immersions $G_\tau$
 such that $G_\tau$ is fixed on $\gamma$, $\eta$ is a gradient-like vector field
	  for $F_0\circ G_\tau$, for all $\tau \in [0,1]$, and $\gamma$ is in good position with respect to $G_1$.
	\item If $N=N_1\sqcup\dots\sqcup N_\ell$ and $G_0(N_1),\dots,G_0(N_\ell)$ are pairwise disjoint, then $G_\tau(N_1),\dots,G_\tau(N_\ell)$ are also pairwise disjoint for all $\tau \in [0,1]$.
      \end{itemize}
    \end{lemma}

    \begin{proof}
    For each point $z$ in $\gamma$, let $\nu_z$ denote the set of points in $\O'$ that lie on the trajectory
    of $\nabla F_0$ through $G_0(z)$.
    The set $C=\bigcup_{z\in\gamma}\nu_z$ is two-dimensional and contains $G_0(\gamma)$ by construction.
    This set can be reinterpreted as follows: the flow of $\nabla F_0$ induces a projection $\Pi$ of $\Omega'$
    onto the level set $F_0^{-1}(a-\varepsilon)$; see Lemma~\ref{lem:projection_vs_gradient} above and the discussion
    in Subsection~\ref{sub:LR2_move}. Suppose $\Pi$ is generic with respect to $\gamma$, that is, it satisfies the following two conditions. %\ypar{Made the conditions into an enumerate environment.}
\begin{enumerate}
  \item First $\beta^{\mathbf{\ell}} \Pi|_\gamma$ is smooth
    of expected dimension (where $\beta^{\mathbf{i}}\gamma$ denotes the Thom--Boardman stratification, see Subsection~\ref{sub:thom_boardman}).
    As $\dim\gamma=1$, the expected dimension of $\beta^1\gamma$ is negative. This means that $\beta^{\mathbf{1}}\gamma$ is empty, so that $\Pi|_\gamma$ is a generic immersion.
  \item  Secondly, we assume that the set of double points of the image is of expected dimension. This set is of codimension $(n+k-3)$ in $\Pi(\gamma)$. As we assumed that $k\ge 2$, the set of double points is empty unless $n=1$ and $k=2$. This means that $\Pi^{-1}\Pi(\gamma)$ is either a smooth surface (if $n>1$ or $k>3$) or a surface with $1$-dimensional set of singularities: each such singularity is a product of an ordinary double point of a curve and a segment.
\end{enumerate}
    These two conditions imply that in both cases, $\Pi^{-1}\Pi(\gamma)$ is Whitney stratified (Definition~\ref{def:Whitney}).  The set $C$ is an open subset of $\Pi^{-1}\Pi(\gamma)$, so it inherits Whitney stratification. In particular, stratified general position arguments can be applied to $C$ as follows.

    If $k>2$, the set $C$ can be made disjoint from $G_0(N)\setminus G_0(\gamma)$ (which is of codimension $k>\dim C$)
    by perturbing $\Pi$. This proves
    the first item.

    For the remaining part of the proof, we will assume that $k=2$.
    As $C$ is stratified of dimension $2$, a general position argument guarantees that $C$ intersects $G_0(N)\setminus G_0(\gamma)$ at finitely many points,
    and each such point is a smooth point of $C$ (because the singular set of $C$ has dimension $1<\codim G_0(N)\setminus G_0(\gamma)$).
    Moreover, $C$ avoids the second stratum of $G_0(N)$. %\ypar{Deleted next sentence, since it is already included in the definition of $C$.}
    %Among the intersection points of $C\cap (G_0(N)\setminus G_0(\gamma))$
    %we only consider those that are above $\gamma$ (reached from $\gamma$ via the flow of $\nabla F_0$ in positive time).
    Let
    $w_1,\dots,w_t\in N$ be such that $C\cap (G_0(N)\setminus G_0(\gamma))=\{G_0(w_1),\dots,G_0(w_t)\}$;
    see Figure~\ref{fig:gamma_lift}.
    \begin{figure}
      \input{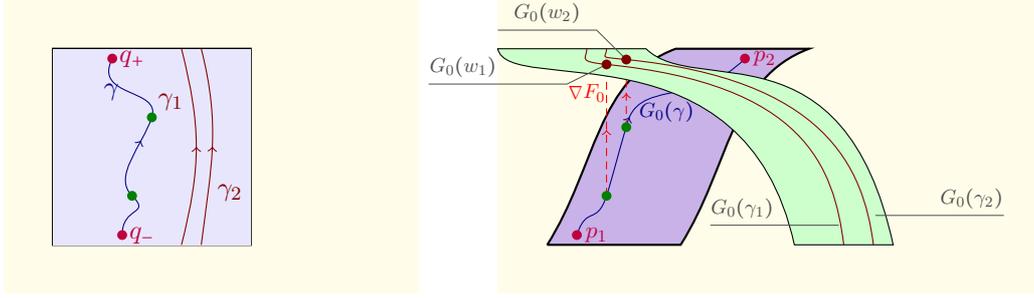}
      \caption{Notation of proof of Lemma~\ref{lem:good_position}.  }\label{fig:gamma_lift}
    \end{figure}
   % \npar{The labels on this diagram overlap with the lines in a few places. }{Is it better? MP: Yes.}

   Similarly to Subsection~\ref{sub:f_ing_immersion}, define $\Sigma \subseteq N$ to be the subset of points $u\in N$ such
    that $\nabla F_0$ is tangent at $G_0(u)$ to the branch of $G_0(N)$.
    Denote  the collection of trajectories of $\eta$ through points on $w_1,\dots,w_t$ by $\gamma_1,\dots,\gamma_t$.
   The curves $\gamma_1,\dots,\gamma_t$ are one-dimensional, while $\Sigma\subseteq N$ is of codimension $k$.
    By general position arguments (choosing appropriate $\Pi$) we may and will assume that all the curves $\gamma_i$ are disjoint from $\Sigma$.

    The original trajectory $\gamma$ is not any of the $\gamma_1,\dots,\gamma_t$, therefore, each of the $\gamma_i$
    either does not terminate at $q_+$,
    or does not start at $q_-$ (or both).
    Suppose $\gamma_1$ does not terminate at $q_+$. As $\gamma_1$ is a trajectory of $\eta$, the fact that $\gamma_1$ does not terminate at $q_+$ implies that $\gamma_1$ is disjoint from $\cK_+$.

    Our aim is to lift $\gamma_1$ so that the point $w_1$ lands above the level $f^{-1}(b+\varepsilon)$, and then to apply induction. To accomplish the first goal, we use the procedure of the proof of Proposition~\ref{prop:create_isotopy}. Set $c_1=f(w_1)$ and $\varepsilon'>0$ such that $c_1-\varepsilon'>a$.

     %Let $U_1$ be the intersection of $\gamma_1$ with the level set $f^{-1}(a+\varepsilon)$.
    Choose a neighbourhood $U_1$ of $\gamma_1\cap f^{-1}(c_1-\varepsilon',b+\varepsilon)$ disjoint from $\cK_+$ and all the other curves $\gamma_2,\dots,\gamma_t$. We assume that all trajectories
    of $\eta$ passing through $U_1$ avoid $\Sigma$ in $f^{-1}([c_1-\varepsilon',b+\varepsilon])$. Such a choice of $U_1$
    is possible, since $\gamma_1$ avoids $\Sigma$ itself. Pick a function
  \[\mu \colon f^{-1}[a+\varepsilon,b+\varepsilon]\to [0,1]\]
  satisfying analogues of
  items~\ref{item:mu_0}, \ref{item:mu_1}, and~\ref{item:mu_2}. Namely:
  \begin{itemize}
    \item $\partial_\eta\mu\ge 0$;
    \item $\mu$ is supported on $U_1$;
    \item $\mu=1$ on $\gamma_1$.
  \end{itemize}
  Existence of such $\mu$ follows from the same arguments as in Lemma~\ref{lemma:mu-exists}.
  %Subsection~\ref{sub:LR2_move} for the case of a function satisfying \ref{item:mu_0}, \ref{item:mu_1} and~\ref{item:mu_2}.
  Note that in this specific case, we need not care about the condition~$\partial_{\upsilon}\mu\ge 0$ as in Lemma~\ref{lem:mu_partial}.
  In fact, $\mu$ is supported on $U_1$ and we have chosen $U_1$ so as to avoid $\Sigma$. The vector field $\upsilon$ was defined
  only at points on $\Sigma$.
  %\npar{The statement says that $G_\tau$ is an immersion. Is that really what is meant here? Not directly a property of $\mu$. Suggest to remove this sentence since the same thing is said, more clearly, in the bullet points below. }{Rephrased}

  This function, together with the choice of the rescaling function $R$ and the lift function $P$ as in the proof of Proposition~\ref{prop:create_isotopy}, allows us to create a regular homotopy lifting
    the trajectory $\gamma_1$ and its neighbourhood
    along the trajectories of $\nabla F_0$, to
    arrange  that no point of $\gamma_1$ belongs to the same trajectory of $\nabla F_0$ as a point on $\gamma$.
    We have the following properties.
    \begin{itemize}
      \item As we already mentioned, the procedure is needed only for $k=2$.
     % \item \psout{The condition $\partial_{\upsilon}\mu\ge 0$ at all points of $\Sigma$ required for the maps to be immersions (compare Subsection~\ref{sub:f_ing_immersion}) is trivially satisfied, because the support of $\mu$ was chosen to be disjoint from $\Sigma$.}
	%\npar{This seems to be repetition from above. Suggest to remove the sentence about this point above, which is phrased confusingly anyway.}{You mean this one? I agree.}
      \item If $N$ is presented as a disjoint sum of $N_i$ with $G_0(N_i)$ pairwise disjoint, then $\gamma_1$ belongs to the same component as $\gamma$. While we might create some new self-intersections, the only new self-intersections we create are within $G(N_1)$.
          %the component of $G_0(N_i)$ containing~$N_1$.
    \end{itemize}
    If $\gamma_1$ terminates at $q_+$, then it cannot start at $q_-$. We push it down instead of up, by the same procedure. An inductive argument constructs the required regular homotopy.
\end{proof}

\subsection{From~\ref{item:LD1} and~\ref{item:LD0} to~\ref{item:LD2}}\label{sub:LD2}
%\ypar{LD1 can be deleted while fighting for LD3, but it is subsumed by (stronger) LD4. I changed the title of the section and tried to clarify the logic throughout the proof.}

\begin{lemma}\label{lem:good_means_ld2}
  Suppose $(F_0,G_0,\eta)$ satisfies~\ref{item:LD1} and~\ref{item:LD0}.
  %If $k>2$, then $G_1(\cK_-)$ and $G_1(\cK_+)$ are mapped to the first stratum. If $k=2$,
  There exists a regular homotopy $G_\tau$ such that $\eta$
  is a gradient-like vector field for $F_0\circ G_\tau$ for all $\tau$ and $G_1(\cK_-)$ and such that $G_1(\cK_+)$ belong to the first stratum
  of $G_1(N)$, i.e.\ such that \ref{item:LD2} is satisfied.  %\npar{Are (LD1) and (LD2) preserved?  I think it would help to mention that here, since it's later important that we get all of LD1 - LD4 simultaneously. }{Added a remark after the statement}
  The path $(F_\tau,G_\tau)$ with $F_\tau=F_0$ is a regular double path.

  Moreover:
  \begin{itemize}
    \item if $G_0(N)$ is an embedding, then $G_\tau=G_0$.
    \item if $G_0(N)$ is a disjoint union of $G_0(N_1),\dots,G_0(N_\ell)$ $($with $N=N_1\sqcup\dots\sqcup N_\ell)$,
  then $G_\tau(N)$ is a disjoint union of $G_\tau(N_1),\dots,G_\tau(N_\ell)$ for all $\tau\in[0,1]$.
  \end{itemize}
\end{lemma}
\begin{remark*}
  The condition~\ref{item:LD2} is stronger than \ref{item:LD0}, that is $G_1$ satisfies \ref{item:LD0} as well. However, the procedure of arranging for \ref{item:LD2} might create extra self-intersections, hence we do not control critical points at deeper strata. That is, $(F_1,G_1)$ need not satisfy \ref{item:LD1}. These critical points will be dealt later, with \ref{item:LD3}, which in turn implies \ref{item:LD1}.
\end{remark*}
\begin{proof}
  If $G_0(N)$ is an embedding then it automatically satisfies~\ref{item:LD2}.

  The proof of Lemma~\ref{lem:good_means_ld2} uses the whole machinery of the proof of Proposition~\ref{prop:create_isotopy}. The presence of the trajectory $\gamma$ prevents us from applying that proposition directly, because in the proof of Proposition~\ref{prop:create_isotopy},
  we were assuming that $\cK_+$ and $\cK_-$ are disjoint. However, thanks to \ref{item:LD0},
  %\npar{?? This is the one we are trying to show. Should it say \ref{item:LD0} here?}{Yes, this is the curse of adding LD2 after the proof was almost complete and the references are completely mess. LD0 in text corresponds to LD-2 in the pdf file.}
   we can apply these arguments with
  only minor modifications. %\ypar{Added this}

  Choose $\varepsilon>0$. As in the proof of Proposition~\ref{prop:create_isotopy}, define $Z$ by Condition~\ref{cond:on_Z}.
  We note that $\cK_-\cap f_0^{-1}(a+\varepsilon)$ belongs to $Z$. Recall that $Z$ is an open subset of the level set $f_0^{-1}(a+\varepsilon)$.

  Let $U_0\subseteq f_0^{-1}(a+\varepsilon)$ be
  a neighbourhood of $\gamma\cap f_0^{-1}(a+\varepsilon)$ such that if $z\in U_0$, if $\lambda_z$ is a trajectory
  of $\eta$ through $z$ (its part in $N' := G_0^{-1}(\Omega')$; see~\eqref{eqn:defn-some-submanifolds}), and if $w\in G_0(\lambda_z)$, then the trajectory of $\nabla F_0$ through $w$
  does not hit $G_0(N)$ at any point in $\O'$ except at $w$ itself. This is a condition analogous to Definition~\ref{def:good_position}.
  Note that $\gamma$ is in good position, so $\gamma\cap f_0^{-1}(a+\varepsilon)$ indeed belongs to $U_0$. In particular $U_0$ is not empty. The condition defining $U_0$ is open, so we may indeed choose $U_0$ to be open.

  Let $V$ be an open set of $f_0^{-1}(a+\varepsilon)$ such that $\ol{V}\subseteq Z$, $V$ contains $(\cK_-\cap f_0^{-1}(a+\varepsilon))\setminus U_0$,
  and $V$ is disjoint from some open subset $U_1$ containing $\cK_+$. Notice that the only point on $f_0^{-1}(a+\varepsilon)$
  that belongs to $\cK_-\cap\cK_+$ is the intersection point of $\gamma$ with the level set (and that belongs to $U_0$). Therefore,
  such a $V$ exists.  Let $Z_\eta$ and $V_\eta$ be
  the sets of points
  in $f^{-1}[a+\varepsilon,b+\varepsilon]$ that are reached from $Z$, respectively $V$, by a trajectory of $\eta$; compare Figure~\ref{fig:V_and_U}.
  \begin{figure}
     \input{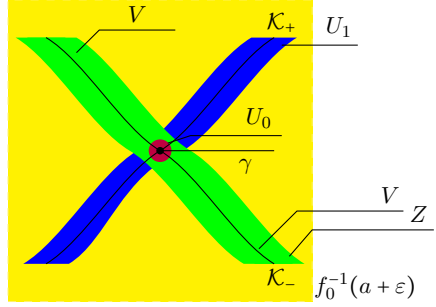}
     \caption{Notation of the proof of Lemma~\ref{lem:good_means_ld2}. The figure represents the level set $f_0^{-1}(a+\varepsilon)$.}\label{fig:V_and_U}
  \end{figure}
  %\npar{In this figure, the line for $V$ hits the letter $Z$.}{Done}

  To ensure property~\ref{item:LD2} holds, we construct yet another cut-off function $\mu$, with similar properties to the function with the same name from Section~\ref{sec:lifting-paths-of-rearangment}, but different from it.
  Let \[\mu\colon f_0^{-1}[a+\varepsilon,b+\varepsilon]\to[0,1]\] be a smooth function such that  %\npar{Is this the same $\mu$ as at the end of the  previous subsection?  If they are the same then something should be done.}{Not the same one, but similar.}
  \begin{itemize}
    \item $\mu$ is supported on $Z_\eta$;
    \item $\mu\equiv 1$ on $V_\eta$;
    \item $\partial_\eta\mu\ge 0$ with equality only at points where $\mu$ is zero or one;
    \item $\partial_{\upsilon}\mu(u)\ge 0$ for all $u\in \Sigma\cap Z_\eta$ with equality only for when $\mu$ is zero or one, here
      $\upsilon$ is as in Subsection~\ref{sub:f_ing_immersion}.
  \end{itemize}
  The construction of $\mu$ is analogous to the construction in Lemma~\ref{lemma:mu-exists}. %\ypar{added ref to new lemma again.}
  %from Subsection~\ref{sub:f_ing_immersion}, so we omit the details. \npar{$\mu$ also appeared in the previous subsection, \ref{sub:LR2_move}. Should we mention that too? }{Added an explanation.}
  Using $\mu$, and constructing an appropriate scaling function $R$ and  a lift function $P$ as in Subsection~\ref{sub:LR2_move}, we define the isotopy $G_{\tau}$, $\tau\in[1,2]$,
   that lifts all the points $w$ with $f_1(w)\in[a+2\varepsilon,b]$ whose trajectory passes through $V$ above the level set $F_0^{-1}(b)$.
   The construction of $G_\tau$ is via the same formula as in the proof of Proposition~\ref{prop:create_isotopy}; see especially \eqref{eq:Gtau_def}.

   The definition of $G_\tau$ implies that the arguments of Lemma~\ref{lem:wtf_is_eta} apply to show that $F_0\circ G_\tau$ is an $\eta$-path of functions.
   %\npar{In the lift of rearrangements section, this was a lemma. Is it clear, or is it a lemma?}{Corrected}
   To show that $G_1$ maps $\cK_-$ to the first stratum, we use arguments analogous to the proof of Lemma~\ref{lem:reduce_to_1}.
   Take $w\in\cK_-$, and let $w'\in f_0^{-1}(a+\varepsilon)$ be the point on the same trajectory of $\eta$ as $\cK_-$. Consider the following
   two cases.
   \begin{itemize}
     \item If $w'\in V$, then
       \begin{itemize}
	 \item either $f_0(w)>a+\varepsilon$ and so $F_0\circ G_1(w)>b+\varepsilon/2$, and we are done;
	 \item or $f_0(w)<a+\varepsilon$ and then $G_\tau(w)$ belongs to the first stratum by construction. %This is the only situation that happened in Lemma~\ref{prop:create_isotopy}.
       \end{itemize}
     \item If $w'\notin V$, then, as $w'\in\cK_-$, we infer that $w'\in U_0$ (by construction, $V\cup U_0$ contains $\cK_-\cap f_0^{-1}(a+\varepsilon)$). Then, by definition of $U_0$, the trajectory of $\nabla F_0$ through $G_0(w)$
       does not hit $G_0(N')$, so $G_\tau(w')$ belongs to the first stratum.
   \end{itemize}
   After this isotopy $\cK_-$ is mapped to the first stratum.

   An analogous operation improves $G_0$ in such a way that $\cK_+$ is
   eventually mapped to the first stratum. The concatenation of the two paths (one making $\cK_-$ map to the first stratum, the other
   making $\cK_+$ map to the first stratum) is the desired path of immersions~$G_\tau$. We perturb it rel $G_0$ to be a regular $F_0$-path in the sense of Definition~\ref{def:another_F_word}, so that $(F_\tau,G_\tau)$, with $F_\tau=F_0$ is a regular double path.

   To prove the second item of the lemma,
   we use the same method as in the proof of Lemma~\ref{lem:rearrange_to_keep_non_intersecting}. The conditions
   on $\nabla F_0$, guaranteed by Lemma~\ref{lem:vector_for_LD2}, ensure that we do not create intersections between different components of $N$.
  \end{proof}

\subsection{Condition~\ref{item:LD3} and the vector field $\xi$}\label{sub:LD34}
Throughout Subsection~\ref{sub:LD34} we assume that $(F_0,G_0)$ satisfies \ref{item:LD2}. Recall that the condition~\ref{item:LD3}
tells about the existence of a suitable vector field $\xi$ and the lack of critical points in $F^{-1}[a-\varepsilon,b+\varepsilon]$
other than $p_-$ and $p_+$. The next lemma constructs the vector field. That vector field will be used to move all the unnecessary critical
points away from the interval $F^{-1}[a-\varepsilon,b+\varepsilon]$.

  \begin{lemma}\label{lem:LD3}
    Suppose $F$ is immersed Morse, $G\colon N\hookrightarrow \O$ is a generic immersion, $f=F\circ G$ and $\eta$ is gradient-like for $f$.
    Assume $G$ maps $\cK_-$ and $\cK_+$ to the first stratum, i.e.\ \ref{item:LD2} holds.
    There exists a grim vector field $\xi$ for $F$
    such that the pull-back $\wt{\eta}$ of $\xi$ to $N$ agrees with $\eta$ in an open subset of $N$ containing $\cK_-\cup\cK_+$.
    If $G$ is an embedding, we can guarantee that $\eta'=\eta$. %\npar{Is this supposed to give part of LD4? We should say what this lemma is doing, with reference to the guiding  conditions.}{Added an extended comment above.}
  \end{lemma}
  As the pull-back $\wt{\eta}$ of $\xi$ agrees with $\eta$ in an open subset of $\cK_-\cup\cK_+$, Lemma~\ref{lem:uni_common} gives us
  a left-homotopy of the $\eta$-path and the $\wt{\eta}$-path of death.
  %\npar{Right now there is nothing about left-homotopic here, but that is part of LD4.}{I'm adding a comment about it.}

  \begin{proof}[Proof of Lemma~\ref{lem:LD3}]
    As $(F,G)$  satisfies~\ref{item:LD2}, the map $G$ takes a neighbourhood $U_1\subseteq N$ of $\cK_-\cup\cK_+$ to the first stratum.
    Take $U_2\subseteq N$ open such that $\cK_-\cup\cK_+\subseteq U_2\subseteq \ol{U}_2\subseteq U_1$.
    By Proposition~\ref{prop:grim_extend}, applied to $U=G(U_1)$ there exists a grim vector field $\xi$ for $F$, which is equal to $DG(\eta)$
    on the whole of $G(U_2)$. We perturb $\xi$ so that it is Morse--Smale.

    By definition, the pull-back $\wt{\eta}$ of $\xi$ agrees with $\eta$ on $U$.
    %We declare $\eta'$ to be the pull-back. %Already defined as the pull back in the statement, so strange to then define it like that here.
  \end{proof}

  To guarantee \ref{item:LD3}, we need to rearrange critical points to move them away from between $F(p_-)=a$ and $F(p_+)=b$.
  This includes the critical points that have not been rearranged before and the critical points that might have been created while enforcing \ref{item:LD2}.
  \begin{lemma}\label{lem:LD3a}
    Suppose $(F,G)$ satisfies \ref{item:LD2} and $\xi$ is as constructed in Lemma~\ref{lem:LD3}. Then, there exists
    a path of functions $(F_\tau,G_\tau)$ such that $F_\tau$ is a $\xi$-path, $F_\tau\circ G_\tau$ is a path
    of excellent Morse functions supported away from $\cK_-\cup\cK_+$, $G_\tau=G$ $($in particular, $(F_\tau,G_\tau)$ is a regular double path$)$
    and moreover, $F_1$ has no critical
    points in $F_1^{-1}[a-\varepsilon,b+\varepsilon]$ other than $p-$ and $p_+$.
  \end{lemma}
  \begin{proof}
  We proceed in a standard way. That is, we take the first critical point above $p_-$ with $h_p+d_p\le h_-+1$.
  We move it below $p$ using Rearrangement Theorem~\ref{thm:grim_rearrangement} (compare the proof of Lemma~\ref{lem:extra_cond}).
  Inductively, we construct a special $\xi$-path that moves all critical points with $h_p+d_p\le h_-+1$ below $p_-$. We can do it
  in such a way that the path is supported away from $\cL_-$. Next,
  the same argument allows us to move all critical points $p$ with $h_p+d_p\ge h_++1$ above the level set $p_+$. This done,
  since $h_+=h_-+1$, there are no more critical points between $p_-$ and $p_+$. We let $(F_\tau,G_\tau)$, for $\tau\in[0,1]$,
 denote  the resulting path. This is a $\xi$-path, a concatenation of $\xi$-paths of rearrangments. Therefore $F_\tau\circ G_\tau$
 is an $\wt{\eta}$-path. Notice that we do not rearrange pairs of critical points at depth $1$, that is, $F_\tau\circ G_\tau$
 is a path of excellent functions.
 \end{proof}

 \begin{remark*}
 Unlike in item \ref{item:LD1}, our vector field $\xi$ extends $\wt{\eta}$. That is, the condition that the membranes of points being rearranged miss $\cL_-$ and $\cL_+$ is included in the Morse--Smale condition, and need not be additionally guaranteed by the condition on a safe rearrangement.
 \end{remark*}
% We let $(F_{\tau},G_{\tau})$ be the path of rearrangements moving all critical points out of the region $F_4^{-1}(a,b)$. Note
%  that this is a $\xi$-path, with $G_{\tau}\equiv G$, and so $\wt{f}_{\tau}:=F_{\tau}\circ G_{\tau}$ is an $\eta'$-path.

  %\npar{Has LD4 been shown now?  We need to say the strategy a little more clearly here, i.e.\ what this paragraph is doing.}{Added the strategy at the beginning. Made the last paragraph a remark.}

  For the reader's convenience we summarise the properties of the functions we have constructed.
  \begin{corollary}
    The pair $(F_1,G_1)$ that is the outcome of Lemma~\ref{lem:LD3a} satisfies conditions \ref{item:LD1}, \ref{item:LD0}, \ref{item:LD2},
    and~\ref{item:LD3}.
  \end{corollary}
  \begin{proof}
    The vector field $\xi$ from Lemma~\ref{lem:LD3} is grim for $F_1$, because Lemma~\ref{lem:LD3a} constructed $F_1$ as the end of a $\xi$-path. Next, Lemma~\ref{lem:LD3a} does not change the immersion $G$, in particular \ref{item:LD2} is preserved and ~\ref{item:LD3} is arranged to hold. Next, condition~\ref{item:LD2} generalises (and implies) the condition~\ref{item:LD0}: the former one tells about the behaviour of $G$ on $\cK_-\cup\cK_+$, the latter
    controls $G$ on $\gamma=\cK_-\cap\cK_+$. Finally, condition~\ref{item:LD3} controls all critical points between $p_-$ and $p_+$, while
    condition~\ref{item:LD1} controls only some of them. That is, \ref{item:LD3} implies \ref{item:LD1}.
  \end{proof}
%  \npar{We should discuss here that LD1, LD2, and LD3 have been preserved. Worth mentioning here that LD4  implies LD1 and LD3 implies LD2, since we need that later.  So then we can say that we have arranged LD4 while preserving LD3, and this implies that we have arranged LD1 -- LD4 simultaneously.}{MB: done.}

  \subsection{The Finger Move Theorem}\label{sub:finger_mover}
  Before we complete the proof of Lifting Paths of Death Lemma~\ref{lem:lift_death}, we need the most important ingredient, namely the Finger Move Theorem, which we now state.  The proof is deferred to Part~\ref{part:finger}.

\begin{theorem}[Finger Move]\label{thm:new_finger_move}
  Let $G\colon N\looparrowright \O$ be a generic immersion
  with $\dim N=n$ and $\dim\O=n+2$. Let $F\colon\O\to\R$ be an immersed Morse function with respect to $M=G(N)$ and let $\xi$
  be a grim vector field for $F$ satisfying the Morse--Smale condition.
  Set $f=F\circ G \colon N \to \R$ and let $\eta$ be the pull-back of $\xi$ as in Section~\ref{sub:pull_back}, a gradient-like vector field on $N$.  Assume that $\xi$ and $\eta$ are Morse--Smale.

  Suppose that there exist critical points $q_-,q_+\in N$ of indices $h$ and $h+1$ respectively, such that there is a single trajectory $\gamma$ of $\eta$
  connecting $q_-$ and $q_+$, and $G$ maps the following subsets to the first stratum of $M=G(N)$, where $a:=F(p_-)$, $b:=F(p_+)$: (i) $\gamma$; (ii) the part of the stable manifold of $q_+$ in $F^{-1}(a,b)$; and (iii) the part of the unstable manifold of $q_-$ in $F^{-1}(a,b)$.
   Suppose also there exists $\varepsilon>0$ such that $\xi$
   has no critical points in the region $F^{-1}(a-\varepsilon,b+\varepsilon)$ other than $p_-$ and $p_+$. Let $r>0$ be the number of trajectories of $\xi$ connecting $p_-=G(q_-)$ to $p_+=G(q_+)$ that do not lie on $M = G(N)$.

  Then there exists a path $(F_\tau,G_\tau)$, for $\tau\in[0,2]$, and a grim vector field $\xit$ for $F_2$,
  such that $F_\tau$ for $\tau\in[0,1]$ is an arbitrarily small perturbation of $F_0$, $F_{1+\tau}=F_1$ for $\tau\in[0,1]$, and $G_\tau=G_0$ for $\tau\in[0,1]$. The path $(F_\tau,G_\tau)$ and $\xit$ satisfy the following properties.
  \begin{enumerate}[label=(FM-\arabic*)]
    \item The Morse function $F_2$ has at most four more critical points than $F_0$ and these critical points belong to the second stratum of $G_2(N)$. Also they lie in the image of a single connected component of $N$ under $G_2$. \label{item:FM_two_more}
    \item $F_2=F_0$ near all critical points of $F_0$. \label{item:FM_support_of_F2}
    \item The trajectories of $\xit$ connecting $p_-$ to $p_+$ are precisely the trajectories of $\xi$ connecting $p_-$ to $p_+$, except that one trajectory outside of $G_0(N)$ $($and outside of $G_2(N))$ is removed. \label{item:FM_one_less}
    \item Set $f_\tau=F_\tau\circ G_\tau$. The path $f_{1+\tau}$ is a constant path for $\tau\in[0,1]$. Moreover, $\eta$ is a gradient-like vector field for all $f_\tau$. \label{item:gradient}
    \item There is a pull-back $\etat$ of $\xit$ via $G_2$ such that the paths of death starting from $F_2\circ G_2$ and
      constructed with $\eta$ and $\etat$ are left-homotopic. \label{item:pull_back}
    \item The map $G_\tau(N)$ fails to be a generic immersion only for one time value of $\tau$ $($but still it is an immersion$)$.
      The double points that are created are within the same connected component of $G_0(N)$. \label{item:connected}
  \end{enumerate}
\end{theorem}

  \subsection{Condition~\ref{item:LD4}}\label{sub:finger}
  The next result uses the finger move for its proof, which will be the topic of Part~\ref{part:finger}.

  \begin{lemma}\label{lem:we_use_finger}
    Suppose $(F,G)$ satisfy \ref{item:LD1}--\ref{item:LD3}. Then.
    \begin{itemize}
      \item If $k>2$, then $\xi$ automatically satisfies~\ref{item:LD4};
      \item If $k=2$, then there exists a regular homotopy $G_\tau$ and a family of functions $F_\tau$, with $(F_\tau,G_\tau)$ a regular double path, as well as a grim vector field $\wt{\xi}$
	for the Morse function $F_2$ such that
	\begin{itemize}
	  \item $\xi$ satisfies \ref{item:LD2}, \ref{item:LD3}, and~\ref{item:LD4};
	  \item with $\wt{\eta}$ the pullback of $\wt{\xi}$, the paths of death starting from $F\circ G$ with $\eta$ and the path
	    of death starting from $F_2\circ G_2$ with $\wt{\eta}$ are lax homotopic over $F_{2\tau}\circ G_{2\tau}$.
	\end{itemize}
    \end{itemize}
  \end{lemma}
  \begin{proof}
    Suppose first $k>2$. Then, by Lemma~\ref{lem:intersectiondimension} applied to $p_-=G_0(q_-)$ and $p_+=G_0(q_+)$, it follows
    that $\dim\Ha(p_-)\cap\Hd(p_+)\cap \O[0]\cap F^{-1}(c)<0$. In particular, the only trajectory of $\xi$ connecting
    $p_-$ to $p_+$ is the image of the trajectory $\gamma$ of $\eta$ connecting $G_0$. This means that~\ref{item:LD4}
    is automatically satisfied.

    From now on, suppose that $k=2$. Lemma~\ref{lem:intersectiondimension} implies that the number of trajectories of $\xi$
    from $p_-$ to $p_+$ that lie in $\Omega\setminus G_0(N)$ is finite. Let $\ell$ be this number. If $\ell=0$, nothing needs to be done.

    Assume $\ell>0$.
    We first check that the properties \ref{item:LD1}---\ref{item:LD3} imply that the assumption of the Finger Move Theorem~\ref{thm:new_finger_move} are satisfied. The vector field $\xi$ is Morse--Smale as well as its pull-back $\wt{\eta}$.
    The curve $\gamma$ is mapped to the first stratum by \ref{item:LD0}, and the parts of the stable/unstable manifolds are of $q_-$ and $q_+$,
    that is, $\cK_-$ and $\cK_+$, are mapped to the first stratum by \ref{item:LD2}. The assumption of the Finger Move Theorem~\ref{thm:new_finger_move} on the lack of critical points is guaranteed by~\ref{item:LD3}.

    Apply the finger move. That is, by Theorem~\ref{thm:new_finger_move}, we can create a path $(F_\tau,G_\tau)$, $\tau\in[0,2]$,
    and a new grim vector field $\xit$ for $F_2$. The path has
    all the properties \ref{item:FM_two_more} -- \ref{item:connected}.
    By Lemma~\ref{lem:proof_of_gradient}, $\eta$ is gradient-like for $F_2\circ G_2$. By Lemma~\ref{lem:uni_common}
    with $h_{\sigma,0}=F_{2\tau}\circ G_{2\tau}$, the $\eta$-paths
    of death constructed from $F_2\circ G_2$ and from $F\circ G$ are lax homotopic over $F_{2\tau}\circ G_{2\tau}$. By \ref{item:pull_back},
    the  pull-back $\etat$ of $\xit$ via $G_2$ has the property that the $\etat$-path of death starting from $F_2\circ G_2$
    and the $\eta$-path of death starting from $F_2\circ G_2$ are left-homotopic. That is, to say,
    the $\etat$ path of death starting from $F_2\circ G_2$ and the $\eta$-path of death starting from $F\circ G$ are lax homotopic
    over $F_{2\tau}\circ G_{2\tau}$.
    Moreover, by \ref{item:FM_one_less},
    $\xit$ has $\ell-1$ trajectories outside of $G_2(N)$ connecting $p_-$ with $p_+$. All this means that $(F,G)$ satisfies condition~\ref{item:LD2}.

    Now, by \ref{item:FM_two_more}, $F_2$ acquires four critical points on the second stratum. We perform rearrangements, using $\xit$ as the guiding vector field, to move the four new critical points away
    from the region $F_2^{-1}[a,b]$. The rearrangement is done as in Lemma~\ref{lem:LD3a}. Call $(F_3,G_3)$ the resulting function. As the rearrangement is done using the vector field $\xit$,
    and $\etat$ is a pull-back, the path $F_{2+\tau}\circ G_{2+\tau}$ is an $\etat$-path of excellent Morse functions. In particular, by Lemma~\ref{lem:uni_common}, the $\etat$-paths of death starting with $F_3\circ G_3$ and starting with $F_2\circ G_2$ are lax homotopic over $F_{2+\tau}\circ G_{2+\tau}$. That is to say, the $\etat$-path of death starting from $F_3\circ G_3$ is lax homotopic over $F_{3\tau}\circ G_{3\tau}$
    to the original
    path of death. Note that the pair $(F_3,G_3)$ satisfies:
    \begin{itemize}
      \item the condition \ref{item:LD2}, with $\xit$ being the corresponding grim vector field. This is because $(F_2,G_2)$ already satisfied that condition;
      \item the condition \ref{item:LD3}, because we just move all critical points that were created by the finger move away from $F_2^{-1}[a,b]$;
      \item the conditions \ref{item:LD1} and \ref{item:LD0} as consequences of conditions~\ref{item:LD2} and~\ref{item:LD3}.
    \end{itemize}

   This places us in a position to use induction. Starting with $(F_3,G_3)$, we create further paths by a subsequent use of the finger move. More precisely, the $j$-th path
   denoted by $(F_{3j-3+\tau},G_{3j-3+\tau})$, $\tau \in [0,3]$, is constructed using the finger move applied to $(F_{3j-3},G_{3j-3})$
   and the vector field $\xitt{j}$, grim for $F_{3j-3}$, with $\xitt{0}=\xi$, $\xitt{1}=\xit$.
   The vector field $\xitt{j}$ has $\ell-j$ trajectories connecting $p_-$ with $p_+$. Moreover, the induction argument shows that the grim path
   of death %\ypar{I just want the phrase `grim path of death' to appear at least once. MP: ha, very good.}
   starting from $F_{3j-3}\circ G_{3j-3}$, and using the pull-back of $\xitt{j}$, is lax homotopic to the original one over $F_{(3j-3)\tau}\circ G_{(3j-3)\tau}$.

   Each of the vector fields $\xitt{j}$ satisfies \ref{item:LD3}. Eventually,
   $(F_{3\ell},G_{3\ell})$ and $\xitt{\ell}$ satisfy~\ref{item:LD2}, \ref{item:LD3}, and~\ref{item:LD4}. We conclude the proof by reparametrizing the path to the interval $[0,1]$.
  \end{proof}

  \subsection{The proof of the Death Lifting Lemma~\ref{lem:lift_death}}\label{sub:summary_death}
  Let $F_0=F$ and $G_0=G$. We set $f_0=F_0\circ G_0$, $\eta$ a gradient-like vector field for $f_0$, and $f_\tau$ an $\eta$-path of death cancelling $q_-$ and $q_+$. The proof follows the pattern of the proof of the Rearrangement Lifting Lemma~\ref{lem:lift_rearr}, but is technically more involved. The reader should keep in mind that throughout the proof a few new vector fields are constructed, like $\wt{\eta}$, $\eta'$, and $\etat$. The key property is that except for $\etat$, all these vector fields agree with the original vector field $\eta$ near $\cK_-\cup\cK_+$, so that all paths of death created with these vector fields are lax-homotopic to the original one. The case of $\etat$ is a bit special, as it arises from the finger move, but its behaviour is controlled by \ref{item:pull_back} of Finger Move Theorem~\ref{thm:new_finger_move}. A schematic of the proof is presented in Figure~\ref{fig:concat_death}.
  \begin{figure}
    \input{pictures/concat_death.tex}
    \caption{A schematic of the proof of Death Lifting Lemma~\ref{lem:lift_death}.}\label{fig:concat_death}
  \end{figure}
Lemma~\ref{lem:rearrange_while_keeping} takes care of \ref{item:LD1}. It creates a path $(F_\tau,G_\tau)$
  that replaces $(F_0,G_0)$ by a $(F_1,G_1)$ (with $G_\tau=G_0$) and $(F_1,G_1)$ satisfying \ref{item:LD1}.
  Set $\wt{f}_\tau=F_\tau\circ G_\tau$. Clearly $\eta$ need not be gradient-like for $\wt{f}_1$.
  However, the arguments of Lemma~\ref{lem:extend_eta} from Subsection~\ref{sub:summary_rearr} allow us to construct a vector field $\wt{\eta}$,
  gradient like for $\wt{f}_\tau$ and such that $\wt{\eta}=\eta$ near $\cK_-\cup\cK_+$.

  If $G_1$ is an embedding (that is if $G_0$ is), \ref{item:LD2} is already satisfied.
  Otherwise, we create a path whose destination satisfies that condition. We choose a generic gradient-like vector field $\nabla F_0$ and, if necessary, replace $G_1$ by a path of immersions (denoted $G_\tau$ in Subsection~\ref{sub:LD0}, but here we call it $G_{1+\tau}$, for $\tau\in[0,1]$, to avoid a notation clash), with $F_{1+\tau}\equiv F_1$ such that the trajectory $\gamma$ is in good position with respect to $G_2$ and $F_{1+\tau}\circ G_{1+\tau}$ is an $\wt{\eta}$-path of Morse functions. This is done in~Lemma~\ref{lem:good_position}.

  With this choice, $(F_2,G_2)$ satisfies \ref{item:LD0}. However, $(F_2,G_2)$ might acquire some critical points on deeper strata,
  as the path $(F_{1+\tau},G_{1+\tau})$ might potentially create extra self-intersections.
  The key point is that no critical points at depth $0$ are created, because $F_{1+\tau}$ is constant. That is, $(F_2,G_2)$ again satisfies
  \ref{item:LD1}.

  Once \ref{item:LD1} and~\ref{item:LD0} are satisfied, Lemma~\ref{lem:good_means_ld2} constructs a path of functions, which we here denote $(F_{2+\tau},G_{2+\tau})$, such that $F_{2+\tau}\circ G_{2+\tau}$ is a $\wt{\eta}$-path of Morse functions. Moreover $(F_3,G_3)$ satisfies \ref{item:LD2}.

  Lemma~\ref{lem:LD3} applied to $(F_3,G_3,\wt{\eta})$ constructs a vector field $\xi$, grim for $F_3$, such that the pull-back, $\eta'$
  agrees with $\wt{\eta}$ near $\cK_-\cup\cK_+$. Note that $\eta'$ agrees also with the original vector field $\eta$ in an open subset
  containing $\cK_-\cup\cK_+$. In Lemma~\ref{lem:LD3a}, we construct a path, which we now
  denote $(F_{3+\tau},G_{3+\tau})$, that moves all critical points away from between $p_-$ and $p_+$. This is a path
  of rearrangement,  such that $F_{3+\tau}\circ G_{3+\tau}$ is an $\eta'$ path.

  The resulting functions $(F_4,G_4)$ satisfy conditions \ref{item:LD2} and \ref{item:LD3}, so they also satisfy \ref{item:LD1} (which is weaker than \ref{item:LD3}) and \ref{item:LD0} (which is weaker than~\ref{item:LD2}).

 We continue with the proof of Death Lifting Lemma~\ref{lem:lift_death} by studying $\xi$ in detail.
 The vector field $\xi$ might have more trajectories
 connecting $p_-$ and $p_+$. We apply finger moves through Lemma~\ref{lem:we_use_finger} applied to $(F,G)=(F_4,G_4)$, the vector field $\xi$
 and its pull-back $\eta'$. We obtain a path, which we denote $(F_{4+\tau},G_{4+\tau})$, $\tau\in[0,1]$,
 and a vector field $\xit$, grim for $F_5$.

 Lemma~\ref{lem:we_use_finger} implies that $(F_5,G_5,\xit)$ satisfies
 \ref{item:LD3} and~\ref{item:LD4}, as well as all the previous conditions.
 Lemma~\ref{lem:cond_lift_death} applied to that system of functions constructs a weak lift of
 the $\etat$ path of death starting from $F_5\circ G_5$.
 Denote it by  $(F_{5+\tau},G_{5+\tau})$, $\tau\in[0,1]$. Set $\wt{f}_{5+\tau}:=F_{5+\tau}\circ G_{5+\tau}$.
 \begin{lemma}\label{lem:lax_f5}
   The path $\wt{f}_{5+\tau}$ is lax-homotopic to the path $f_\tau$ over $\wt{f}_{\tau}$, $\tau\in[0,5]$.
 \end{lemma}
 \begin{proof}
   We make this proof in several steps. %\mpar{Should probably be `A path' instead of `the path' here.}
   \begin{itemize}
     \item An $\wt{\eta}$-path of death starting from $\wt{f}_1$ is lax homotopic to an $\eta$-path of death starting from $f_0=\wt{f}_0$ over $\wt{f}_\tau$ (which is $f_\tau$) by Lemma~\ref{lem:uniqueness_after_rear}.
     \item An $\wt{\eta}$-path of death starting from $\wt{f}_2$ is lax homotopic to an $\wt{\eta}$-path of death starting from $\wt{f}_1$ over $\wt{f}_{1+\tau}$ by Lemma~\ref{lem:uni_common}.
     \item An $\wt{\eta}$-path of death starting from $\wt{f}_3$ is lax homotopic to an $\wt{\eta}$-path of death starting from $\wt{f}_2$ over $\wt{f}_{2+\tau}$ by Lemma~\ref{lem:uni_common}.
     \item An $\eta'$-path of death starting from $\wt{f}_4$ is lax homotopic to an $\wt{\eta}$-path of death starting from $\wt{f}_3$ over $\wt{f}_{3+\tau}$. Here we use Lemma~\ref{lem:uniqueness_after_rear}, because we change the vector field, however, we do not change it near $\cK_-\cup\cK_+$.
     \item An $\etat$-path of death starting from $\wt{f}_5$ is lax homotopic to an $\eta'$-path of death starting from $\wt{f}_4$ over $\wt{f}_{4+\tau}$. This is a consequence of Lemma~\ref{lem:we_use_finger}, which in turn relies on the property~\ref{item:pull_back} of the Finger Move Theorem~\ref{thm:new_finger_move}.
     \item By construction $\wt{f}_{5+\tau}$ is a weak lift of an $\etat$-path of death starting from $\wt{f}_5$. As such,  $\wt{f}_{5+\tau}$ is left-homotopic to an $\etat$-path of death starting from~$\wt{f}_5$.
   \end{itemize}
   We now concatenate the homotopies. Eventually we get that $\wt{f}_{5+\tau}$ is lax homotopic to $f_\tau$ over $\wt{f}_{5\tau}$.
 \end{proof}
 Lemma~\ref{lem:lax_f5} will put us in a position to use Lemma~\ref{lem:lax_to_lift}. The following result is needed to check all the assumptions.
 \begin{lemma}\label{lem:f5a}
   The path $(F_{6\tau},G_{6\tau})$ is a regular double path.
 \end{lemma}
 \begin{proof}
   The proof requires a case-by-case analysis of the components of the path.
   \begin{itemize}
     \item The path $(F_{\tau},G_{\tau})$ changes only $F_\tau$, fixing $G_\tau=G_0$ being a regular immersion. Moreover $F_\tau$ is an $\cF^1$-path, compare Lemma~\ref{lem:rearrange_while_keeping},
   so $(F_\tau,G_\tau)$ is a regular double path.
 \item The path $(F_{1+\tau},G_{1+\tau})$ created in Lemma~\ref{lem:good_position} is such that $F_{1+\tau}$ is fixed, and $G_{1+\tau}$ is a path of immersions (an $F_1$-regular path in the sense of Definition~\ref{def:another_F_word}), in particular $(F_{1+\tau},G_{1+\tau})$ is
   a regular double path.
 \item Likewise, Lemma~\ref{lem:good_means_ld2} creates a regular double path $(F_{2+\tau},G_{2+\tau})$.
 \item The path $(F_{3+\tau},G_{3+\tau})$ is a path of rearrangments with $G_{3+\tau}=G_3$, see Lemma~\ref{lem:LD3a}. That is, it is a regular double path.
 \item Lemma~\ref{lem:we_use_finger} creating the path $(F_{4+\tau},G_{4+\tau})$ creates a regular double path.
 \item The weak lift $(F_{5+\tau},G_{5+\tau})$ is created using Lemma~\ref{lem:cond_lift_death} creating a regular double path, actually, that lemma does not change $G_{5+\tau}$, only $F_{5+\tau}$.
   \end{itemize}
   Concatenation of regular double paths is a regular double path.
 \end{proof}
 Give Lemmata~\ref{lem:lax_f5} and~\ref{lem:f5a}, we promote $(F_{6\tau},G_{6\tau})$ to a weak lift of $f_\tau$ using Lemma~\ref{lem:lax_to_lift}. This completes the proof of the Death Lifting Lemma~\ref{lem:lift_death}.

\section{The Path Lifting Theorem}\label{sec:path_lifting_proof}

We can now combine the results from Section~\ref{sec:elementary_lift} on lifting $\cF^0$-paths and paths of birth, rearrangement and death, to obtain the following result.

\begin{theorem}[Path Lifting]\label{thm:path_lifting}
  Let $N$ be a closed manifold of dimension~$n$ and let $f_\tau$, for $\tau\in[0,1]$, be an $\cF^1$-path of functions such that there are no rearrangements for which
  a critical point of higher index goes below a critical point of a lower index.
  Let $G_0\colon N\to\O$ be a generic immersion,  %$G_0^{-1} \bd(\O) = \bd N$,
  where $\O$ is a compact manifold of dimension $n+k$. Suppose $F_0\colon\O\to\R$ is an immersed Morse function such that $F_0\circ G_0=f_0$.
  %Assume that Morse functions are constant on connected components of the boundary $\bd \O$.

  If $k\ge 2$, then there exists a weak lift $($Definition~\ref{def:lifting}$)$ of $f_\tau$.  That is, there exist:
  \begin{itemize}
    \item a regular path $($an $\cF^1$-path$)$
     of functions $\wt{f}_\tau\colon N\to\R$ such that $\wt{f}_0=f_0$ and $\wt{f}_1=f_1$ and the paths $\wt{f}_\tau$ and $f_\tau$
      are $\cF^1$-homotopic;
    \item a regular double path $(F_\tau,G_\tau)$ such that $F_\tau\circ G_\tau=\wt{f}_\tau$ with the following properties for $G_\tau$.
      \begin{enumerate}
	\item[(i)] If connected components of $N$ have disjoint image under $G$, then this also holds for $G_\tau$ for each $\tau \in [0,1]$.
	\item[(ii)] if $G_0$ is an embedding and either $k\ge 3$ or $f_\tau$ has no deaths, then $G_\tau$ is an ambient isotopy.
      \end{enumerate}
    \item The family $F_\tau$ has no births or deaths on the zeroth stratum, that is to say, for each $\tau \in [0,1]$, $F_\tau$ is a classical Morse function when
      regarded as a function from $\Omega$ to $\R$, forgetting the stratification.
  \end{itemize}

  \smallskip
  If $N$ and $\Omega$ are compact with nonempty boundary,
  $G_0$ is a neat immersion and $f_\tau$ is neat as a path of functions on $N$, while $F_0$
  is a neat function, then there is a neat regular double path $(F_\tau,G_\tau)$ satisfying all the required properties.
\end{theorem}

\begin{proof}
  Let $T$ be the number of `events' on $f_{\tau}$, where by an event we mean a birth, a rearrangement,
  or a death.
  As $f_{\tau}$ is an $\cF^1$-path we make sure that at each of these events, the path $f_{\tau}$ is an elementary path, respectively of birth, of rearrangement, or of death.
  We may and shall assume that these events occur at distinct times.

  We perform induction over $T$. If $T=0$, then $f_{\tau}$ is a path of excellent Morse functions. We use Lemma~\ref{lem:lift_morse}.

  Otherwise suppose that at $\tau_0$ there appears one of the events: birth, death or rearrangement, and that $\tau_0>0$ is minimal with this property.
   Let $\delta>0$ be such that the path $f_\tau$, $\tau\in[\tau_0-\delta,\tau_0+\delta]$ is an elementary path. We split our path $f_{\tau}$ into
  three parts: the first part, for $\tau\in[0,\tau_0-\delta]$, is a $\cF^1$-path of Morse functions, which can be lifted using Lemma~\ref{lem:lift_morse}. For
  $\tau\in[\tau_0-\delta,\tau_0+\delta]$ we lift the path using Lemma~\ref{lem:lift_birth} (if there is a birth at $\tau_0$), Lemma~\ref{lem:lift_rearr} (if there
  is a rearrangement at $\tau_0$), or Lemma~\ref{lem:lift_death} (if there is a death at $\tau_0$). Note that if there is a rearrangement at $\tau_0$, a critical point with a smaller index can never go above a critical point of higher index by the assumptions of the theorem. Therefore,
  the assumptions of Lemma~\ref{lem:lift_rearr} are satisfied.

  The path for $\tau\in[\tau_0+\delta,1]$ has $T-1$ events and we use the inductive hypothesis.

  \smallskip
  Now we provide an argument that the result works for $N$ and $\O$ with nonempty boundary. First, we perturb $f_\tau$ to a $\cF^1$-path that is neat.
  This is done using methods from Subsection~\ref{sub:neat}. Therefore, we only need to prove that we can lift a neat path of Morse functions
  to a neat path, and that we can lift an elementary path of neat functions to a neat path on~$\O$.

  The first step, that is, lifting a neat path of Morse functions to a neat path, is done by the last part of Lemma~\ref{lem:lift_morse}. Next,
  we note that each kind of elementary path (births, rearrangements and deaths) is lifted in such a way that the initial Morse function $F_0$
  and the initial embedding $G_0$ are unchanged away from a set of the form $F_0^{-1}(c,d)$: for example, for rearrangements, if $a$ and $b$
  are critical values of points that are going to be rearranged, we set $c=a-\varepsilon$, $d=b+\varepsilon$ for $\varepsilon\ll 1$. This means
  that each time we lift an elementary path, we do not alter the function $F_0$ near the boundary of $\O$, nor do we alter the function $G_0$ near the boundary of $N$. Each time we lift an elementary path, we might need to shrink the neighbourhoods of $\bd N$ and $\bd\O$, on which $G_\tau$, respectively $F_\tau$ are independent of $\tau$, but we perform only finitely many such operations. Hence, the resulting regular double path
  $(F_\tau,G_\tau)$ is neat.
\end{proof}

\section{Path lifting for $N$ having components of varying dimensions}\label{sec:multidim}

Now we discuss the case that $N=N_1\sqcup \cdots\sqcup N_r$ is a union of components of dimensions $n_1,\dots,n_r$, with $n_i\le\dim\Omega-2$ for all $i$. We suppose that $G_0(N_i)\cap G_0(N_j)=\emptyset$ unless $i=j$. Set $k_i=\dim\Omega-\dim N_i$ to be the codimension of~$N_i$.

First, the notions of immersed Morse functions and of grim vector fields can be generalised in an obvious way. With the stratification $\O[d]$ given by points at depth $d$ (even if the connected components of $\O[d]$ might have different codimensions), the local form of a grim vector field is still given by \eqref{eq:localgrim},
with the necessary adjustments of the range of indices of the $x$ and $y$ variables. The ascending and descending membranes of a critical point $p\in G_0(N_i)$ have dimensions governed by the $n_i$, the depth of $p$, and the codimension $k_i$. Existence of Morse functions and of grim vector fields is proved in this situation without changes.

In the Cancellation Theorem~\ref{thm:grimcanc} we work with the image of a single connected component of $N$. Indeed, the two critical points that are to be cancelled are connected by a trajectory of the grim vector field that lies in the first stratum. That is, everything happens within a single connected component of $N$. The dimension restriction is not used.

The situation is different with the rearrangement theorem, where we might rearrange critical points belonging to different components. The key property we use is dimension counting. It is routine to check that if $k_i\ge 2$ for all $i$, then  critical points $p_-$ and $p_+$ with $\ind p_-\ge \ind p_+$ can be rearranged regardless of their depth and of the dimension of the manifolds they belong to. One way to see this (and to avoid lengthy calculations) is that rearrangement is possible in codimension two, and the dimension of membranes at the same depth does not increase if the codimension $k_i$ increases. Another way is to check that if $p_-\in N_i$ and $p_+\in N_j$, $i\neq j$, then the only trajectories connecting $p_-$ to $p_+$ are necessarily in the zeroth stratum, since $G_0(N_i) \cap G_0(N_j) = \emptyset$. The calculation of dimensions of the membranes is straightforward.

We pass to a discussion of path lifting. Lifting of births is purely local. In general, we lift a path governed by a vector field $\eta$ on $N$
(an $\eta$-path of rearrangment or an $\eta$-path of death), this terminology is valid even if the components of $N$ have different dimensions.
Going into technical details, the partial rearrangement, i.e.\ ensuring \ref{item:LR1} and \ref{item:LD1}, relies on dimension counting arguments. These arguments go through if the components have various dimensions (the same argument as in the previous paragraph applies). The most difficult case is still if $p_-$ and $p_+$ both belong to a single codimension two component.

Arranging for properties \ref{item:LR2} and \ref{item:LD2} to hold is one of the most sophisticated technical part of the proof of the Path Lifting Theorem. Recall that we use an extra gradient-like vector field, which we call $\nabla F_0$, and a special cut-off function $\mu$ supported in a neighbourhood of the ascending manifold $\cK_-$ of $q_-$. The function $\mu$ is nonzero only on the component of $N$ containing $q_-$, that is, the whole discussion constructing $\mu$ is independent of the dimensions of the components of $N$ not containing $q_-$. The vector field $\nabla F_0$ is constructed so as to avoid self-intersections of different components of $G_0(N)$.
This holds even if the $N_i$ have different dimensions.

Finally, we need to discuss the finger move. It is used if the critical points $q_-$ and $q_+$ belonging to the \emph{same} component of $N$
are connected by a trajectory of a vector field $\xi$ that lies in the zero stratum. As $q_-$ and $q_+$ belong to the same component, and the component has codimension two, the finger move concerns only that single connected component of $N$. Other components do not interfere, regardless of their dimension.

\part{The finger move}\label{part:finger}
%\section{Overview of Part~\ref{part:finger}}

Throughout Part~\ref{part:finger} we assume that the codimension $k=2$.
Our work is focused on the image of $N$. Therefore, unlike in Part~\ref{part:pathlifting}, we will mostly work on $M:=G_0(N)$.

\begin{proof}[Plan of proof of Theorem~\ref{thm:new_finger_move}]
  The proof of Theorem~\ref{thm:new_finger_move} requires Sections \ref{sec:coorsystem} through \ref{sec:proof_of_pull_backs}.
As a guide for the reader through the proof, we present the plan.

The first challenge is to find a suitable coordinate system.
The construction of this coordinate system is done in Sections~\ref{sec:coorsystem} and~\ref{sec:middle_and_inner} and consists of several steps. In fact,
the coordinate system we define
is not merely a local coordinate system near a critical point, but it is suitably extended along a ``guiding curve.'' In a sense,
this a semi-local coordinate system.

The actual finger move is described in Section~\ref{sec:fingermove}. This section begins with an explicit formula
for the finger move in  Subsection~\ref{sub:highfinger}. Formally, this induces an isotopy $G_{1+\tau}$ of Theorem~\ref{thm:new_finger_move}.
One of the key properties is that $G_{1+\tau}$ is level-set-preserving, that is $F\circ G_{1+\tau}=F\circ G_1=F\circ G_0$ for all $\tau \in [0,1]$. Condition~\ref{item:connected} follows from the construction of $G_\tau$, and it is addressed at the end
of Subsection~\ref{sub:highfinger}.

In Subsection~\ref{sub:newvector} we construct a suitable vector field $\xih$, which is tangent to $G_2(M)$. In Subsection~\ref{sub:critical} we prove
that $\partial_{\xih} F \ge 0$.
Already the vector field $\xih$ satisfies the most
important property of Theorem~\ref{thm:new_finger_move}, namely \ref{item:FM_one_less}. However, $\xih$ is not a grim vector field on $F$;
moreover, $F$ is not an immersed Morse function for $G_2(N)$. To remedy this, in Subsection~\ref{sub:explicit_perturbation} we construct
an explicit perturbation $F_\tau$, $\tau\in[0,1]$, of $F$ and an explicit perturbation of the vector field $\xi_\tau$ and of
the function $F_\tau$.

The key property \ref{item:FM_one_less} is proved in Section~\ref{sec:membranesofxi} as Theorem~\ref{thm:oneless} with
two key lemmas (Lemma~\ref{lem:safeentry} and~\ref{lem:safeexit}) proved in Section~\ref{sec:middle_and_inner}. Since $F_1$ is a perturbation
of $F$, the proof of Theorem~\ref{thm:oneless} holds both for $F$ and $\xih$ and for $F_1$ and $\xit$.

Finally, we study the pull-back of $\xit$, denoted by $\etat$, which is a vector field on $N$, gradient-like for $f_2=F_2\circ G_2$.
Properties~\ref{item:gradient} and~\ref{item:pull_back} are proved in Section~\ref{sec:proof_of_pull_backs}.
By a careful analysis of the ascending and descending manifolds of $\etat$, in the same section we show that the grim paths of death
constructed via $\etat$ and via $\eta$ are homotopic.
\end{proof}

\section{The outer shell}\label{sec:coorsystem}
As noted above, the proof of Finger Move Theorem~\ref{thm:new_finger_move} begins with a construction of a suitable open subset with concrete coordinates, in which the finger move
will be performed. In fact,
we will construct three nested neighbourhoods of one of the critical points. %, and the finger move will be supported inside the smallest of these neighbourhoods, with various controls on the changes in the other neighbourhoods, and no change outside the largest of the three.

Subsections~\ref{sub:fingernbhd1} through~\ref{sub:last_technical_finger} construct the outer shell, $V_{\oout}$. The construction
is rather technical. For the reader who wants to skip the details, we summarise the properties of $V_{\oout}$ in Subsection~\ref{sub:fingernbhd7}.

%Subsection~\ref{sub:fingernbhd7} constructs the middle and inner shells $V_{\iinn}$ and $V_{\mmid}$ and describes the r\^oles of the three sets.

\subsection{The set $V_{\oout}$.}\label{sub:fingernbhd1}

In this section we construct the biggest (outer) neighbourhood of one of the critical points, which we will denote $V_{\oout}$. %Later on

We set
\[c_-=F(p_-),\ c_+=F(p_+)\]
and we denote by $\Xi_s$ the flow of the vector field $\xi$.

Suppose that there are no other critical points in $F^{-1}([c_-,c_+])$.  Note that this can always be arranged by the Global Rearrangement Theorem~\ref{thm:grim_global_rearrangement}.
Choose local coordinates around $p_+$ given by $(x_1,\ldots,x_n,y_1,y_2)$, so that $p_+$ is given by $(0,\ldots,0)$,  $M$ is locally given by $\{y_1=y_2=0\}$,
and the grim vector field around $p_+$ is given by $(-x_1,\ldots,-x_{h+1},x_{h+2},\ldots,x_n,y_1^2+y_2^2,0)$.
Let $U_{\ccoor}$ be an open set containing $p_+$ on which these coordinates are defined.
In these local coordinates $\Hd(p_+)$ is given by
\[\Hd(p_+)=\{x_{h+2}=\dots=x_n=y_2=0\}\cap\{y_1\le 0\}.\]

Choose $c_{\bbot}\in (c_-,c_+)$ such that for any $c\in [c_{\bbot},c_+]$, the intersection of the descending membrane $\Hd(p_+)$ with the level set $F^{-1}(c)$
is contained in the coordinate neighbourhood $U_{\ccoor}$. Likewise, choose a regular value $c_{\ttop}>c_+$ in such a way that
$\Ha(p_+)\cap F^{-1}(c)\subseteq U_{\ccoor}$ for any $c\in[c_+,c_{\ttop}]$;
see Figure~\ref{fig:uout}.
It is clear that such $c_{\bbot}$ and $c_{\ttop}$ exist.
\begin{figure}
\input{pictures/uout.tex}
\caption{The set $U_{\ccoor}$ and the level sets $c_{\bbot}$, $c_{\ttop}$.}\label{fig:uout}
\end{figure}
The set $U_{\oout}$ will be contained inside $F^{-1}[c_{\bbot},c_{\ttop}]$.
To define it choose $\varepsilon_{\oout}>0$ and let $Y_{\oout}$ be an open tubular neighbourhood of $\Hd(p_+)\cap F^{-1}(c_{\bbot})$ in the level set $F^{-1}(c_{\bbot})$
consisting of points at distance less than $\varepsilon_{\oout}$ from $\Hd(p_+)\cap F^{-1}(c_{\bbot})$.
We will assume that $\varepsilon_{\oout}>0$ is small, and the precise meaning of `small' in this context will be clarified soon. For now define $U_{\oout}$ as
\begin{equation}\label{eq:uout}
  \ol{U}_{\oout} := \ol{\bigcup_{s\in\R}\Xi_s(Y_{\oout})}\cap F^{-1}[c_{\bbot},c_{\ttop}].\ \ U_{\oout}=\Int \ol{U}_{\oout}.
\end{equation}
That is, $U_{\oout}$ is the grim neighbourhood of $p_+$ in the sense of Definition~\ref{def:grimneigh}. Clearly, $\ol{U}_{\oout}$ is the closure of $U_{\oout}$.

\begin{lemma}
If $Y_{\oout}$ is sufficiently thin $($that is, if $\varepsilon_{\oout}$ is sufficiently small$)$,
then $U_{\oout}$ is contained in the coordinate neighbourhood $U_{\ccoor}$.
\end{lemma}

\begin{proof}
  Let $Y_n$ be the set $Y_{\oout}$ defined with $\varepsilon_{\oout}=\frac1n$ and let $U_{n}$ be the corresponding set $U_{\oout}$ (that is,
  $U_n$ is constructed via \eqref{eq:uout} by taking $Y_{\oout}=Y_n$). By construction, $\ol{U}_{n+1}\cap F^{-1}(c_{\bbot},c_{\ttop})\subseteq U_{n}$.
  It follows from the proof of
Lemma~\ref{lem:grimneigh} that $\bigcap U_n=(\Ha(p_+)\cup\Hd(p_+))\cap F^{-1}(c_{\bbot},c_{\ttop})$, so that $\bigcap U_n\subseteq U_{\ccoor}$. Since $U_{\ccoor}$ is open, it follows
that for large $n$ we must have $U_n\subseteq U_{\ccoor}$ as well.
\end{proof}

We have constructed the outer neighbourhood $U_{\oout}$ starting from the neighbourhood $Y_{\oout}$.
We leave ourselves the possibility of shrinking $Y_{\oout}$ further by decreasing $\varepsilon_{\oout}$,
but henceforth the levels $c_{\bbot}$ and $c_{\ttop}$ will remain fixed.

\subsection{The position of $\Ha(p_-)$.}\label{sub:fingernbhd2}

By the Morse--Smale condition, the membranes $\Ha(p_-)$ and $\Hd(p_+)$ intersect transversely along a finite number of trajectories. Intersecting
$\Ha(p_-)\cap\Hd(p_+)$ with the level set $F^{-1}(c_{\bbot})$  yields a finite number of points $z_0,z_1,\ldots,z_r$.
Here $z_0$ is assumed to lie on $M$, that is $z_0$ corresponds to the trajectory of $\xi$ on $M$ that connects $p_-$ and $p_+$. %By the assumptions, there is exactly

If $Y_{\oout}$ is small enough, the intersection $\Ha(p_-)\cap Y_{\oout}$ consists of $r+1$ discs of dimension $n-h$.
%the intersection $\Ha(p_-)\cap Y_{\oout}$ consists of $r$ or $r+1$ discs of dimension $n-h$. %(depending on whether there is a trajectory on $M$ connecting $p_-$ with $p_+$).
In fact, each connected component of $\Ha(p_-)\cap Y_{\oout}$ that does not contain $\Hd(p_+)\cap Y_{\oout}$
is separated from it by some distance, and thus shrinking $Y_{\oout}$ (decreasing $\varepsilon_{\oout}$) we can make this component disjoint from $U_{\oout}$;
see Figure~\ref{fig:removeredundantcomponents}.
\begin{figure}
\input{pictures/removeredundant.tex}
\caption{The membranes $\Ha(p_-)$ and $\Hd(p_+)$ at the level set $c_{\bbot}$. The set $Y_{\oout}$ is shrunk to avoid the component on
the left.}\label{fig:removeredundantcomponents}
\end{figure}
The $r+1$ discs $D_0,D_1,\ldots,D_r$ forming $\Ha(p_-)\cap F^{-1}(c_{\bbot})\cap U_{\oout}$ are transverse to $\Hd(p_+)$, because $\Ha(p_-)$ intersects $\Hd(p_+)$
transversely (this is the Morse--Smale condition)
and they intersect $\Hd(p_+)$ at the points $z_0,z_1,\ldots,z_r$. These discs can be isotoped in such a way that for $k>0$ the disc $D_k$ around $z_k$ has the form
\begin{equation}\label{eq:maft}
x_{1}=s_{1,k},\ldots,x_{h+1}=s_{h+1,k}\textrm{ for }k=1,\ldots,r,
\end{equation}
for some constants $s_{i,\ell}\in\R$. For one of the points $z_m$, after applying the aforementioned isotopy, we can and will assume that all of the constants $s_{i,m}=0$, for $i=1,\dots,h+1$ and that
$s_{1,j}\neq 0$ for all $j\neq m$.
The choice of $m$ is arbitrary if $\dim\Ha(p_-)\cap F^{-1}(c_{\bbot})>1$ (that is, if $h<n-1$); if $\dim\Ha(p_-)\cap F^{-1}(c_{\bbot})$ is 1-dimensional, then we take $m$ to be the index of the point
\emph{nearest} to $M$ on the guiding curve, see Subsection~\ref{sub:fingernbhd2} below.

Given \eqref{eq:maft}, the coordinates of the point $z_{\ell}$ are
\begin{equation}\label{eq:maft_2}
(s_{1,\ell},\ldots,s_{h+1,\ell},0,\ldots,0,c_{\bbot}-s_{1,\ell}^2-\ldots-s_{h+1,\ell}^2,0).
\end{equation}
We remark that the $y_1$ coordinate of $z_{\ell}$ is determined by the fact that $F(z_{\ell})=c_{\bbot}$. We also isotope the disc $D_0$ within $M$ in such a way that it is given by \eqref{eq:maft},
but with the extra conditions that the $y_1$ coordinate of $z_0$ is equal to $0$, and the $y_2$ coordinate on $D_0$
is either non-positive or non-negative.

According to the Isotopy Insertion Lemma~\ref{lem:isoinject}, the isotopy of the discs required to arrange the local description above of the intersection of $\Ha(p_-)$ with $\Hd(p_+)$, can be obtained by locally altering the vector field $\xi$ in the set
$F^{-1}(c',c_{\bbot})$, for some $c'<c_{\bbot}$ close to $c_{\bbot}$.
From now on we will assume that such an alteration has
been made, and that the intersection $\Ha(p_-)\cap Y_{\oout}$ is given by \eqref{eq:maft}.

\begin{figure}
\input{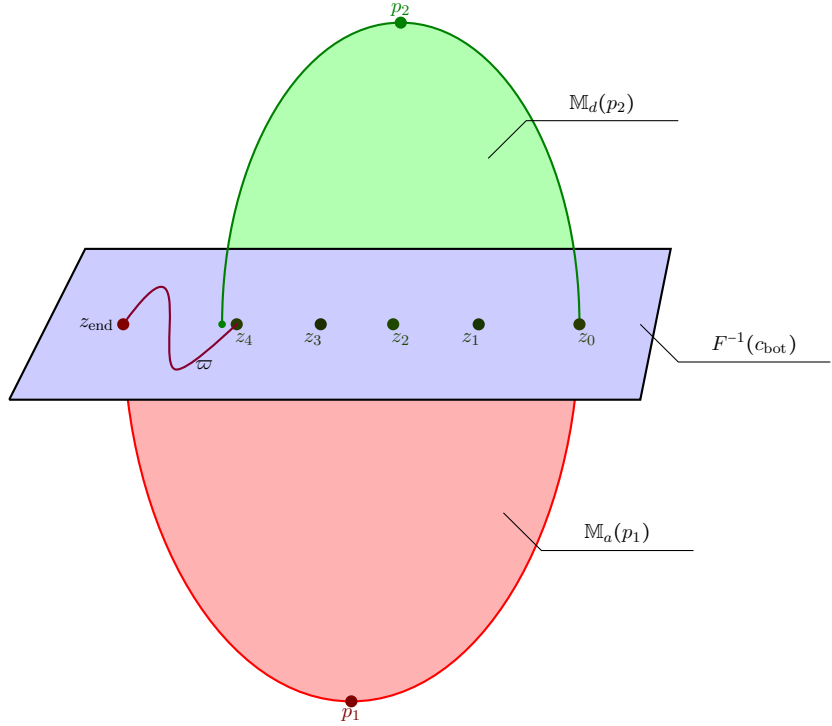}
\caption{The notation of Section \ref{sub:fingernbhd2}.  For clarity only the part of the membrane $\Hd(p_+)$ in $\{F\ge c_{\bbot}\}$ is drawn,
as well as the part of the membrane $\Ha(p_-)$ below the level set $F^{-1}(c_{\bbot})$.}\label{fig:membranes}
\end{figure}

\subsection{The guiding curve.}\label{sub:fingernbhd3}

Consider the point $z_m$ on the intersection $\Ha(p_-)\cap\Hd(p_+)\cap F^{-1}(c_{\bbot})$, for which we have set $s_{1,m}=\ldots=s_{h+1,m}=0$.
Take an embedded curve $\guiding\colon[0,\varepsilon_{\oout}]\to\Ha(p_-)\cap F^{-1}(c_{\bbot})$ given by
\begin{equation}\label{eq:gammashort}
\guiding(s)=(0,\ldots,0,0,\ldots,0,0,s);
\end{equation}
see Figure~\ref{fig:membranes}, in which $m=4$. Since $x_1,\ldots,x_n,y_1,y_2$ is an orthonormal coordinate system and $Y_{\oout}$ is the set
of points at distance less than or equal to $\varepsilon_{\oout}$
 from $\Hd(p_+)\cap F^{-1}(c_{\bbot})$, we have that $\guiding(\varepsilon_{\oout})\in\partial Y_{\oout}$.
Next, we want to extend $\guiding$ to a map $\guiding\colon[0,1]\to \Ha(p_-)\cap F^{-1}(c_{\bbot})$ such that
\begin{align*}
\guiding((\varepsilon_{\oout},1])\cap Y_{\oout}&=\emptyset,\\
\guiding(1)&\in M\cap F^{-1}(c_{\bbot}),  \text{ and}\\
\guiding(1)&\notin \Hd(p_+).
\end{align*}
That is, $\guiding$ is a path in $\Ha(p_-)$ and in the level set $F^{-1}(c_{\bbot})$, from an intersection point in $\Ha(p_-) \cap \Hd(p_+)$, to a point on $M$ in the boundary of $\Ha(p_-)\cap F^{-1}(c_{\bbot})$;
see Figure~\ref{fig:membranesgamma}.
The curve $\guiding$ will correspond to the guiding curve in Section~\ref{sec:61}.

The construction of $\guiding$ is as follows. The set $\Ha(p_-)\cap F^{-1}(c_{\bbot})$ is path connected. Its boundary, the unstable manifold of $p_-$, consists of at least two points. Precisely one point of the boundary
belongs to $\Hd(p_+)$: this point corresponds to the unique trajectory on $M$ connecting $p_-$ with $p_+$. There exists a point $z_{\eend}$ on $\bd(\Ha(p_-)\cap F^{-1}(c_\bbot))$ disjoint from $\Hd(p_+)$.
By path connectedness, we extend the curve $\guiding$ to a curve in $\Ha(p_-)\cap F^{-1}(c_{\bbot})$ in such a way that $\guiding(1)=z_{\eend}$ and $\guiding$ is transverse to $M$ at $z_{\eend}$.

It remains to show that $\guiding$ can be chosen in such a way that
$\guiding((\varepsilon_{\oout},1])\cap Y_{\oout}=\emptyset$. Note that $\Hd(p_+)$ intersects $\Ha(p_-)\cap F^{-1}(c_{\bbot})$ along points $z_0,\dots,z_r$ and so $Y_{\oout}$ intersects $\Ha(p_-)\cap F^{-1}(c_{\bbot})$
along open balls near these points. If $\dim\Ha(p_-)\cap F^{-1}(c_{\bbot})>1$, the complement $\Ha(p_-)\cap F^{-1}(c_{\bbot})\setminus Y_{\oout}$ is still path connected so the choice of $\guiding$ does not pose
any problems.

If $\dim\Ha(p_-)\cap F^{-1}(c_\bbot)=1$, though, $\Ha(p_-)\cap F^{-1}(c_{\bbot})$ is an arc with end points $z_0$ and $z_{\eend}$ (there is only one possibility of chosing $z_{\eend}$). The set $Y_{\oout}$ separates this arc.
In order that $\guiding$ exists, we need to choose the point $z_m$, from which $\guiding$ emerges, to be the \emph{nearest} point (among $z_1,\dots,z_r$) to $z_{\eend}$ on $\Ha(p_-)\cap F^{-1}(c_{\bbot})$, compare
Figure~\ref{fig:membranes}, where $m=4$ is the only possible choice. It might also happen that \eqref{eq:gammashort} defining $\guiding$ points in the other direction, i.e.\ in Figure~\ref{fig:membranes}, the curve from $z_4$
goes towards $z_{3}$ instead of $z_{\eend}$ as the parameter $y_2$ increases. If this happens, we globally replace $y_2$ by $-y_2$.  Note that $F$ does not depend on $y_2$ and $\xi$ depends on $y_2$ only via $y_2^2$, so this change takes one coordinate system to another one,
in which both $F$ and $\xi$ have the same shape.

%Extend $\guiding$ to a smooth curve in $\Ha(p_-)\cap F^{-1}(c_{\bbot})$ connecting $\guiding(\varepsilon_{\oout})$
%with a point $z_{\eend}\in M\cap F^{-1}(c_{\bbot})$, and parameterise $\guiding$ in such a way that the endpoint on $M$ corresponds to $\guiding(1)$. We will assume
%that $\guiding$ is transverse to $M$ at $\guiding(1)$. We also insist that $z_{\eend}$ be disjoint from the descending membrane of $p_+$. This is always possible. In fact, for
%any choice of $z_{\eend}$ it belongs to the ascending membrane of $p_-$, so if it belongs also to the descending membrane of $p_+$, it belongs to the trajectory
%connecting $p_-$ and $p_+$. We assumed there is exactly one such trajectory  and the boundary of $\Ha(p_-)\cap F^{-1}(c_{\bbot})$ consists of more than one point.
%Thus, there is always a choice of $z_{\eend}$ which does not belong to the descending manifold of $p_+$.

%We want to make sure that $\guiding((\varepsilon_{\oout},1])\cap Y_{\oout}=\emptyset$ as in Figure~\ref{fig:membranesgamma}.

\begin{figure}
\input{pictures/membranesgamma.tex}
\caption{The guiding curve $\guiding$.}\label{fig:membranesgamma}
\end{figure}

\smallskip
We have constructed a guiding curve $\guiding$.
For technical reasons we need to extend $\guiding$ slightly across the point $\guiding(1)$, so we will assume, for some $\eg > 0$, that
$\guiding\colon [0,1+\eg] \to  F^{-1}(c_{\bbot})$
is defined as a small extension of $\guiding \colon [0,1] \to F^{-1}(c_{\bbot})$, and that $\guiding$ is transverse to $M$ at $\guiding(1)$. This is why $\guiding$ extends slightly beyond $M$ in Figures~\ref{fig:membranesgamma} and~\ref{fig:xout}.

\subsection{Coordinates on a neighbourhood of the guiding curve.}\label{sub:fingernbhd4}

Now we are going to extend the coordinate system from $U_{\oout}$ to a neighbourhood of~$\guiding$.
First the coordinate system will be constructed in the level set $F^{-1}(c_{\bbot})$, then in later subsections it will be extended to $F^{-1}[c_{\bbot},c_{\ttop}]$.

\begin{figure}
\input{pictures/xout.tex}
\caption{The set $X_{\oout}$.}\label{fig:xout}
\end{figure}
Choose a small neighbourhood $X_{\oout}$ of $\guiding(\varepsilon_{\oout}/2,1+\eg)$ in $F^{-1}(c_{\bbot})$, shown in Figure~\ref{fig:xout}.
We will assume that $X_{\oout}$ has a product structure, that is $X_{\oout}\cong \guiding(\varepsilon_{\oout}/2,1+\eg)\times D$
for some $n$-dimensional disc $D$ ($\O$ is $(n+2)$-dimensional so level sets are $(n+1)$-dimensional) and that $X_{\oout}$ intersects the manifold $M$ away from the descending manifold of $p_+$. In the local coordinate system in $Y_{\oout}\cap X_{\oout}$, we will assume that $X_{\oout}$ is given by
$\{x_1^2+\cdots+x_n^2\le\wt{\varepsilon}_{\oout}^2,y_2\in(\varepsilon_{\oout}/2,\varepsilon_{\oout})\}\cap Y_{\oout}$,
where $\wt{\varepsilon}_{\oout}\in(0,\varepsilon_{\oout}/2)$ is a real number.
Recall that $y_1$ is fixed by being on the level set $F^{-1}(c_{\bbot})$, and $y_2$ changes along $\guiding$.

With respect to the product structure $X_{\oout}\cong\guiding\times D$,
we write $\partial X_{\oout}=D_-\cup R_{\ttang}\cup D_+$, where $D_- = \guiding(\varepsilon_{\oout}/2)\times D$, $D_+=\guiding(1+\eg)\times D$ and
$R_{\ttang}=\guiding\times\bd D$. %\footnote{MP: I changed a $D_0$ to $D$ here, is that correct? I couldn't figure out what $D_0$ could mean in this context.}
Here $D_-$ is given by $X_{\oout}\cap\{y_2=\varepsilon_{\oout}/2\}\cap\{x_1^2+\cdots+x_n^2\le\wt{\varepsilon}_{\oout}^2\}$.
As $\wt{\varepsilon}_{\oout}<\varepsilon_{\oout}/2$ we have that $D_-\subseteq Y_{\oout}$.

Choose an auxiliary nonvanishing vector field $\auxfield$ on $X_{\oout}$ such that the following hold:
\begin{enumerate}[label=(L-\arabic*)]
  \item\label{item:L1} The vector field $\auxfield$ is tangent to $\guiding$ and to $\Ha(p_-)\cap F^{-1}(c_{\bbot})$.
\item\label{item:L2} The vector field $\auxfield$ is transverse to $D_-$ and to $D_+$. Moreover $\auxfield$ points into $X_{\oout}$ on $D_-$ and out of $X_{\oout}$ on $D_+$.
\item\label{item:L3} The vector field $\auxfield$ is tangent to $R_{\ttang}$.
\item\label{item:L4} For any $w\in X_{\oout}\sm (D_+ \cup D_-)$, the trajectory of $\auxfield$ through $w$ hits $D_+$ in the future and hits $D_-$ in the past.
\item\label{item:L5} The time it takes the trajectory of $\auxfield$ to go from $D_-$ to $D_+$ is equal to $1+\eg$.
\item\label{item:L6} In the local coordinates on $U_{\oout}\cap X_{\oout}$, the vector field $\auxfield$ is given by $\dy{2}$.
\end{enumerate}
The existence of such a $\auxfield$ is clear because $X_{\oout}$ is homeomorphic to a disc. Notice that \ref{item:L6} is compatible with
\ref{item:L1}---\ref{item:L5}.
Condition \ref{item:L5} can be obtained after all the other conditions by rescaling.
The flow of $\auxfield$ allows us to extend the coordinate system from $U_{\oout}\cap X_{\oout}$
to the whole of $X_{\oout}$.

Take a point $u \in X_{\oout}$ and let $\theta$ be the trajectory of $\auxfield$ through $u$.  Parameterise $\theta$ as $\theta(s)$ so
that $\theta(\varepsilon_{\oout}/2)=w \in D_-$ and $\theta'(s) = \auxfield_{\theta(s)}$ for all $s \in [\varepsilon_{\oout}/2,1+\eg)$.
Then, identifying $w$ with its expression in coordinates in $U_{\oout}$, define the coordinates of $u$ to be $w+(0,\ldots, 0,0,s_u)$, where~$s_u$
is the unique value  of $s \in [\varepsilon_{\oout},1+\eg)$ such that $\theta(s_u)=u$.

\begin{figure}
\input{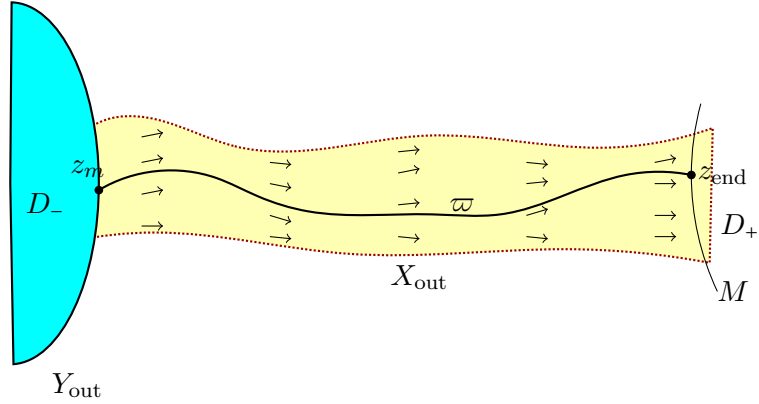}
\caption{The auxiliary vector field $\auxfield$ and its flow.}\label{fig:auxillaryvector}
\end{figure}

We conclude this subsection by making an assumption,
 whose r\^{o}le is to simplify the formulae that will appear in the sequel.

\begin{lemma}\label{lem:simplifyinglemma}
By possibly altering the vector field $\xi$ below the level set $F^{-1}(c_{\bbot})$, we may assume that
the connected component of $M\cap F^{-1}(c_{\bbot})\cap X_{\oout}$ containing $\guiding(1)$ is given by the set of equations $x_1=0$, $y_2=1$ and
$y_1=x_1^2+\dots+x_{h+1}^2-x_{h+2}^2-\ldots-x_n^2$.
It follows in particular that $\Ha(p_-)\cap X_{\oout}$ is contained in $\{y_2\le 1\}$.
\end{lemma}

\begin{proof}[Sketch of proof]
  We know that $M\cap X_{\oout}\cap F^{-1}(c_{\bbot})$
  is an $(n-1)$--dimensional
 disc containing $(0,\ldots,0,0,1)$ and also the equations
$x_1=0,y_2=1$, $y_1=x_1^2+\cdots+x_{h+1}^2-x_{h+2}^2-\dots-x_n^2$ define an $(n-1)$--dimensional disc in $X_1$ containing $(0,\ldots,0,0,1)$. These
two discs are isotopic,
%\footnote{Why? Because we chose $X_{\oout}$ so that this is unknotted?}\footnote{MB: there's no room for knottedness. The discs are small (we can choose sufficiently small nbhd of $z_{\eend}$), so being isotopic is trivial}
 so by an isotopy of $M$ we can achieve that they coincide.
Such an isotopy forces a change in $\xi$ below the level set of $F^{-1}(c_{\bbot})$, by the Isotopy Insertion Lemma~\ref{lem:isoinject}; compare Section~\ref{sub:fingernbhd2}.
\end{proof}

\subsection{Flowing the guiding curve to different level sets.}\label{sub:last_technical_finger}

Recall that we defined $\Xi_s$ to be the flow of $\xi$ and define
\[\ol{Z}_{\oout} := \bigcup_{s \in \R} \Xi_s(\ol{X}_{\oout})\cap F^{-1}[c_{\bbot},c_{\ttop}],\ \ Z_{\oout}=\Int\ol{Z}_{\oout}.\]
Then set
\[V_{\oout} := U_{\oout}\cup Z_{\oout}.\]
\begin{figure}
\input{pictures/setsX2andZ2.tex}
\caption{The sets $Z_{\oout}$ and $U_{\oout}$.}\label{fig:setsX2andZ2}
\end{figure}
Using the flow of $\xi$, we can extend the coordinate system from $U_{\oout}\cup X_{\oout}$ to $V_{\oout}$. Note that the fact that $\guiding(1)$ does not belong
to the descending membrane of $p_+$ implies that $Z_{\oout}$ is a product $X_{\oout}\times(c_{\bbot},c_{\ttop})$
and the vector field $\xi$ does not vanish on $Z_{\oout}$.
\begin{remark*}
  This is a place where we assume that there are no critical points of $F$ other than $p_-$ and $p_+$. Otherwise, we should strive
  to show that there are no trajectories connecting $\ol{X}_{\oout}$ with a critical point.
\end{remark*}

Recall that $U_{\oout}$ is a neighbourhood of the critical point $p_+$, whereas $X_{\oout}$, a neighbourhood of the part of $\guiding$ that is outside $U_{\oout}$, is a subset of the level set $F^{-1}(c_{\bbot})$; see Figure~\ref{fig:setsX2andZ2}.

We want to extend our coordinate system to $Z_{\oout}$.
The extension on $Z_{\oout}$ is given by the requirement that
the vector field $\xi$ is given by
\begin{equation}\label{eq:formofxi}
(-x_1,\ldots,-x_{h+1},x_{h+2},\ldots,x_n,y_1^2+y_2^2,0).
\end{equation}
The procedure for extending the coordinate system is analogous to extending the coordinate system to $X_{\oout}$ using  the
auxiliary vector field $\auxfield$,
which was described above.  %Of course, it differs in that $\xi$ is the gradient-like vector field, and is not auxiliary.

Our aim is to show that in the new coordinate system $F$ is given by
$-x_1^2-\dots-x_{h+1}^2+x_{h+2}^2+\dots+x_n^2+y_1$ (up to a constant).
In order to obtain this condition and \eqref{eq:formofxi} at the same time, we need to rescale $\xi$
by a positive factor. The procedure is as follows.
%\footnote{[MB] Rephrased a wordy statement in a more concise way.}
%There will be one more technical issue, because we want to control the formula for the function $F$.
%Ideally we would like to have the form
%This condition and \eqref{eq:formofxi} are not automatically satisfied, because altogether they give formula
%for $\partial_\xi F$, which is independent on the choice of coordinate system. This means that if we want to keep
%$F$ in that form and $\xi$ as in \eqref{eq:formofxi} we need to rescale $\xi$ by an appropriate factor.

%This is not guaranteed by the form of the vector field \eqref{eq:formofxi}. In fact, as described below,
%we can control the form of $F$ at the price of rescaling the vector field $\xi$.

Suppose $\phi\colon Z_{\oout}\to (0,\infty)$ is a smooth function
and consider the vector field $\phi\xi$.
Take a point $u\in Z_{\oout}$.
Suppose there is a trajectory $\theta_\phi$ of $\phi\cdot \xi$ such that $\theta_\phi(0)=u'\in X_{\oout}$. Suppose
$u'$ has coordinates $(x'_1,\ldots,x'_n,y'_1,y'_2)$.
For $s\ge 0$, we set the coordinates of $\theta_\phi(s)$ to be
$(x_1^\phi,\dots,y_2^\phi)$, where %\footnote{MB: We never use $u$ in the proof, so we need not introduce a new variable.}
\begin{equation}\label{eq:xi_phi}
  \begin{split}
    x_i^\phi& =e^{-s}x'_i\textrm{ for }i\le h+1,\ \ x_i^\phi=e^sx'_i\textrm{ for }i>h+1,\\ y_1^\phi&=y'_2\tan(sy'_2+\arctan(y'_1/y'_2)),\ \  y_2^\phi=y_2'.
  \end{split}
\end{equation}
For simplicity, we do not write explicitly the dependence of
$x_1^\phi,\dots,y_2^\phi$ on the variable $s$.
In these coordinates, %$\phi\xi$ changes the coordinates is
the vector field $\phi\xi$ has the form
\begin{equation}\label{eq:formofxiphi}
  (-x_1^\phi,\ldots,-x_{h+1}^\phi,x_{h+2}^\phi,\ldots,x_n^\phi,(y_1^\phi)^2+(y_2^\phi)^2,0).
\end{equation}
To justify this, one must think of $(x'_1,\ldots,x'_n,y'_1,y'_2)$ as constants, and differentiate $x_i^\phi$ and $y_i^\phi$ with respect to the variable $s$. This yields:
\begin{align*}
\frac{{d}x_i^\phi}{{d}s} &= \frac{{d}}{{d}s}  (-e^{-s}x_i') = - x_i^\phi, \text{ for } i \leq h+1, \\
\frac{{d}x_i^\phi}{{d}s}& = \frac{{d}}{{d}s}  (e^{s}x_i') = x_i^\phi \text{ for } i > h+1, \\
\frac{{d} y_1^\phi}{{d}s}& = (y_2')^2 \sec^2\big(sy_2' + \arctan(y_1'/y_2')\big) = \\=&(y_2')^2\big(1 + \tan^2\big(sy_2' + \arctan(y_1'/y_2')\big)\big) = (y_2')^2 +(y_1')^2.
\end{align*}
Also note that at $s=0$ we recover $x_i^\phi= x_i'$ and $y_i^\phi = y_i'$. Therefore the path $s\mapsto (x_1^\phi,\dots,x_n^\phi,y_1^\phi,y_2^\phi)$
is the trajectory of the vector field \eqref{eq:formofxi}, but it is also, by the definition, the trajectory $\theta_\phi(s)$ of $\phi\xi$. Therefore, indeed, $\phi\xi$
is given by \eqref{eq:formofxi}. %\footnote{[MB] Added a few sentences.}

The next result adjusts the function $\phi$ in such a way that $F$ has the desired form in the new coordinates.

\begin{lemma}\label{lem:functionisgood}
  There exists a unique choice of smooth function $\phi \colon Z_{\oout} \to (0,\infty)$ such that, in the coordinate system $(x_1^\phi,\dots,y_2^\phi)$, $F$ has the form
\begin{equation}\label{eq:formofF}
  F=-(x_1^\phi)^2-\ldots-(x_{h+1}^\phi)^2+(x_{h+2}^\phi)^2+\dots+(x_n^\phi)^2+(y_1^\phi)+c_+,
\end{equation}
where we recall that $c_+=F(p_+)$.
%\footnote{for some $c_+ \in \R$?}\footnote{MB: explained.}
\end{lemma}

\begin{proof}
%The idea of the proof is to rescale the vector field in such a way that $\partial_\xi F$ is equal to $\partial_\xi(-x_1^2-\dots-x_{h+1}^2+x_{h+2}^2+\dots+x_n^2+y_1)$.
%The problem is that we can not simply multiply the vector field by the ratio of these two expressions, because the coordinates themselves depend on $\xi$.

Throughout the proof, whenever we talk of a trajectory of a vector field (usually $\xi$ or $\phi\xi$), we mean
the part of the trajectory that stays in $Z_{\oout}$.

Choose a point $u'\in X_{\oout}$, whose coordinates are
$x_1',\ldots,x_n',y_1',y_2'$. Notice that $y_2'\ge\varepsilon_{\oout}/2>0$. Define
an abstract function $\nu\colon\R_{\ge 0}\to\R$ by
\[
\nu(s)=-e^{-2s}({x'_1}^2+\ldots+{x'_{h+1}}^2)+e^{2s}({x'_{h+2}}^2+\ldots+{x'_n}^2)+y_2'\tan(s{y_2'}+\arctan(y_1'/y_2')).
\]
Observe that this is always an increasing function.

Let $\theta$ be the trajectory of $\xi$ such that $\theta(0)=u'$.
The functions $s\mapsto F(\theta(s))$ and $s\mapsto \nu(s)$ are
both smooth, strictly increasing functions of a real positive argument that agree at $s=0$ and $\nu$ is unbounded. By one-dimensional
inverse function theorem, there exists an increasing function $\psi\colon\R_{\ge 0}\to\R_{\ge 0}$ with $\psi(0)=0$ such
that
\begin{equation}\label{eq:def_of_psi}
  F(\theta(\psi(s)))=\nu(s).
\end{equation}
Note that the function $\psi$ depends implicitly on the choice of $u'$. The formula
\begin{equation}\label{eq:def_of_phi}
  \phi(\theta(\psi(s)))=\frac{d}{ds}\psi(s)
\end{equation}
%\begin{equation}\label{eq:def_of_phi}
defines the value of $\phi$ on the trajectory of $\xi$ through $u'$.
Different choices of initial points $u'$ allow us
to define $\phi$ on the whole of $Z_{\oout}$. With this definition, a straightforward argument involving
smooth dependence of solutions on intial conditions, shows that $\phi$ is a smooth function on the whole of $Z_{\oout}$.
%\footnote{[MB] Rephrased, is it better?}

%Let $\theta_\phi$ be the trajectory of the vector field $\phi\xi$ such that $\theta_\phi(0)=u'$.
We claim that with the choice of $\phi$ from \eqref{eq:def_of_phi}
%\footnote{MB: commented out the repetition of definition of $\theta_\phi$.}
\begin{equation}\label{eq:theta_is_psi}
  \theta_\phi(s)=\theta(\psi(s)).
\end{equation}
To see this, notice that the two functions are equal at $s=0$. We calculate the differentials of both sides. By the definition of $\theta_\phi$:
\begin{equation}\label{eq:theta_is_psi2}
  \frac{d}{ds}\theta_\phi(s)=\phi(\theta_\phi(s))\xi(\theta_\phi(s))
\end{equation}
On the other hand:
\begin{equation}\label{eq:theta_is_psi3}
  \begin{split}
    \frac{d}{ds}\theta(\psi(s))=&\frac{d}{ds}\psi(s)\theta'_s(\psi(s))\stackrel{(*)}{=}
    \phi(\theta(\psi(s)))\theta'(\psi(s))\stackrel{(**)}{=}\\
    &=\phi(\theta(\psi(s)))\xi(\theta(\psi(s))).
\end{split}
\end{equation}
Here (*) is precisely \eqref{eq:def_of_phi}, whereas (**) holds because $s\mapsto \theta$ is the trajectory of $\xi$.
Therefore the two functions, $\theta_\phi(s)$ and $\theta(\psi(s))$ satisfy
the same differental equation
\begin{equation}\label{eq:flow_def}\frac{d}{ds}y(s)=\phi(y(s))\xi(y(s))\end{equation} with the same initial condition.
Hence, they are equal. Note that the differential equation \eqref{eq:flow_def} is exactly
the flow of the vector field $\phi\xi$.

Combining \eqref{eq:def_of_psi} with \eqref{eq:theta_is_psi} we see that
\[F(\theta_\phi(s))=\nu(s).\]
With the choice of coordinate system as in \eqref{eq:xi_phi} we note that $F$ has the desired form.
\end{proof}

By construction, the function $\phi$ is uniquely determined on $Z_{\oout}$, in particular $\phi\equiv 1$ on $U_{\oout}\cap X_{\oout}$.
We extend $\phi$ by $1$ to the whole of $X_{\oout}$ and to a smooth positive function on the whole of $\Omega$ (which we take equal to $1$ away from
a small neighbourhood of $Z_{\oout}\cap X_{\oout}$). We rescale now $\xi$ to $\phi\xi$; this preserves all the membranes.
We also drop superscript $\phi$ from the coordinates $x_1,\dots,y_2$.
By Lemma~\ref{lem:functionisgood}
the function $F$ has the form \eqref{eq:formofF} on the whole of $Z_{\oout}$.

\smallskip
We have now completed the definition of the set $V_{\oout}$ and the coordinate
system on it. In what follows, we will refer to $V_{\oout}$ as the \emph{outer shell}.
%We summarise its properties of in Section~\ref{sub:fingernbhd7}.\footnote{[MB] I commented out the part on $\Gamma$. We never use this part.}

\subsection{Properties of the outer shell}\label{sub:fingernbhd7}

Before we pass to the construction of $V_{\mmid}$ and $V_{\iinn}$, we sum up the construction of $V_{\oout}$.
It is an open subset of $\Omega$, endowed with a coordinate system $(x_1,\ldots,x_n,y_1,y_2)$, such that the following holds.
\begin{enumerate}[label=(V-\arabic*)]
  \item\label{item:V1} The vector field $\xi$ is given by $(-x_1,\ldots,-x_{h+1},x_{h+2},\ldots,x_n,y_1^2+y_2^2,0)$. In particular
$\Hd(p_+)$ is given by $\{x_{h+2}=\ldots=x_n=0,y_2=0,y_1\le 0\}$.
\item\label{item:V2} The manifold $M$ is given by $\{x_1=0,y_2=1\}\cup\{y_1=0,y_2=0\}$. The second component contains the critical
  point $p_+$ with coordinates $(0,\ldots,0)$.
\item\label{item:V3} The function $F$ is given by $-x_1^2-\ldots-x_{h+1}^2+x_{h+2}^2+\dots+x_n^2+y_1+c_+$.
\item\label{item:V4} The set $V_{\oout}$ is a grim neighbourhood of $p_+$; in particular the entrance set of $V_{\oout}$ is contained
  in $F^{-1}(c_{\bbot})$ and the exit set belongs to $F^{-1}(c_{\ttop})$.%\footnote{MB: Added this condition.}
\end{enumerate}

\section{The middle and the inner shells}\label{sec:middle_and_inner}

\begin{figure}
\input{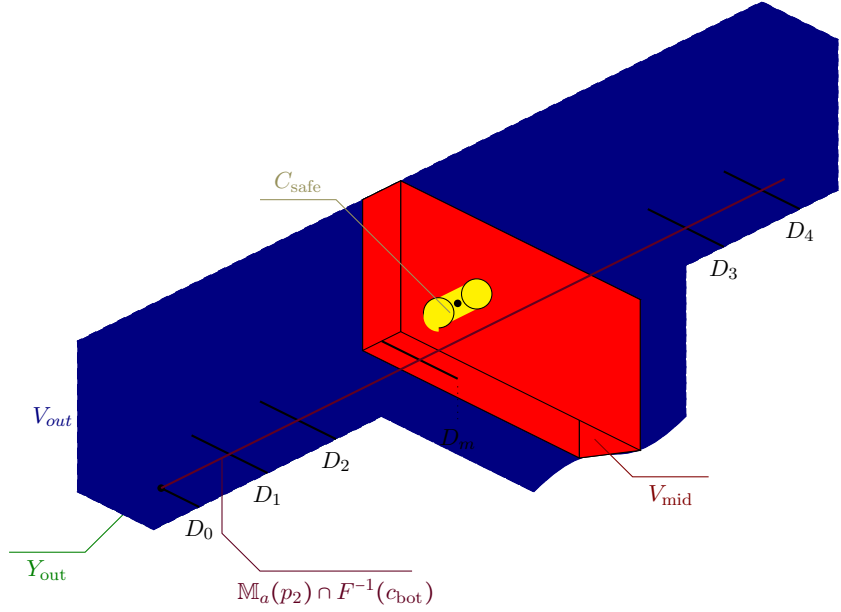}
\caption{The middle shell.}\label{fig:oneless}
\end{figure}

\subsection{Definining the middle shell}
The middle and inner shells $V_{\mmid}$ and $V_{\iinn}$ are constructed as open subsets satisfying $\ol{V}_{\iinn}\subseteq V_{\mmid}$, $V_{\mmid}\subseteq V_{\oout}$.
The objective of defining $V_{\iinn}$ and $V_{\mmid}$
is that the finger move will be done inside $V_{\iinn}$, and the vector field $\xi$ will be changed inside $V_{\mmid}$. The coordinate system
was defined on a larger set $V_{\oout}$ so that we can control  all the trajectories of $\xi$ reaching $p_+$.

Set
\[V_{\mmid}: =\{(x,y)\in V_{\oout} \colon |x_1|< \efng\};\]
see Figure~\ref{fig:oneless}.
The main objective will be to fix the quantity $\efng$. Before we do this, let us denote.
\[\ol{Y}_{\mmid}:= \ol{V}_{\mmid}\cap F^{-1}(c_{\bbot}),\ \ Y_{\mmid}=\Int Y_{\mmid},\]
where the interior is taken in the level set $F^{-1}(c_{\bbot})$; see Figure~\ref{fig:ymid}. %\footnote{MB: rephrased.}% for an illustration of the intersection of $\ol{V}_{\mmid}$ with $F^{-1}(c_{\bbot})$.
\begin{figure}
\input{pictures/ymid.tex}
\caption{The set $Y_{\mmid} := V_{\mmid}\cap F^{-1}(c_{\bbot})$.}\label{fig:ymid}
\end{figure}

As a first step towards adjusting the value of $\efng$, we prove the following result.
\begin{lemma}\label{lem:basic_efng}
  There exists $\efng^0>0$ such that if $\efng<\efng^0$, the middle set
  the set $Y_{\mmid}$ is disjoint from the discs $D_i$ if $i\neq m$; see Figure~\ref{fig:ymid}.
\end{lemma}
\begin{proof}
  By \eqref{eq:maft},
  the disc $D_i$ is contained in $\{x_1=s_{i,1}\}$, and for $i\neq m$, $s_{i,1}\neq 0$. We take $\efng<|s_{i,1}|$ for all $i\neq m$.
\end{proof}

%We require that only one of the discs $D_1,\ldots,D_r$
%defined by $\Ha(p_-)\cap F^{-1}(c_{\bbot})\cap Y_{\oout}$ has nonempty intersection with $V_{\mmid}\cap F^{-1}(c_{\bbot})$,
%as shown in Figure~\ref{fig:ymid}.
The set $V_{\mmid}$ is definitely not a grim neighbourhood of $V_{\mmid}\cap F^{-1}(c_{\bbot})$, because it does
not contain the whole of $\Hd(p_+)\cap F^{-1}(c_{\bbot})$. However,
if we restrict to the subspace $\{x_1=0\}$, then $V_{\mmid}$ turns out to be a grim neighbourhood of $p_+$.
We record this property for future use.
%We have the following property of $V_{\mmid}$.

\begin{lemma}\label{lem:vmid_is_grim}
The set $V_{\mmid}\cap\{x_1=0\}$ is a grim neighbourhood of $p_+$ in $V_{\oout}\cap\{x_1=0\}$.
\end{lemma}

\begin{proof}
  By definition we have $V_{\mmid}\cap\{x_1=0\}=V_{\oout}\cap\{x_1=0\}$. The set $V_{\oout}$ is the grim neighbourhood of $p_+$, see \ref{item:V4}.
The hypersurface $\{x_1=0\}$ is preserved by $\xi$. The statement follows quickly.
\end{proof}

\subsection{Trajectories starting from $D_i$}\label{sub:traj}
The choice of $\efng$ is determined by the control of trajectories starting at $D_i$ and coming close to $p_+$. We will consider
two subsets $\Csafe$ and $\Cvs$ (`vs' standing for `very safe') and shrink $\efng$ in such a way that any trajectory from $D_i$ enters $V_{\mmid}$ through $\Cvs$
and, if it exits $\Csafe$ later on, it does it above the level set $c_+$. The finger move will not alter the vector field outside
of $V_{\mmid}\setminus\Csafe$. This means that the finger move will not create any new trajectories from $D_i$ to $p_+$; compare Theorem~\ref{thm:oneless}.

To begin with, let $\esafe,\evs$ be two positive numbers with $\esafe>\evs$. Define
\begin{align}\label{eq:csafe_def}
  \Csafe&=\{x_2^2+\cdots+x_n^2+y_1^2+y_2^2\le \esafe^2, x_1\in[-\efng,\efng]\}\\
  \Cvs&=\{x_2^2+\cdots+x_n^2+y_1^2+y_2^2\le \evs^2, x_1\in[-\efng,\efng]\}.
\label{eq:cvs_def}
\end{align}
The purpose of Subsection~\ref{sub:traj} is to prove the following two results.
The first one tells, that we can find a sufficiently thin neighbourhood of $\{x_1=0\}$ (controlled by the parameter $\efng$)
a neighbourhood $\Cvs$ of the critical point, such that if a trajectory starts from a point $D_i$ and hits $V_{\mmid}$,
then either it hits $\Cvs$ (which is unavoidable, since there are trajectories of $\xi$ from $D_i$ to $p_+$),
or it enters $V_{\mmid}$ at the level set, where it has no chance to hit $p_+$ afterwards.
\begin{lemma}\label{lem:safeentry}
  There exists $\efng^0>0$ such that for any $\efng<\efng^0$ there exists %$\esafe^0$, such that for any $\esafe<\esafe^0$ there exists
  $\evs^0>0$ with the property that if $\evs<\evs^0$, then the following holds.

If $w_1\in D_i$, $i\neq m$ and the trajectory of $\xi$ through $w_1$
enters $V_{\mmid}$ through $w_2\in\partial V_{\mmid}$ $($we require that $w_2$ is the point of the first entrance of the trajectory to $V_{\mmid}$ after it hits $w_1)$, then:
\begin{itemize}
\item[(a)] either $w_2\in \Cvs$;
\item[(b)] or $F(w_2)>c_+.$
\end{itemize}
\end{lemma}
\begin{proof}
Choose a point $w_1\in D_i$. Write its coordinates as
\[w_1=(s_{i,1},\ldots,s_{i,h+1},x_{h+2},\ldots,x_{n},y_{1,1},y_2),\]
where
\begin{equation}\label{eq:y1,1}
y_{1,1}=c_{\bbot}-c_++s_{i,1}^2+\cdots+s_{i,h+1}^2-x_{h+2}^2-\cdots x_n^2;
\end{equation}
see \eqref{eq:maft} and \eqref{eq:maft_2}.
Denote
\[\Delta_-=s_{i,1}^2+\cdots+s_{i,h+1}^2,\ \ \Delta_+=x_{h+2}^2+\cdots+x_n^2.\]
Note that $\Delta_-$ is fixed for all the points in $D_i$, while $\Delta_+$ can vary. There is
a uniform upper bound on $\Delta_-$ and $\Delta_+$ on the whole of $D_0\cup\dots\cup D_r\setminus D_m$.

Let $\gamma(s)$ be the trajectory of $\xi$ such that  $\gamma(0)=w_1$.
If $\gamma$ leads out of $V_{\oout}$ without
hitting $V_{\mmid}$, then, by \ref{item:V4}, it leads out at the level set $c_{\ttop}$, so it never reaches $p_+$; in that case there
is nothing to prove. So assume $\gamma$ reaches $V_{\mmid}$ and let
$\sh$ be the time of the first entrance of $\gamma$ into $V_{\mmid}$. Set $w_2=\gamma(\sh)\in\partial\ol{V_{\mmid}}$.

Given \ref{item:V1}, the coordinates of $w_2$ are calculated as follows,
\begin{equation}\label{eq:w2_entry}
  w_2=\gamma(\sh)=(e^{-\sh}s_{i,1},\ldots,e^{-\sh}s_{i,h+1},e^{\sh}x_{i,h+2},\ldots,e^{\sh}x_{i,n},y_1(\sh),y_2).
\end{equation}
where $y_1(\sh)$ can also be determined, compare \eqref{eq:xi_phi}, but we do not use the explicit formula here.

By the definition of $V_{\mmid}$, the $x_1$-coordinate of $w_2$ is
$\pm\efng$. Therefore
%\footnote{MP: I changed $c$ to $s$ here. Then, where does this come from? Can you reference or repeat the formula that is used here. }\footnote{MB: it is now explained above.}
\[\sh=\ln\frac{|s_{i,1}|}{\efng}.\]
Note that by making $\efng$ sufficiently small, we can make $\sh$ as large as we please.

Using \eqref{eq:w2_entry} and \ref{item:V3}, we compute
\begin{equation}\label{eq:fw2}
  F(w_2)= - e^{-2\sh}\Delta_-+e^{2\sh}\Delta_+ +y_1(\sh)+c_+.
\end{equation}
In the following estimate we recall $\sh$ is a function of $w_1$, but $\sh$ can be made uniformly large for all $w_1\in D_i$ (and for all
indices $i\neq m$) by choosing sufficiently small
$\efng$

\begin{lemma}\label{lem:stbound}
For any $\delta>0$ there exists $t_0>0$ such that for every $w_1 \in D_i$, if $\sh = \sh(w_1)>t_0$, then $y_1(\sh)>-\delta-\Delta_+$.
\end{lemma}
%Note that the value $t_0$ can be chosen the same for all $w_1$ in $D_i$.

\begin{proof}
  The proof relies on explicit calculations. While we could use the formula as in \eqref{eq:xi_phi} for $y_1$,
  we prefer to use more legible, but more rough, estimates.
  Let $\upsilon(s)$ be the $y_1$-coordinate of the point $\gamma(s)$. Note that $\upsilon(s)$ satisfies:
\begin{equation}\label{eq:dups}
  \begin{split}
\frac{d}{ds}\upsilon&=\upsilon^2+y_2^2\\
\upsilon(0)&=c_{\bbot}-c_++\Delta_--\Delta_+.
\end{split}
\end{equation}

The first equation holds because the $y_1$ coordinate of $\xi$ is $y_1^2 + y_2^2$. The second holds by \eqref{eq:y1,1}.
We have $\frac{d}{ds}\upsilon(s)\ge \upsilon^2$.  Let $\upsilon_0(s)$ be the solution to the ODE $\frac{d}{ds}\upsilon_0(s) = \upsilon_0^2$ with initial condition $\upsilon_0(0) = \upsilon(0).$
That is, \[\upsilon_0(s)= \frac{1}{\frac{1}{\upsilon(0)}- s}.\]
Gronwall's inequality from the classical theory of ODEs implies that $\upsilon(s) \geq \upsilon_0(s)$ for $s \geq 0$, so %we infer that
\[\upsilon(\sh)\ge \frac{1}{\frac{1}{\upsilon(0)}-\sh}.\]
Now $\frac{1}{\frac{1}{c_{\bbot}-c_++\Delta_--\Delta_+}-\sh}$ tends to zero as $\sh\to\infty$. %Also $\Delta_+$ is bounded as $w_1$ varies in $D_i$.
Hence, for sufficiently large $\sh$ we have $\upsilon(\sh) \ge -\delta-\Delta_+$.
\end{proof}

Resuming the proof of Lemma~\ref{lem:safeentry},
choose $\delta<\evs^2/10$ and adjust $\efng$ according to Lemma~\ref{lem:stbound}. Recall that $\Delta_-$ is constant on $D_i$ and, making $\efng$ smaller if necessary,
choose $\efng$ in such a way that $e^{-2\sh}\Delta_{-}<\evs^2/10$ as well.
Then \eqref{eq:fw2} implies that
\begin{align}\label{eq:fvs2}
  F(w_2) = - e^{-2\sh}\Delta_-+e^{2\sh}\Delta_+ +y_1(\sh)+c_+ & \geq - e^{-2\sh}\Delta_-+e^{2\sh}\Delta_+ -\delta - \Delta_+ +c_+ \\&  \geq (e^{2\sh}-1)\Delta_+ +c_+-\evs^2/5. \notag
\end{align}

%Recall that the coordinates of $w_2$ are
%\[(e^{-\sh}s_{i,1},\dots,e^{-\sh}s_{i,h+1},e^{\sh}x_{h+2},\dots,e^{\sh}x_n,\upsilon(\sh),y_2),\]
%where $\upsilon(s)$ is as in the proof of Lemma~\ref{lem:stbound}.
Suppose $w_2\notin\Cvs$.
Recall that $\Cvs$ was given by \eqref{eq:cvs_def}, that is
$|x_1|\le \efng$ and
$x_2^2+\dots+y_2^2\le\evs^2$.
If $w_2\notin\Cvs$, by
\eqref{eq:w2_entry}, we obtain
the following conditions:
\begin{equation}\label{eq:w2_not_in_cvs}
  \left|e^{-\sh}s_{i,1}\right|>\efng,\textrm{ or }
  e^{-2\sh}\Delta_-+e^{2\sh}\Delta_++y_1(\sh)^2+y_2^2>\evs^2.
\end{equation}
However, we know that $w_2$ is on the boundary of $V_{\mmid}$,
so $|e^{-\sh}s_{i,1}|=\efng$. Therefore, the second inequality
in \eqref{eq:w2_not_in_cvs} has to hold.
There are four terms on the left-hand side of the second inequality,
so at least one of them is greater than one fourth of the right-hand side.
We will list relevant cases and try to show that either
there is an immediate contradiction (using the fact that $\sh$ is large), or that $F(w_2)>c_+$.
These cases are:
\begin{itemize}
\item $e^{-2\sh}\Delta_->\evs^2/4$. This is impossible if $\sh>s_0$ for some fixed $s_0$, which can be made uniform
  for all $w_1$.
\item $e^{2\sh}\Delta_+>\evs^2/4$. Then $(e^{2\sh}-1)\Delta_+>\frac29\evs^2$ if $\sh$ is large,
  and so by \eqref{eq:fvs2} we conclude that $F(w_2)>c_+$.
\item $y_1(\sh)<-\sqrt{\evs^2/4} = -\evs/2$. But then by Lemma~\ref{lem:stbound} we infer that $-\evs/2 > -\delta - \Delta_+> -\evs^2/10 - \Delta_+$, which implies that $\Delta_+\ge \evs/2-\evs^2/10$. This
is greater than $\evs^2/4$ if $\evs$ is small. We reduce to the previous case.
\item $y_1(\sh)>\sqrt{\evs^2/4} = \evs/2$ and $e^{-2\sh}\Delta_-<\evs^2/4$.
  Then by \eqref{eq:fw2} we obtain that $F(w_2)\ge c_++\evs/2-\evs^2/10>c_+$, again provided $\evs$ is small.
\item $y_2^2(\sh)\ge \evs^2/4$. The $y_2$-coordinate is constant along the trajectory of $\xi$, hence the $y_1$ coordinate of $\gamma(s)$
  satisfies the inequality
  $\frac{d}{ds}\upsilon(s)\ge\evs^2/4$ (compare \eqref{eq:dups})
  hence $\upsilon(\sh)\ge \upsilon(0)+\sh\evs^2/4$. With $\sh$ sufficiently large we arrive
  at the previous case, that is $y_1(\sh)>\evs/2$.
  %In \eqref{eq:fw2}, the term $y_1(\sh)$ grow linearly
  %and it is unbounded, while $\Delta_-$ is bounded on the whole of $D_i$, hence, if $\sh$ is large enough, we have $F(w_2)>c$.
\end{itemize}
We have assumed that $w_2$ is not contained in $\Cvs$, and have shown that $F(w_2) > c_+$, as desired.
\end{proof}

The second result after Lemma~\ref{lem:safeentry} also adjusts the value of $\efng$, but with respect to another parameter.
We fix a neighbourhood $\Csafe$ as above, whose size is controlled by a parameter $\esafe$. We prove we can choose $\efng$ small enough
that if a trajectory enters $\Cvs$ then on exiting the larger set $\Csafe$, it is at a level set above $p_+$, so that it has
no chances to reach $p_+$ in the future.
\begin{lemma}\label{lem:safeexit}
For any $\esafe>0$, there exists $\evs^0>0$ such that if $\evs<\evs^0$, then
if a trajectory of $\xi$  enters $V_{\mmid}$ through $w_2\in \Cvs$, and exits $\Csafe$ at a point $w_3$,
then $F(w_3)>c_+$.
\end{lemma}
\begin{proof}
  %In this proof we use a slightly different notation than
  %in the proof of Lemma~\ref{lem:safeentry}.
  %Let $w_2$ the entry point of the trajectory to $\Cvs$.
  We let the coordinates of $w_2$ be
\[w_2=(x_1,\ldots,x_n,y_1,y_2).\]
We also set
\[\Delta_-=x_1^2+\cdots+x_{h+1}^2,\ \Delta_+=x_{h+2}^2+\cdots+x_n^2.\]
As $w_2\in\Cvs$ we have that
\begin{equation}\label{eq:asw2}\Delta_-,\,\Delta_+,\, y_1^2,\,y_2^2\le\evs^2,\ \
|x_1|<\efng.
\end{equation}
The trajectory of $\xi$ through $w_2$ leaves $\Csafe$.
We let $s_0$ be the time after which this trajectory leaves $\Csafe$ for
the first time and let $w_3$ be the point at which this trajectory leaves $\Csafe$.
The coordinates of $w_3$ are clearly
\begin{equation}\label{eq:w3_exit}
  w_3=(e^{-s_0}x_1,\ldots,e^{-s_0}x_{h+1},e^{s_0}x_{h+2},\ldots,e^{s_0}x_{n},\wt{y}_1,y_2).
\end{equation}
Here, $\wt{y}_1$ is the $y_1$-coordinate of $w_3$.
We have $\wt{y}_1\ge y_1$ and $y_1\ge -\evs$ by \eqref{eq:asw2}, so
\begin{equation}\label{eq:asw21}
  \wt{y_1}\ge -\evs.
\end{equation}
By property \ref{item:V3} combined with \eqref{eq:w3_exit}:
\begin{equation}\label{eq:fvs3}
  F(w_3)-c_+= e^{2s_0}\Delta_+-e^{-2s_0}\Delta_-+\wt{y}_1.
\end{equation}
In \eqref{eq:csafe_def} $\Csafe$ is given by two conditions: $|x_1|\le\efng$, $x_2^2+\dots+\dots+y_2^2\le\esafe^2$.
As $w_3\in\partial\Csafe$, one of the two inequalities must be equality.
We argue as in the proof of Lemma~\ref{lem:safeentry}: either the first inequality is an equality, or one of the four terms
$x_2^2+\dots+x_{h+1}^2$, $x_{h+2}^2+\dots+x_n^2$, $y_1^2$, $y_2^2$ at $w_3$ should be greater or equal than $\esafe^2/4$.
The coordinates of $w_3$ are given by \eqref{eq:w3_exit}, so $w_3\in\partial\Csafe$ implies that one of the following holds.
%or $x_2^2+\dots+y_2^2=\esafe^2$. Using \eqref{eq:w3_exit}, we conclude that one of the following has to hold.
%\footnote{MP: How do we show this? I can't follow this proof. Which parts am I supposed to take from the previous lemma? Where does this list come from? I think I will wait to try to check this after you have added a bit more details/explanation. }\footnote{MB is it better now?}
\begin{itemize}
  \item[(a)] $e^{-2s_0}\Delta_-\ge\esafe^2/4$;
  \item[(b)] $e^{2s_0}\Delta_+\ge\esafe^2/4$;
  \item[(c)] $\wt{y}_1\ge\esafe/2$;
  \item[(d)] $\wt{y}_1\ge-\esafe/2$;
  \item[(e)] $|y_2|\ge\esafe/2$;
  \item[(f)] $|e^{-2s_0}x_1|\ge\efng$.
\end{itemize}
%In fact, if none the conditions (a)--(f) hold, then
%$w_3$ belongs to the interior of $\Csafe$.
We now show that among these possibilities some are self-contradictory (given \eqref{eq:asw2}),
the others lead to $F(w_3)>c_+$.
\begin{itemize}
  \item Case (a) is impossible, because $\Delta_-<\evs^2$; compare \eqref{eq:asw2}.
  \item If (b) holds, then in \eqref{eq:fvs3} we use that $\Delta_-\le\evs^2$, $e^{s_0}>1$ and \eqref{eq:asw21}
to get $F(w_3)-c_+\ge \esafe^2-\evs^2-\evs$. If $\evs$ is small enough compared with $\esafe$, we have $F(w_3)-c_+>0$.
  \item In case (c) we have $F(w_3)-c_+\ge \esafe-e^{-2s_0}\Delta_-\ge \esafe-\evs^2$. This difference is positive
if $\evs$ is sufficiently small.
  \item Case (d) contradicts \eqref{eq:asw21}.
  \item Case (e) cannot occur, because $y_2$ is constant on the trajectory, so we must have $|y_2|\le \evs$.
  \item Finally, we discuss case (f). Note that $|x_1|<\efng$ by \eqref{eq:asw2}. As $s_0>0$,
we have that $|e^{-2s_0}x_1|<\efng$.
\end{itemize}
We have shown that $F(w_3)>c_+$, so the lemma is proved.
\end{proof}

%Given Lemmas~\ref{lem:safeentry} and~\ref{lem:safeexit},
\begin{condition}\label{con:efng}
  The value of $\efng$ is chosen in such a way that the
statements of Lemmas~\ref{lem:basic_efng} and~\ref{lem:safeentry} hold.
\end{condition}
The values of $\esafe$ and $\evs$
will be adjusted in later subsections according to Lemmas~\ref{lem:safeentry} and~\ref{lem:safeexit}.

\subsection{The inner shell $V_{\iinn}$}\label{sub:Vinn}

We will also define the set $V_{\iinn}$. First introduce the following piece of notation.
For every $\alpha\in[0,1]$ we denote
\begin{equation}\label{eq:tc}
c_{-\alpha}=c_++(c_{\bbot}-c_+)\alpha \textrm{ and }c_{+\alpha}=c_++(c_{\ttop}-c_+)\alpha.
\end{equation}
%\footnote{MP: I changed one of the $c$s to $\alpha$. It was incomprehensible with all those $c$s!!}

For $\eeta>0$  consider a cylinder%\footnote{[MB] Previous definition of $V_{\iinn}$ was that it was a nowhere-dense subset.}
\begin{equation}\label{eq:cfng_def}
  V_{\iinn}(\eeta)=\{x_1^2+\cdots+x_n^2\le\eeta^2, y_1\in [c_{-1/2}-c_+,c_{+1/2}-c_+], y_2\in[-\eeta,1+\eeta]\} \subseteq V_{\oout}.
\end{equation}
\begin{lemma}\label{lem:eta_small_enough}
  For $\eeta$ small enough, the interior of $V_{\iinn}(\eeta)$ is a subset of $V_{\mmid}$.
\end{lemma}
\begin{proof}
  The intersection of all of the $V_{\iinn}(\eeta)$ is the 2-dimensional rectangle $(0,\dots,0,t,s)$, $t\in[c_{-1/2}-c_+,c_{+1/2}-c_+]$, $s\in[0,1]$. This
  rectangle belongs to $V_{\mmid}$ and we can choose a neighbourhood $U_{\fng}$ of the rectangle contained in $V_{\mmid}$. For $\varepsilon_{R_0}$ sufficiently small,
  the interior of $V_{\iinn}(\varepsilon_{R_0})$ (interior in $F^{-1}[c_{-1/2},c_{+1/2}]$) is contained in $U_{\fng}$. Hence, for $\eeta=\varepsilon_{R_0}/2$,
  the whole of $V_{\iinn}(\eeta)$ belongs to $V_{\mmid}$.
\end{proof}
\begin{condition}\label{con:eeta}
From now on set $\eeta$ in such a way that
\[V_{\iinn}:=V_{\iinn}(\eeta)\]
is contained in $V_{\mmid}$; compare Figure~\ref{fig:vinn}.
\end{condition}

%Note that $R$ will be later adjusted to $\efng$; see Remark~\ref{rem:order_of_choosing} below.

%We set $V_{\iinn}$ to be an open subset of $V_{\mmid}\cap F^{-1}(c_{-1/2},c_{+1/2})$ containing the interior of $V_{\iinn}$ such  $\ol{V}_{\iinn}\subseteq V_{\mmid}$;
%\footnote{[MB] an attempt to make more precise the definition of $V_{\iinn}$. Needs revision.}
\begin{figure}
  \input{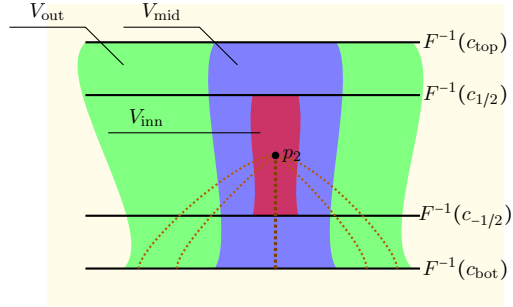}
  \caption{The position of the sets $V_{\iinn}$, $V_{\mmid}$ and $V_{\oout}$ presented in a two-dimensional schematic. The dotted lines represent the trajectories
  of $\xi$ connecting $p_-$ and $p_+$.}\label{fig:vinn}
\end{figure}

The purpose of the sets $V_{\iinn},V_{\mmid}$ and $V_{\oout}$ will be as follows.
\begin{itemize}
\item $V_{\iinn}$ is the set where the finger move is conducted.
\item $V_{\mmid}$ is the set where the vector field $\xi$ is changed after the finger move, so that it is still a grim vector field of some function.
\item $V_{\oout}$ is the set where we control the trajectories of the vector field $\xi$.
\end{itemize}

\section{The finger move}\label{sec:fingermove}

Having defined a suitable neighbourhood of the guiding curve in Subsection~\ref{sec:coorsystem}, we
pass to an explicit description of
the finger move in Subsection~\ref{sub:highfinger}, that is, of the isotopy $G_{1+\tau}$. Note that we assume that $G_\tau\equiv G$ for
$\tau\in[0,1]$.
In Subsection~\ref{sub:newvector} we show how to change the vector field $\xi$ to a new vector field $\xih$.
In Subsection~\ref{sub:critical} we study the critical points of $\xih$ and show that $\partial_{\xih}F\ge 0$,
which shows that $\xih$ is a good candidate for a grim vector field.
It turns out that $\xih$ vanishes on the newly created sphere of double points. In fact, the function $F$ is not an immersed Morse function on $G_2(M)$
(it has non-isolated critical points on the second stratum after the finger move).
Therefore, we define an explicit perturbation of both $F$ and
$\xih$ near that sphere of double points. This is done
in Subsection~\ref{sub:explicit_perturbation}, where we define the path $F_\tau$ for $\tau\in[0,1]$ and a perturbation of $\xih$, denoted by $\xit$.
Property \ref{item:FM_two_more} of Finger Move Theorem is proved in Subsection~\ref{sub:explicit_perturbation}, and \ref{item:FM_support_of_F2} follows immediately from the construction
of $F_\tau$.
The key property of $\xih$, namely \ref{item:FM_one_less}, is proved in Section~\ref{sec:membranesofxi}.

\subsection{Explicit description of the finger move}\label{sub:highfinger}

As described in Section \ref{sec:61} via an explicit example, we will start %above the critical level $c_+$, more precisely
at the level $c_{+1/4}$.
We will push $M$ along the guiding curve $\guiding$.
The length of the push will vary as the height varies.

%The cylinder $V_{\iinn}$ can be regarded as a neighbourhood of the end point $z_{\eend}$ of the curve $\guiding$,  pushed to the level set $c_+$.

%The value of $\eeta$ and $R$ will be adjusted to $\efng$ in the following way (see Remark~\ref{rem:order_of_choosing} below for
%the order of choosing various constants).
%\begin{lemma}\label{lem:eta_small_enough}
%  For $\eeta$ small enough and $R$ large enough, the interior of $V_{\iinn}(R,\eeta)$ is a subset of $V_{\iinn}$.
%\end{lemma}
%\begin{proof}
%%et $V_{\iinn}(R,\eeta)$ be the set $V_{\iinn}$ defined with given $R$ and $\eeta$.
%  Each of the sets $V_{\iinn}(R,\eeta)$ contains $\Gamma$.
%  The intersection of all $V_{\iinn}(R,\varepsilon)$ is precisely $\Gamma$.
%Its interior is a neighbourhood of the set
%$x_1=\cdots=x_n=0$, $y_2=1$, $y_2\in [c_{-1/4}-c_+,c_{1/4}-c_+]$, which is a compact subset of $V_{\iinn}$. The intersection
%of $V_{\iinn}(R,\eeta)$ is contained in $V_{\iinn}$  and $V_{\iinn}(R,\eeta)$ form a family of compact subsets such that $V_{\iinn}(R,\eeta)\subseteq V_{\iinn}(R',\eeta')$ if $R\geq R'$ and $\eeta\le\varepsilon'_R$. It follows that for $R$ sufficiently large and $\eeta>0$ sufficiently small, $V_{\iinn}(R,\eeta)$ must be a subset of $V_{\iinn}$.
%\end{proof}
%
By \ref{item:V2}, the intersection of $M$ with $V_{\oout}$ consists of two connected components. The component containing $p_+$ is
\[\Mtwo:=M\cap U_{\oout}.\]
The second component belongs to $\{y_2=1\}$. We consider its part contained in $V_{\iinn}$.
\[\Mone=M\cap V_{\iinn}\cap\{y_2=1\}.\]
Now we define the finger move, that is, the path $G_{1+\tau}$
of immersions such that $G_1=G$.
Choose $u\in N$. If $G(u)\notin\Mone$, we set $G_{1+\tau}(u)=G(u)$ for all $\tau\in[0,1]$.
If $G(u)=(x_1,\dots,y_2)\in\Mone$, we set
\begin{equation}\label{eq:Gtau_def_finger}
  G_{1+\tau}(u) := (x_1,\dots,y_1,(1-\tau)y_2+\tau\ftheta(x_1,\dots,x_n,y_1))\in V_{\iinn},
\end{equation}
$\tau \in [0,1]$, for a function $\ftheta \colon \R^{n+1} \to \R$
which we will define shortly.  Note that as $F$ does not depend on the $y_2$ coordinate, see property \ref{item:V3}, we immediately obtain.

\begin{lemma}\label{lem:independent}
  The composition $F\circ G_{1+\tau}$ is equal to $F\circ G$.
\end{lemma}

%Before we define the function $\ftheta$, we study for the moment the image of $G_2(N)$.
By definition, the finger move replaces the original part of $M$, $\Mone$ by
\begin{equation}\label{eq:mprimeta}
\Mfng:=\{x_1=0, y_2=\ftheta(x_1,\ldots,x_n,y_1)\},
\end{equation}
that is to say $G_2(N) = (M \sm \Mone) \cup \Mfng$.
We define
$\Mnew=G_2(N)$.
To complete the construction, we need to define $\ftheta$. Set
\begin{equation}\label{eq:Theta}
  \ftheta(x_1,\dots,y_1)=\fa(y_1)\min\{x_1^2+x_2^2+\dots+x_n^2,\eeta^2\}+\fb(y_1),
\end{equation}
where $\fa \colon \R \to [0,\infty)$ and $\fb \colon \R \to \R$ %\footnote{added domain and codomain of these functions.}
are two functions such that
\begin{equation}\label{eq:fafb}\fa(y_1)\eeta^2+\fb(y_1)\equiv 1.
\end{equation}
The function $\fa$ is nonnegative.   %Outside $V_{\iinn}$ and with $y_2=1$ (this is the case othe minimum in \eqref{eq:Theta} is $\eeta^2$, so $\ftheta \equiv 1$.
When all $x_i=0$, the function is $\fb(y_1)$.
Moreover we require that:
\begin{condition}\label{con:fa_fb}\
\begin{enumerate}[label=(A-\arabic*)]
  \item $\fb\equiv 1$ (and so $\fa\equiv 0$) for $y_1\notin[c_{-1/4}-c_+,c_{+1/4}-c_+]$; \label{item:A1}
  \item $\fb$ has a minimum at $y_1=0$, it is non-increasing for $y_1>0$ and non-decreasing for $y_1<0$;\label{item:A2}
  \item $\fb(0)<0$ and $|\fb(0)|<\eeta$, so that $\Mfng\subseteq V_{\iinn}$;\label{item:A3}
  \item $\fb'(y)\neq 0$ whenever $\fb(y)=0$.\label{item:A4}
\end{enumerate}
\end{condition}
A graph of an exemplary function for $\fb$ is presented in Figure~\ref{fig:beta}. A schematic of the finger move is depicted in Figure~\ref{fig:schematicfinger}.

\begin{figure}
  \input{pictures/beta.tex}
  \caption{A graph of $\fb \colon \R \to \R$.}\label{fig:beta}
\end{figure}
\begin{figure}
  \input{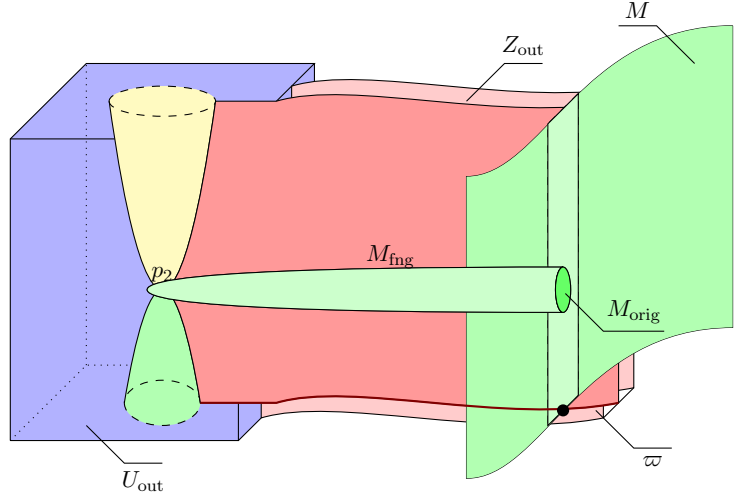}
  \caption{A schematic of the finger move.}\label{fig:schematicfinger}
\end{figure}
%Let $\Mnew=G_2(N) = (M\setminus \Mone)\cup \Mfng$. %Clearly $\Mnew=G_2(N)$.

\begin{remark*}
As the reader may have noticed, $\Mnew$ has corners along the common boundary of $\Mfng$ and $\Mone$.
These corners can be avoided by taking a smooth approximation of $\min(x_1^2+x_2^2+\ldots+x_n^2,\eeta^2)$ in \eqref{eq:mprimeta}. We will not do that explicitly because
it would make the exposition even more complicated.
\end{remark*}

Note that $\Mnew$ has extra double points. In fact, we have the following result towards proving property~\ref{item:connected}
\begin{lemma}\label{lem:gtau_is_immersion}
  The map $G_{1+\tau}$ is an immersion and it is regular for all values of $\tau$ except for a single value of $\tau$.
\end{lemma}
\begin{proof}
  Since $G_1$ is an immersion, and $G_\tau$ changes only the coordinate $y_2$, which was constant on $\Mone$,
  it is clear that $G_{\tau}$ is an immersion.
  In fact, an easy calculation shows that the double point set of $G_\tau(\Mone)\cap \Mtwo$
  intersect along the sphere $x_1=y_1=y_2=0$ and $\{x_2^2+\dots+x_n^2=\varepsilon_R^2-\frac{1}{\tau \fa(0)}\}$.
  The only $\tau$ for which the immersion is not regular, is when the radius of the sphere is zero, that is  $\tau=\varepsilon_R^2/\fa(0)$.
\end{proof}
From this lemma we deduce that
$\Mtwo\cap\Mfng$ is given by
\begin{equation}\label{eq:sigmadef}
  \Sigma=\{x_2^2+\dots+x_n^2=\erho^2\}\cap\{x_1=y_1=y_2=0\},\end{equation}
  where
%  \begin{equation}\label{eq:rhodef}
%\erho^2=-\frac{\fb(0)+1}{\fa(0)}.
%\end{equation}
  \begin{equation}\label{eq:rhodef}
\erho^2=-\frac{\fb(0)}{\fa(0)}.
\end{equation}
The double point set is an $(n-2)$-dimensional sphere.
Note that the two branches of $M$ intersect transversely.

The last observation in this subsection prove~\ref{item:connected}.
\begin{lemma}\label{lem:item_connected}
  The components $\Mone$ and $\Mtwo$ of $M\cap V_{\oout}$ belong to the same connected component of $M$.
\end{lemma}
\begin{proof}
  The part $\Mtwo$ is a subset of the connected component of $M$ containing $p_+$. Note that $p_+$ and $p_-$
  are connected by a single trajectory of $\xi$, hence $\Mtwo$ belongs to the same connected component of $M$ as $p_-$.
  The intersection of the ascending membrane $\Ha(p_-)$ with $M$ belongs to the same connected component of $M$ as $p_-$.
  The endpoint of the guiding curve $\guiding(1)$ belongs to the same connected component of $M$ as $p_-$, because it belongs
  to the ascending membrane. But $\Mone$ contains $\guiding(1)$ by construction.
\end{proof}

\subsection{Construction of $\xih$}\label{sub:newvector}

The next goal is to modify the vector field $\xi$  to obtain a vector field $\xih$, which will be perturbed to $\xit$ later on.
The definition of $\xih$ is more important, because it will already have the property \ref{item:FM_one_less} from Theorem~\ref{thm:new_finger_move} -- one fewer trajectory between $p_-$ and $p_+$. Note that
\begin{equation}\label{eq:kappa0}
\partial_\xi\ftheta(x_1,\ldots,y_2)=0\textrm{ if } x_1^2+x_2^2+\ldots+x_n^2\ge\eeta^2,
\end{equation}
which follows immediately from \eqref{eq:Theta}, since then $\ftheta$ is constant.
It is easy to see that the vector field
\begin{equation}\label{eq:wtxi_def}
\wt{\xi} := (-x_1,\ldots,-x_{h+1},x_{h+2},\ldots,x_n,y_1^2+y_2^2,\partial_\xi\ftheta)=\xi+\partial_\xi\ftheta\dy{2}
\end{equation}
is tangent to $\Mfng$.
Moreover, by \ref{item:V3} we have $\dy{2}F=0$, hence
\begin{equation}\label{eq:wtxiF}
  \partial_{\wt{\xi}}F=\partial_\xi F
\end{equation}
everywhere on $V_{\oout}$.

The two vector fields $\xi$ and $\wt{\xi}$ do not match on $\Mfng\cap \Mtwo$. To rectify this, we will  interpolate between $\xi$ and $\wt{\xi}$, to obtain a third vector field, which will be $\xih$.
The most naive approach, that is, to define the new vector field $\xih$
as a linear combination of $\xi$ and $\wt{\xi}$ multiplied by two functions vanishing
on $\Mfng$ and $\Mtwo$ respectively, does not quite work, because a nontrivial linear combination of $\xi$ and $\wt{\xi}$ does not have
to be tangent to the stratum $\Mfng\cap \Mtwo$.

Let $e_1$ be the vector field
\begin{equation}\label{eqn:defn-of-e1}
e_1 := (0,x_2,\ldots,x_n,0,0).
\end{equation}
We claim that
\begin{equation}\label{eq:e1ftheta}
\partial_{e_1}\ftheta>0
\end{equation}
everywhere on $\Sigma$. To see this,
note that if $x_1^2+\dots+x_n^2\le \varepsilon_R^2$ (which holds
on $\Sigma$), then
\[\partial_{e_1}\ftheta=2(x_2^2+\dots+x_n^2)\fa(y_1)+\fb(y_1).\]
On $\Sigma$, $x_2^2+\dots+x_n^2=\erho^2$, so $\partial_{e_1}\ftheta=2\erho^2\fa(0)+\fb(0)$. By \eqref{eq:rhodef},
we obtain $\partial_{e_1}\ftheta=\erho^2\fa(0)>0$.

%The meaning of $\Csafe$ is that
%the vector field $\xi$ will not be changed in $\Csafe$. As $p_+\in \Csafe$ this means that the vector field is not changed near the critical point $p_+$.
%The cylindrical shape of $\Csafe$ will allow us to control the trajectories of the new vector field $\xih$ in Section~\ref{sec:membranesofxi} below and eventually ensure property \ref{item:FM_one_less}.

\begin{figure}
\input{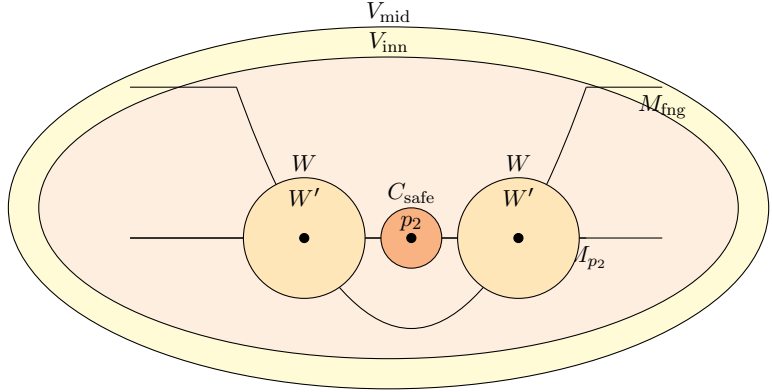}
\caption{The
subsets $V_{\iinn}$, $V_{\mmid}$, $W$ and $W'$ that we use to study the vector field $\xih$. We draw the intersection with the hyperplane $\{x_1=0\}$.}\label{fig:nbhds}
\end{figure}

%The finger move changes the manifold $M$ only inside $V_{\iinn}$.
%In particular, the double point set $\Mfng\cap\Mtwo$ belongs
%to $V_{\iinn}$.
Choose $W\subseteq V_{\iinn}$ to be a neighbourhood of
the double point set
$\Sigma$ satsfying the following two conditions; see Figure~\ref{fig:nbhds}.
\begin{condition}\label{con:onW}
  The set $W$ is disjoint from $p_+$, and $\partial_{e_1}\ftheta>0$ everywhere on $W$.
\end{condition}
The latter is possible because of \eqref{eq:e1ftheta}.
Let $W'$ be a neighbourhood of $\Sigma$ in $\O$
such that the closure of $W'$ is contained in $W$.

Now we adjust $\esafe$ to satisfy the assumptions of Lemma~\ref{lem:safeexit} and by the following condition:
\begin{condition}\label{con:csafe}
  $\Csafe$ is a subset of $V_{\iinn}$, the closures of $\Csafe$ and $W$ are disjoint, and $\Csafe$ is disjoint from $\Mfng$.
\end{condition}

Set $\rhofng \colon V_{\mmid} \to \R_{\geq 0}$ to be the square of the distance to $\Mfng$,  and  set $\rho_+\colon V_{\mmid} \to \R_{\geq 0}$
 to be the square of the distance to $\Mtwo$. Finally, choose two bump functions $\phi_1,\phi_2\colon\O\to[0,1]$, where $\phi_1$ is a function
supported on $V_{\mmid}$, equal to $1$ on $W$, and $\phi_2$ is a function supported in $W$ and equal to $1$ on $W'$

In $W$, define the decomposition of the tangent bundle
$T\Omega$ into $\Ttan$, $E_1=\sspan\{\dy{1},(\partial_{e_1}\ftheta)\dy{2}+e_1\}$ and $E_2=\sspan\{e_1,\dx{1}\}$.
By Condition~\ref{con:onW}, $E_1\cap E_2=0$.
The subbundle $\Ttan$ can be defined as the orthogonal complement of $E_1$ in $\sspan\{\dx{1},\ldots,\dx{n}\}$.
Recall that we are still using  the metric for which the coordinate vectors $\dx{1},\ldots,\dy{2}$ are
orthonormal.

The bundle $\Ttan$ is orthogonal to $E_1\oplus E_2$,
but $E_1$ is not necessarily orthogonal to $E_2$.
By definition
$(\Ttan\oplus E_2)|_{\Mtwo\cap W}$ is the tangent bundle to $\Mtwo\cap W$, and $(\Ttan\oplus E_1)|_{\Mfng\cap W}$ is the tangent bundle to $\Mfng\cap W$.
Moreover,
$\Ttan$ restricted to $\Mtwo\cap \Mfng$ is the tangent bundle to $\Mtwo\cap \Mfng$.

According to the decomposition $T\Omega=\Ttan\oplus E_1\oplus E_2$
we write $\xi=\xtan+\xi_1+\xi_2$ and $\wt{\xi}=\wxtan+\wt{\xi}_1+\wt{\xi}_2$ on $W$. Notice that $\wt{\xi}-\xi$ is parallel to $\dy{2}$,
which belongs to $E_1\oplus E_2$, hence $\wxtan=\xtan$ in $W$.
Set
\[\UUU=\bigg(\frac{\rhofng}{\rhofng+\rho_+}\bigg)^{\phi_3}.\]
Then $\UUU \equiv 1$ on $\Csafe$. %while outside a small neighbourhood of $\Csafe$, $\phi_3=1$, and so it has no effect.
Also $\UUU\equiv 0$ on the whole of $\Mfng$ and $\UUU\equiv 1$ on the whole of $\Mtwo$ ($\rho_+=0$ there); see
Figure~\ref{fig:upsilon}. The function $\UUU$ is not well-defined on $\Mfng\cap \Mtwo$, but, as we will see below, this does not affect the correctness
of the definition of $\xih$.
The interpolated vector field is given by:
\begin{equation}\label{eq:defxiprim}
\xih=\xi^a+\xi^b+\xi^c+\xi^d,
\end{equation}
\textrm{ where:}
\begin{align*}
  \xi^a&=(1-\phi_1)\xi&
\xi^b&=\phi_1\left(\UUU\xtan+(1-\UUU)\wxtan\right)\\
\xi^c&=\phi_1(1-\phi_2)\left(\UUU(\xi_1+\xi_2)+(1-\UUU)(\wt{\xi}_1+\wt{\xi}_2)\right) &
\xi^d&=\phi_1\phi_2\left(\rho_+\wt{\xi}+\rhofng\xi\right).
\end{align*}
\begin{figure}
\input{pictures/upsilon.tex}
\caption{Figure~\ref{fig:nbhds} revisited. We indicate the values of $\rhofng$, $\rho_+$ and $\UUU$.}\label{fig:upsilon}
\end{figure}
The form of $\xih$ in different regions is given in Lemma~\ref{lem:series_2} and Lemma~\ref{lem:series_4}.
\begin{remark*}
Notice that $\xtan$, $\wxtan$, and $\xi_i,\wt{\xi}_i$ for $i=1,2$, are defined only on $W$. However, outside $W$ the vector field $\xih$
depends only on the sum $\xtan+\xi_1+\xi_2$ and $\wxtan+\wt{\xi}_1+\wt{\xi}_2$. Therefore, to define $\xih$ everywhere on $V_{\mmid}$, extend $\xi_i$, $\wt{\xi}_i$ over $V_{\mmid}$ in such a way that $\xtan+\xi_1+\xi_2=\xi$ and $\wxtan+\wt{\xi}_1+\wt{\xi}_2=\wt{\xi}$, and then $\xih$
will not depend on the extension.
\end{remark*}

One can think of $\xih$ on $\Mtwo$ as taking the vector field $\xi$ and slowing it in the transverse direction to $\Mtwo$
as the trajectory draws near $\Mtwo$.

In the following series of lemmas we explain the formulae for $\xih$ and show that $\xih$ is
tangent to $\Mfng$ and to $\Mtwo$.
\begin{lemma}\label{lem:series_1}
  If $w\notin V_{\mmid}$, then $\xih(w)=\xi(w)$.
\end{lemma}
\begin{proof}
  We have $\phi_1(w)=0$, so $\xi^a=\xi$ and $\xi^b=\xi^c=\xi^d=0$.
\end{proof}
\begin{lemma}\label{lem:series_2}
  If $w\in V_{\mmid}\setminus V_{\iinn}$, then $\xih$ is a linear combination of $\xi$ and $\wt{\xi}$; more precisely
  \begin{equation}\label{eq:series_3}\xih=((1-\phi_1)+\phi_1\UUU)\xi +\phi_1(1-\UUU)\wt{\xi}.\end{equation}
\end{lemma}
\begin{proof}
  This follows from the fact that $\phi_2\equiv 0$ on $V_{\mmid}\setminus V_{\iinn}$.
\end{proof}
\begin{lemma}\label{lem:series_3}
  If $w\in V_{\mmid}\setminus V_{\iinn}$ and $w\in \Mfng$, then $\xih$ is tangent to $\Mfng$; more precisely
  $\xih=\xi$ on $\Mfng\cap(V_{\mmid}\setminus V_{\iinn})$.
\end{lemma}
\begin{proof}
  On $\Mfng\setminus V_{\iinn}$ we have $x_2^2+\dots+x_n^2\ge\eeta^2$. Therefore, by the definition
  of $\wt{\xi}$ we have $\wt{\xi}=\xi$. We conclude by \eqref{eq:series_3}.
\end{proof}
\begin{lemma}\label{lem:series_4}
  Suppose $w\in V_{\iinn}\setminus W$. Then $\xih(w)=\UUU\xi(w)+(1-\UUU)\wt{\xi}(w)$. In particular, $\xih(w)=\xi(w)$ if $w\in \Mtwo$ and $\xih(w)=\wt{\xi}(w)$
  if $w\in \Mfng$.
\end{lemma}
\begin{proof}
  We have $\phi_1\equiv 1$ and $\phi_2\equiv 0$, hence $\xi^a=\xi^d=0$ and so
  \[\xih=\xi^b+\xi^c=\UUU\xi+(1-\UUU)\wt{\xi}.\]
  The statement follows, since $\UUU(w)=0$ if $w\in \Mfng$ and $\UUU(w)=1$ if $w\in \Mtwo$.
\end{proof}
As $\Csafe\subseteq V_{\iinn}\setminus W$ and $\UUU\equiv 1$ on $\Csafe$ we obtain the following corollary.
\begin{corollary}\label{cor:csafe}
  On $\Csafe$ we have $\xih=\xi$.
\end{corollary}
\begin{lemma}\label{lem:series_5}
  Suppose $w\in W\setminus W'$. Then $\xih$ is a nonnegative linear combination of $\xi$, $\wt{\xi}$ and $\xtan$, more precisely
\begin{equation}\label{eq:xiprimonW}
\xih=(1-\phi_2)(\UUU\xi+(1-\UUU)\wt{\xi})+\phi_2(\rho_+\wt{\xi}+\rhofng\xi+\xtan).
\end{equation}
\end{lemma}
\begin{proof}
We use the fact that $\xtan=\wxtan$ in $W$.
\end{proof}
\begin{lemma}\label{lem:series_6}
  Suppose $w\in W\setminus W'$. If $w\in \Mtwo$, then $\xih(w)$ is tangent to $\Mtwo$, and if $w\in \Mfng$, then $\xih(w)$ is tangent  to $\Mfng$.
\end{lemma}
\begin{proof}
 This follows from \eqref{eq:xiprimonW}. If $w\in \Mtwo$, then $\UUU(w)=1$ and $\rho_+(w)=0$, while if $w\in \Mfng$, then $\UUU(w)=0$ and $\rhofng(w)=0$. The vector field $\xtan$
 is tangent both to $\Mtwo$ and to $\Mfng$.
\end{proof}
\begin{lemma}\label{lem:series_7}
  If $w\in W'$ then $\xi(w)$ is still given by \eqref{eq:xiprimonW}, but the formula simplifies further to
\begin{equation}\label{eq:this_means_that}
  \xih(w)=\xtan(w)+\rho_+\wt{\xi}_1(w)+\rhofng\xi_2(w).
\end{equation}
In particular $\xih$ is well-defined (even if $\UUU$ is not) and $\xih$ is tangent to $\Mfng$, $\Mtwo$ and $\Mfng\cap \Mtwo$.
\end{lemma}
\begin{proof}
  We have $\phi_1(w)=\phi_2(w)=1$, so $\xi^a=\xi^c=0$. Moreover $\xtan=\wxtan$, hence $\xi^b=\xtan$. The formula for $\xi^d$
  simplifies to $\xi^d=\rho_+\wt{\xi}_1+\rhofng\xi_2$. Hence we obtain \eqref{eq:this_means_that} and from this formula the tangency of $\xih$
  follows immediately.
\end{proof}
Combining the above lemmas we obtain the following corollary.
\begin{corollary}
  The vector field $\xih$ is tangent to $\Mtwo$, $\Mfng$ and $\Mtwo\cap \Mfng$.
\end{corollary}
\begin{figure}
\input{pictures/xineigh.tex}
\caption{Graphical presentation of the statements of Lemmas~\ref{lem:series_1} through~\ref{lem:series_7}.}\label{fig:xinbhds}
\end{figure}

\subsection{Positivity of $\xih$}\label{sub:critical}

We pass to the study of the critical points of $\xih$.

\begin{proposition}\ \label{prop:criticalF}
\begin{itemize}
\item[(a)]The vector field $\xih$ satisfies $\partial_{\xih}F\ge 0$.
\item[(b)] Away from $\Mfng\cap \Mtwo$, the vector field $\xih$ has the same critical points as $\xi$ and they have the same local behaviour.
\end{itemize}
Moreover on $\Mfng\cap \Mtwo$, $\xih$ vanishes at points where $x_2=\dots=x_{h+1}=0$ or $x_{h+2}=\dots=x_n=0$.
%and
%these points are critical points of $F$ restricted to $\Mfng\cap \Mtwo$.
\end{proposition}

%\begin{remark}\label{remark:need-to-change-F}
%  The last part of the proposition suggests that we need to change the function $F$ (and the vector field $\xih$) near $\Mfng$ and $\Mtwo$.
%  In fact, unless $n=2$, there are infinitely many critical points
%  of $F$ on $\Mfng\cap\Mtwo$.
%  An explicit perturbation
%  is done in Subsection~\ref{sub:explicit_perturbation}.
%\end{remark}

\begin{proof}
  By~\eqref{eq:wtxiF},  $\partial_{\wt{\xi}} F\ge 0$.
%This is not enough, though, to show that $\partial_{\xih} F\ge 0$,
%because in $W$, $\xih$ is not a linear combination of $\xi$ and $\wt{\xi}$.
  Inside $W$, $\xih$ is a linear
combination of $\xi$, $\wt{\xi}$ and $\xtan$ (see Lemmas~\ref{lem:series_5} and~\ref{lem:series_7})
with nonnegative coefficients. Therefore, in order to
show that $\partial_{\xih}F\ge 0$, it is enough to show that $\partial_{\xtan} F \ge 0$.

Recall from \eqref{eqn:defn-of-e1} that $e_1$ is the vector field $(0,x_2,\ldots,x_n,0,0)$ in $V_{\iinn}$.
The projection of $\xi$ to the space orthogonal to $\dx{1},\dy{1}$ and $\dy{2}$ is equal to
the vector field $\xmo$ defined as
\begin{equation}\label{eq:xmo}
  \xmo := (0,-x_2,\ldots,-x_{h+1},x_{h+2},\ldots,x_n,0,0).
\end{equation}
  Then
\begin{equation}\label{eq:xoo}
  \xtan=\xmo-\frac{\langle\xmo,e_1\rangle}{\langle e_1,e_1\rangle}e_1.
\end{equation}
Set $\Delta=-x_2^2-\cdots-x_{h+1}^2+x_{h+2}^2+\cdots+x_n^2$.
Then the following holds.
\begin{align*}
  \langle\xmo,e_1\rangle&=\Delta &
  \langle e_1,e_1\rangle&=x_2^2+\cdots+x_n^2 &
  \smfrac{1}{2}\partial_{e_1}F&=\Delta &
\smfrac{1}{2}\partial_{\xmo}F&=x_2^2+\cdots+x_n^2.
\end{align*}
Therefore
\begin{equation}\label{eq:pxtanF}
\smfrac{1}{2}\partial_{\xtan}F=x_2^2+\cdots+x_n^2-\frac{\Delta^2}{x_2^2+\cdots+x_n^2}\ge 0.
\end{equation}
The last inequality holds because $|\Delta|\le x_2^2+\ldots+x_n^2$.
 This concludes the proof that $\partial_{\xih}F\ge 0$ in $W$,
and so proves point (a) of the proposition.

\smallskip
We pass to the study of the critical points of $\xih$. Firstly, outside $W$ the vector field $\xih$ is a linear combination
of $\xi$ and $\wt{\xi}$ with coefficients summing to $1$. Now $\wt{\xi}-\xi$ is parallel to $\dy{2}$, so it is linearly independent
of $\xi$. Thus, for  $w\notin W$, $\xih(w)=0$ only if $\xi(w)=0$,
that is, only if $w$ is a critical point of $\xi$.
By construction, $\xih=\xi$ near each critical point of $\xi$.
%Each such point is either outside of $V_{\oout}$
%(and hence $\xih=\xi$ there), or this is the critical point $p_+$, and by construction we also have $\xih=\xi$ in the neighbourhood of $p_+$.
This shows point (b) outside $W$.

\smallskip
Now we analyse the critical points of $\xih$ inside $W$.
First, by \eqref{eq:wtxiF}, $\partial_{\xi}F(w)=\partial_{\wt{\xi}}F(w)$ on $V_{\oout} \supseteq W$.
If $w\not\in \Mtwo\cap \Mfng$, then in \eqref{eq:xiprimonW} (which holds also inside $W'$ by Lemma~\ref{lem:series_7}),
either the coefficient
at $\xi$ or the coefficient at $\wt{\xi}$ is nonzero. Moreover, $\partial_{\xi}F(w)=\partial_{\wt{\xi}}F(w)>0$, because within $F^{-1}([c_{\bbot},c_{\ttop}])$, we have
$\partial_{\xi}F=0$ only at the point $p_+$, which does not belong to $W$.  Since we have already shown that $\partial_{\xtan}F\ge 0$,
we conclude, using that  all coefficients in \eqref{eq:defxiprim} are positive, that $\partial_{\xih}F(w)>0$, and so $\xih(w)\neq 0$.
This concludes the proof of point (b) inside $W$.
\end{proof}

\subsection{The path $F_\tau$}\label{sub:explicit_perturbation}
In this subsection we construct the path $F_\tau$
and the vector field $\xit$. We also prove property \ref{item:FM_two_more}
of Finger Move Theorem~\ref{thm:new_finger_move}.

%As we mentioned above (Remark~\ref{remark:need-to-change-F}),
The function $F$ is -- in general -- not an immersed Morse
function for $\Mnew$. More precisely we have the following result.
\begin{lemma}\label{lem:crits_of_F}
  Recall that $\Sigma=\Mtwo\cap\Mfng$. The function $F$ restricted to $\Sigma$
  has critical points on the whole subset $\Sigma\cap\{x_2=\dots=x_{h+1}=0\}$ and $\Sigma\cap\{x_{h+2}=\dots=x_n=0\}$.
\end{lemma}

\begin{proof}
  By \eqref{eq:sigmadef} the set $\Sigma$ is given as $x_2^2+\dots+x_n^2=\erho^2$, $x_1=y_1=0$. On this set, by property~\ref{item:V3},
  the function $F$ is given by $-x_2^2-\dots-x_{h+1}^2+x_{h+2}^2+\dots+x_n^2+c_+$. Then $F$ is constant on the intersection of
  $\Sigma$ with $\{x_2=\dots=x_{h+1}=0\}$ and $\{x_{h+2}=\dots=x_n=0\}$.
\end{proof}

Our goal is to construct an explicit
perturbation of $F$ and of $\xih$ inside $W' \subseteq W$.
To this end choose a real vector $(\mu_2,\dots,\mu_n)$ with $|\mu_i|<\frac1n$ for each $i$. On our coordinate chart in $V_{\oout}$, define $L \colon V_{\oout} \to \R$ and a vector field $\xi^\mu$ by:
\[
  L:=2\sum_{j=2}^n\mu_jx_j,\ \
  \xi^\mu:=\sum_{j=2}^n\mu_j\frac{\partial}{\partial x_j}.
\]
Let $\xtan^{\mu}$ be the projection of the vector field $\xi^\mu$ to the sub-bundle
orthogonal to $e_1$. Decrease the coefficients
$\mu_j$ so that $|L(z)|<1$ for all $z\in W$.
Recall that $\rho_2 \colon V_{\mmid} \to \R_{\geq 0}$ is the square of the distance to $M_{p_+}$ and $\rhofng \colon V_{\mmid} \to \R_{\geq 0}$ is the square of the distance to $\Mfng$.
For $\varepsilon_{\lambda}>0$ sufficiently small define the set $U_\lambda$ by the inequalities $\{\rhofng\le\varepsilon_\lambda,\rho_2\le \varepsilon_\lambda\}$. If $\varepsilon_\lambda\ll 1$, we have $U_\lambda\subseteq W'$. From now on we assume that this is the case.
Let also $U_{\lambda}' \subseteq U_{\lambda}$ be given by $\{\rhofng\le\frac12\varepsilon_{\lambda},\rho_2\le \frac12\varepsilon_{\lambda}\}$.
The choice of the factor $\frac12$ in the definition of $U'_\lambda$ is arbitrary.

\begin{theorem}\label{thm:perturb}
  There exists $\varepsilon_{\lambda}>0$, a smooth function $\feta\colon\O\to[0,1]$ vanishing outside $U_{\lambda}$ and equal to $1$ in $U'_{\lambda}$, and a parameter $\sigma_0>0$ such that for any $\sigma\in(0,\sigma_0]$
  the function
  \[F-\sigma\feta L\]
  is an immersed Morse function for $G_2(N)$. It has at most four critical points when restricted to $\Sigma$.
  Moreover, the vector field
  \[\xih_\sigma:=\xih-\sigma\feta\xtan^{\mu}\]
  is a grim vector field for $F-\sigma\feta L$.
\end{theorem}

\begin{proof}
  Fix $\varepsilon_{\lambda}>0$ such that $U_{\lambda}\subseteq W'$. Choose $\feta \colon \Omega \to [0,1]$ in such
  a way that
  \begin{equation}\label{eq:greek_feta}
    \|\nabla\feta \| \le 4/\varepsilon_\lambda.
  \end{equation}
That this can be done is easy to believe if one notices that it is essentially equivalent to approximating a piecewise linear function $f \colon [0,2] \to [0,1]$ with $f(x)=1$ for $x \leq 1/2$ and $f(x)=0$ for $x \geq 1$ by a smooth function with derivative $|f'(x)|<4$ for all $x \in [0,2]$.
  Choose $c,K_\xi$, and $K_F$ to be real positive numbers satisfying
  \[\partial_\xi F(z)>c,\ \|\xih(z)\| <K_\xi, \text{ and } \|\nabla F(z)\| <K_F,\]
  for all $z\in \overline{W'}$. Such $c$ exists, because $\partial_\xi F$ is positive on the whole of $\overline{W'}$.

  Our discussion takes place entirely in $W'$.
By \eqref{eq:this_means_that} we have
$\xih=\xtan+\rho_2\xi+\rhofng\wt{\xi}$.
%Therefore $\partial_{\xih} F \geq \rho_2\partial_{\xi} F + \rhofng\partial_{\wt{\xi}} F$.
By \eqref{eq:pxtanF} and \eqref{eq:wtxiF},
$\partial_{\xih} F \geq (\rho_2+\rhofng\partial_{\xi}) F$.
Hence,
\[\partial_{\xih_\sigma}(F-\sigma\feta L)\ge (\rhofng+\rho_2)\partial_\xi F-\sigma\partial_{\xih}(\feta L)-\sigma\feta\partial_{\xtan^{\mu}}(F-\sigma\feta L).\]
  We first show that this quantity is nonzero on $U_\lambda\setminus U_{\lambda'}$.

  In $U_\lambda\setminus U_\lambda'$ we have $\rho_2 > \frac{1}{2} \varepsilon_{\lambda}$ or $\rhofng > \frac{1}{2} \varepsilon_{\lambda}$,  $\partial_\xi F = \partial_{\wt{\xi}}F>c$, and $\partial_{\xtan}F\ge 0$ by \eqref{eq:pxtanF}. Therefore
  \begin{equation}\label{eq:partial1}
  \partial_{\xih} F > \frac12 c\varepsilon_{\lambda}.
\end{equation}
  Now observe that $\partial_{\xih}\feta\le \|\xih \| \,\| \nabla\feta\| \le  4K_\xi\varepsilon_\lambda^{-1}$. Hence, using that  $|L|<1$, $\|\nabla L \| <1$, and that $\feta$ is valued in $[0,1]$, we have:
  \begin{equation}\label{eq:partial2}
    |\partial_{\xih}\feta L| = |\feta\partial_{\xih}L+L\partial_{\xih}\feta | \le K_\xi(1+4\varepsilon_\lambda^{-1}).
  \end{equation}
  The vector field $\xi^{\mu}$ has length at most $1$ since $|\mu_i| <\frac{1}{n}$, so $\xtan^\mu$ also has length at most $1$ as well, being an orthogonal projection. Therefore we have
  \begin{equation}\label{eq:partial3}
    |\partial_{\xtan^{\mu}}F | \le \| \nabla F \|\cdot \|\xtan^{\mu}\|\le K_F.
  \end{equation}
  In addition, also using that  $|L|<1$, $\|\nabla L \| <1$ on $U_{\lambda}$, and that $\feta$ is valued in $[0,1]$:
  \begin{equation}\label{eq:partial4}
    |\partial_{\xtan^{\mu}}\feta L| = |L\partial_{\xtan^{\mu}}\feta+\feta\partial_{\xtan^\mu}L| \le 4\varepsilon^{-1}_{\lambda}+1.
\end{equation}

Combining \eqref{eq:partial1}, \eqref{eq:partial2}, \eqref{eq:partial3} and \eqref{eq:partial4} we obtain
\[
  \partial_{\xih_\sigma} (F-\sigma\feta L)
  \ge \frac12c\varepsilon_{\lambda}-\sigma K_\xi(1+4\varepsilon_\lambda^{-1})-\sigma K_F-\sigma^2(1+4\varepsilon^{-1}+1))
\]

In the above inequality, only the first term does not depend on $\sigma$.
Therefore, there exists $\sigma_0>0$ such that if $|\sigma|<\sigma_0$,
then $\partial_{\xih_\sigma}(F-\sigma\feta L)$ is positive everywhere on
$U_\lambda\setminus U_{\lambda}'$.

\smallskip
We now estimate $\partial_{\xi_\sigma}(F-\sigma\feta L)$ inside $U_{\lambda'}$.
 Then $\feta\equiv 1$ so we have
 \[\xih_\sigma=\xih-\sigma\xtan^\mu.\]
 Let $\xi'_\sigma=\xtan-\sigma\xtan^\mu$ so that $\xih_\sigma=\xi'_\sigma+\rho_2\xi+\rhofng\wt{\xi}$.
As $\partial_\xi F>c$, for $\sigma$ sufficiently small, we have
$\partial_\xi (F-\sigma L),\partial_{\wt{\xi}}(F-\sigma L)>0$.
We need to show that $\partial_{\xi'_\sigma}(F-\sigma L)\ge 0$.

%Recall from \ref{item:V3} that $F=-x_1^2-x_2^2-\dots-x_{h+1}^2+x_{h+2}^2+\dots+x_n^2+y_1+F(p_+)$.
Set $H_\sigma=(F-\sigma L)+x_1^2-y_1$.
By definition, $H_\sigma$ does not depend on $y_1$ and $x_1$, because these terms in the definition of $H_{\sigma}$ cancel with the corresponding terms in $F$; compare~\ref{item:V3}. %Note that $(F-\sigma L)-H_\sigma = -x_1^2 +y_1$.
By \ref{item:V1}, $\partial_{\xi} (F-\sigma L-H_\sigma)=\partial_{\wt{\xi}} (F-\sigma L-H_\sigma)\ge 0$  and $\partial_{\xtan}(F-\sigma L-H_\sigma)=\partial_{\xtan^\mu}(F-\sigma L-H_\sigma)=0$.
%In particular, if we show that
 Therefore we strive to show that $\partial_{\xi'_\sigma}H_\sigma\ge 0$, and
 this will imply that $\partial_{\xi'_\sigma} (F-\sigma L)\ge 0$.

 Note that $\nabla H_\sigma=2\xmo-2\sigma\xi_\mu$.
The orthogonal projection of $\xmo-\sigma\xi_\mu$ along $e_1$ is equal precisely $\xtan-\sigma\xtan^\mu=\xi'_{\sigma}$. This shows that
\begin{equation}\label{eq:nablaGs}
\partial_{\xi'_{\sigma}}H_\sigma\ge 0.
\end{equation}
As alluded to above, this implies that $\partial_{\xih_\sigma}(F-\sigma L)\ge 0$ on $U'_\lambda$.

The remaining part of the proof is devoted to checking when $\partial_{\xih_\sigma}(F-\sigma L)=0$. First of all, if $\rhofng>0$ or $\rho_2>0$,
then $\partial_{\xih_\sigma}(F-\sigma L)>0$, because $\partial_{\xi'_\sigma}(F-\sigma L)\ge 0$. Therefore, if $\partial_{\xih_\sigma}(F-\sigma L)(z)=0$,
then $\rhofng(z)=\rho_2(z)=0$, that is, $z\in\Sigma$.

Next, if $\partial_{\xih_\sigma}(F-\sigma L)(z)=0$, we need to have
equality in \eqref{eq:nablaGs}. This is possible only if $2\xmo-2\sigma\xi_\mu$
is parallel to $e_1$. %Explicit calculations show
%that is possible for at most four points on the sphere $\Sigma$.
For $2\xmo-2\sigma\xi_\mu$ to be parallel to $e_1$ we have, for some $\kappa \in \R$, that
$(-1-\kappa)x_i = \mu_i$ for $2 \leq i \leq h$ and $(1-\kappa)x_i = \mu_i$ for $h+1 \leq i \leq n$.
Solving for $x_i$ and substituting into $x_2^2 + \cdots + x_n^2 = \varepsilon_{\Sigma}^2$ yields a quartic equation for $\kappa$. Thus there are at most 4 solutions.
\end{proof}

Given Theorem~\ref{thm:perturb} we choose $\sigma_1<\sigma_0$.
We define
\begin{equation}\label{eq:xibis}
  \xit=\xih_{\sigma_1}.
\end{equation}
We also define the path $F_\tau$ that will satisfy the requirements
of Finger Move Theorem~\ref{thm:new_finger_move}.
For each $\tau \in [0,1]$
we set
\[F_{\tau}:= F-\tau\sigma_1\feta L.\]

\begin{remark*}
  We leave ourselves the possibility of decreasing the value of $\sigma_1$ later on. The precise order of choosing different variables is explained in
  Remark~\ref{rem:order_of_choosing}.
\end{remark*}

\section{Property~\ref{item:FM_one_less}}\label{sec:membranesofxi}

Our next result proves property \ref{item:FM_one_less} from the Finger Move Theorem~\ref{thm:new_finger_move}.
Our result is valid both for $F_1$ and the grim vector field $\xit$,
as well as for the original function $F$ and the unperturbed
vector field $\xih$.

Recall that $\Ha(p_-) \cap F^{-1}(c_{\bbot}) \cap U_{\oout}$ consists of discs $D_0,D_1,\dots,D_r$ containing points $z_0,z_1,\dots,z_r \in \Ha(p_-) \cap \Hd(p_+) \cap F^{-1}(c_{\bbot})$, and that we seek to remove $z_m \in D_m$.
Recall that $Y_{\mmid}$ is the interior of $\ol{V}_{\mmid} \cap F^{-1}(c_{\bbot})$ and that $V_{\mmid} \subseteq F^{-1}(c_{\bbot},c_{\ttop})$.

\begin{theorem}\label{thm:oneless}
There are no trajectories of $\xit$ from $p_-$ to $p_+$ that go through $Y_{\mmid}$. If $\efng$ is very small, then
the only trajectories from $p_-$ to $p_+$ are those trajectories of $\xi$ that go through discs $D_0,D_1,\ldots,D_r$ $($see \eqref{eq:maft}$)$
for $r\neq m$.

In other words, the number of trajectories of $\xi$ from $p_-$ to $p_+$ is one less than the number of trajectories of $\xit$ from $p_-$ to $p_+$.
\end{theorem}

\begin{proof}
Suppose $\gamma$ is a trajectory of $\xit$ connecting $p_-$ with $p_+$. Assume $\gamma$
does not go through $Y_{\mmid}$. As below the level set $c_{\bbot}$, $\xi=\xit$, $\gamma$ is a trajectory of $\xi$ at least until
it reaches the level set $c_{\bbot}$. Let $w_1$ be the intersection of $\gamma$ with $F^{-1}(c_{\bbot})$. Clearly $w_1\in\Ha(p_-)$.

%Note that if $w_1\notin Y_{\oout}$, then $\gamma$ does not hit $p_+$.
By \ref{item:V4}, and since $\xi=\xit$ on $\partial V_{\oout}$,
a trajectory of $\xit$ can enter $V_{\oout}$ only through $Y_{\oout}$. Hence, $w_1\in Y_{\oout}$. As $\Ha(p_-)\cap Y_{\oout}$ is the
union of discs $D_0,\dots,D_r$, we note that $w_1\in D_i$ for some $i$.

We consider now two cases. Either $i=m$, or $i\neq m$.

We begin with the first case, that is, $i=m$. Our aim is to show that this case is not possible.
The set $P:=\{x_1=0\}$ is easily seen to be invariant
under the flow of $\xit$, moreover $D_m\subseteq\{x_1=0\}$. As the trajectory hits the point $w_1\in P$, the trajectory stays on $P$
until reaching $p_+$, compare Lemma~\ref{lem:vmid_is_grim}. % (we will show that such a trajectory can not possibly reach $p_+$)

%Now define $P=\{x_1=0\}\subseteq V_{\mmid}$.
Consider the subsets
\[P_{\pm}=\{\pm \big(y_2-\fa(y_1)\min\{x_2^2+\dots+x_n^2,\eeta^2\}+\fb(y_1)\big)\ge 0\}.\]
Note that we can also write $P_{\pm}$ as  $\{\pm (y_2 -\ftheta) \geq 0\}$, where $\ftheta$ is the function from \eqref{eq:Theta}.
We have $P_+\cup P_-=P$ and $P_+\cap P_-\subseteq \Mfng$, see Figure~\ref{fig:ppm}. As $\xit$ is tangent to $\Mfng$, there can be no trajectory of $\xit$ that goes from $P_-$ to
$P_+$ and does not belong to $\Mfng$. By construction, $p_+\in P_+$.
\begin{figure}
  \input{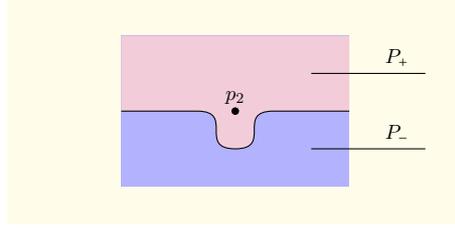}
  \caption{
  The sets $P_+$ and $P_-$ from the proof of Theorem~\ref{thm:oneless}. The trajectories of $\xit$ that enter through $P_-$ will never reach $p_+$.}\label{fig:ppm}
\end{figure}

Recall that $w_1\in D_m$ is the point of entrance of $\gamma$ to $V_{\mmid}$.
Denote by $(x_1,x_2,\ldots,x_n,y_1,y_2)$ its coordinates. We know that $y_2<1$ and $x_1=\dots=x_{h+1}=0$, compare \eqref{eq:maft}.
On the other hand, we have $F(w_1)=c_{\bbot}$. By \ref{item:V3},
$y_1+x_{h+2}^2+\cdots+x_n^2=c_{\bbot}-c_+$, so $y_1\le c_{\bbot}-c_+$. Here, recall that $c_+ = F(p_+)$.  It follows that
$y_1\notin[c_{-1/4}-c_+,c_{+1/4}-c_+]$.
Property \ref{item:A1} implies that $\fb(y_1)=1$, so it follows from \eqref{eq:Theta} that $\ftheta(w_1)=1$. As $y_2<1$, we infer that $w_1\in P_-$. Hence, $\gamma$ does not reach $p_+$.

\smallskip
We pass to the proof of the second part, that is, $w_1\in D_i$ for $i\neq m$.
As $V_{\mmid}$ is disjoint from $D_i$ by Lemma~\ref{lem:basic_efng}, the trajectory $\gamma$ must enter $V_{\mmid}$ (the point $p_+$
belongs to the interior of $V_{\mmid}$). Let $w_2\in\partial V_{\mmid}$ be the point of the entrance. The trajectory $\gamma$ is a trajectory
of $\xi$ until it reaches $w_2$. Hence, we can use
Lemmas~\ref{lem:safeentry} and~\ref{lem:safeexit}.
We have the following cases.
\begin{itemize}
\item The point $w_2$ is outside $\Cvs$. By Lemma~\ref{lem:safeentry}, $F(w_2)>c_+=F(p_+)$. As $\partial_{\xit}F\ge 0$,
such a trajectory will never reach $p_+$.
\item The trajectory hits $V_{\mmid}$ at $w_2\in \Cvs$ and exits $\Csafe$ at a point $w_3$. By Lemma~\ref{lem:safeexit} we have that $F(w_3)>c_+$,
  so the trajectory will never reach $p_+$ again.
\item The trajectory enters $\Cvs$ and does not lead out of $\Csafe$. Then for all the time it is a trajectory of $\xi$, because $\xi=\xit$ in $\Csafe$.
\end{itemize}
All this means that every trajectory
of $\xit$ that connects $p_-$ and $p_+$ and passes through $D_i$ with $i\neq m$ is actually a trajectory of $\xi$.

Thus we have indeed reduced the number of trajectories by $1$, namely, by eliminating the trajectory passing through $D_m$.
\end{proof}

We conclude the section with an important remark.
\begin{remark}\label{rem:order_of_choosing}
The order of choosing and adjusting various constants, after $V_{\mmid}$ has been constructed, is the following.
\begin{itemize}
  \item Adjust $\efng$ according to Lemma~\ref{lem:safeentry}; see Condition~\ref{con:efng}.
  \item Choose $\eeta$ in such a way that Lemma~\ref{lem:eta_small_enough} is satisfied; see Condition~\ref{con:eeta};
  \item The functions $\fa$, $\fb$ are adjusted to $\eeta$ as in Condition~\ref{con:fa_fb}. %As an outcome we obtain the value of $\erho$ in \eqref{eq:rhodef}.
\item The subsets $W$ and $W'$ are defined as neighbourhoods of the double point set \eqref{eq:sigmadef}. They are
  chosen to be disjoint from $p_+$ and contained in $V_{\mmid}$. %depending on $\erho$.
  %We expect that $\erho$ is small enough so that $W$ and $W'$ are subsets of $U_{\oout}\capV_{\iinn}$.
\item Take a $\esafe>0$ such that $\Csafe\subseteq V_{\mmid}$ is disjoint from $W$.
  This choice is
  made in Condition~\ref{con:csafe}.
\item Adjust $\evs$ according to Lemmas~\ref{lem:safeentry} and~\ref{lem:safeexit}.
\item Adjust the value $\sigma_1<\sigma_0$ in the definition of $\xit$ and $F_\tau$ in \eqref{eq:xibis}. The adjustment is used in
  Section~\ref{sec:proof_of_pull_backs} while proving property~\ref{item:gradient} of Finger Move Theorem~\ref{thm:new_finger_move}.
\end{itemize}
\end{remark}

\section{Proof of properties \ref{item:gradient} and \ref{item:pull_back}}\label{sec:proof_of_pull_backs}

We recall the statements of \ref{item:gradient} and \ref{item:pull_back}.

\begin{itemize}
  \item[(FM-4)] Set $f_\tau=F_\tau\circ G_\tau$. The path $f_{1+\tau}$ is a constant path for $\tau\in[0,1]$. Moreover, $\eta$ is a gradient-like vector field for all $f_\tau$.
  \item[(FM-5)] There is a pull-back $\etat$ of $\xit$ via $G_2$ such that the grim paths of death starting from $F_2\circ G_2$ and
      constructed with $\eta$ and $\etat$ are left-homotopic.
\end{itemize}

We begin with the following result.
  \begin{lemma}\label{lem:f_stays}
    For any $u\in N$ and for any $\tau \in [0,2]$  we have $F\circ G_\tau(u)=F\circ G_0(u)$.
  \end{lemma}
  \begin{proof}
    We use \eqref{eq:Gtau_def_finger}. Note that if $u\notin G^{-1}(\Mone)$, then $G_\tau(u)=G(u)$, so there is nothing to prove.
    If $G(u)=(x_1,\dots,y_2)\in\Mone$, the finger move affects only the $y_2$ coordinate of $G(u)$, but $F$ does not depend on
    the $y_2$ coordinate of a point in $V_{\oout}$.   \end{proof}

  \begin{corollary}\label{cor:eta_gradient}
    The vector field $\eta$ is a gradient-like vector field for $F\circ G_\tau$ for all $\tau$.
  \end{corollary}
  Next we pass to the function $F_2\circ G_2$. Recall that the path
  $F_\tau$ was constructed in Theorem~\ref{thm:perturb}. The size
  of perturbation was controlled by a parameter $\sigma_1<\sigma_0$,
  compare \eqref{eq:xibis}.
  \begin{lemma}\label{lem:proof_of_gradient}
    Suppose $\sigma_1$ is sufficiently small. Then $\eta$ is a gradient-like vector field for $F_\tau\circ G_2$, for all $\tau \in [0,2]$.
  \end{lemma}
  \begin{proof}
    We use notation from Subsection~\ref{sub:explicit_perturbation}.
    We have $F_\tau=F-\tau\sigma_1\feta L$. Let $Z=G_2^{-1}(U_\lambda)$.
    As $\feta$ has support on $U_\lambda$, it is enough to prove the statement on $Z$.
    Note that $Z$ is separated from the critical set of $F\circ G_2$.
    As $\eta$ is gradient-like for $F\circ G_2$, there exists $c>0$ such that
    $\partial_\eta(F\circ G_2)(u)>c$ for all $u\in Z$.
    Denote by $K_\phi$ the supremum of $|\partial_\eta(\feta L\circ G)|$
    on $Z$. Choose $\sigma_1<c/K_\phi$. For all $\sigma<\sigma_1$
    we have $\partial_\eta F_{\tau}\circ G_2(u)>0$, if $u\in Z$.
    %Recall that we define $F_\tau$, for $\tau \in [0,1]$, by taking $\sigma \in [0,\sigma_1]$ such that $\sigma/\sigma_1 = \tau$, and setting $F_{\tau}:= F_{\sigma}.$  So in $Z$ we have $\partial_\eta F_\tau\circ G_2 > 0$ for all $\tau \in [0,1]$.
   %As $F_\sigma\circ G_2=F\circ G_2$ away from $Z$, we deduce that $\eta$ is a gradient-like vector field for $F_\tau\circ G_2$ for all $\tau\in[0,1]$. Note that $F_{1+\tau}=F_1$ for $\tau\in[0,1]$ by construction, so $\eta$ is a gradient-like vector field for $F_\tau\circ G_2$ for all $\tau \in [0,2]$, as desired.
  \end{proof}
 Now  define $\etat$ to be a pull-back of $\xit$ via $G_2$ in the sense of Section~\ref{sub:pull_back}. As $\eta$ itself is a pull-back of $\xi$ via $G$, we may and will assume that $\etat(u)=\eta(u)$ if $G_2(u)=G(u)$
  and $\xi(G(u))=\xit(G(u))$. This means that $\etat(u)=\eta(u)$
  away from $G_2^{-1}(V_{\oout}\setminus\Csafe)$.

  By construction,
  $\etat$ is a gradient-like vector field for $f_2=F_2\circ G_2$.
  For $\tau\in[0,1]$, set:
  \[\eta_\tau=(1-\tau)\eta+\tau\etat.\]
  %The goal is to show that the grim paths of death constructed with the $\eta_\tau$ are left-homotopic.
  %We aim to eventually invoke Lemma~\ref{lem:uniqueness_of_death}.
  %We begin with the following result.
  \begin{lemma}
    The vector field $\eta_\tau$ is a gradient-like vector field for $f_2 \colon N \to \R$.
  \end{lemma}
  \begin{proof}
    By construction, $\eta$ and $\etat$ agree away from the preimage of $G_2^{-1}(V_{\oout}\setminus\Csafe)$. In particular, $\eta$ and $\etat$ agree near all critical point of $f_2$. Away from these points,
    $\partial_{\eta}f_2>0$ by Corollary~\ref{cor:eta_gradient}, and
    $\partial_{\etat}f_2>0$ as $\etat$ is gradient-like for $f_2$.
    The statement promptly follows.
  \end{proof}

  \begin{theorem}\label{thm:paths_of_death_are_equal}
    The grim paths of death starting at $f_2$ constructed by $\eta$ and $\eta_\tau$ are left-homotopic.
  \end{theorem}

  \begin{proof}

    Our goal is to eventually use the Uniqueness of Death Lemma~\ref{lem:uniqueness_of_death}. To this end, we need to study the stable and unstable manifolds of $q_-$ and $q_+$ with respect to $\eta_\tau$, where  $q_- := G^{-1}(p_-)$ and $q_+ := G^{-1}(p_+)$, i.e.\ the preimages in $N$ of $p_-, p_+ \in \O$.

  Consider the level set $F^{-1}(c_{\bbot})$. Since this is the level set below the finger move is performed, we have that for any $w\in F^{-1}(c_{\bbot})$,
  $F(w)=F_\tau(w)$. Moreover, if $u\in f^{-1}(c_{\bbot})$, then $G_\tau(u)=G(u)$ for all $\tau$.
  Our first step toward the proof of Theorem~\ref{thm:paths_of_death_are_equal} is the following result, which relies on Theorem~\ref{thm:oneless}. %We are using notation from Section~\ref{sec:coorsystem}.
  The subset $W \subseteq V_{\iinn}$ was chosen in Section~\ref{sub:newvector}, as a small neighbourhood of the double point set $\Mfng\cap \Mtwo$, with $p_+\not\in W$.

%For example, one could define $W=\{|x_1|+|y_1|+|y_2|<\erho/4, |x_2^2+\cdots+x_n^2-\erho^2|<\erho^2/16\}$

  \begin{lemma}\label{lem:W_means_nop1}
    Suppose a trajectory of $\xi$ hits $\ol{W}$ below the level set $c_+$. If it hits $p_-$ in the infinite past, then it crosses $F^{-1}(c_{\bbot})$ at a point of $D_m$.
  \end{lemma}
  \begin{proof}
    Any trajectory of $\xi$ entering $V_{\oout}$ and hitting $p_-$ in the infinite past must cross one of the discs $D_0,\dots,D_r$.
    Consider such a trajectory and let $w_1\in D_0\cup\dots D_r$ be the point of its entrance to $V_{\oout}$.

    Assume the trajectory reaches $W$ before the level set $c_+$. Suppose towards contradiction that $w_1\in D_j$, $j\neq m$.
    As $W\subseteq V_{\mmid}$, the trajectory must enter $V_{\mmid}$. Let $w_2$
    be the point it enters $V_{\mmid}$. By Lemma~\ref{lem:safeentry}, either $F(w_2)>c_+$ and we are done, or $w_2\in\Cvs$. In the latter case, the trajectory has to leave $\Csafe$ before entering $W$.
    By Lemma~\ref{lem:safeexit}, such a trajectory leaves $\Csafe$ above the level set $c_+$. The contradiction shows that $j=m$.
  \end{proof}

  The above result has implications for trajectories of $\eta_\tau$ on $N$. The following corollary explains the behaviour of the unstable manifold of $q_-$ with respect to $\eta_\tau$.

  \begin{corollary}\label{cor:W_meansnop1}
    Suppose a trajectory of $\eta_\tau$ for $\tau\in[0,1]$ hits $G^{-1}(W\cap \Mtwo)$. Then it does not reach $q_-$ in the infinite past.
  \end{corollary}

  \begin{proof}
    Assume for a contradiction that $\gamma_u$ is a trajectory of $\eta_\tau$ that starts at $q_-$ and reaches $G^{-1}(W\cap\Mtwo)$.
    %Choose a parametrisation of $\gamma_u$ in such a way that
    %$\gamma_u(0)=u_0\in f^{-1}(c_{\bbot})$.
    Let $t_0>0$ be such that $\gamma(t_0)$ hits $G^{-1}(V_{\mmid})$.
    Then, for $t<t_0$, $\gamma_u$ is a trajectory of $\eta$, because $\eta_\tau$ agrees
    with $\eta$ away from $G^{-1}(V_{\mmid})$. Next, on $G^{-1}(V_{\mmid}\setminus V_{\iinn})$, $\eta=\etat$, so also $\eta=\eta_\tau$,
    by Lemma~\ref{lem:series_3}. Similarly, on $G^{-1}(V_{\mmid}\setminus W)$ we use Lemma~\ref{lem:series_4} to show that $\eta=\etat=\eta_\tau$.
    Note that Lemma~\ref{lem:series_4} cannot be used inside $G^{-1}(W)$, because we use a perturbation of $\xih$ to $\xit$. Anyway, we proved that until it reaches $G^{-1}(W)$, the trajectory $\gamma_u$ is a trajectory of $\eta$.
    The trajectory $\gamma_u$ between $q_-$ and $G^{-1}(W)$ does not hit any preimage of the second stratum. By Pull-back Lemma~\ref{lem:pull_back_lemma}, the image $G(\gamma_u)$ is a trajectory of $\xi$ starting at $p_-$ and reaching $W$.
    %Now $G$ maps a trajectory of $\eta$ to a trajectory of $\xi$. Therefore, $G(\gamma_u)$ is a trajectory of $\xi$ that starts at $p_-=G(q_1)$
    %and reaches $W$.
    By Lemma~\ref{lem:W_means_nop1}, we conclude that $u_0\in G^{-1}(D_m)$.
    But $D_m$ is disjoint from $M$. The contradiction
    shows that $\gamma_u$ cannot reach $q_-$ in the infinite past.
  \end{proof}

    \begin{lemma}\label{lem:Mone_means_no_q2}
    Suppose a trajectory of $\eta_\tau$ passes through $u\in f^{-1}(c_{\bbot})$. If this trajectory hits $G^{-1}(\Mone)$, then it does not reach $q_+$ in the infinite future.
  \end{lemma}

  \begin{proof}
    First we prove this result for $\tau=0$, that is for $\eta_\tau=\eta$.
    Since $\Mone$ belongs to the subset $\{y_2\equiv 1\}$, a trajectory of $\xi$ passing through this
    subset will have $y_2=1$, at least until it leaves $V_{\oout}$. But any trajectory leaving $V_{\oout}$ leaves it above the level set $F(p_+)$, so the trajectory will not reach $p_+$.
    From this it follows that if a trajectory of $\eta$ leaves $G^{-1}(\Mone)$, it will not reach $q_+$.

    Also note that if a trajectory of $\xi$ leaves $\Mone$, then it will not reach $\Mtwo$. This is because $\Mone$ has $y_2$--coordinate
    equal to $1$ and the $y_2$ coordinate near $\Mtwo$ is close to zero.

    Suppose a trajectory $\gamma_u(t)$ with $\gamma_u(0)=u$ of $\eta_{\tau}$ passes through $G^{-1}(\Mone)$ and reaches $q_+$. In general, such a trajectory can
    hit $G^{-1}(\Mone)$ infinitely many times, but there has to be $t_{\max}$ such that $\gamma_u(t_{\max})\in\partial G^{-1}(\Mone)$ and $\gamma_u(t)\notin G^{-1}(\Mone)$
    for $t>t_{max}$.
    Now from the time moment $t_{\max}$ onwards, $\gamma_u$ is a trajectory of $\eta$: it does not return to $G^{-1}(\Mone)$, it does not hit $G^{-1}(\Mtwo)$, and
    away from $G^{-1}(\Mone\cup\Mtwo)$ we have $\eta=\eta_\tau$, because $\eta$ was not modified there.
    We have already shown that no trajectory of $\eta$ can leave $\Mone$ and hit~$q_+$.
  \end{proof}

  For a fixed $\tau\in[0,1]$, let $A_\tau$ be the intersection of the unstable (ascending) manifold of $q_-$ (with respect to $\eta_\tau$) with the level set $f^{-1}(c_{\bbot})$.
  Let $B_\tau$ be the intersection of the stable (descending) manifold of $q_+$ (with respect to $\eta_\tau$) with the level set $f^{-1}(c_{\bbot})$.

  \begin{lemma}\label{lem:AtauBtau}
    For any $\tau$ the submanifolds $A_\tau$ and $B_\tau$ intersect at a single point. Near this point $A_\tau=A_0$ and $B_\tau=B_0$.
  \end{lemma}

  \begin{proof}
    Below the level set $c_{\bbot}$ we have $\eta=\eta_\tau$, hence $A_\tau=A_0$. On $B_\tau$ we specify the subset
    $B_\tau^0$, as the set of points $u\in B_\tau$ such that the trajectory of $\eta_\tau$ through $u$ reaches $G^{-1}(\Mtwo\cap W)$ in the future. As $W$ is open, $B_\tau^0$ is an open subset of $B_\tau$.
    Let $B_\tau^1$ be the complement of $B_\tau^0$ in $B_\tau$. By Lemma~\ref{lem:Mone_means_no_q2}, if $u\in B_\tau^1$, then the trajectory of $u$ is in fact a trajectory
    of $\eta$. In particular $B_\tau^1=B_0^1$.

    By Corollary~\ref{cor:W_meansnop1} the whole intersection of $A_\tau$ and $B_\tau$ occurs in the interior of $B_\tau^1$. That is,
    $A_\tau\cap B_\tau=A_\tau\cap B_\tau^1=A_0\cap B_0^1=A_0\cap B_0$ as desired.
  \end{proof}

  Now we finish the proof of Theorem~\ref{thm:paths_of_death_are_equal}. Lemma~\ref{lem:AtauBtau} implies that the assumptions
  of the Uniqueness of Death Lemma~\ref{lem:uniqueness_of_death} are satisfied.
  This concludes the proof of Theorem~\ref{thm:paths_of_death_are_equal} and therefore of the last property of Theorem~\ref{thm:new_finger_move}.
\end{proof}

\part{Examples and applications}\label{part:examples}

%\section{Overview of Part~\ref{part:examples}}

This part gives examples and applications of our theory.  We already explained some applications to link homotopy in the introduction.

First we prove the Singular Concordance implies Regular Homotopy Theorem~\ref{thm:concordance} in Section~\ref{sec:concordance_implies}.
Next, we illustrate the Singular Concordance implies Regular Homotopy Theorem~\ref{thm:concordance} using some (previously known) low dimensional cases.
We also give an application to homotopies between two surfaces in a 4-manifold.

In Section~\ref{sec:braid} we prove Proposition~\ref{prop:braids}, which states that any classical link with $n$ components is link homotopic
to the closure of a braid with $n$ strands. While the result is well-known, our proof illustrates the methods in the proof of
the Singular Concordance implies Regular Homotopy Theorem~\ref{thm:concordance}.

Section~\ref{sec:61} deals with another example. The stevedore knot $6_1 \subseteq S^3$ is slice, that is it bounds a smoothly embedded disc in $D^4$. For one such disc, the function `distance to origin'
restricted to this disc has two minima and one saddle. It is not possible to cancel a minimum--saddle pair  in the embedded case, even though such
cancellation is possible restricting to the disc, forgetting about the embedding. To see this, observe that if this cancellation could be achieved, then the stevedore knot would be unknotted. However, after a finger move, the embedded cancellation becomes possible. This example serves as an alternative geometric explanation of the finger move. It is our interpretation of the method sketched by Habegger in~\cite{Habegger-1992-1}.

Section~\ref{section:regular-homotopies-surfaces} investigates regular homotopies between surfaces in 4-manifolds, and shows that every regular homotopy can be expressed as a combination of finger moves and Whitney moves, and that all the finger moves can be assumed to happen first.  This fact is often quoted in the literature.
%, but as far as we are aware a rigorous proof has not been given.

\section{Immersed link concordance implies regular link homotopy}\label{sec:concordance_implies}

Having  the Path Lifting Theorem~\ref{thm:path_lifting} at our disposal, we can prove the main result of the article.

%Fix a closed $(n-1)$-manifold $A$ and a compact $(n-1+k)$-manifold $Y$, both smooth.
% Consider a link concordance $G\colon A\times[0,1]\to Y\times[0,1]$ that is also an immersion, such that $G|_{A\times\{0\}}=g_0\times\{0\}$, $G|_{A\times\{1\}}=g_1\times\{1\}$ with $($link$)$ immersions $g_i \colon A\to Y$.
%
%If the codimension $k \geq 2$, then there is a regular link homotopy $G_\tau$ $($rel.\ boundary$)$ with $G_0=G$ and such that $G_1$ is a level-preserving link immersion.   In particular $g_0$ and $g_1$ are regularly link homotopic.
%
%Moreover, if the codimension $k \geq 3$ and $G$ is an embedding, then $G_\tau$ may be chosen to be an embedding for all $\tau \in [0,1]$. In particular $g_0$ and $g_1$ are ambiently isotopic.

\begin{theorem}[Immersed link concordance implies regular  link homotopy]\label{thm:concordance}
  Suppose $A$ is a closed $(n-1)$-dimensional manifold and let $Y$ be a compact $(n+k-1)$-dimensional manifold. Assume that $g_0,g_1\colon A\to Y$ are  concordant generic immersions, via a generic immersion $G\colon A\times[0,1]\to Y\times[0,1]$ such that $G|_{A\times\{0\}}=g_0\times\{0\}$, $G|_{A\times\{1\}}=g_1\times\{1\}$.
\begin{enumerate}
  \item\label{item-conc-implies-link-hom-thm-1}
  If $k>2$, and $G$ is an embedding,
  then there are rel.\ boundary isotopies of diffeomorphisms $H_t$ of $Y \times [0,1]$ and $K_t$ of $A \times [0,1]$, such that
  \[G' := H_1 \circ G \circ K_1 \colon A \times [0,1] \to Y \times [0,1] \]
  is a level-preserving embedding.  In particular $g_0$ and $g_1$ are isotopic.
  %If $G$ is an embedding then $G'$ is the trace of an ambient isotopy.
\item\label{item-conc-implies-link-hom-thm-2}
If $k \geq 2$,
 then there is a rel.\ boundary regular homotopy $H_t \colon A \times [0,1] \to Y \times [0,1]$ with $H_0=G$ such that $H_1$ is a level-preserving generic immersion.   In particular $g_0$ and $g_1$ are regularly homotopic.   Moreover, for every pair of connected components of $A \times [0,1]$ that has disjoint image under $G$, the same holds for $H_t$, for all $t \in [0,1]$.
    In particular, if $H_0$ is an immersed link concordance then $H_1$ is a regular link homotopy.
  \end{enumerate}
\end{theorem}

Before we prove Theorem~\ref{thm:concordance} we recall one more result of Cerf.

\begin{proposition}[see \expandafter{\cite[Theorem V.1.1]{Cerf}}]\label{prop:cerf2}
  Suppose $f_0,f_1\colon N\to\R$ are two Morse functions with the property that whenever $q_-,q_+$ are two critical points of $f_0$ $($respectively $f_1)$, with $\ind q_-<\ind q_+$, then $f_0(q_+)>f_0(q_-)$ respectively $f_1(q_+)>f_1(q_-)$.

  Then there exists a $\cF^1$-path of functions $f_\tau$ such that for every rearrangement occurring along~$f_\tau$, no critical point
  of higher index goes below a critical point of lower index.
\end{proposition}

\begin{proof}[Proof of Theorem~\ref{thm:concordance}]
  Set $\O=Y\times[0,1]$ and let $F\colon\O\to[0,1]$ be the projection onto the first factor. Let $N=A\times[0,1]$.
  Perturb $F$ if necessary to make it an immersed Morse function with respect to $M=G(N)$, without introducing critical points of $F$ on $\O \sm M$.
  Set $f\colon  N\to[0,1]$ to be $f=F\circ G$.
  Then $f$ is a Morse function on $N$.

  If $F$ has no critical points (and no critical points on all strata of $M$), we conclude that $G(A\times\{0\})$ is isotopic to $G(A\times\{1\})$. In fact, we can identify
  $F^{-1}(t)$ with $Y$ and under this identification, $M\cap F^{-1}(t)$ is a set $A_t$ isotopic to $A_0$. If $F$ has no critical points on the zeroth and first stratum of $M$,
  we conclude that $A\times\{0\}$ is regularly homotopic to $A\times\{1\}$ by
  Theorem~\ref{prop:crossingforhomotopy}.
  %the Crossing Deeper Strata Theorem~\ref{thm:deeperstrata}.
The aim is to achieve these conditions on $F$ in cases \eqref{item-conc-implies-link-hom-thm-1} and \eqref{item-conc-implies-link-hom-thm-2} respectively.

  Now, $f$ is a Morse function on $A\times[0,1]$, equal to $0$ on $A\times\{0\}$ and equal to $1$ on $A\times\{1\}$.
  Let $f_1$ be the Morse function on $A\times[0,1]$ that is a projection onto the second factor.
  Note that $f_1$ and $f$ can be made equal on a neighbourhood of $A\times\{0,1\}$. Then  we connect $f$ and $f_1$
  by a neat path of functions.
  By Lemma~\ref{lem:regular_path}, we may perturb this path
  to a neat $\cF^1$-path of functions $f_{\tau}$ connecting $f$ to a function $f_1$.
  %This path need not be very neat. However, we can argue as follows.

  By Proposition~\ref{prop:cerf2},
  we may and shall assume that:
  \begin{itemize}
    \item $f_{\tau}$ is an elementary path near each event (birth, death and rearrangement);
    \item $f_{\tau}$ has no rearrangements such that a critical point of higher index goes below a critical point of lower index.
  \end{itemize}
  These are precisely the assumptions for the Path Lifting Theorem~\ref{thm:path_lifting}. This means that there exists a regular double path $(F_{\tau},G_\tau)$ that is a weak lift of $f_{\tau}$.
 Moreover $F_{\tau}$ is a constant path near $F^{-1}(0)$ and $F^{-1}(1)$. The path of immersions $G_\tau$ changes the map $G$ by a regular homotopy, but only introduces intersection points within each connected component. If $k \geq 3$ and $G$ is an embedding then $G_\tau$ is an ambient isotopy.

 The functions $F_{\tau}$, regarded as ordinary Morse functions on $\O$, have no critical points.
 We invoke Proposition~\ref{prop:crossingforhomotopy} to conclude that $G(A\times\{0\})$ is regularly homotopic to $G(A\times\{1\})$.
 The regular homotopy does not create any intersections between different connected components of $A$. Moreover, if $k>2$ and $G$ was initially an embedding, then so is $G_\tau$ for all~$\tau$. Thus $M$ is the trace of an ambient isotopy. %\ypar{Rephrased}
 \begin{comment}
 By stability~\cite[Proposition III.2.2]{GG}, as in Lemma~\ref{lem:lift_morse},  there exist isotopies $\Phi_\tau\colon\O \to \O$ and $\Upsilon_\tau \colon [0,1] \to [0,1]$
  such that $F_\tau= \Upsilon_\tau \circ F_0\circ\Phi_{\tau}$ for all $\tau \in [0,1]$. We may and will assume that $\Phi_\tau|_{Y\times\{0,1\}}$ is the identity.

  Consider the immersion $\Phi_1\circ G_1\colon A\times[0,1]\to\O$. The function $\Upsilon_1 \circ F_0$ has no critical points on the zeroth and the first
  strata of $M=\Phi_1(G_1(A\times[0,1]))$, because $F_1\circ G_1=f_1$, and $f_1$ is projection $A \times [0,1]=N \to [0,1]$.  Therefore $F_0$ also has no critical points in these strata of $M$, since $\Upsilon_1$ is a diffeomorphism of $\R$.  By Theorem~\ref{thm:deeperstrata}, $M$ is the trace of a regular homotopy between $M\cap Y\times\{0\}$ and $M\cap Y\times\{1\}$.    If $k>2$ and $G$ was initially an embedding, then $M$  is also embedded,
  and since there are no critical points $M$ is the trace of an ambient isotopy.
\end{comment}
\end{proof}

\section{A 3-dimensional example: the proof of Proposition~\ref{prop:braids}}\label{sec:braid}

We recall the statement of Proposition~\ref{prop:braids} for the convenience of the reader.

\begin{proposition}\label{prop:braids-restated}
Every classical link in $S^3$ with $n$ components is link
homotopic to the closure of a braid with $n$ strands.
\end{proposition}

\begin{proof}
Consider a link in $S^3$ with $n$ components.  Choose a disc $D^2$ intersecting each of the~$n$ components of the link transversely in a single point, and thicken it to obtain $D^2 \times I$ intersecting the link in $\{q_i\} \times I$ for some points $q_i \in D^2$.
Remove the interior of $D^2 \times I$ from $S^3$, and identify the remainder also with $D^2 \times I$.
We obtain a concordance
$G\colon  \sqcup^n I \hra D^2 \times I$ from $\{(q_i,0)\}_{i=1}^n$ to $\{(q_i,1)\}_{i=1}^n$, which is called a \emph{string link} in
\cite{Habegger-Lin:1990-1}. An example is shown in Figure~\ref{fig:string link}.  Write $G$ for this string link.

\begin{figure}[ht]
  \input{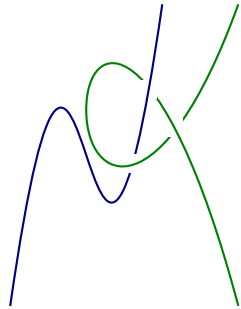}
\caption{A string link.}\label{fig:string link}
\end{figure}

After a perturbation, we may assume that the composition $\pi_2\circ G \colon n\cdot I\to I =[0,1]$ (where $\pi_2$ is a projection onto the second factor) is a Morse
function. If this Morse function has no critical points then the string link is a
braid, and we are done.
We will explain how to perform a link homotopy to reduce the number of local minima by one, and thus simultaneously  the number of local maxima by one (which must occur for Euler characteristic reasons). An induction on the number of critical points then completes the proof.

By Theorem~\ref{thm:grim_global_rearrangement}, by an isotopy we may arrange that all the local minima occur below the local
maxima, as is already the case in Figure~\ref{fig:string link}.  Assume that the local minima occur in $D^2 \times [0,1/2]$, and the local maxima are in $D^2 \times [1/2,1]$, with $1/2 \in I$ a regular value.
Then the disc $L:= D^2 \times \{1/2\} \subseteq D^2 \times I$ is an intermediate level set of the Morse function $\pi_2$.
The image $G(\sqcup^n I)$ intersects $L$ transversely in $k\geq n+2$ points $p_1,\dots,p_k$; compare Figure~\ref{fig:string1}.

Take two such points $p_i$ and $p_j$ and suppose they are connected by an arc in $G(\sqcup^n I)$ that does not intersect $L$. In Figure~\ref{fig:string1}
such pairs of points are $(p_1,p_2)$, $(p_2,p_4)$, $(p_3,p_5)$ and $(p_3,p_6)$. Each such pair corresponds either
to a local minimum or to a local maximum of the function $\pi_2\circ G$. For two such points, write $a_{ij} = a_{ji}$ for the projection of the corresponding
arc to the disc $L$. Such arcs are also drawn in Figure~\ref{fig:string1}. In the language of Section~\ref{sec:hypermembranes},
these arcs are the intersections of the ascending and descending membranes of the critical points with the level set $L$.
We assume that the arcs $a_{ij}$ intersect one another  transversely.
Moreover, the arcs corresponding to maxima (respectively minima) are mutually disjoint, while the projections of minima and maxima arcs can intersect each other.
The transverse intersections assumption is equivalent to the Morse--Smale condition (Definition~\ref{def:Morse_smale_immersed}).

If we consider the projection of $G(\sqcup^n I)$ to the disc $L$, in addition to the arcs just discussed we also see two arcs per component of $G(\sqcup^n I)$, running from one of the $p_k$ to $(q_\ell,1/2) \in L = D^2 \times \{1/2\}$, for some $\ell$. These correspond to the projections of monotone arcs with one endpoint $p_k$ and one endpoint either $q_\ell^+ := q_\ell \times \{1\}$ or $q_\ell^{-} = q_\ell \times \{0\}$. (Recall that $q_\ell^{\pm}$ are the endpoints of the $\ell$-th component of $G(\sqcup^n I)$.)  These arcs in $L$ can intersect anything, including themselves; we do not need to control these intersections. They represent a portion of the link components where the embedding is already monotone increasing with respect to the $I$ coordinate.  Denote the arc in $L$ by $b_\ell^{\pm}$, using $b_\ell^+$ if one endpoint of the preimage in $D^2 \times I$ is $q_\ell^+ \in D^2 \times \{1\}$, and $b_\ell^-$ if the endpoint is $q_\ell^- \in D^2 \times \{0\}$.
Figure~\ref{fig:level} shows the configuration in $L$ coming from Figure~\ref{fig:string link}.

\begin{figure}
  \input{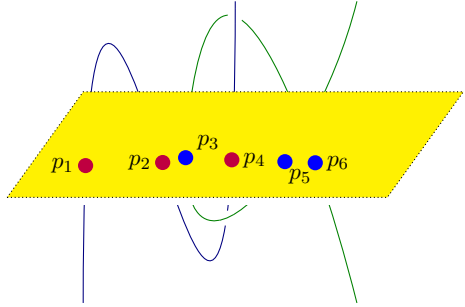}
  \caption{Intersecting the string link of Figure~\ref{fig:string link} with an intermediate level set $L$. }\label{fig:string1}
\end{figure}

\begin{figure}[ht]
  \input{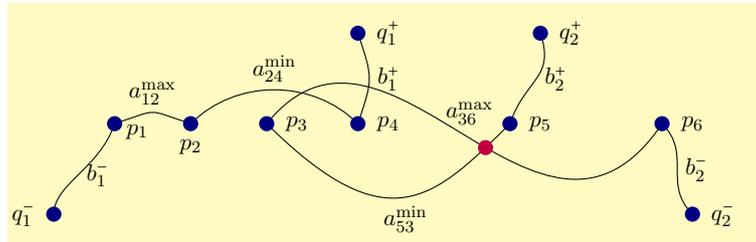}
\caption{The projection of the arcs containing the critical points, and the arcs $b_i^{\pm}$ to the level set $L$. In the language of Definition~\ref{def:hypermembrane}, the arcs between two $p_i$ are the intersections
of the ascending and descending membranes of the local minima and maxima with~$L$.  The points $q_i \times \{0\}$ and $q_i \times \{1\}$ have been displaced from one another so their projections can be shown as distinct points $q_i^{\pm}$. Arcs labelled with $\min$ represent the projections of minima, while arcs labelled with $\max$ represent the projections of maxima.}\label{fig:level}
\end{figure}

Let $a_{\min}$ and $a_{\max}$ be two arcs among the $\{a_{ij}\}$ that meet at one of their end points, with $a_{\min}$ (respectively $a_{\max}$) the projection of an arc containing a local minimum $w_{\min}$ of $\pi_2 \circ G$ (respectively, a local maximum $w_{\max}$).  Choose these to be the first such pair that arises while travelling along a link component starting at $q_{\ell}^-$ for some $\ell$, i.e.\ $a_{\max}$ must be adjacent to $b_\ell^-$.
If this is their only intersection point then isotope the arc $a_{\min}$ that is below the intermediate level to slightly above it, and isotope the arc $a_{\max}$ that is above the intermediate level to below it. Then further isotope both arcs downwards, without creating maxima nor minima, so that the arc $a_{\min}$, apart from the endpoint that is furthest along the $\ell$th component, is moved back slightly below the intermediate level; move this endpoint so as to lie on $L$.  The result is to cancel the critical points $w_{\min}$ and $w_{\max}$. In the model case presented in Figure~\ref{fig:level}, we can perform this operation for the arcs $a_{12}$ and $a_{24}$ connecting
$(p_1,p_2)$ and $(p_2,p_4)$. The effect on the projection to $\ell$ is that the arcs $a_{\max}$ and $a_{\min}$ are both added to $b_\ell^-$.
%\npar{I've modified the proof here to make sure that we do the operation as we go along a link component, and to make sure both arcs are added to the monotone part where we don't control intersections. Before, the move would have introduced new intersections that were not allowed. By adding them both to $b_{\ell}^-$ we no longer have to control the intersections.}{good point!}

%\begin{figure}[ht]
%\vspace{3cm}
%\caption{A Morse cancellation, i.e.\  the family of functions
%$f_t(x)=tx+x^3$.}\label{fig:cancellation}
%\end{figure}
%\todoMP{Picture of Morse cancellation.}
To fix notation, let $i, j$, and $k$ be defined by setting $a_{\min} = a_{ij}$ and $a_{\max} = a_{jk}$.
Even if the arcs $a_{\min}$ and $a_{\max}$ meet in their interior, that is if $(a_{ij} \cap a_{jk}) \sm \{p_j\}$ is nonempty (this is the case in Figure~\ref{fig:level} for arcs $a_{53}$ and $a_{36}$ connecting $(p_3,p_5)$ and $(p_3,p_6)$), we can still perform the same cancellation motion --- push $a_{ij}$ above $L$, push $a_{jk}$ below $L$, and then move all apart from the endpoint $p_i$ of $a_{ij}$ to be below $L$. Since the maxima (minima) arcs are mutually disjoint, this only creates intersections between the two pushed arcs. Since they share a common point $p_j$, the arcs $a_{ij}$ and $a_{jk}$ are projections of arcs of the same component of the string link $G(\sqcup^n I)$, and so we have performed a link homotopy.
An example is drawn in Figure~\ref{fig:string3}, and
the effect of the homotopy in Figure~\ref{fig:string3} on the configuration of arcs is shown in Figures~\ref{fig:string4} and~\ref{fig:string5}.

The effect of the cancellation on the configuration of arcs diagram, in $L$, is to remove the points of intersection $p_j$ and $p_k$ of $G(\sqcup^n I)$ with $L$. The arcs $a_{ij}$ and $a_{jk}$ are joined with $b_\ell^-$.
%their adjacent arcs $a_{ri}$ and $a_{ks}$ respectively, or with $b_i^+$, respectively
%$b_k^-$, if the other arc with endpoint $p_i$, respectively $p_k$, is one of the monotone $b$ arcs.
The point $p_i \in L$ at the far end of $a_{ij}$ remains. The number of maxima and the number of minima of $\pi_2 \circ G$ have both been reduced by one.  Since $\sqcup^n I$ is compact, a finite induction suffices to remove all maxima and minima.  Once all that remains are the monotone arcs $b_i^{\pm}$, we have a braid, and it has been obtained by a link homotopy from our original string link~$G$.  The closure of this braid is therefore link homotopic to the original link in $S^3$ we started with.
\begin{figure}
  \input{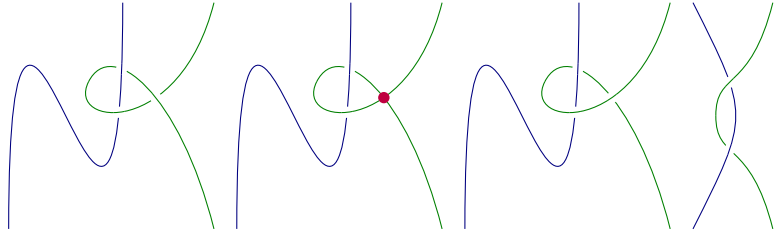}
  \caption{A homotopy of the string link that cancels critical points, even when there are intersections between membranes.}\label{fig:string3}
\end{figure}
\begin{figure}
  \input{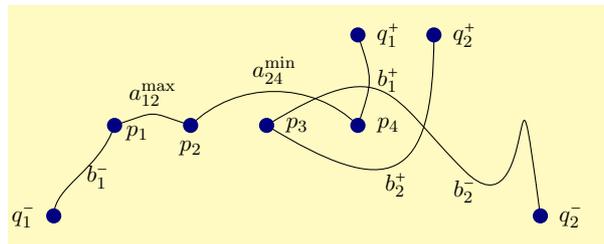}
  \caption{The effect of the homotopy in the first three frames of Figure~\ref{fig:string3} on the configuration of arcs in $L$.}\label{fig:string4}
\end{figure}
\begin{figure}
  \input{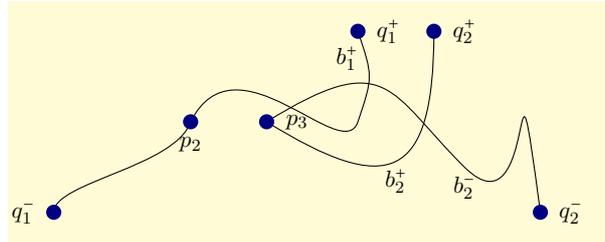}
  \caption{The effect of the homotopy in the final two frames of Figure~\ref{fig:string3} on the configuration of arcs in $L$.}\label{fig:string5}
\end{figure}
\end{proof}

\section{A 4-dimensional example}\label{sec:61}

We provide another example, this time involving surfaces in 4-dimensional space. We explain Figure~\ref{fig:stevedore} from the
introduction in detail.
This example enables us to illustrate the  finger move, and the idea of the proof of Theorem~\ref{thm:concordance}. We use the simplest nontrivial slice knot $6_1$, the stevedore knot, which is shown in  Figure~\ref{fig:one}.
This example motivates the more general construction of the finger move from
Part~\ref{part:finger}.

\begin{figure}[h]
	\input{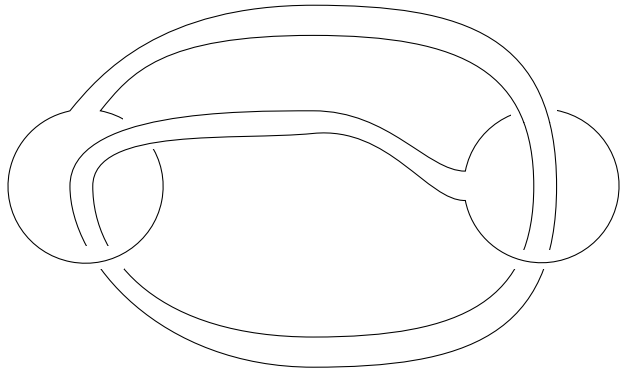}
	\caption{The stevedore knot $6_1$.}\label{fig:one}
\end{figure}

The stevedore knot $6_1$ bounds a disc $\Delta$ in $D^4$, but no  embedding of the disc can be built from a
single $0$-handle and no $1$-handles, in such a way that the handle decomposition comes from a Morse function on $D^4$ with a single index $0$ critical point; if it could, the knot would be trivial.
However, the knot bounds a disc admitting a handle decomposition with two $0$-handles and one $1$-handle.
This can be seen in Figure~\ref{fig:two}.
Our aim is to cancel a $0$--$1$ handle pair at the expense of a regular homotopy of $\Delta$. This will yield a link null-homotopy of $6_1$. Of course we know this is possible, because every knot in $S^3$ is null homotopic. Our aim is to give an explicit demonstration of the finger move.

To fix notation, consider $F \colon D^4 \to [0,1]$ given by the radius, and assume $F$ is an immersed Morse function with respect to $\Delta$. Denote the critical points corresponding to 0-handles by $p_-,p_-'$, and the critical point
corresponding to the $1$-handle by $p_+$.  Choose a grim vector field $\xi$ on $D^4$ for $F$.

\begin{figure}
  \input{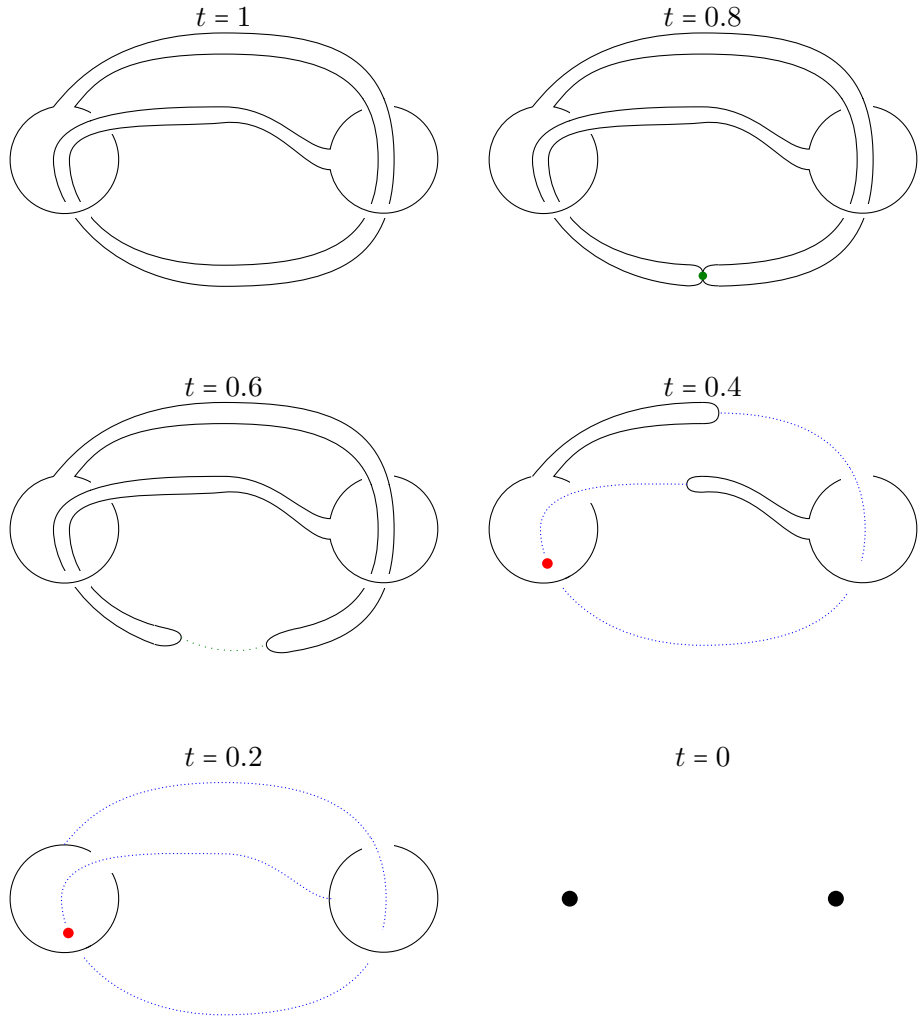}
  \caption{A movie indicating that the stevedore knot $6_1$ bounds a disc $\Delta$ in $D^4$ such that the distance-to-origin function $t$ on $D^4$ restricts
  to a Morse function on $\Delta$ with precisely two minima and one saddle.}\label{fig:two}
\end{figure}
%\npar{How about time labels on the last two pictures as well? }{I added labels, shifting to $1,0.8,0.6,0.4,0.2,0$ to make them evenly separated. If you want to change them again, it is the line `draw (0,2.7) node {$t=0.8$};' in each scope environment.}

In the absolute setting, i.e.\ if we were to ignore the fact that they lie on an embedded submanifold and just consider trajectories of $\xi$ on $\Delta$, either pair of 0- and 1-handles could be cancelled. This is not possible when we consider the ambient trajectories of $\xi$ as well. The obstruction to cancellation can be expressed in terms of the membranes $\Hd^0(p_+)$ and $\Ha^0(p_-)$ or $\Ha^0(p_-')$.  When the membranes of the critical points have nontrivial intersection in their interior, this prevents one from cancelling the critical points, even when there is a unique trajectory on $\Delta$ between the critical points.
As mentioned before, this was studied in \cite{Pe,Sha,BP}.
Each such intersection between membranes corresponds to a trajectory that starts at the index zero critical point, leaves the embedded disc immediately and then later arrives at the index~$1$ critical point. The
hypothesis of the Cancellation Theorem~\ref{thm:grimcanc} that there be only one trajectory of a gradient-like immersed vector field between the critical points, and that this must lie on $\Delta$, is not satisfied.  The ascending membrane of one of the
index zero critical points $p_-, p_-'$, and the descending membrane of the index one critical point $p_+$, are depicted in Figure~\ref{fig:membranes3}.  We can see that they intersect.
\begin{figure}
  \includegraphics[width=5cm]{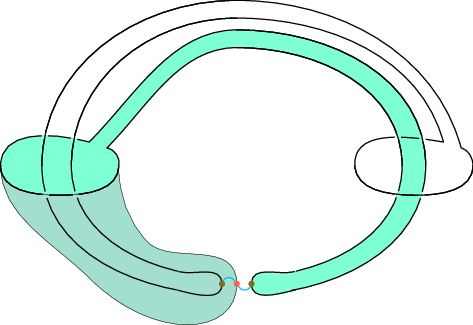}
	\caption{The ascending membrane (shaded) of one of the two index $0$ critical points $p_-$ of the disc $\Delta$, below the index 1 critical level set. The small arc represents
	the descending membrane of the index $1$ critical point $p_+$ of $\Delta$. The intersection of the two membranes is one point in the middle of the arc.}\label{fig:membranes3}
\end{figure}
%\begin{figure}
%  \input{pictures/stev_finger.tex}
%\end{figure}

However, we can cancel the intersection points in the interiors of membranes after a suitable regular homotopy of $\Delta$. This homotopy can be constructed as follows.
Choose a level set $c$ in between the two critical level sets $c_-$ and $c_+$ for $p_-$ and $p_+$ respectively.  Choose a curve $\gamma$ that lies in this level set, that lies on the ascending membrane of $p_-$, and which connects the intersection point of the ascending and descending membranes $\Ha^0(p_-) \cap \Hd^0(p_+) \cap F^{-1}(c)$ to a point on the ascending sphere $\Ha^1(p_-)$ of $p_-$.   See Figure~\ref{fig:guiding}.
\begin{figure}
  \includegraphics[width=5cm]{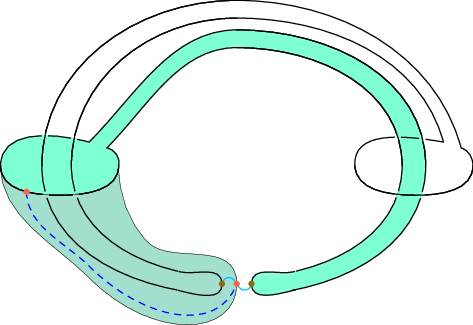}
  \caption{The guiding curve $\gamma$ represented as a dotted arc connecting the intersection point of the membranes with the link.}\label{fig:guiding}
\end{figure}
Now take the curve $\gamma$ and flow it upwards, to a level in between the level set~$c$ and a level set just above the index $1$ critical point. That is, to the level $c_+ +\delta$, where $\delta>0$ is a small real number.  Let $\gamma_t$ be the result of flowing $\gamma$ to the level set $t$. So we are considering $\gamma_{c_+ + \delta}$.

Above $c_+$, there is a problem with flowing at the endpoint $\gamma(0)$ of $\gamma$ that lies on $\Hd^0(p_+)$.  We therefore consider, in the level sets $F^{-1}(\wt{c})$, for $\wt{c}$ in $(c_+,c_+ +\delta]$, the closure of $F^{-1}(\wt{c}) \cap \bigcup_{t \in \R} \gamma_{t}$; see Figure~\ref{fig:guiding2}.
\begin{figure}
  \input{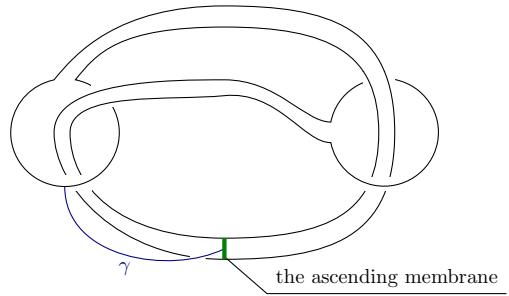}
  \caption{The guiding curve $\gamma$ above the level set $c_+$.}\label{fig:guiding2}
\end{figure}
%with the ascending membrane $\Ha^0(p_0) \cap F^{-1}(t)$ of the index 1 critical point at the level set $t$.

%\begin{figure}
%	\input{pictures/knoc_with_membrane.tex}
%	\caption{A guiding curve $\gamma$.}\label{fig:guiding}
%\end{figure}

Now we describe the homotopy for the finger move.  It is supported in the level sets $[c_+-\delta,c_+ + \delta]$. We describe the intersection of the finger-moved $\Delta$ with $F^{-1}([c_+-\delta,c_+ + \delta])$,  starting at the level set $c_++\delta$, \emph{above} the critical level $F(p_+)=c_+$ for the index 1 critical point $p_+$.  We will work downwards.
We push a neighbourhood of the endpoint of $\gamma$ that lies on the ascending sphere of $p_-$ along $\gamma$ by a `finger move' as in Figures~\ref{fig:dragging}.

\begin{figure}
	\input{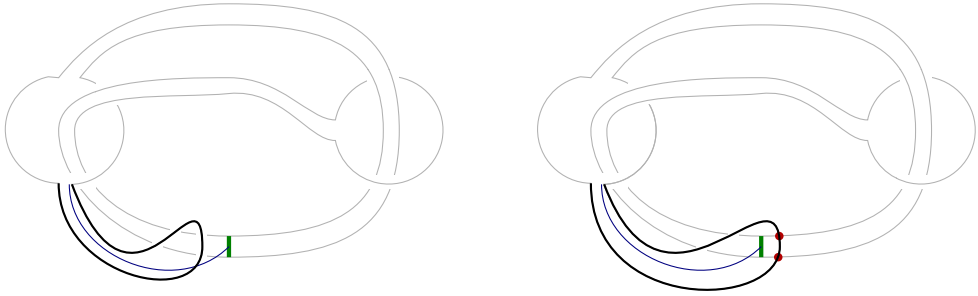}
\caption{The outcome of dragging part of the knot along the guiding curve, shown at the level sets $F^{-1}(c_+ + 3\delta/4)$ on the left, and at $F^{-1}(c_+ + \delta/2)$ on the right.}\label{fig:dragging}
\end{figure}
The length of the finger depends on the level $c_+ + t\delta$, $t \in [-1,1]$.
For $c_++\delta = c_{\ttop}$, the length is zero. As $t$ decreases from $1$, the finger gets longer. At some level set $c_{\mathrm{int}} \in (c_+,c_++\delta)$, the finger crosses the two lower strands of the knot, creating a pair of  self-intersections; see Figure~\ref{fig:dragging}(left). This happens at the two points of $\Ha^1(p_+) \cap F^{-1}(c_{\mathrm{int}})$; see Figure~\ref{fig:dragging}(right).
Moving the finger slightly further in the level sets below $c_{\mathrm{int}}$, we no longer have self-intersections.
The knot after this move, which is an unknot, is depicted on the left of Figure~\ref{fig:knoc_after}.

\begin{figure}
  \input{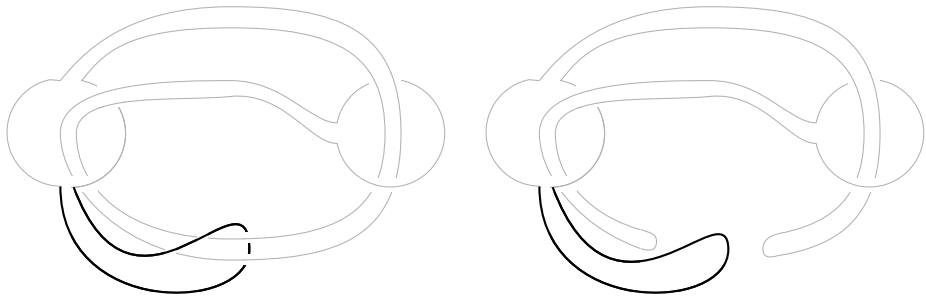}
  \caption{Left: the maximum extent of the finger move, in $F^{-1}(c+\delta/4)$. The finger is pushed past the band. The resulting knot is the unknot. Right: the finger is dragged back to where it started, below the level set $F^{-1}(c_+)$. The right figure shows $F^{-1}(c_+-\delta/4)$.}\label{fig:knoc_after}
\end{figure}

Next, we have to undo the finger, in the reverse of the above process.  We will do this within the level sets $[c_+-\delta,c_+)$.  We keep the finger constant between $c_{\mathrm{int}}$ and $c_+$, then below $c_+$ we begin to retract it.   We do this in such a way that the disc $\Delta$ is unaltered by the move at the level set $c_+-\delta$, and the finger reduces in length gradually as we move down from $c_+$ to $c_+-\delta$.  Below the index one critical point $p_+$ i.e.\ below the level set $c_+$, the band is cut by the saddle point, so the finger can shorten without introducing any more self intersections of $\Delta$.
Compare the right of Figure~\ref{fig:knoc_after}.

The finger move transforms the disc $\Delta$ into an immersed disc $\Delta'$. The two discs only differ in the region $F^{-1}([c_+-\delta,c_++\delta])$.  The Morse function on $D^4$ induces an immersed Morse function $F$ on $\Delta'$. Apart from the two index zero critical points $p_-$ and $p_-'$,
and one index $1$ critical point $p_+$, $F$ now has two critical points on the double point stratum (recall that any isolated double point is automatically a critical point). There is a single trajectory on $\Delta'$
connecting $p_-$ to $p_+$, and a single trajectory on $\Delta'$ connecting $p_-'$ to $p_+$. But now the trajectory connecting $p_-$
to $p_+$, that used to run outside $\Delta$, has no corresponding trajectory in the complement of $\Delta'$.   See Figure~\ref{fig:newstev1a}.
\begin{figure}
  \includegraphics[width=5cm]{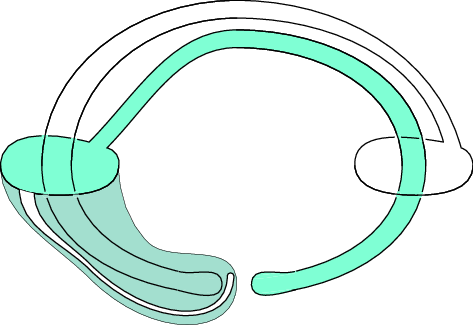}
  \caption{After drilling a neighbourhood of the guiding curve, the membranes do not intersect any more.}\label{fig:newstev1a}
\end{figure}
\begin{comment}
\begin{figure}
  \input{pictures/new_stev1.tex}
  \caption{The disc $\Delta$ in the level set $F^{-1}(c_+-\delta/4)$. The intersections of the ascending membranes $\Ha(p_-)$ and $\Ha(p_-')$ of $p_-$ and $p_-'$ with this level set are shown in red and blue. The dashed arc indicates the intersection of the descending membrane $\Hd(p_+)$ of $p_+$ with this level set.  Green dots show intersections between membranes in this level set. The guiding curve $\gamma$ is also depicted.}
  \label{fig:newstev1}
\end{figure}

\begin{figure}
  \input{pictures/new_stev2.tex}
  \caption{The disc $\Delta'$ in the level set $F^{-1}(c_+-\delta/4)$. After the finger move, there is no intersection between $\Hd(p_+)$ and $\Ha(p_-)$ in $D^4 \sm \Delta'$, so the critical points are in cancelling position.}
  \label{fig:newstev2}
\end{figure}
\end{comment}
%\npar{There is currently a problem with how Figure 50 is displayed. }{Figure 50 was meant to be a part of Figure 49, so it is clipped and rescaled. I corrected slightly sizes, but I'm not sure if it is optimal right now.}

So we are ready to apply Cancellation Theorem~\ref{thm:grimcanc} to cancel the pair of critical points~$p_-$ and~$p_+$.  Note that the disc below the level set $c_{\mathrm{int}}$ is a slice disc for the unknot.
Some care must be taken to show that no new trajectories from $p_-$ to $p_+$ appear as a result of the finger move. It turns out that broken trajectories, that is trajectories
going from $p_-$, through a critical point on a deeper stratum, and ending up in $p_+$, can be excluded by a dimension counting argument. Ordinary
trajectories do not appear since we have just removed them with the finger move.
Thus we can cancel $p_-$ and $p_+$, leaving just a single index $0$ critical point $p_-'$.  The new disc $\Delta'$ therefore determines a link homotopy from the original knot $6_1$ to the unknot.  The above construction is described precisely, and in generality, in carefully chosen coordinates, in Section~\ref{sec:fingermove}.

In this dimension, the finger move introduced a 0-sphere $S^0$ of self intersections. In higher dimensions, the finger move introduces a higher dimensional sphere $S^k$ of new self-intersections.

\section{Codimension one links} \label{sec:cod1}

The Schoenflies theorem holds for $n\neq 3$, proven by elementary  methods for $n<3$ and using the $h$-cobordism theorem for $n>3$~\cite[p.~112]{Mil65}.
More precisely, given a smooth embedding $\varphi \colon S^n \hra \R^{n+1}$, the $h$-cobordism theorem implies that  for $n>3$ the (closure of the) bounded part of its complement is diffeomorphic to $D^{n+1}$. Since embeddings of codimension 0 discs (in any connected manifold) are all isotopic, up to a reflection, it follows that the image $K:= \varphi(S^n)$ is ambiently isotopic to
the unknot $U:= S^n\subseteq \R^{n+1}$, the round $n$-sphere.

We note that precomposing $U$ by a self-diffeomorphism of $S^n$ can change the isotopy class of the embedding $\varphi:S^n\hra \R^{n+1}$ and in fact, Cerf's pseudo-isotopy theorem implies that  this precomposition inducess bijections
\[
\pi_0\Diff(S^n) \cong \pi_0 \Emb(S^n,\R^{n+1})  \quad \text{for }  n>4.
\]
By Kervaire-Milnor~\cite{Kervaire-Milnor:1963-1} the (orientation preserving part of the) left hand mapping class group is isomorphic to the group $\Theta_{n+1}$ of exotic $(n+1)$-spheres, a highly nontrivial (finite) group.

That is why we do not work with {\em parametrized knots} $\varphi$ in this codimension one discussion but with their images, i.e.\ smooth codimension one submanifolds $K=\varphi(S^n)$ (that happen to be diffeomorphic to $S^n$).
In that spirit, we defined a codimension one link $L\subseteq\R^{n+1}$ to be an {\em ordered} sequence $(K_1, \dots, K_r)$ of disjoint codimension one knots.

\begin{definition}\label{def:unparametrized}
Given unparametrized knots $K_0, K_1 \subseteq \R^{n+1}$, we say that a concordance $g \colon S^n \times [0,1] \to \R^{n+1} \times [0,1]$ connects $K_0$ and $K_1$ if $g|_{S^n \times \{i\}}$ is a parametrisation of $K_i$ for $i=0,1$.  Similarly for homotopy between $K_0$ and $K_1$ and unparametrized link concordance respectively link homotopy. Alternatively, we could require maps of manifolds whose domain happens to be diffeomorphic to $S^n \times [0,1]$, but that notion is equivalent.
\end{definition}

We note that two embeddings $\varphi_0, \varphi_1 \colon  S^n \hra \R^{n+1}$ that differ by an {\em orientation preserving} diffeomorphism of $S^n$ are homotopic without changing their image, because the diffeomorphism is homotopic to the identity. So changing the parametrisations of the components of an embedding $\varphi \colon \sqcup^r S^n \hra \R^{n+1} $ in this way preserves its link homotopy and link concordance class.

However, if we precompose $\varphi$ by a reflection on one component then its link concordance class can change, depending on whether that component was innermost or not. In the former case, the component is null homotopic in the complement of the other components and hence homotopic to the precomposition with a reflection. In the latter case, there is at least one other component on the inside and hence the winding number around that component detects the orientation.

\begin{remark}\label{rem:oriented edges}
This discussion leads naturally to the notion of {\em oriented} unparametrized links $L=(K_1,\dots,K_r)$ for which the notion of link concordance is slightly more interesting. Recalling the \emph{dual tree} from Definition~\ref{def:dual}, it turns out that the orientations of the non-innermost submanifolds $K_i$ lead to {\em orientations} of those edges in $t(L)$ that are ``internal'', i.e.\ not adjacent to a leave, and so they correspond to innermost components $K_i$, see the discussion in the next paragraph. Then a link concordance preserves these orientations if and only if the trees are isomorphic as rooted, edge-ordered and internally oriented trees.

Furthermore, it follows that  isomorphism classes of such trees also classify the set of parametrised links $\Emb(\sqcup^r S^n,\R^{n+1})$, modulo link concordance.
\end{remark}

Now we give the proof of Proposition~\ref{prop:cod1}, whose statement we recall here.

\begin{proposition}\label{prop:cod1-body}
For $n\neq 3$, two smooth codimension one links $L, L' \subseteq\R^{n+1}$ are ambiently isotopic if and only if they are link concordant.
Moreover, a generically immersed link concordance induces an isomorphism on dual trees, and if there is an isomorphism $t(L)\cong t(L')$ of rooted, edge-ordered trees, then it is unique and is induced by an ambient isotopy from $L$ to $L'$.
\end{proposition}

\begin{proof}%[Proof of Proposition~\ref{prop:cod1}]
The arguments will all use the following notion of an {\em innermost} component of $L$. Start with the first component $K_1$ and consider its ``inside'' ball $J_1$, the closure of the bounded part of $\R^{n+1} \smallsetminus K_1$. If there is another component $K_j$ inside $K_1$, i.e.\ in $J_1$, then consider its inside ball $J_j$ etc. By finiteness, there must be a component $K_i$ of $L$ that has no other component inside and hence it bounds a ball in $\R^{n+1} \smallsetminus L$ - we call $K_i$ innermost and note that there may be several innermost components in $L$. By induction on the number of components $r$ of $L$, where we only need to add an innermost component, one can easily show the following.
\begin{enumerate}
\item If $L$ has $r$ components then $\R^{n+1} \smallsetminus L$ has $r+1$ connected components, exactly one of which is unbounded.
\item For every component  $K$ of $L$, there are exactly two connected components of $\R^{n+1} \smallsetminus L$ whose closures meet $K$.
\item We have that $t(L)$ is a well-defined tree, with an ordering of the edges induced from the ordering of the components of $L$.
\item If $C$ is a connected component of $\R^{n+1} \smallsetminus L$ representing a vertex $v_C$ in $t(L)$ then the edges adjacent to $v_C$ correspond to the components of $L$ that lie in the closure (or equivalently, in the point set boundary) of $C$.
\item The component $C$ is innermost if and only if $v_C$ is a univalent vertex.
\item If there is an isomorphism $t(L)\cong t(L')$ of rooted edge-ordered trees, it is unique, and it is realised by an ambient isotopy from $L$ to $L'$.
\end{enumerate}
The subtle part of Proposition~\ref{prop:cod1} is to show that link concordance implies ambient isotopy. By Proposition~\ref{prop:immersed} we can approximate the link concordance by an immersed link concordance
\[
G \colon (\sqcup^r S^n) \times [0,1] \imra \R^{n+1} \times [0,1].
\]
from $L$ to $L'$, that we may assume is in general position as in Definition~\ref{def:generic-immersion}

Then we show that $G$ induces an isomorphism $t(L)\cong t(L')$ which is implied by the following result, Lemma~\ref{lem:tree isomorphism}, which proves that the obvious inclusion-induced maps on vertices and edges are an isomorphism of rooted, edge-ordered trees.
\end{proof}

\begin{figure}[h]
  \begin{tikzpicture}
    \draw (0,0) node {\includegraphics[width=6cm]{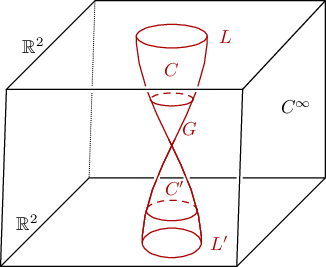}};
    %\draw[red!65!black] (0.3,-1.35) node[scale=1.1] {$C'$};
    %\draw[red!65!black] (0.3,1.8) node[scale=1.1] {$C$};
    %\draw[red!65!black] (1.2,-2) node[scale=1.1] {$L'$};
    %\draw[red!65!black] (1.4,1.6) node[scale=1.1] {$L$};
    %\draw[red!65!black] (-0.3,-0.3) node[scale=1.1] {$C_\infty$};
  \end{tikzpicture}
\caption{An `hourglass' knot concordance $G$ between two copies $L$ and $L'$ of a standard $S^1 \subseteq \R^2$. This $G$ is not generically immersed. The connected components of $\R^2 \times [0,1] \sm G$ are $C$, $C'$, and the unbounded region $C^{\infty}$. With $j \colon \pi_0(\R^2 \sm L) \to \pi_0(\R^2 \times [0,1] \sm G)$ and $j' \colon \pi_0(\R^2 \sm L') \to \pi_0(\R^2 \times [0,1] \sm G)$ the inclusion-induced maps, we have $\{C_\infty, C\}=\Image(j)\neq\Image(j') = \{C_\infty, C'\}$. Although $L$ and $L'$ are isotopic, and $t(L) \cong t(L')$, this isomorphism is not induced by the concordance $G$. Since $G$ is not generically immersed, this is consistent with Proposition~\ref{prop:cod1-body}. }
\label{fig:need immersion}
\end{figure}

\begin{lemma}\label{lem:tree isomorphism}
The two inclusions from the boundary induce injective maps $j, j'$
\[
 \pi_0(\R^{n+1} \smallsetminus L) \xrightarrow{j} \pi_0(\R^{n+1} \times [0,1] \smallsetminus G) \xleftarrow{j'}   \pi_0(\R^{n+1} \smallsetminus L')
\]
with equal images, $\Image(j) = \Image(j')$.
\end{lemma}

\begin{remark}\label{rem:need immersion}
The injectivity of the maps $j,j'$ also holds for an arbitrary (non-immersed) link concordance by Alexander duality, as will become apparent from the proof of Lemma~\ref{lem:tree isomorphism}. However, equality of their images does not hold in general, as illustrated by the example in Figure~\ref{fig:need immersion}.
\end{remark}

\begin{proof}[Proof of Lemma~\ref{lem:tree isomorphism}]
The first step can be proven using Alexander duality, but we use a transversality argument instead.

\emph{Step 1}. Generically, a smooth circle $c \colon S^1 \to  \R^{n+1} \times [0,1]$ intersects a smooth concordance $g \colon S^n \times [0,1]\to \R^{n+1} \times [0,1]$ in an even number of points. To see this, note that the circle bounds a smooth map $D^2 \to  \R^{n+1} \times [0,1]$  which generically intersects $g$ in a compact $1$-manifold with boundary $g \pitchfork c$, so this consists of an even number of points.

\emph{Step 2}. Injectivity of $j \colon \pi_0(\R^{n+1} \smallsetminus L) \hra \pi_0(\R^{n+1} \times [0,1] \smallsetminus G)$. Assume that there are two components $C_1, C_2$ of $\R^{n+1} \smallsetminus L$ that map to the same component, $j(C_1) = j(C_2)$. That means that there is a path $\gamma_{12}$ in $\R^{n+1} \times [0,1] \smallsetminus G$ from a point $p_1 \in C_1$ to $p_2\in C_2$. Consider the geodesic in $t(L)$ from the vertex $v_{C_1}$ to $v_{C_2}$ and assume that $K_i$ is a component of $L$ that corresponds to an edge on that geodesic. This means that there is a path $\gamma$ in $\R^{n+1}$ from $p_1$ to $p_2$ that intersects $K_i$ generically an odd number of times. So $\gamma_G \cup \gamma$ is a circle in $\R^{n+1} \times [0,1]$ that intersects the component $G_i$ of $G$ generically an odd number of times, a contradiction to the observation in Step~1.

\emph{Step 3}. Setting up an induction on $r$, the number of components of $L$ to show the equality $\Image(j)=\Image(j')$. Let $K_i$ be an innermost component of $L$ and let $\widehat L:= L \smallsetminus K_i$.
Then $\widehat G:= G \smallsetminus G_i$ is an immersed concordance from $\widehat L$ to $\widehat L' := L' \smallsetminus K_i'$ and so the statement holds by induction. The base case that $G$ is empty is easily seen to hold. To show it for $G$, it suffices to prove that there is a path $\gamma$ in $\R^{n+1} \times [0,1] \smallsetminus G$ that starts on the inside of  $K'_i$ and ends on the inside of $K_i$.
This is one of the places  where we use that $G$ is a generic immersion. Namely, we show that such a path  $\gamma$ exists arbitrarily close to $G_i$, and hence $\gamma$ automatically can be assumed to miss the other components of $G$. So we are left with showing the inductive step purely in terms of the innermost component we removed.

Step 4. To prove the inductive step, we need to show that for a generically immersed concordance $g$ between knots $K$ to $K'$, there is a path $\gamma$ in $\R^{n+1} \times [0,1] \smallsetminus g$ that starts on the inside of $K'$, ends on the inside of  $K$, and stays arbitrarily close to $g$.

To prove this statement, consider the normal bundle $\nu g$ of $g \colon S^n \times [0,1] \imra \R^{n+1} \times [0,1]$, a 1-dimensional vector bundle on $S^n \times [0,1]$. This is a trivial bundle since its restriction to $S^n \times \{0\}$ can be trivialized by the normal vector that points into the inside of $K$. As a consequence, we can extend that section to a non-vanishing section $\sigma$ of $\nu g$. Using the exponential map locally, where $g$ is an embedding, we get a map $s \colon S^n \times [0,1] \to \R^{n+1} \times [0,1]$ that we think of as a `parallel' of $g$.  Here we used again that $g$ is a generic immersion; this step is not possible with the `hourglass' concordance in Figure~\ref{fig:need immersion}.

If $g$ were an embedding then $s$ would also be an embedding (and actually be an honest parallel), lying in the complement of $g$. However, the double point manifold of $g$ has codimension one in $S^n \times [0,1]$, the triple point manifold has codimension~2, etc. Therefore, we can at least pick a smooth arc $\alpha\colon [0,1] \to S^n \times [0,1]$ from $(x,0)$ to $(x,1), x\in S^n$, that misses the triple, quadruple and higher singular points of $g$ and intersects the double point set generically  in finitely many points $\{t_1,\dots, t_m\}\subseteq [0,1]$.

By taking $s \circ \alpha \colon [0,1] \to \R^{n+1} \times [0,1]$, we obtain an embedded  arc $\beta$ in $\R^{n+1} \times [0,1]$ that starts on the inside of  $K'$, ends on the inside of  $K$, and meets $g$ transversely in $2m$ points. One pair of these intersections comes from each point $\alpha(t_i)$ lying in the double point set of $g$. So $\beta$ follows the normal vector field along $\alpha$. At each double point of $\alpha$, we have the local picture in Figure~\ref{fig:gamma}. The easiest way to describe it is to start with $\alpha$ given locally in $\R^2$ by the coordinate axes, intersecting in the origin $(0,0)$. Observe that $\R^2 \smallsetminus \alpha$ has four connected components, namely the four (open) quadrants.
By transversality, $g$ is given locally by the product of this $\R^2$ with $\R^n$, with $\{0\} \times \R^n$ becoming the local double point manifold of $g$.
Note that $g$ itself may twist around $\alpha$ from one double point of $\alpha$ to the next and that $\beta$ just twists along as its normal vector field. So we have no control over how $\beta$ moves between these double points, it just continues to stay normal to $g$ and glues together to a smoothly embedded arc (since we are assuming $n>0$).

\begin{figure}[h]
  \begin{tikzpicture}
    \draw (0,0) node {\includegraphics[width=8cm]{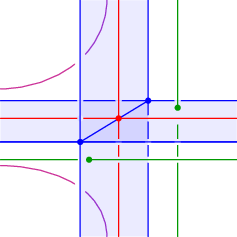}};
    \draw[red] (0.2,3.8) node {$\alpha$};
    \draw[red] (-3.8,0.2) node {$\alpha$};
    \draw[green!60!black] (2.2,3.8) node {$\beta$};
    \draw[green!60!black] (3.8,-1.6) node {$\beta$};
    \draw[blue] (-1.6,3.8) node {$g$};
    \draw[blue] (3.8,1.0) node {$g$};
    \draw[purple] (-3.8,1.2) node {$\gamma$};
    \draw[purple] (-3.8,-2.0) node {$\gamma$};
\end{tikzpicture}
\caption{Constructing the paths $\alpha$, its parallel $\beta$, and ultimately $\gamma$.}
\label{fig:gamma}
\end{figure}

From the local description we see, at each double point of $\alpha$, a unique turn that we can introduce just before $\beta$ hits $g$. We can keep the new local arc $\gamma$ normal to both sheets of $g$ and in addition disjoint from $g$ as in Figure~\ref{fig:gamma}. This comes from the fact that in a coordinate ball around the double point, the complement of $g$ has exactly four connected components, aforementioned the quadrants in $\R^2$ times $\R^n$. The required global path $\gamma$ in the complement of $g$ is then constructed by the following algorithm.
\begin{enumerate}
\item Start with $\gamma(0):=\beta(0)$ lying on the inside of $K'$ and let $\gamma(t):=\beta(t)$ for all $t<t_1$, the first double point of $\alpha$.
\item At that double point, let $\gamma$ make the unique smooth turn that keeps it disjoint from $g$ and then continue normally to $g$ along $\alpha$. This continuation may be either following $\beta$, or may be on the other side, i.e.\ following the negative $-\sigma$ of the section $\sigma$ used to define $s$ and $\beta$, applied to $\alpha$. We denote this curve by $-\beta$. The submanifold $\gamma$ may also have a different local orientation to $\pm \beta$.
\item This process of building an embedding $\gamma$ of a compact connected 1-manifold continues uniquely, up to isotopy, turning at all double points $\{t_1,\dots, t_m\}$ of $\alpha$ in the prescribed way, and staying normal to $g$ along $\alpha$ otherwise.
\item The resulting 1-manifold can be parametrised to become an embedding of $[0,1]$  with a second boundary point  $\gamma(1)$. By construction, $\gamma(1)$ is equal to either $\sigma$ or $-\sigma$ applied to a boundary point of $\alpha$. So $\gamma(1)$  equals $\sigma(\alpha(0)) = \beta(0)$, $-\sigma(\alpha(0)) = -\beta(0)$, $\sigma(\alpha(1)) = \beta(1)$, or $-\sigma(\alpha(1)) = -\beta(1)$.  However $\gamma(1)$ cannot be equal to $\gamma(0) = \beta(0)$ since the manifold $\gamma$ has two boundary points.
\end{enumerate}
The last step is to argue that $\gamma(1)= \pm \beta(1)$ and lies on the inside of  $K$, because the other two possibilities contradict Step 1 above. If $\gamma(1)= -\beta(0)$ still lies in $\R^{n+1} \smallsetminus K'$, then we can glue $\gamma$ with the short path $\delta \colon [0,1] \to \R^{n+1}$ sending $t\mapsto (\smfrac{1}{2}(t+1)\cdot \beta(0)$ through $K'$ to get a circle that intersects $g$ once, which is a contradiction. Similarly, if $\gamma(1)$ lies in $\R^{n+1} \smallsetminus K$, but outside of $K$, then we can connect $\gamma(1)$  by a path in the unbounded component, far away from $g$, to $-\beta(0)$. Using $\delta$ again we obtain a circle that intersects $g$ once, which is another contradiction. Hence $\gamma(1)$ is equal to whichever of $\pm \beta(1)$ lies inside $K$. This completes the proof of Step 3, and hence completes the proof of Lemma~\ref{lem:tree isomorphism}.
\end{proof}

\section{Regular homotopies of surfaces}\label{section:regular-homotopies-surfaces}

  A regular homotopy between immersed compact surfaces $G_i \colon \Sigma \looparrowright M$, $i=0,1$, in a closed 4-manifold $M$, is a smooth homotopy $G \colon \Sigma \times I \to M$ through immersions, with $G_t \colon \Sigma \looparrowright M$ an immersion, and with $G_t^{-1}(\bd M) = \bd \Sigma$ for all $t \in I$.  We consider the trace $G' \colon \Sigma \times I \to M \times I$ that sends $G'(s,t) = (G(s,t),t)$. This is a level-preserving generic immersion.  The following result is often used in the 4-manifolds literature.

\begin{theorem}\label{thm:reg-homotopies-surfaces}
  Let $G \colon \Sigma \times I \looparrowright M \times I$ be a generic immersion between generically immersed surfaces, e.g.\ the trace of a regular homotopy between $G_0(\Sigma)$ and $G_1(\Sigma)$.  Then $G$ is regularly homotopic to a concatenation of finger moves, Whitney moves, and ambient isotopies leading from $G_0$ to $G_1$.  Moreover, we may assume that all the finger moves occur before all the Whitney moves.
\end{theorem}

\begin{proof}
  By general position (Lemma~\ref{lem:generic-immersion}), we may and shall assume that $G$ induces a stratification on $M \times I$ with:
  \begin{itemize}
    \item  $\O[d] = \emptyset$ for $d \geq 3$,
    \item  $\O[2]$ the double point set, a collection of disjointly embedded 1-manifolds,
    \item   $\O[1]$ the complement in $G(\Sigma \times I)$ of the double point set, and
    \item $\O[0]$ the complement $M \times I \sm G(\Sigma \times I)$.
  \end{itemize}
  By Lemma~\ref{lem:cR_function}, we may assume after a perturbation of $G$ that the projection $F \colon M \times I \to I$ is an immersed Morse function with respect to $G(\Sigma \times I)$.
  %\npar{What actually happens is the Morse function is perturbed. Need to explain how to reverse this and move the immersion instead. See remark below.}{No, what happens is that we actually perturb $G$. Changed the reference.}
By Theorem~\ref{thm:concordance}, $G$ is smoothly regularly homotopic rel.\ boundary to the trace of a regular homotopy.

Since $G$ is level preserving, the only critical points occur on $\O[2]$.  Since $\O[2]$ is a 1-manifold, the critical points are maxima and minima, i.e.\ index 0 and index 1 critical points of the Morse function restricted to the stratum $\O[2]$.  Minima correspond to finger moves of $G_t(\Sigma)$ and maxima correspond to Whitney moves.  By the Rearrangement Theorem~\ref{thm:grim_global_rearrangement}, we may rearrange the critical points so that maxima are above minima.
  Consider a single pair of a Whitney move and a finger move, where the Whitney move occurs first.
  To see that the rearrangement theorem applies, note that the descending membrane of a finger move is 3-dimensional, and thus 2-dimensional in each level set, while the ascending membrane of a Whitney move is 2-dimensional, and thus 1-dimensional in each level set. The level sets are 4-dimensional, so by the Morse-Smale condition the membranes may be assumed disjoint.
  More precisely, Theorem~\ref{thm:grim_global_rearrangement} alters the Morse function by a 1-parameter family to achieve the rearrangement, with a corresponding 1-parameter family of grim vector fields.  It does so without introducing any new critical points. We perform an ambient isotopy that returns the Morse function to the original projection; see Lemma~\ref{lem:lift_morse}.
  %\npar{Again, is there something we can refer to earlier that will make the change entirely using a motion of the homotopy and not a change in the Morse function? I think we do something similar using stability in the proof of singular concordance implies homotopy? }{Referred to Lemma~\ref{lem:lift_morse}.}
  The rearrangement switches the order of our given pair of a Whitney move and finger move.  By applying this as many times as necessary, it follows that by a regular homotopy $G$ can be made into the trace of a regular homotopy, and all the finger moves i.e.\ minima on $\O[2]$, can be placed before the Whitney moves i.e.\ the maxima on $\O[2]$.
\end{proof}

\bibliographystyle{alpha}
\def\MR#1{}
\bibliography{research}

\end{document}